%% file: p_greater_one.tex
\documentclass[aos,preprint]{imsart}

\RequirePackage[OT1]{fontenc}
\RequirePackage{amsthm,amsmath,amsfonts,hyperref}
\RequirePackage[numbers]{natbib}
\RequirePackage[colorlinks,citecolor=blue,urlcolor=blue]{hyperref}
\usepackage{amssymb}
\usepackage[dvips]{graphics}
\usepackage{graphicx}
\usepackage[retainorgcmds]{IEEEtrantools}
\usepackage{algorithm}
\usepackage{algorithmic}
\usepackage{latexsym}
\usepackage{bbm}
\usepackage[caption=false,font=footnotesize]{subfig}

\makeatother

\DeclareMathOperator*{\argmin}{arg\,min}

\startlocaldefs
\numberwithin{equation}{section}
\theoremstyle{plain}

\newtheorem{definition}{Definition}

\newtheorem{theorem}{Theorem}[section]
\newtheorem{lemma}{Lemma}

\newtheorem{corollary}{Corollary}

\newtheorem{proposition}{Proposition}

\endlocaldefs

\begin{document}

\begin{frontmatter}
\title{Overcoming The Limitations of Phase Transition by Higher Order Analysis of Regularization Techniques}
\runtitle{Higher Order Analysis of Regularization Techniques}

\begin{aug}
\author{\fnms{Haolei} \snm{Weng}\thanksref{m1}\ead[label=e1]{hw2375@columbia.edu}},
\author{\fnms{Arian} \snm{Maleki}\thanksref{m2}\ead[label=e2]{arian@stat.columbia.edu}}
\and
\author{\fnms{Le} \snm{Zheng}\thanksref{m3}
\ead[label=e3]{le.zheng.cn@gmail.com}
\ead[label=u1,url]{http://www.foo.com}}

\runauthor{Haolei Weng, Arian Maleki and Le Zheng}

\affiliation{Columbia University\thanksmark{m1}\thanksmark{m2}\thanksmark{m3}}

\address{H. Weng\\
Department of Statistics \\
Columbia University \\
1255 Amsterdam Avenue \\
New York, NY, 10027 \\
USA \\
\printead{e1}\\
}

\address{A. Maleki \\
Department of Statistics \\
Columbia University \\
1255 Amsterdam Avenue \\
New York, NY, 10027 \\
USA \\
\printead{e2}\\
}

\address{L. Zheng \\
Department of Electrical Engineering \\
Columbia University \\
1300 S. W. Mudd Building, MC 4712 \\
500 W. 120th Street \\
New York, NY 10027 \\
USA \\
\printead{e3}\\
}

\end{aug}

\begin{abstract}
We study the problem of estimating a sparse vector $\beta \in \mathbb{R}^p$ from the response variables $y= X\beta+ w$, where $w \sim N(0, \sigma_w^2 I_{n\times n})$, under the following high-dimensional asymptotic regime: given a fixed number $\delta$, $p \rightarrow \infty$, while $n/p \rightarrow \delta$. We consider the popular class of $\ell_q$-regularized least squares (LQLS), a.k.a. bridge estimators, given by the optimization problem:
\begin{equation*}
\hat{\beta} (\lambda, q ) \in \arg\min_\beta \frac{1}{2} \|y-X\beta\|_2^2+ \lambda \|\beta\|_q^q,
\end{equation*}
and characterize the almost sure limit of $\frac{1}{p} \|\hat{\beta} (\lambda, q )- \beta\|_2^2$, and call it asymptotic mean square error (AMSE).  The expression we derive for this limit does not have explicit forms and hence is not useful in comparing LQLS for different values of $q$, or providing information in evaluating the effect of $\delta$ or sparsity level of $\beta$. To simplify the expression, researchers have considered the ideal ``error-free'' regime, i.e. $w=0$,  and have characterized the values of $\delta$ for which AMSE is zero. This is known as the phase transition analysis. 
 
In this paper, we first perform the phase transition analysis of LQLS. Our results reveal some of the limitations and misleading features of the phase transition analysis. To overcome these limitations, we propose the small error analysis of LQLS. Our new analysis framework not only sheds light on the results of the phase transition analysis, but also describes when phase transition analysis is reliable, and presents a more accurate comparison among different regularizers. 
\end{abstract}

\begin{keyword}[class=MSC]
\kwd[Primary ]{60K35}
\kwd{60K35}
\kwd[; secondary ]{60K35}
\end{keyword}

\begin{keyword}
\kwd{Bridge regression}
\kwd{phase transition}
\kwd{comparison of estimators}
\kwd{small error regime}
\kwd{second order term}
\kwd{asymptotic mean square error}
\kwd{optimal tuning}
\end{keyword}

\end{frontmatter}

\section{Introduction}\label{sec:intro}

\subsection{Objective}
Consider the linear regression problem where the goal is to estimate the parameter vector $\beta \in\mathbb{R}^p$ from a set of $n$ response variables $y \in \mathbb{R}^n$, under the model $y= X\beta+ w$. This problem has been studied extensively in the last two centuries since Gauss and Legendre developed the least squares estimate of $\beta$. The instability or high variance of the least squares estimates led to the development of the regularized least squares. One of the most popular regularization classes is the $\ell_q$-regularized least squares (LQLS), a.k.a. bridge regression \cite{frank1993statistical, fu1998penalized}, given by the following optimization problem:
\begin{equation}\label{eq:lqls}
\hat{\beta} (\lambda, q ) \in \arg\min_\beta \frac{1}{2} \|y-X\beta\|_2^2+ \lambda \|\beta\|_q^q,
\end{equation}
where $\| \beta\|_q^q = \sum_{i=1}^p |\beta_i|^q$ and $1 \leq q \leq 2$.\footnote{Bridge regression is a name used for LQLS with any $q \geq 0$. In this paper we focus on $1\leq q\leq 2$. To analyze the case $0\leq q< 1$, \cite{zheng2015does} has used the replica method from statistical physics.}  LQLS has been extensively studied in the literature. In particular, one can prove the consistency of $\hat{\beta} (\lambda, q )$ under the classical asymptotic analysis ($p$ fixed while $n \rightarrow \infty$) \cite{knight2000asymptotics}. However, this asymptotic regime becomes irrelevant for high-dimensional problems in which $n$ is \textit{not} much larger than $p$. Under this high dimensional setting, if $\beta$ does not have any specific ``structure'', we do not expect any estimator to perform well. One of the structures that has attracted attention in the last twenty years is the sparsity, that assumes only $k$ of the elements of $\beta$ are non-zero and the rest are zero. 
To understand the behavior of the estimators under structured linear model in high dimension, a new asymptotic framework has been proposed in which it is assumed that $X_{ij} \overset{iid}{\sim} N(0, 1/n)$, $k,n, p \rightarrow \infty$, while $n/p \rightarrow \delta$ and $k/p \rightarrow \epsilon$, where $\delta$ and $\epsilon$ are fixed numbers \cite{donoho2005sparse, donoho2009message, amelunxen2014living, el2013robust, bradic2015robustness}. 

One of the main notions that has been widely studied in this asymptotic framework, is the phase transition \cite{donoho2005sparse, donoho2009message, amelunxen2014living, stojnic2009various}. Intuitively speaking, phase transition analysis assumes the error $w$ equals zero and characterizes the value of $\delta$ above which an estimator converges to the true $\beta$ (in certain sense that will be clarified later). While there is always error in the response variables, it is believed that phase transition analysis provides reliable information when the errors are small. 
In this paper, we start by studying the phase transition diagrams of LQLS for $1 \leq q \leq 2$. Our analysis reveals several limitations of the phase transition analysis. We will clarify these limitations in the next section. We then propose a higher-order analysis of LQLS in the small-error regime. As we will explain in the next section, our new framework sheds light on the peculiar behavior of the phase transition diagrams, and explains when we can rely on the results of phase transition analysis in practice.

\subsection{Limitations of the phase transition and our solution} \label{limitations}
In this section, we intuitively describe the results of phase transition analysis, its limitations, and our new framework. 
Consider the class of LQLS  estimators and suppose that we would like to compare the performance of these estimators through the phase transition diagrams. For the purpose of this section, we assume that the vector $\beta$ has only $k$ non-zero elements, where $k/p \rightarrow \epsilon$ with $\epsilon \in (0,1)$. Since phase transition analysis is concerned with $w=0$ setting, it considers $\lim_{\lambda \rightarrow 0} \hat{\beta} (\lambda, q )$ which is equivalent to the following estimator:
\begin{eqnarray}\label{eq:ellqmin}
&&\arg \min_\beta \| \beta\|_q^q,  \nonumber \\
 && {\rm subject}   \ {\rm to \ } y= X\beta. 
\end{eqnarray}
Below we informally state the results of the phase transition analysis. We will formalize the statement and describe in details the conditions under which this result holds in Section \ref{sec:our}.\\

\noindent \textbf{Informal Result 1.}   For a given $\epsilon>0$ and $q \in [1,2]$, there exists a number $M_q(\epsilon)$ such that as $p \rightarrow \infty$, if $\delta \geq M_q(\epsilon)+ \gamma$ ($\gamma>0$ is an arbitrary number), then  \eqref{eq:ellqmin} succeeds in recovering $\beta$, while if $\delta \leq M_q(\epsilon)- \gamma$, \eqref{eq:ellqmin}  fails.\footnote{Different notions of success have been studied in the phase transition analysis. We will mention one notion later in our paper. } 

\vspace{.3cm}

\noindent The curve $\delta = M_q(\epsilon)$ is called the phase transition curve of \eqref{eq:ellqmin}. We will show that $M_q(\epsilon)$ is given by the following formula:
\begin{eqnarray}
M_q(\epsilon) = \left\{ \begin{array}{l}
1 \ \ \mbox{if} \ 2\geq q>1,\\
\inf_{\chi\geq 0} (1-\epsilon) \mathbb{E} \eta_1^2(Z; \chi) + \epsilon (1+ \chi^2) \ \ \mbox{if} \ q=1,
\end{array} \right.
\end{eqnarray}
where $\eta_1(u;\chi) = (|u|- \chi)_+ {\rm sign}(u)$ denotes the soft thresholding function and $Z\sim N(0,1)$. While the above phase transition curves can be obtained with different techniques, such as statistical dimension framework proposed in \cite{amelunxen2014living} and Gordon's lemma applied in \cite{oymak2013squared, thrampoulidis2016precise}, we will derive them as a simple byproduct of our main results in Section \ref{sec:our} under message passing framework. Also, the phase transition analysis of the regularized least squares has already been performed in the literature \cite{donoho2005neighborliness, donoho2009message, oymak2016sharp}. Hence, we should emphasize that the presentation of the phase transition results for bridge regression is not our main contribution here. We rather use it to motivate our second-order analysis of the asymptotic risk, which will appear later in this section. Before we proceed further let us mention some of the properties of $M_1(\epsilon)$ that will be useful in our later discussions.

\begin{lemma}\label{lem:M1epsprop}
$M_1(\epsilon)$ satisfies the following properties:
\begin{enumerate}
\item[(i)] $M_1(\epsilon)$ is an increasing function of $\epsilon$.
\item[(ii)] $\lim_{\epsilon \rightarrow 0} M_1(\epsilon) =0$.
\item[(iii)] $\lim_{\epsilon \rightarrow 1} M_1(\epsilon) =1$.
\item[(iv)] $M_1(\epsilon)>  \epsilon$, for $\epsilon \in (0,1)$.
\end{enumerate}
\end{lemma}

\begin{proof}

Define $F(\chi, \epsilon) \triangleq (1-\epsilon) \mathbb{E} \eta^2_1(Z; \chi) + \epsilon (1+ \chi^2)$. It is straightforward to verify that $F(\chi, \epsilon)$, as a function of $\chi$ over $[0,\infty)$, is strongly convex and has a unique minimizer. Let $\chi^*(\epsilon)$ be the minimizer. We write it as $\chi^*(\epsilon)$ to emphasize its dependence on $\epsilon$. By employing the chain rule we have
\begin{eqnarray*}
\frac{dM_1(\epsilon)}{ d \epsilon} &=& \frac{\partial F(\chi^*(\epsilon), \epsilon)}{ \partial  \epsilon}  + \frac{\partial F(\chi^*(\epsilon), \epsilon)}{ \partial  \chi} \cdot \frac{d\chi^*(\epsilon)}{ d \epsilon} =   \frac{\partial F(\chi^*(\epsilon), \epsilon)}{ \partial  \epsilon} \nonumber \\
&=& 1+ (\chi^*(\epsilon))^2 - \mathbb{E}\eta^2_1(Z; \chi^*(\epsilon))> 1+ (\chi^*(\epsilon))^2 - \mathbb{E} |Z|^2 \\
&=&(\chi^*(\epsilon))^2>0,
\end{eqnarray*}
which completes the proof of part (i). To prove (ii) note that
\begin{eqnarray}
0 &\leq& \lim_{\epsilon \rightarrow 0} \min_{\chi\geq 0}  (1- \epsilon) \mathbb{E}\eta^2_1(Z; \chi)+ \epsilon (1+ \chi^2) \nonumber \\
& \leq&  \lim_{\epsilon \rightarrow 0} (1- \epsilon) \mathbb{E} \eta^2_1(Z; \log (1/ \epsilon))+ \epsilon (1+ \log^2(1/\epsilon)) \nonumber \\
&=&  \lim_{\epsilon \rightarrow 0} 2(1- \epsilon) \int_{ \log(1/\epsilon)}^{\infty} (z- \log(1/\epsilon))^2 \phi(z) dz \nonumber \\
&=& \lim_{\epsilon \rightarrow 0} 2(1- \epsilon) \int_{0}^{\infty} z^2 \phi(z + \log(1/\epsilon)) dz \nonumber \\
&\leq& \lim_{\epsilon \rightarrow 0} 2(1- \epsilon) {\rm e}^{-\frac{\log^2(1/\epsilon)}{2}}  \int_{0}^{\infty} z^2 \phi(z) dz =0, \nonumber
\end{eqnarray}
where $\phi(\cdot)$ is the density function of standard normal. Regarding the proof of part (iii), first note that as $\epsilon \rightarrow 1$, $\chi^*(\epsilon) \rightarrow 0$. Otherwise suppose $\chi^*(\epsilon) \rightarrow \chi_0>0$ (taking a convergent subsequence if necessary). Since $\mathbb{E}\eta^2_1(Z;\chi^*(\epsilon)) \leq \mathbb{E}|Z|^2=1$, we obtain 
\[
  \lim_{\epsilon \rightarrow 1} F(\chi^*(\epsilon),\epsilon)= 1+ \chi^2_0 >1.
\]
On the other hand, it is clear that 
\[
\lim_{\epsilon \rightarrow 1}F(\chi^*(\epsilon), \epsilon)\leq \lim_{\epsilon \rightarrow 1}F(0,\epsilon)=1. 
\]
A contradiction arises. Hence the fact $\chi^*(\epsilon) \rightarrow 0$ as $\epsilon \rightarrow 1$ leads directly to 
\[
\lim_{\epsilon \rightarrow 1} M_1(\epsilon ) =\lim_{\epsilon \rightarrow 1} F(\chi^*(\epsilon),\epsilon)=1.
\]
Part (iv) is clear from the definition of $M_1(\epsilon)$. 
\end{proof}

Figure \ref{fig:phasetrans1} shows $M_q(\epsilon)$ for different values of $q$. We observe several peculiar features: (i) As is clear from both Lemma \ref{lem:M1epsprop} and Figure \ref{fig:phasetrans1}, $q=1$ requires much fewer observations than all the other values of $q>1$ for successful recovery of $\beta$. (ii) The values of the non-zero elements of $\beta$ do not have any effect on the phase transition curves. In fact, even the sparsity level does not have any effect on the phase transition for $q>1$. (iii) For every $q>1$, the phase transition of \eqref{eq:ellqmin} happens at exactly the same value. 

\begin{figure}
\begin{center}
\includegraphics[width=7cm]{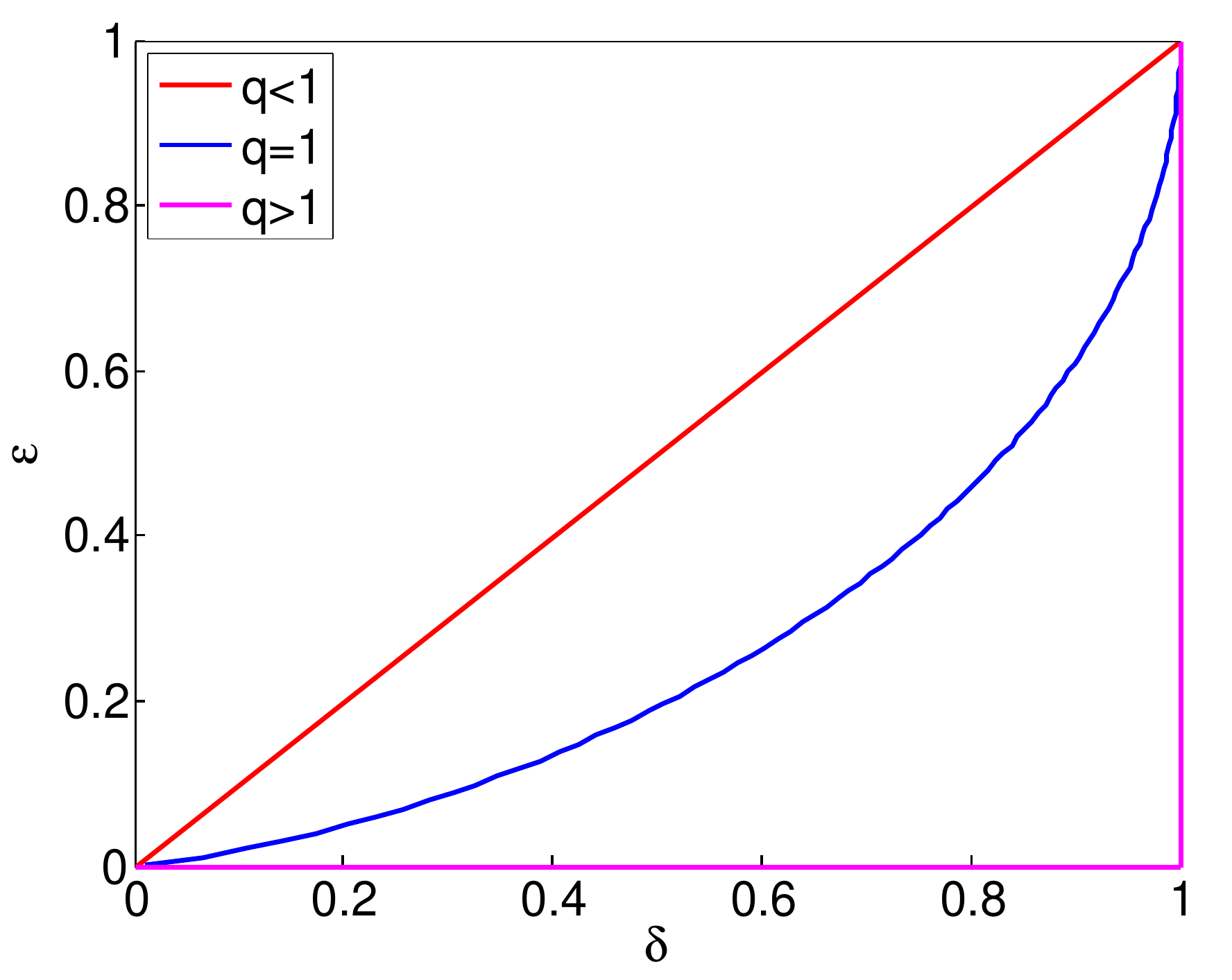}
\caption{Phase transition curves of LQLS for (i) $q<1$: these results are derived in \cite{zheng2015does} by the non-rigorous replica method from statistical physics. We have just included them for comparison purposes. In this paper we have focused on $q\geq 1$. (ii)  $q=1$: the blue curve exhibits the phase transition of LASSO. Below this curve LASSO can ``successfully'' recover $\beta$. (iii) $q>1$: The magenta curve represents the phase transition of LQLS for any $q>1$. This figure is based on Informal Result 1 and will be carefully defined and derived in Section \ref{sec:our}.  }
\label{fig:phasetrans1}
\end{center}
\end{figure}

These features raise the following question: how much and to what extent are these phase transition results useful in applications, where at least small amount of error is present in the response variables? For instance, intuitively speaking, we do not expect to see much difference between the performance of LQLS for $q=1.01$ and $q=1$. However, according to the phase transition analysis, $q=1$ outperforms $q=1.01$ by a wide margin. In fact the performance of LQLS for $q=1.01$ seems to be closer to that of $q=2$ than $q=1$. Also, in contrast to the phase transition implication, we may not expect LQLS to perform the same for $\beta$ with different values of non-zero elements. The main goal of this paper is to present a new analysis that will shed light on the misleading features of the phase transition analysis. It will also clarify when and under what conditions the phase transition analysis is reliable for practical guidance. 

 In our new framework, the variance $\sigma_w^2$ of the error $w$ is assumed to be small. We consider \eqref{eq:lqls} with the optimal value of $\lambda$ for which the asymptotic mean square error, i.e., $\lim_{p \rightarrow \infty} \frac{\|\hat{\beta}(\lambda, q) - \beta\|_2^2}{p}$, is minimized. We first obtain the formula for the asymptotic mean square error (AMSE) characterized through a series of non-linear equations. Since $\sigma_w$ is assumed small, we then derive the asymptotic expansions for AMSE as $\sigma_w \rightarrow 0$. As we will describe later, the phase transition of LQLS for different values of $q$ can be obtained from the first dominant term in the expansion. More importantly, we will show that the second dominant term is capable of evaluating the importance of the phase transition analysis for practical situations and also provides a much more accurate analysis of different bridge estimators. Here is one of the results of our paper, presented informally to clarify our claims. All the technical conditions will be determined in Sections \ref{sec:def} and \ref{sec:our}. \\

\noindent \textbf{Informal Result 2.}
If $\lambda_*$ denotes the optimal value of $\lambda$, then for any $q \in (1,2)$, $\delta>1$, and $\epsilon<1$
\begin{eqnarray*}
 \lim_{p\rightarrow \infty} \frac{1}{p}\|\hat{\beta}(\lambda_*,q)-\beta\|_2^2 = \frac{\sigma^2_w}{1-1/\delta} - \sigma^{2q}_w\frac{\delta^{q+1}(1- \epsilon)^2(\mathbb{E} |Z|^q)^2}{(\delta-1)^{q+1}\epsilon  \mathbb{E} |G|^{2q-2}} + o(\sigma_w^{2q}), \hspace{-1.5cm}
\end{eqnarray*}
where $Z \sim N(0,1)$ and $G$ is a random variable whose distribution is specified by the non-zero elements of $\beta$. We will clarify this in the next section. Finally, the limit notation we have used above is the almost sure limit.  

\vspace{.3cm}

As we will discuss in Section \ref{sec:our}, the first term $\frac{\sigma^2_w}{1-1/\delta}$ determines the phase transition. Moreover, we have further derived the second dominant term in the expansion of the asymptotic mean square error. This term enables us to clarify some of the confusing features of the phase transitions. Here are  some important features of this term: (i) It is negative. Hence, the AMSE that is predicted by the first term (and phase transition analysis) is overestimated specially when $q$ is close to $1$. (ii) Fixing $q$, the magnitude of the second dominant term grows as $\epsilon$ decreases. Hence, for small values of $\sigma_w$ all values of $1<q<2$ benefit from the sparsity of $\beta$.  Also, smaller values of $q$ seem to benefit more. (iii) Fixing $\epsilon$ and $\delta$, the power of $\sigma_w$ decreases as $q$ decreases. This makes the absolute value of the second dominant term bigger. As $q$ decreases to one, the order of the second dominant term gets closer to that of the first dominant term and thus the predictions of phase transition analysis become less accurate. We will present a more detailed discussion of the second order term in Section \ref{sec:our}. To show some more interesting features of our approach, we also informally state a result we prove for LASSO. \\

\noindent \textbf{Informal Result 3.}  
Suppose that the non-zero elements of $\beta$ are all larger than a fixed number $\mu$. If $\lambda_*$ denotes the value of $\lambda$ that leads to the smallest AMSE, and if $\delta > M_1(\epsilon)$, then
\begin{eqnarray}\label{asymp:l1}
\lim_{p\rightarrow \infty}{\frac{1}{p}\|\hat{\beta}(\lambda_*,1)-\beta\|_2^2}-\frac{\delta M_1(\epsilon)\sigma^2_w}{\delta-M_1(\epsilon)}  =O(\exp(-\tilde{\mu}/\sigma_w^2)),
\end{eqnarray}
 where $\tilde{\mu}$ is a constant that depends on $\mu$. 
 
 \vspace{.3cm}
 
 As can be seen here, compared to the other values of $q$, $q=1$ has smaller first order term (according to Lemma \ref{lem:M1epsprop}), but much smaller (in magnitude) second order term. The first implication of this result is that the first dominant term provides an accurate approximation of AMSE. Hence, phase transition analysis in this case is reliable even if small amount of noise is present; that is one of the main reasons why the theoretically derived phase transition curve matches the empirical one for LASSO. Furthermore, note that in order to obtain Informal Result 3, we have made certain assumption about the non-zero components of $\beta$. As will be shown in Section \ref{sec:our}, any violation of this assumption has major impact on the second dominant term.  
  
 In the rest of the paper, we first state all the assumptions required for our analysis. We then present the formal statements of aforementioned and related results and provide a more comprehensive discussion.

\vspace{.1cm}

\subsection{Organization of the paper}
The rest of the paper is organized as follows: Section \ref{sec:def} presents the asymptotic framework of our analysis. Section \ref{sec:our} discusses the main contributions of our paper. Section \ref{sec:relatedwork} compares our results with the related work. Section \ref{sec:simanddis} shows some simulation results and discusses some open problems that require further research.  Section \ref{sec:proofthm4full} and Supplementary material are devoted to the proofs of all the results. 

\section{The asymptotic framework}\label{sec:def}
The main goal of this section is to formally introduce the asymptotic setting under which we study LQLS. In the current and next sections only, we may write vectors and matrices as $\beta(p), X(p), w(p)$ to emphasize the dependence on the dimension of $\beta$. Similarly, we may use $\hat{\beta}(\lambda,q,p)$ as a substitute for $\hat{\beta}(\lambda, q)$. Note that since we assume $n/p \rightarrow \delta$, we do not include $n$ in our notations. Now we define a specific type of a sequence known as a converging sequence. Our definition is borrowed from other papers \cite{DoMaMoNSPT, BaMo10, BaMo11} with some minor modifications. 
 
\begin{definition}
A sequence of instances $\{\beta(p), X(p), w(p)\}$ is called a converging sequence if the following conditions hold:
\begin{itemize}
\item[-] The empirical distribution\footnote{It is the distribution that puts a point mass $1/p$ at each of the $p$ elements of the vector.} of $\beta(p) \in \mathbb{R}^p$ converges weakly to a probability measure $p_{\beta}$ with bounded second moment. Further, $\frac{1}{p} \|\beta(p)\|_2^2$ converges to the second moment of $p_{\beta}$.
\item[-] The empirical distribution of $w(p) \in \mathbb{R}^n$ converges weakly to a zero mean distribution with variance $\sigma_w^2$. And, $\frac{1}{n} \|w(p)\|_2^2 \rightarrow \sigma_w^2$. 
\item[-] The elements of $X(p)$ are iid with distribution $N(0, 1/n)$. 
\end{itemize}
\end{definition}
For each of the problem instances in a converging sequence, we solve the LQLS problem \eqref{eq:lqls} and obtain $\hat{\beta}( \lambda,q,p)$ as the estimator. The interest is to evaluate the accuracy of this estimator. Below we define the asymptotic mean square error.  

\begin{definition}
Let $\hat{\beta}(\lambda,q,p)$ be the sequence of solutions of LQLS for the converging sequence of instances $\{\beta(p), X(p), w(p)\}$. The asymptotic mean square error is defined as the almost sure limit of 
\[
{\rm AMSE}(\lambda, q, \sigma_w) \triangleq \lim_{p \rightarrow \infty} \frac{1}{p} \sum_{i=1}^p |\hat{\beta}_i(\lambda,q,p)- \beta_{i}(p)|^2,
\]
where the subscript $i$ is used to denote the $i$th component of a vector.
\end{definition}
Note that we have suppressed $\delta$ and $p_{\beta}$ in the notation of AMSE for simplicity, despite the fact that the asymptotic mean square error depends on them as well. In the above definition, we have assumed that the almost sure limit exists. Under the current asymptotic setting, the existence of AMSE can be proved. In fact we are able to derive the asymptotic limit for general loss functions as presented in the following theorem.

\begin{theorem}\label{thm:eqpseudolip} Consider a converging sequence $\{\beta(p), X(p), w(p)\}$. For any given $q\in [1,2]$, suppose that $\hat{\beta}(\lambda,q,p)$ is the solution of LQLS defined in \eqref{eq:lqls}. Then for any pseudo-Lipschitz function\footnote{A function $\psi:\mathbb{R}^2 \rightarrow \mathbb{R}$ is pseudo-Lipschitz of order $k$ if there exists a constant $L>0$ such that for all $x,y \in \mathbb{R}^2$, we have $|\psi(x)-\psi(y)|\leq L(1+\|x\|_2^{k-1}+\|y\|_2^{k-1}) \|x-y\|_2$. We consider pseudo-Lipschitz functions with order 2 in this paper.} $\psi: \mathbb{R}^2 \rightarrow \mathbb{R}$, almost surely
\begin{equation}\label{eq:lassoobs}
\lim_{p \rightarrow \infty} \frac{1}{p} \sum_{i=1}^p \psi \left(\hat{\beta}_i(\lambda,q,p),{\beta}_{i}(p) \right) = \mathbb{E}_{B,Z} [\psi(\eta_q(B+\bar{\sigma} Z; \bar{\chi} \bar{\sigma}^{2-q}), B)],
\end{equation}
where $B$ and $Z$ are two independent random variables with distributions $p_\beta$ and $N(0, 1)$, respectively; the expectation $\mathbb{E}_{B,Z}(\cdot)$ is taken with respect to both $B$ and $Z$; $\eta_q(\cdot;\cdot)$ is the proximal operator for the function $\| \cdot \|_q^q$\footnote{Proximal operator of $\| \cdot \|_q^q$ is defined as $\eta_q (u; \chi) \triangleq \arg \min_z \frac{1}{2} (u-z)^2 + \chi |z|^q$. For further information on these functions, please refer to the supplementary material. }; and $\bar{\sigma}$ and $\bar{\chi}$ satisfy the following equations:
\begin{eqnarray} \label{eq:fixedpoint}
\bar{\sigma}^2 &=& \sigma_{\omega}^2+\frac{1}{\delta} \mathbb{E}_{B, Z} [(\eta_q(B +\bar{\sigma} Z; \bar{\chi} \bar{\sigma}^{2-q}) -B)^2], \label{eq:fixedpoint11}  \\
\lambda &=& \bar{\chi} \bar{\sigma}^{2-q} \Big(1-\frac{1}{\delta} \mathbb{E}_{B,Z}[\eta_q'(B +\bar{\sigma} Z;  \bar{\chi} \bar{\sigma}^{2-q})] \Big), \label{eq:fixedpoint21} 
\end{eqnarray}
where $\eta_q'(\cdot; \cdot)$ denotes the derivative of $\eta_q$ with respect to its first argument. 
\end{theorem}

The result for $q=1$ has been proved in \cite{BaMo11}. The key ideas of the proof for generalizing to $q \in (1,2]$ are similar to those of \cite{BaMo11}. We describe the main proof steps in Appendix \ref{ssec:App:firstresult} of the supplementary material. According to Theorem \ref{thm:eqpseudolip}, in order to calculate the asymptotic mean square error (or any other loss) of $\hat{\beta}(\lambda,q,p)$, we have to solve \eqref{eq:fixedpoint11} and \eqref{eq:fixedpoint21} for $(\bar{\sigma},\bar{\chi})$. The following lemma shows  these two nonlinear equations have a unique solution. 

\begin{lemma}\label{lem:eq7and8solutionexists}
For any positive values of $\lambda, \delta, \sigma_w>0$, any random variable $B$ with finite second moment, and any $q \in [1,2]$, there exists a unique pair $(\bar{\sigma}, \bar{\chi})$ that satisfies both \eqref{eq:fixedpoint11} and \eqref{eq:fixedpoint21}.
\end{lemma}

The proof of this lemma can be found in Appendix \ref{subsec:solutionfixedpoint} of the supplementary material. Theorem \ref{thm:eqpseudolip} provides the first step in our analysis of LQLS. We first calculate $\bar{\sigma}$ and $\bar{\chi}$ from \eqref{eq:fixedpoint} and \eqref{eq:fixedpoint21}. Then, incorporating $\bar{\sigma}$ and $\bar{\chi}$ in \eqref{eq:lassoobs} gives the following expression for the asymptotic mean square error:
\begin{equation}\label{eq:def:AMSE}
{\rm AMSE}(\lambda, q, \sigma_w) =  \mathbb{E}_{B,Z} (\eta_q(B+\bar{\sigma} Z;  \bar{\chi} \bar{\sigma}^{2-q})- B)^2.
\end{equation}
Given the distribution of $B$ (the sparsity level $\epsilon$ included), the variance of the error $\sigma_w^2$, the number of response variables (normalized by the number of predictors) $\delta$, and the regularization parameter $\lambda$, it is straightforward to write a computer program to find the solution of \eqref{eq:fixedpoint11} and \eqref{eq:fixedpoint21} and then compute the value of AMSE. However, it is needless to say that this approach does not shed much light on the performance of bridge regression estimates, since there are many factors involved in the computation and each affects the result in a non-trivial fashion. In this paper, we would like to perform an analytical study on the solution of \eqref{eq:fixedpoint} and \eqref{eq:fixedpoint21} and obtain an explicit characterization of AMSE in the small-error regime.

\section{Our main contributions}\label{sec:our}

\subsection{Optimal tuning of $\lambda$}
The performance of LQLS, as defined in \eqref{eq:lqls}, depends on the tuning parameter $\lambda$. 
In this paper, we consider the value of $\lambda$ that gives the minimum AMSE. Let $\lambda_{*,q}$ denote the value of $\lambda$ that minimizes AMSE given in \eqref{eq:def:AMSE}. Then LQLS is solved with this specific value of $\lambda$, i.e.,
\begin{equation}
\hat{\beta} (\lambda_{*,q}, q,p ) \in \arg\min_{\beta} \frac{1}{2} \|y-X\beta\|_2^2+ \lambda_{*,q} \|\beta\|_q^q. \label{optimaldef}
\end{equation}
Note that this is the best performance that LQLS can achieve in terms of the AMSE. Theorem \ref{thm:eqpseudolip} enables us to evaluate this optimal AMSE of LQLS for every $q\in[1,2]$. The key step is to compute the solution of \eqref{eq:fixedpoint} and \eqref{eq:fixedpoint21} with $\lambda=\lambda_{*,q}$. Since $\lambda_{*,q}$ has to be chosen optimally, it seemingly causes an extra complication for our analysis. However, as we show in the following corollary, the study of Equations \eqref{eq:fixedpoint} and \eqref{eq:fixedpoint21} can be simplified to some extent.

\begin{corollary}\label{thm:mseoptimalasymptot} 
Consider a converging sequence $\{\beta(p), X(p), w(p)\}$. Suppose that $\hat{\beta}(\lambda_{*,q}, q,p)$ is the solution of LQLS defined in \eqref{optimaldef}. Then for any $q\in [1,2]$
\begin{equation}\label{eq:lassoobs:optimal}
{\rm AMSE}(\lambda_{*,q}, q, \sigma_w) = \min_{\chi \geq 0}\mathbb{E}_{B,Z} (\eta_q(B+\bar{\sigma} Z; \chi )- B)^2,
\end{equation}
where $B$ and $Z$ are two independent random variables with distributions $p_\beta$ and $N(0, 1)$, respectively; and $\bar{\sigma}$ is the unique solution of the following equation:
\begin{eqnarray} \label{eq:fixedpointoptimalnew}
\bar{\sigma}^2 = \sigma_{\omega}^2+\frac{1}{\delta} \min_{\chi \geq 0} \mathbb{E}_{B, Z} [(\eta_q(B +\bar{\sigma} Z; \chi) -B)^2].
\end{eqnarray}
\end{corollary}
The proof of Corollary \ref{thm:mseoptimalasymptot} is shown in Appendix \ref{sec:optimaltuning1} of the supplementary material. Corollary \ref{thm:mseoptimalasymptot} enables us to focus the analysis on a single equation \eqref{eq:fixedpointoptimalnew}, rather than two equations \eqref{eq:fixedpoint} and \eqref{eq:fixedpoint21}. The results we will present in the next section are mainly based on investigating the solution of \eqref{eq:fixedpointoptimalnew}.

\subsection{Analysis of AMSE}\label{secondorder}
In this paper since we are focused on the sparsity structure of $\beta$, from now on we assume that the distribution, to which the empirical distribution of $\beta \in \mathbb{R}^p$ converges, has the form 
\[
p_\beta(b) = (1-\epsilon) \delta_0 (b) + \epsilon g(b),
\]
where $\delta_0(\cdot)$ denotes a point mass at zero, and $g(\cdot)$ is a generic distribution that does not have any point mass at $0$. Here, the mixture proportion $\epsilon \in (0,1)$ is a fixed number that represents the sparsity level of $\beta$. The smaller $\epsilon$ is, the sparser $\beta$ will be. The distribution $g(b)$ specifies the values of non-zero components of $\beta$. We will use $G$ to denote a random variable having such a distribution. Since our results and proof techniques look very different for the case $q>1$ and $q=1$, we study these cases separately. 

\subsubsection{ Results for $q>1$} \label{discussion:q12}
 Our first result is concerned with the optimal AMSE of LQLS  for $1<q \leq 2$, when the number of response variables is larger than the number of predictors $p$, i.e., $\delta>1$. 
\begin{theorem}\label{asymp:lqbelowpt}
Suppose $\mathbb{P}(|G|\leq  t)=O(t)$ (as $t \rightarrow 0$) and $\mathbb{E}|G|^2<\infty$, then for $1<q <2$, $\delta>1$ and $\epsilon \in (0,1)$, we have
\begin{eqnarray}\label{mse:lp}
{\rm AMSE}(\lambda_{*,q}, q, \sigma_w)=\frac{\sigma^2_w}{1-1/\delta}-\frac{\delta^{q+1}(1- \epsilon)^2(\mathbb{E} |Z|^q)^2}{(\delta-1)^{q+1}\epsilon  \mathbb{E} |G|^{2q-2}}\sigma_w^{2q}+o(\sigma_w^{2q}).
\hspace{-1.5cm}
\end{eqnarray}
For $q=2, \delta>1$ and $\epsilon\in (0,1)$, if $\mathbb{E}|G|^2<\infty$, we have
\[
{\rm AMSE}(\lambda_{*,q},q, \sigma_w)= \frac{\sigma^2_w}{1-1/\delta}- \frac{\delta^{3}\sigma^{4}_w}{(\delta-1)^{3}\epsilon  \mathbb{E} |G|^{2}}+o(\sigma^4_w). 
\]
Note that $Z \sim N(0,1)$ and $G\sim g(\cdot)$ are independent.
\end{theorem}

The proof of the result is presented in Appendix \ref{sec:prooftheorem2full} of the supplementary material. There are several interesting features of this result that we would like to discuss: (i) The second dominant term of AMSE is negative. This means that the actual AMSE is smaller than the one predicted by the first order term, especially for smaller values of $q$. (ii) Neither the sparsity level nor the distribution of the non-zero components of $\beta$ appear in the first dominant term, i.e. $\frac{\sigma^2_w}{1-1/\delta}$. As we will discuss later in this section, the first dominant term is the one that specifies the phase transition curve. Hence, these calculations show a peculiar feature of phase transition analysis we discussed in Section \ref{limitations}, that the phase transition of $q \in (1,2]$ is neither affected by non-zero components of $\beta$ or the sparsity level. However, we see that both factors come into play in the second dominant term.  (iii) For the fully dense vector, i.e. $\epsilon =1$, \eqref{mse:lp} may imply that for $1<q<2$,
\begin{eqnarray*}
{\rm AMSE}(\lambda_{*,q},q, \sigma_w)=\frac{\sigma_w^2}{1-1/\delta}+o(\sigma^{2q}_w).
\end{eqnarray*}
 Hence, we require a different analysis to obtain the second dominant term (with different orders). We refer the interested readers to \cite{2016arXiv160307377W} for further information about this case. (iv) For $\epsilon<1$, the choice of $q \in (1,2]$ does not affect the first dominant term. That is the reason why all the values of $q\in (1,2]$ share the same phase transition curve. However, the value of $q$ has a major impact on the second dominant term. In particular, as $q$ approaches $1$, the order of the second dominant term in terms of $\sigma_w$ gets closer to that of the first dominant term. This means that in any practical setting, phase transition analysis may lead to misleading conclusions. {Specifically, in contrast to the conclusion from phase transition analysis that $q\in (1,2]$ have the same performance, the second order expansion enables us to conclude that the closer to $1$ the value of $q$ is, the better its performance will be. Our next theorem discusses the AMSE when $\delta<1$.  

\begin{theorem}\label{asymp:lqabovept}
Suppose $\mathbb{E}|G|^2<\infty$, then for $1<q \leq 2$ and $\delta<1$,
\begin{eqnarray}\label{mse:lp2}
\lim_{\sigma_w \rightarrow 0} {\rm AMSE}(\lambda_{*,q},q,\sigma_w)>0.
\end{eqnarray}
\end{theorem}

The proof of this theorem is presented in Appendix \ref{sec:proofthm3full} of the supplementary material. Theorems \ref{asymp:lqbelowpt} and \ref{asymp:lqabovept} together show a notion of phase transition. For $\delta>1$, as $\sigma_w \rightarrow 0$, $\mbox{AMSE} = O(\sigma_w^2)$, and hence it will go to zero,  while $\mbox{AMSE} \nrightarrow 0$ for $\delta<1$. In fact, the phase transition curve $\delta=1$ can be derived from the first dominant term in the expansion of AMSE. If $\delta=1$, the first dominant term is infinity and there will be no successful recovery, while it becomes zero when $\sigma_w=0$ if $\delta>1$. A more rigorous justification can be found in the proof of Theorems \ref{asymp:lqbelowpt} and \ref{asymp:lqabovept}. Therefore, we may conclude that the first order term contains the phase transition information. Moreover, the derived second order term offers us additional important information regarding the accuracy of the phase transition analysis. To provide a comprehensive understanding of these two terms, in Section \ref{sec:simanddis} we will evaluate the accuracy of first and second order approximations to AMSE through numerical studies.

\subsubsection{Results for $q=1$}

So far we have studied the case $1<q \leq 2$. In this section, we study $q=1$, a.k.a. LASSO. In Theorems \ref{asymp:lqbelowpt} and \ref{asymp:lqabovept}, we have characterized the behavior of LQLS with $q \in (1,2]$ for a general class of $G$. It turns out that the distribution of $G$  has a more serious impact on the second dominant term of AMSE for LASSO. We thus analyze it in two different settings. Our first theorem considers the distributions that do not have any mass around zero. 

\begin{theorem}\label{asymp:sparsel1_1}
Suppose $\mathbb{P}(|G|>\mu)=1$ with $\mu$ being a positive constant and $\mathbb{E}|G|^2<\infty, $ then for $\delta> M_1(\epsilon)$\footnote{Recall $M_1(\epsilon)=\inf_{\chi\geq 0} (1-\epsilon) \mathbb{E} \eta_1^2(Z; \chi) + \epsilon (1+ \chi^2)$ with $Z\sim N(0,1)$.}
\begin{eqnarray}\label{asymp:l1}
\hspace{1.cm}{\rm AMSE}(\lambda_{*,1},1,\sigma_w)=\frac{\delta M_1(\epsilon)}{\delta-M_1(\epsilon)}\sigma^2_w+o\Big(\phi \Big(\sqrt{\frac{\delta-M_1(\epsilon)}{\delta}}\frac{\tilde{\mu}}{\sigma_w}\Big)\Big),
\end{eqnarray}
where $\tilde{\mu}$ is any positive constant smaller than $\mu$ and $\phi(\cdot)$ is the density function of standard normal.
\end{theorem}

The proof of Theorem \ref{asymp:sparsel1_1} is given in Section \ref{sec:proofthm4full}. Different from LQLS with $q\in (1,2]$, we have not derived the exact analytical expression of second dominant term for LASSO. However, since it is exponentially small, the first order term (or phase transition analysis) is sufficient for evaluating the performance of LASSO in the small-error regime. This will be further confirmed by the numerical studies in Section \ref{sec:simanddis}.  Below is our result for the distributions of $G$ that have more mass around zero.

\begin{theorem}\label{asymp:sparsel1_2}
Suppose that $\mathbb{P}(|G|\leq t)=\Theta(t^{\ell})$ (as $t \rightarrow 0$) with $\ell>0$ and $\mathbb{E}|G|^2<\infty$, then for $\delta > M_1(\epsilon)$, 
\begin{eqnarray*}
 - \Theta(\sigma_w^{\ell+2}) &\gtrsim& {\rm AMSE}(\lambda_{*,1},1,\sigma_w)-\frac{\delta M_1(\epsilon)}{\delta-M_1(\epsilon)} \sigma^2_w \\
 &\gtrsim&  -\Theta(\sigma_w^{\ell+2}) \cdot \left(\underbrace{\log \log \ldots \log}_{m\  \rm times} \left(\frac{1}{\sigma_w}\right) \right)^{\ell/2},
\end{eqnarray*}
where $m$ is an arbitrary but finite natural number, and $a \gtrsim b$ means $a \geq b$ holds for sufficiently small $\sigma_w$.
\end{theorem}

The proof of this theorem can be found in Appendix \ref{sec:prooflassogeralepsless1} of the supplementary material. It is important to notice the difference between Theorems \ref{asymp:sparsel1_1} and \ref{asymp:sparsel1_2}. The first point we would like to emphasize is that the first dominant terms are the same in both cases. The second dominant terms are different though. As we will show in Section \ref{sec:proofthm4full} and Appendix \ref{sec:prooflassogeralepsless1}, similar to LQLS for $1<q\leq 2$, the second dominant terms are in fact negative. Hence, the actual AMSE will be smaller than the one predicted by the first dominant term. Furthermore, note that the magnitude of the second dominant term in Theorem \ref{asymp:sparsel1_2} is much larger than that in Theorem \ref{asymp:sparsel1_1}. This seems intuitive. LASSO tends to shrink the parameter coefficients towards zero, and hence, if the true $\beta$ has more mass around zero, the AMSE will be smaller. The more mass the distribution of $G$ has around zero, the better the second order term will be. Our next theorem discusses what happens if $\delta < M_1(\epsilon)$.

\begin{theorem}\label{them:ellabovpt}
Suppose that $\mathbb{E}|G|^2<\infty$. Then for $\delta < M_1(\epsilon)$, 
\begin{eqnarray}\label{mse:lp2}
\lim_{\sigma_w \rightarrow 0} {\rm AMSE}(\lambda_{*,1},1,\sigma_w) >0.
\end{eqnarray}
\end{theorem}

The proof  is presented in Appendix \ref{sec:proofthm6full} of the supplementary material. Similarly as we discussed in Section \ref{discussion:q12}, Theorems \ref{asymp:sparsel1_1}, \ref{asymp:sparsel1_2} and \ref{them:ellabovpt} imply the phase transition curve of LASSO. Such information can be obtained from the first dominant term in the expansion of AMSE as well.

\section{Related work}\label{sec:relatedwork}

\subsection{Other phase transition analyses and $n/ p \rightarrow \delta$ asymptotic results}

The asymptotic framework that we considered in this paper evolved in a series of papers by Donoho and Tanner \cite{donoho2004most, donoho2006high, donoho2005sparse, donoho2005neighborliness}. This framework was used before on similar problems in engineering and physics \cite{guo2005randomly, tanaka2002statistical, coolen2005mathematical}. Donoho and Tanner characterized the phase transition curve for LASSO and some of its variants. Inspired by this framework, many researchers started exploring the performance of different algorithms or estimates under this asymptotic settings \cite{ stojnic2009various, amelunxen2014living, thrampoulidis2016precise, el2013robust,  karoui2013asymptotic, donoho2013high, donoho2013phase, donoho2015variance, bradic2015robustness, donoho2011noise, zheng2015does,  rangan2009asymptotic, krzakala2012statistical, BaMo10, BaMo11}.

Our paper performs the analysis of LQLS under such asymptotic framework. Also, we adopt the message passing analysis that was developed in a series of papers \cite{DoMaMoNSPT, donoho2009message, maleki2010approximate, BaMo10, BaMo11}. The notion of phase transition we consider is similar to the one introduced in \cite{DoMaMoNSPT}. However, there are three major differences: (i) The analysis of \cite{DoMaMoNSPT} is performed for LASSO, while we have generalized the analysis to any LQLS with $1<q\leq2$. (ii) The analysis of \cite{DoMaMoNSPT} is performed on the least favorable distribution for LASSO, while here we characterize the effect of the distribution of $G$ on the AMSE as well. (iii) Finally, \cite{DoMaMoNSPT} is only concerned with the first dominant term in AMSE of LASSO, while we derive the second dominant term whose importance has been discussed in the last few sections. 

Another line of research that has connections with our analysis is presented in a series of papers \cite{oymak2013squared, oymak2016sharp, thrampoulidis2016precise}. In  \cite{thrampoulidis2016precise} the authors have derived a minimax formulation that (if it has a unique solution and is solved) can give an accurate characterization of the asymptotic mean square error. Compared with Theorem \ref{thm:eqpseudolip} in our paper, that result works for more general penalized M-estimators, while Theorem  \ref{thm:eqpseudolip} holds for general pseudo-Lipschitz loss functions. When applying the minimax formulation in \cite{thrampoulidis2016precise} to bridge regression, the AMSE formula in \eqref{eq:def:AMSE} can be recovered. However, the derivation of phase transition curves for bridge regression under optimal tuning is not found in \cite{thrampoulidis2016precise}. Furthermore, \cite{oymak2013squared, oymak2016sharp} proposed a geometric approach to characterize the risk of penalized least square estimates with general convex penalties. In particular, both papers obtained phase transition results based on a key convex geometry quantity called ``Gaussian squared-distences". However, \cite{oymak2013squared} only rigorously proved the negative results (equivalent to Theorems \ref{asymp:lqabovept} and \ref{them:ellabovpt}) and left the positive part as a conjecture. The phase transition results in \cite{oymak2016sharp} are concerned with the prediction errors $\|y-X\hat{\beta}\|_2^2$ and $\|X\beta-X\hat{\beta}\|_2^2$, rather than the estimation error $\|\hat{\beta}-\beta\|_2^2$. Also, neither of the two papers went beyond the first-order or phase transition analysis of the risk.

Several researchers have also worked on the analysis of LQLS for $q <1$  \cite{kabashima2009typical, zheng2015does, rangan2009asymptotic}. These analyses are based on non-rigorous, but widely accepted replica method from statistical physics. The current paper extends the analysis of \cite{zheng2015does} to $q \geq 1$ case, makes the analysis rigorous by using the message passing framework rather than the replica method, and finally provides a higher order analysis.

\subsection{Other analysis frameworks}
One of the first papers that compared the performance of penalization techniques is \cite{hoerl1970ridge} which showed that there exists a value of $\lambda$ with which Ridge regression, i.e. LQLS with $q=2$, outperforms the vanilla least squares estimator. Since then, many more regularizers have been introduced to the literature each with a certain purpose. For instance, we can mention LASSO \cite{tibshirani1996regression}, elastic net \cite{zou2005regularization}, SCAD \cite{fan2001variable}, bridge regression \cite{frank1993statistical}, and more recently SLOPE \cite{bogdan2015slope}. There has been a large body of work on studying all these regularization techniques. We partition all the work into the following categories and explain what in each category has been done about the bridge regression:

\begin{itemize}
\item[(i)] Simulation results: One of the main motivations for our work comes from the nice simulation study of the bridge regression presented in \cite{fu1998penalized}.  This paper finds the optimal values of $\lambda$ and $q$ by generalized cross validation and compares the performance of the resulting estimator with both LASSO and ridge. The main conclusion is that the bridge regression can outperform both LASSO and ridge. Given our results we see that if sparsity is present in $\beta$, then smaller values of $q$ perform better than ridge (in their second dominant term). 


\item[(ii)] Asymptotic study: Knight and Fu \cite{knight2000asymptotics} studied the asymptotic properties of bridge regression under the setting where $n \rightarrow \infty$, while $p$ is fixed. They established the consistency and asymptotic normality of the estimates under quite general conditions. Huang et al. \cite{huang2008asymptotic} studied LQLS for $q<1$ under a high-dimensional asymptotic setting in which $p$ grows with $n$ but is still assumed to be less than $n$. They not only derived the asymptotic distribution of the estimators, but also proved LQLS has oracle properties in the sense of Fan and Li \cite{fan2001variable}. They have also considered the case $p>n$, and have shown that under partial orthogonality assumption on $X$, bridge regression distinguishes correctly between covariates with zero and non-zero coefficients.  Note that under the asymptotic regime of our paper, both LASSO and the other bridge estimators have false discoveries \cite{su2015false} and possibly non-zero AMSE. Hence, they may not provide consistent estimates. We should also mention that the analysis of bridge regression with $q\in [0,1)$ under the asymptotic regime $n/p \rightarrow \delta$ is presented in \cite{zheng2015does}. Finally, the performance of LASSO under a variety of conditions has been studied extensively. We refer the reader to 
\cite{buhlmann2011statistics} for the review of those results.

\item [(iii)] Non-asymptotic bounds: One of the successful approaches that has been employed for studying the performance of regularization techniques such as LASSO is the minimax analysis \cite{bickel2009simultaneous}, \cite{raskutti2011minimax}. We refer the reader to \cite{buhlmann2011statistics} for a complete list of references on this direction. In this minimax approach, a lower bound for the prediction error or mean square error of any estimation technique is first derived. Then a specific estimate, like the one returned by LASSO, is considered and an upper bound is derived assuming the design matrices satisfy certain conditions such as restrictive eigenvalue assumption \cite{bickel2009simultaneous, koltchinskii2009sparsity}, restricted isometry condition \cite{candes2008restricted}, or coherence conditions \cite{bunea2007sparsity}. These conditions can be confirmed for matrices with iid subgaussian elements. Based on these evaluations, if the order of the upper bound for the estimate under study matches the order of the lower bound, we can claim that the estimate (e.g. LASSO) is minimax rate-optimal. This approach has some advantages and disadvantages compared to our asymptotic approach: (i) It works under more general conditions. (ii) It provides information for any sample size. The price paid in the minimax analysis is that the constants derived in the results are usually not sharp and hence many schemes have similar guarantees and cannot be compared to each other. Our asymptotic framework looses the generality and in return gives sharp constants that can then be used in evaluating and  comparing different schemes as we do in this paper. Along similar directions, \cite{koltchinskii2009sparsity} has studied the penalized empirical risk minimization with $\ell_q$ penalties for the values of $q \in [1,1+ \frac{1}{\log p}]$ and has found upper bounds on the excess risk of these estimators (oracle inequalities). Similar to minimax analysis, although the results of this analysis enjoy generality, they suffer from loose constants that impede an accurate comparisons of different bridge estimators.  
\end{itemize}

\begin{figure}
	\centering
	\subfloat[][]{\includegraphics[width=1.9in]{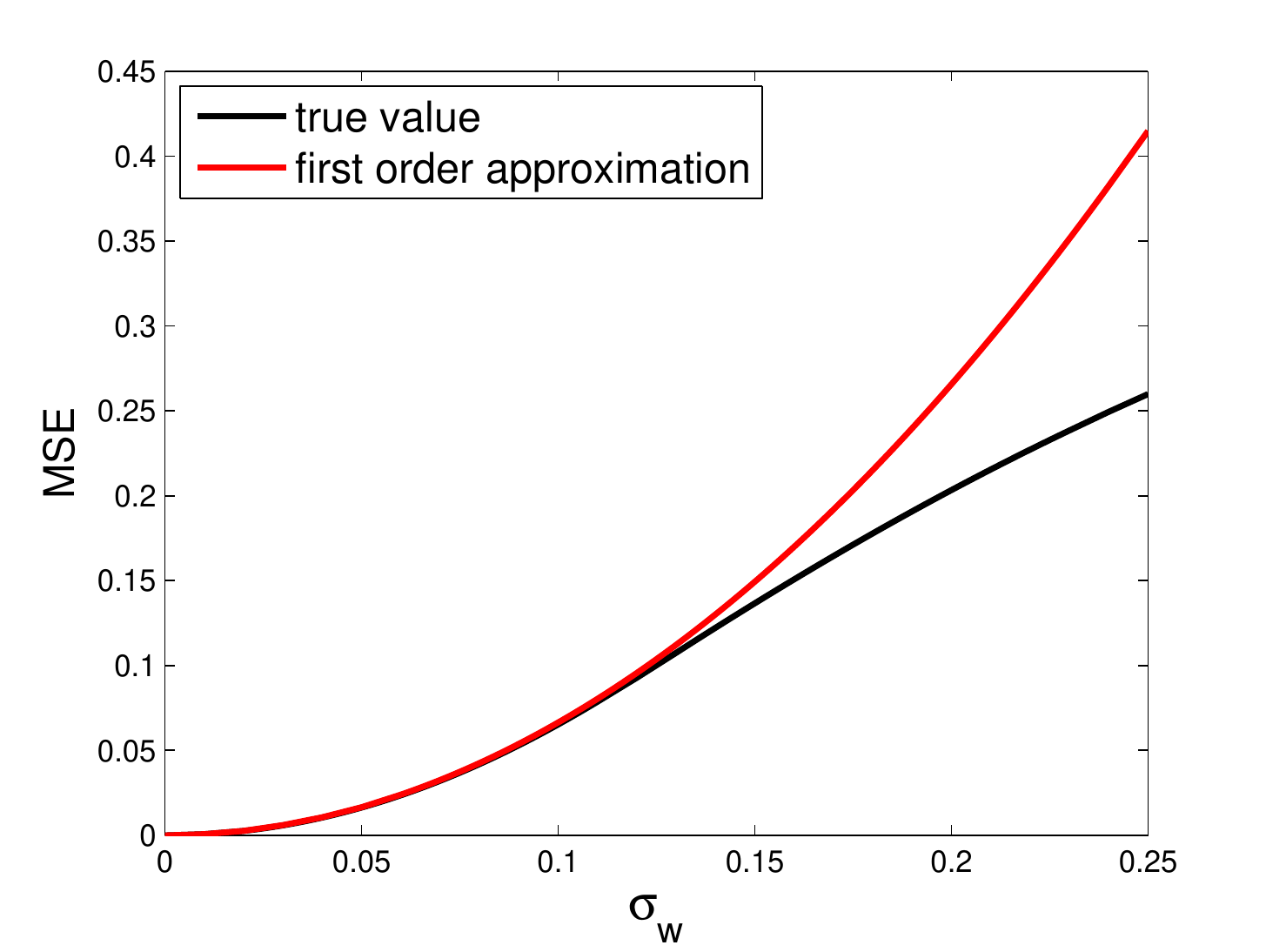}}
	\subfloat[][]{\includegraphics[width=1.9in]{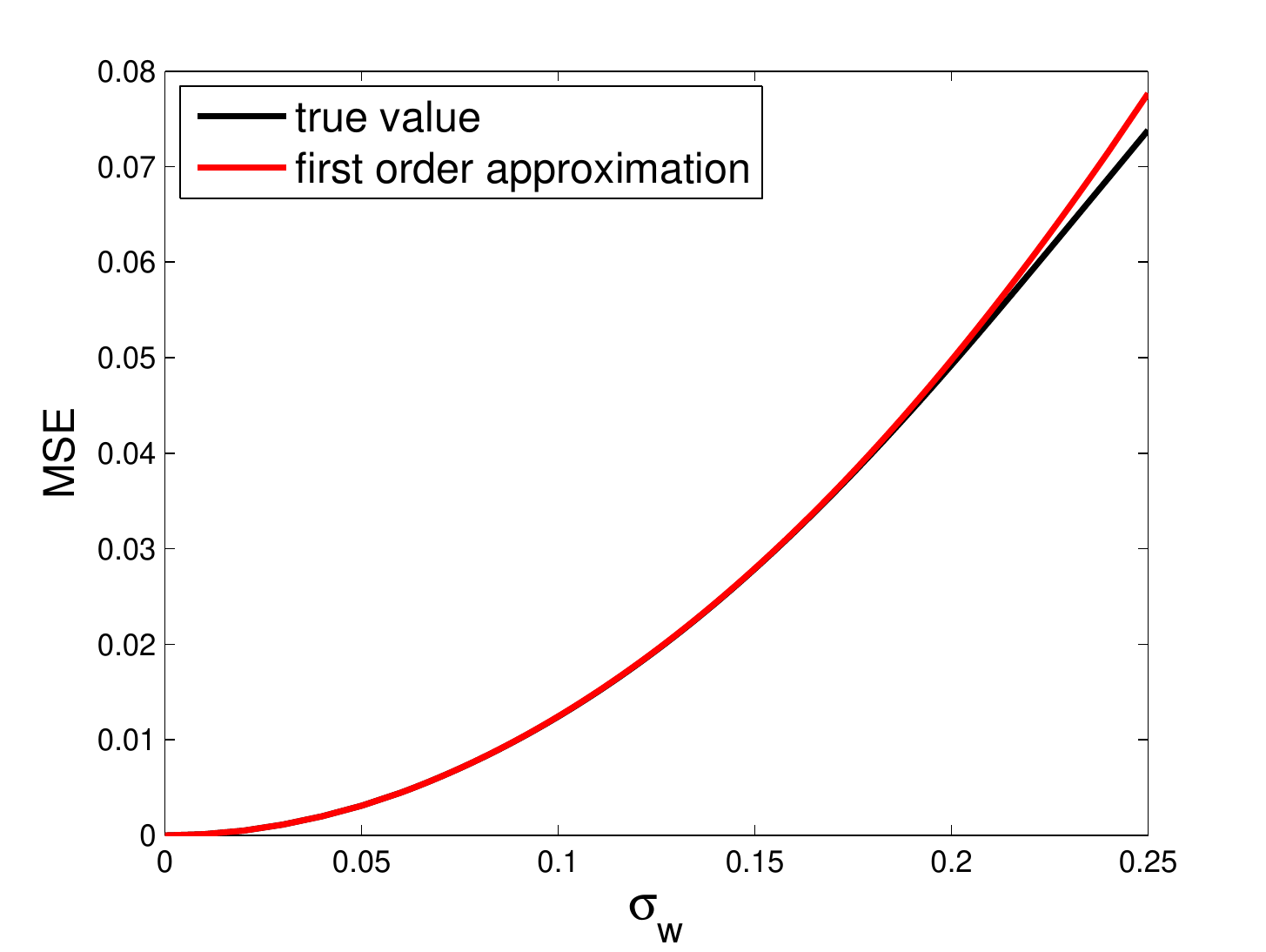}}
	
	\subfloat[][]{\includegraphics[width=1.9in]{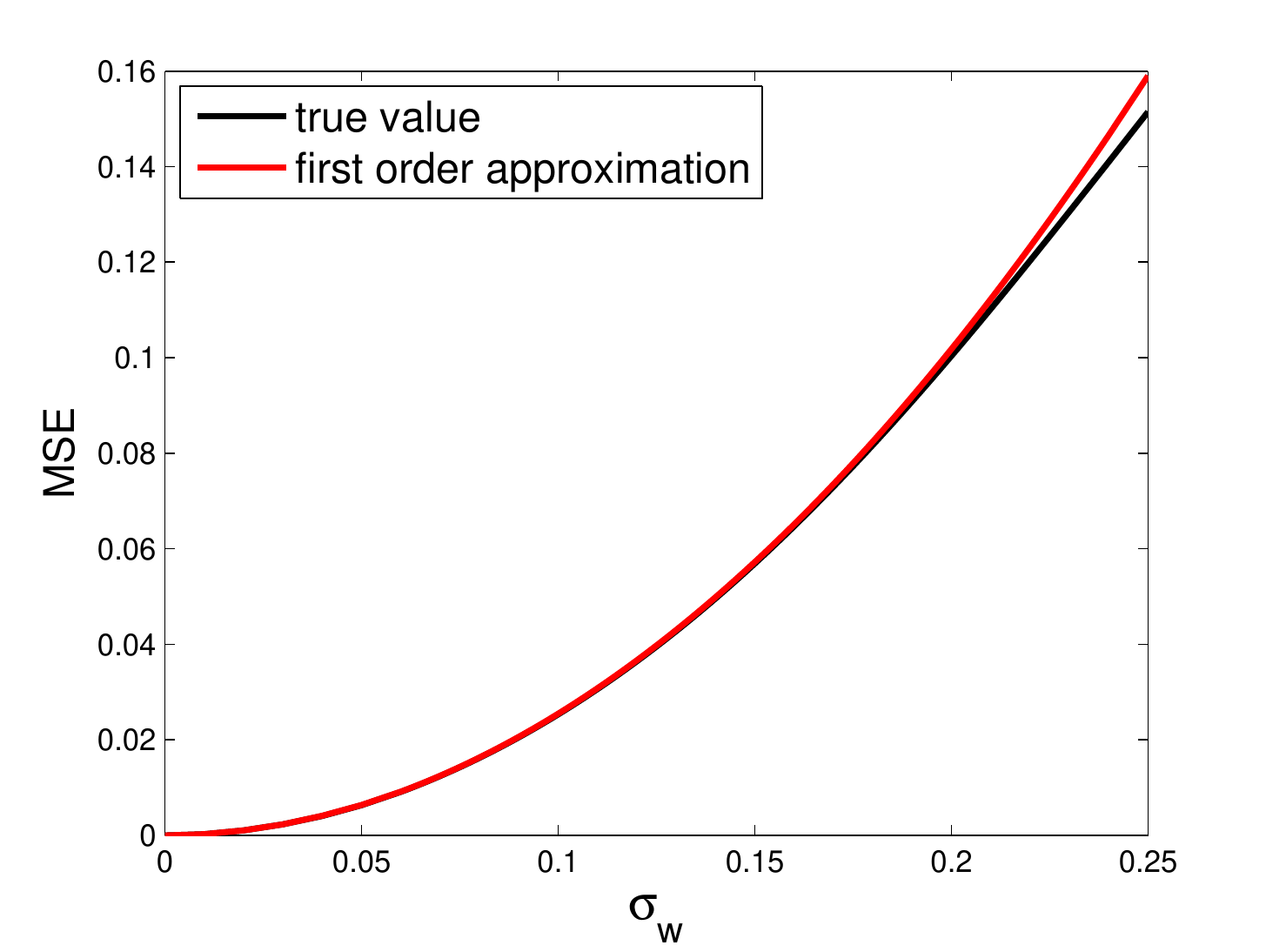}}
	\subfloat[][]{\includegraphics[width=1.9in]{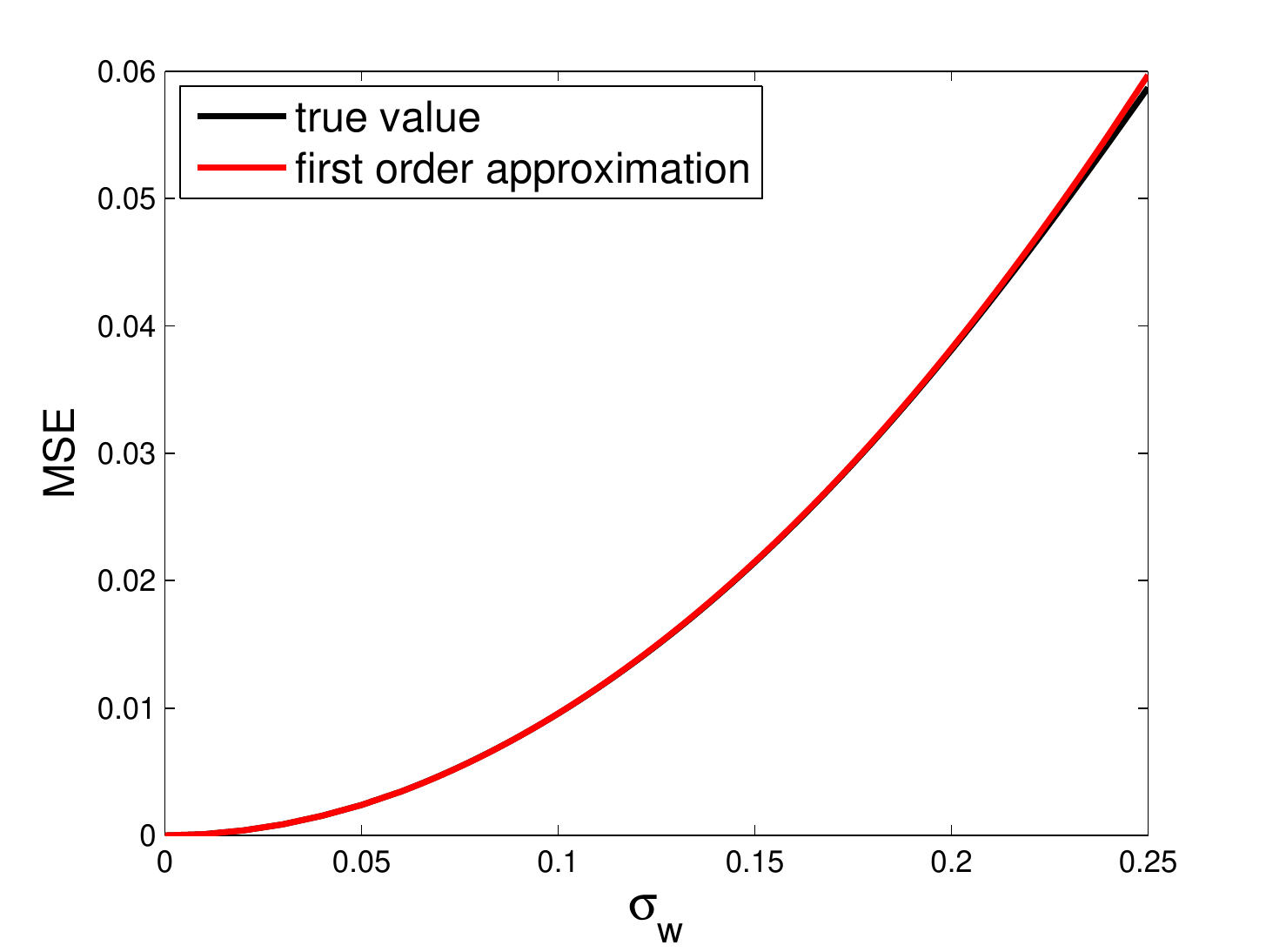}}
	
	\subfloat[][]{\includegraphics[width=1.9in]{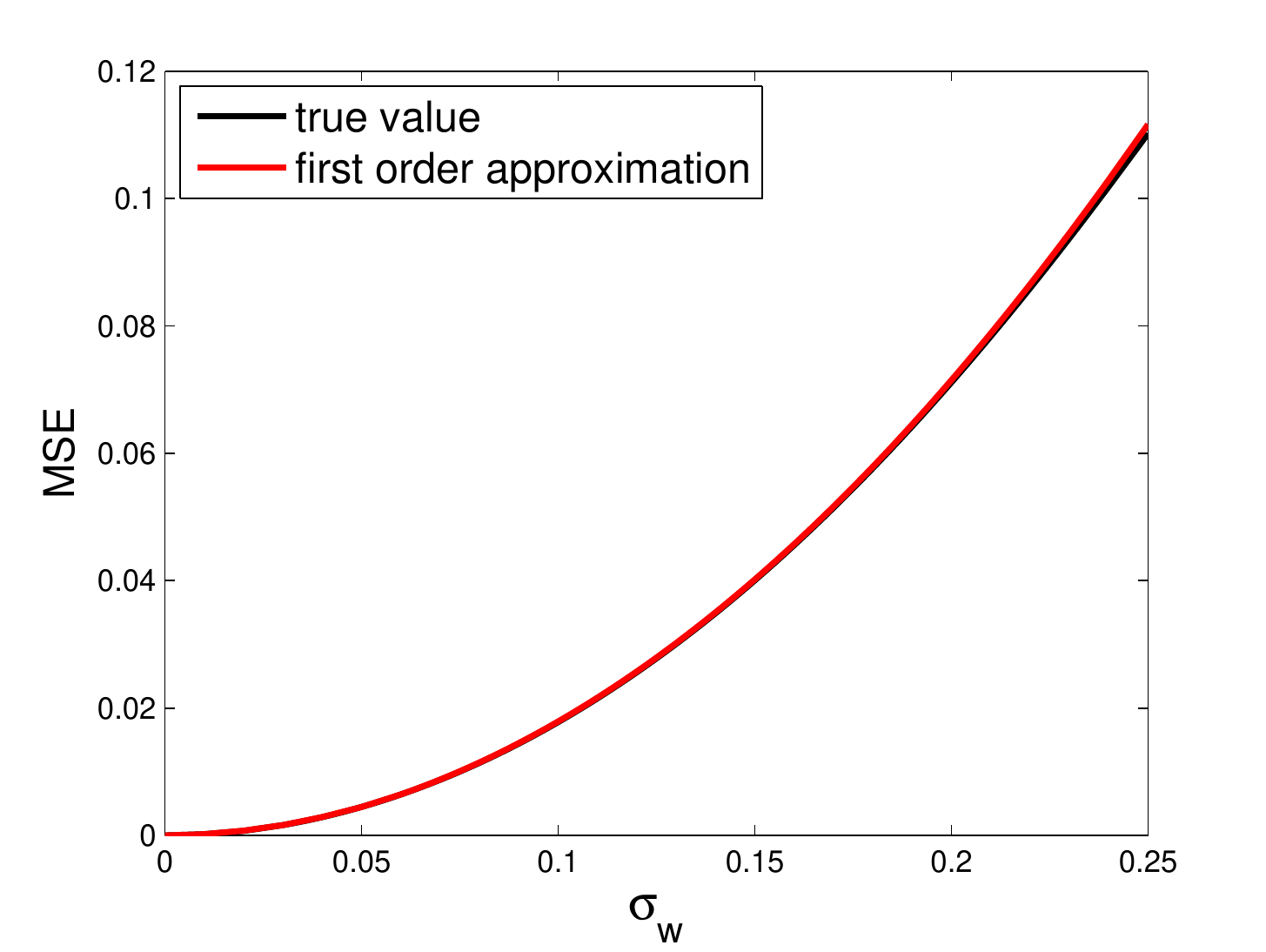}}
	\subfloat[][]{\includegraphics[width=1.9in]{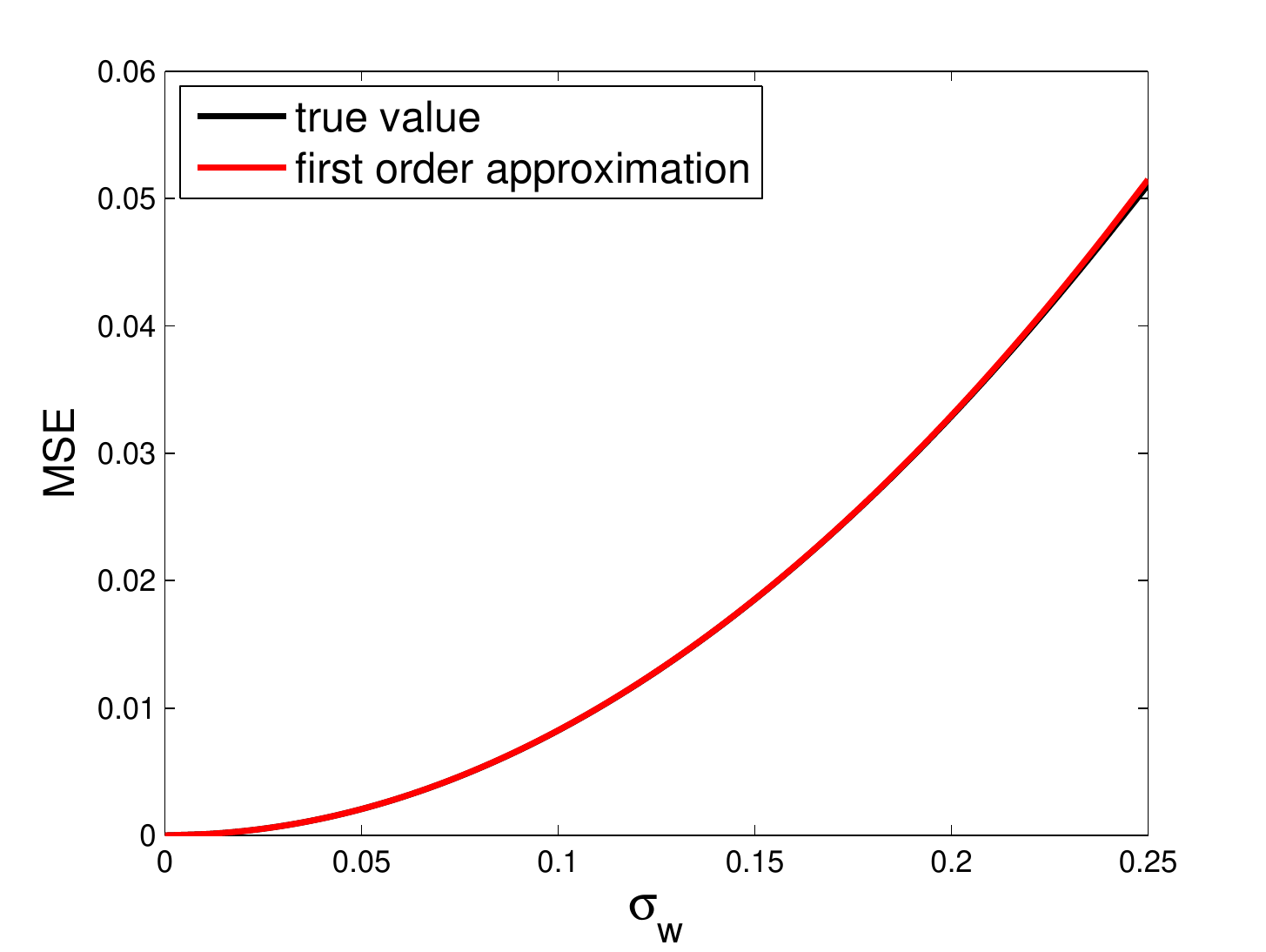}}	

	
	\caption{Plots of actual AMSE and its approximations for (a) $\delta = 1.1$ and $\epsilon =0.7$, (b) $\delta= 1.1$ and $\epsilon =0.25$, (c) $\delta= 1.5$ and $\epsilon =0.7$, (d) $\delta = 1.5$  and $\epsilon =0.25$, (e) $\delta=2$ and $\epsilon =0.7$, (f) $\delta=2$ and $\epsilon =0.25$. }
	
	\label{fig:delta15}
\end{figure}

\section{Numerical results and discussions}\label{sec:simanddis}
\subsection{Summary}

The analysis of AMSE we presented in Section \ref{secondorder} is performed as $\sigma_w \rightarrow 0$. For such asymptotic analysis, it would be interesting to check the approximation accuracy of the first and second order expansions of AMSE over a reasonable range of $\sigma_w$. Towards this goal, this section performs several numerical studies to (i) evaluate the  accuracy of the first and second order expansions discussed in Section \ref{secondorder}, (ii) discover situations in which the first order approximation is not accurate (for reasonably small noise levels) while the second order expansion is, and (iii) identify situations where both first and second orders are inaccurate and propose methods for improving the approximations. Sections \ref{ssec:sim:lasso} and \ref{ssec:sim:bridgeq1} study the performance of LASSO and other bridge regression estimators with $q>1$ respectively. Finally, we should also mention that all the results presented in this paper are concerned with the asymptotic setting $n,p \rightarrow \infty$ and $n/p \rightarrow \delta$. To evaluate the accuracy of these results for finite sample sizes, we have performed additional simulations whose results are presented in Appendix \ref{add:simulations} of the supplementary material.

\subsection{LASSO}\label{ssec:sim:lasso}
One of the conclusions from Theorem \ref{asymp:sparsel1_1} is that the first dominant term provides a good approximation of AMSE for the LASSO problem when the distribution of $G$ does not have a large mass around 0. To test this claim we conduct the following numerical experiment. We set the parameters of our problem instances in the following way:
 
\begin{enumerate}

\item $\delta$ can take any value in $\{1.1, 1.5, 2\}$.

\item $\epsilon$ can take values in $\{0.25, 0.7 \}$.

\item $\sigma_w$ ranges within the interval $[0, 0.25]$.

\item the distribution of $G$ is specified as $g(b)=0.5 \delta_1(b) + 0.5\delta_{-1}(b)$, where $\delta_a(\cdot)$ denotes a point mass at point $a$.
\end{enumerate}
 We then use the formula in Corollary \ref{thm:mseoptimalasymptot} to calculate AMSE$(\lambda_{*,1}, 1, \sigma_w)$. Finally, we compare AMSE$(\lambda_{*,1}, 1, \sigma_w)$, computed numerically from \eqref{eq:lassoobs:optimal} and \eqref{eq:fixedpointoptimalnew},  with its first order approximation provided in Theorem \ref{asymp:sparsel1_1}. The results of this experiment are summarized in Figure \ref{fig:delta15}. As is clear in this figure, the first order expansion gives a very good approximation for AMSE over a large range of $\sigma_w$.

\subsection{Bridge regression estimators with $q>1$}\label{ssec:sim:bridgeq1}
In this numerical experiment, we would like to vary $\sigma_w$ and see under what conditions our first order or second order expansions can lead to accurate approximation of AMSE for a wide range of $\sigma_w$. Throughout this section, we set the distribution of $G$ to $g(b)=0.5 \delta_1(b) + 0.5\delta_{-1}(b)$, as we did in Section \ref{ssec:sim:lasso}. We then investigate different conditions by specifying various values of other parameters in our problem instances. The expansion of AMSE for $q>1$ is presented in Theorem \ref{asymp:lqbelowpt}. For $q\in(1,2)$, recall the two terms in the expansion below
\begin{eqnarray}\label{refresh}
\hspace{0.9cm} {\rm AMSE}(\lambda_{*,q}, q, \sigma_w)=\frac{\sigma^2_w}{1-1/\delta}-\frac{\delta^{q+1}(1- \epsilon)^2(\mathbb{E} |Z|^q)^2}{(\delta-1)^{q+1}\epsilon  \mathbb{E} |G|^{2q-2}}\sigma_w^{2q}+o(\sigma_w^{2q}).
\end{eqnarray}

We expect the first order term to present a good approximation over a reasonably large range of $\sigma_w$, when the second order term is sufficiently small. According to the analytical form of the second order term in \eqref{refresh}, it is small if the following three conditions hold simultaneously: (i) $\delta$ is not close to $1$, (ii) $\epsilon$ is not small, and  (iii) $q$ is not close to 1. Our first numerical result shown in Figure \ref{fig:everythingOk1} is in agreement with this claim. In this simulation we have set three different cases for $\delta, \epsilon$ and $q$ so that they satisfy the above three conditions. The non-zero elements of $\beta$ are independently drawn from $0.5 \delta_1(b) + 0.5 \delta_{-1}(b)$. As demonstrated in this figure, the first order term approximates AMSE accurately. Another interesting finding is that the second order expansion provides an even better approximation.

\begin{figure}
\centering
	\subfloat[][]{\includegraphics[width=1.6in]{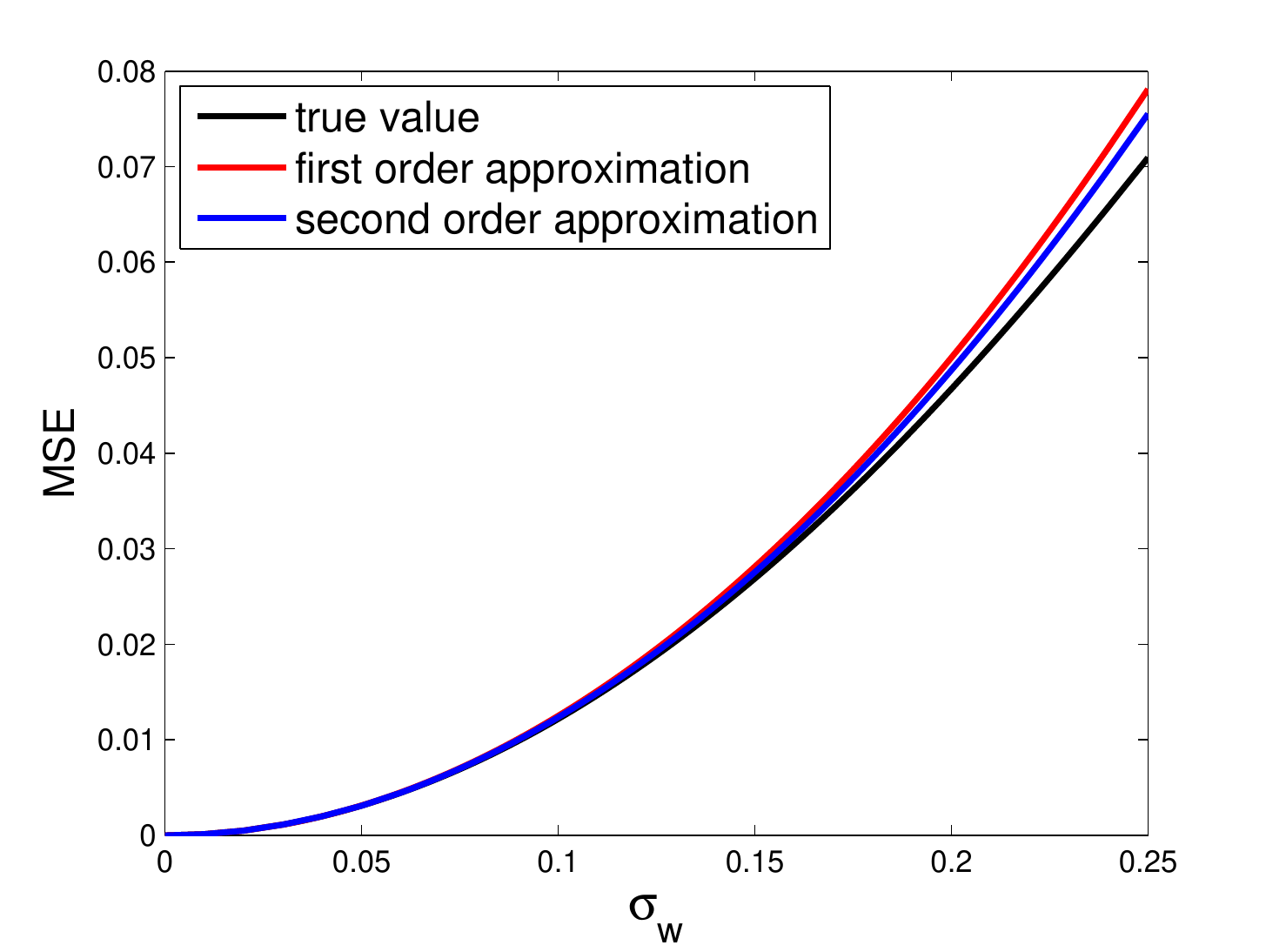}}
	\subfloat[][]{\includegraphics[width=1.6in]{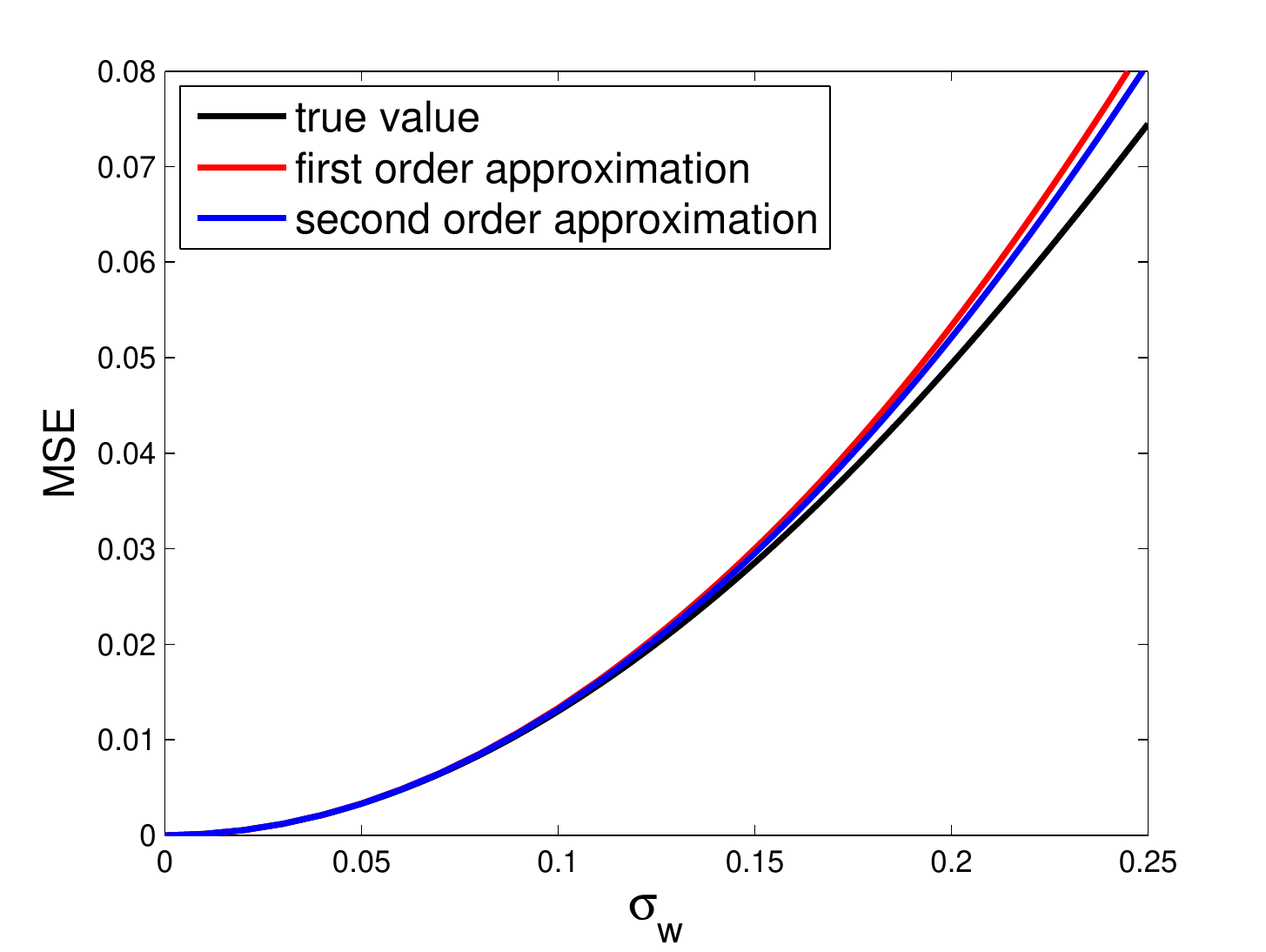}}
	\subfloat[][]{\includegraphics[width=1.6in]{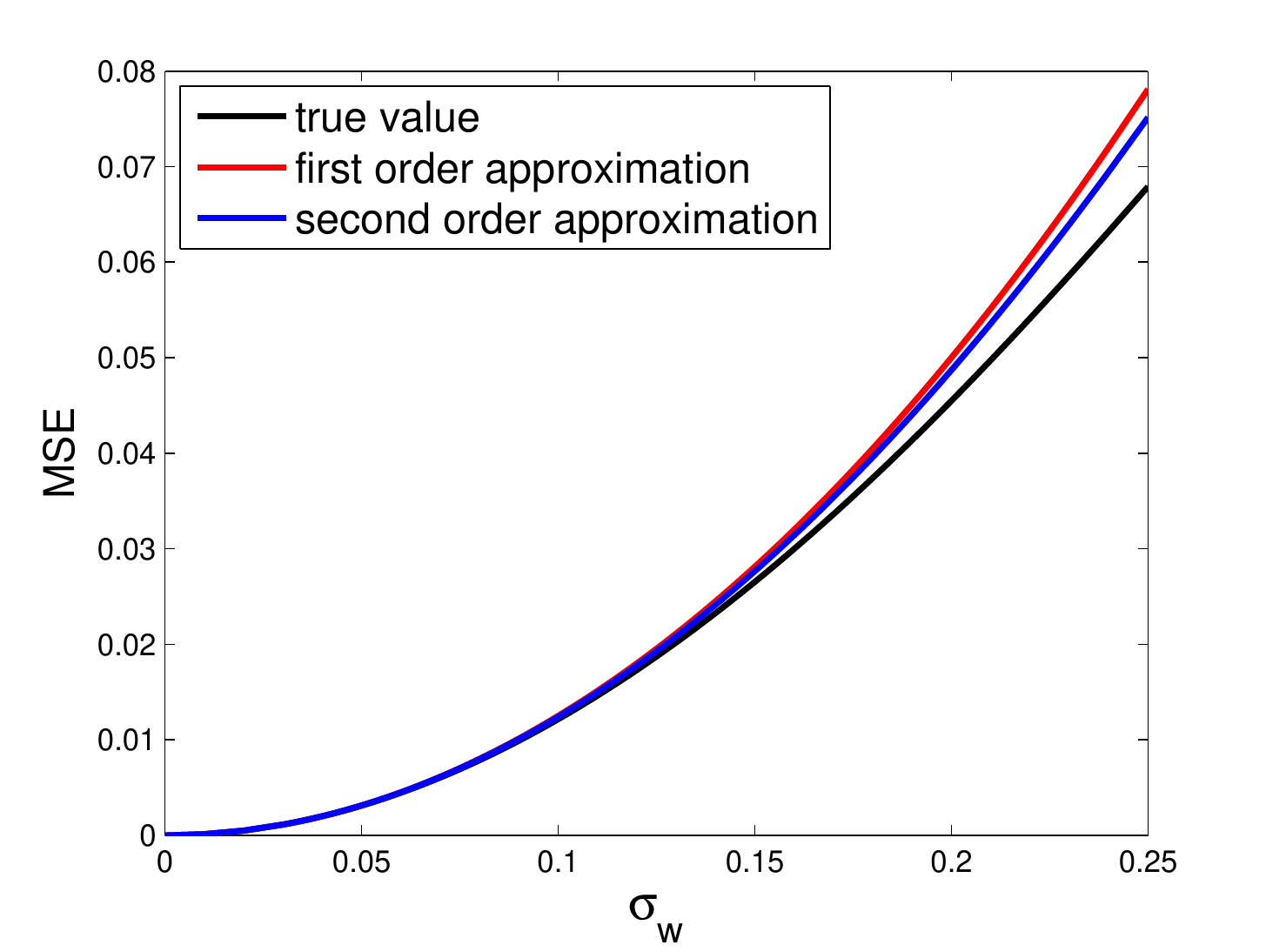}}
          \caption{Plots of actual AMSE and its approximations for (a) $\delta= 5, \epsilon=0.7, q= 1.5$, (b) $\delta=4, \epsilon=0.7, q=1.6$, (c) $\delta=5, \epsilon=0.6, q=1.8.$}
\label{fig:everythingOk1}
\end{figure}

To understand the limitation of the first order approximation, we consider the cases in which the second order term is large and suggests that at least the first order approximation is not necessarily good. This happens when either $\delta$ decreases to $1$, $\epsilon$ decreases to $0$ or $q$ decreases to $1$. The settings of our experiments and the results are summarized below. 

\begin{enumerate}
\item We keep $q=1.5$ and $\epsilon=0.7$ fixed and study different values of $\delta \in \{5,2,1.5, 1.1 \}$. Figure \ref{fig:deltachange} summarizes the results of this simulation. As is clear in this figure (and is consistent with the message of the second dominant term), as we decrease $\delta$ the first order approximation becomes less accurate. The second order approximation in these cases is more accurate than the first order approximation. However interestingly, the second order approximation becomes less accurate as $\delta$ decreases too. These observations suggest that to have a good approximation for the values of $\delta$ that are very close to $1$,  although the second order approximation outperforms the first order, it may not be sufficient and higher order terms are required. Such terms can be derived with strategies similar to the ones we used in the proof of Theorem \ref{asymp:lqbelowpt}. Note that the insufficiency of the second order expansion partially results from the wide range of $\sigma_w \in [0, 0.25]$. If we evaluate the approximation when $\sigma_w$ is small enough, we will expect the success of the second order expansion. 

\begin{figure} 
	\centering
	\subfloat[][]{\includegraphics[width=1.9in]{delta5eps07p15}}
	\subfloat[][]{\includegraphics[width=1.9in]{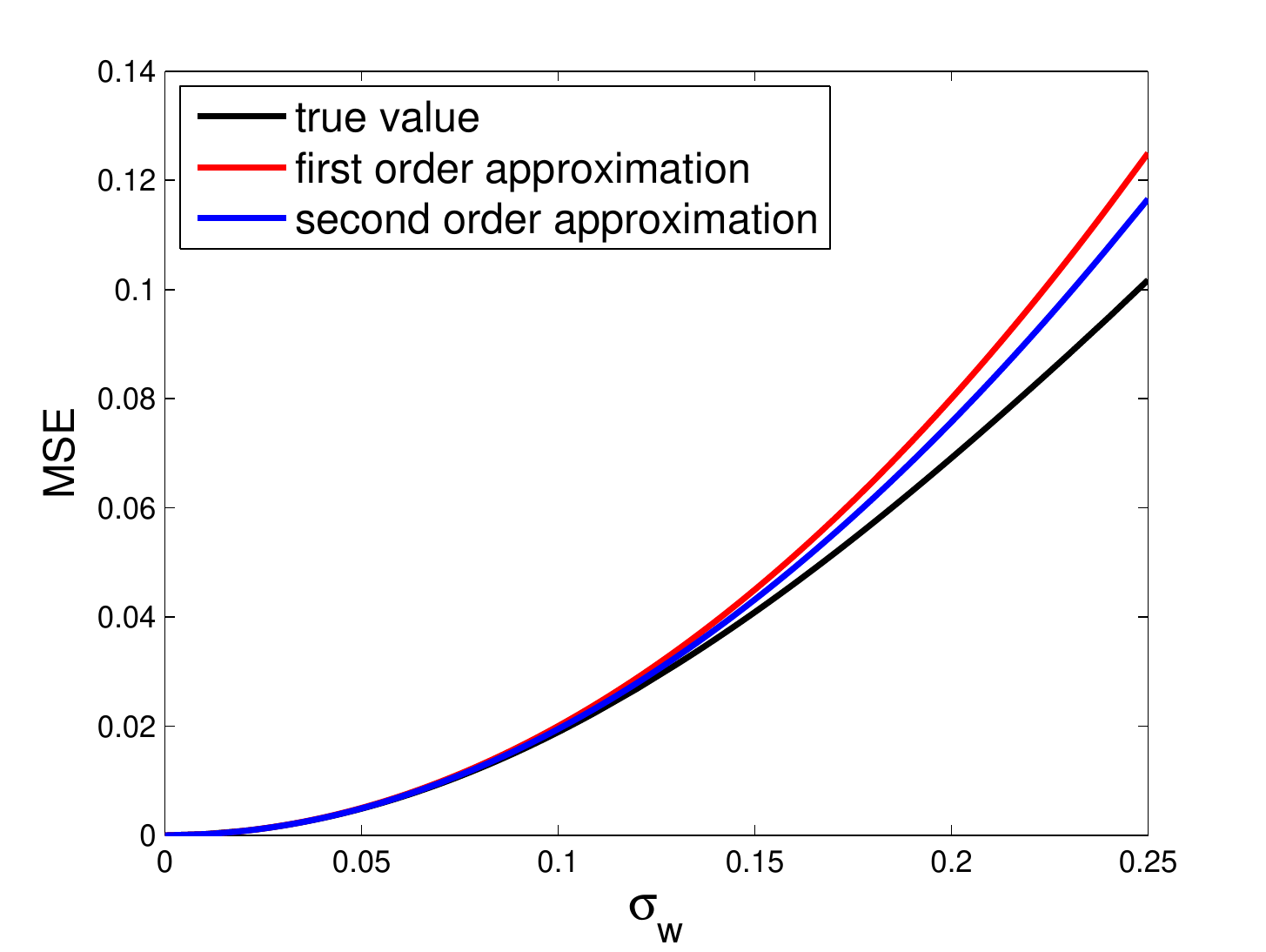}}

	\subfloat[][]{\includegraphics[width=1.9in]{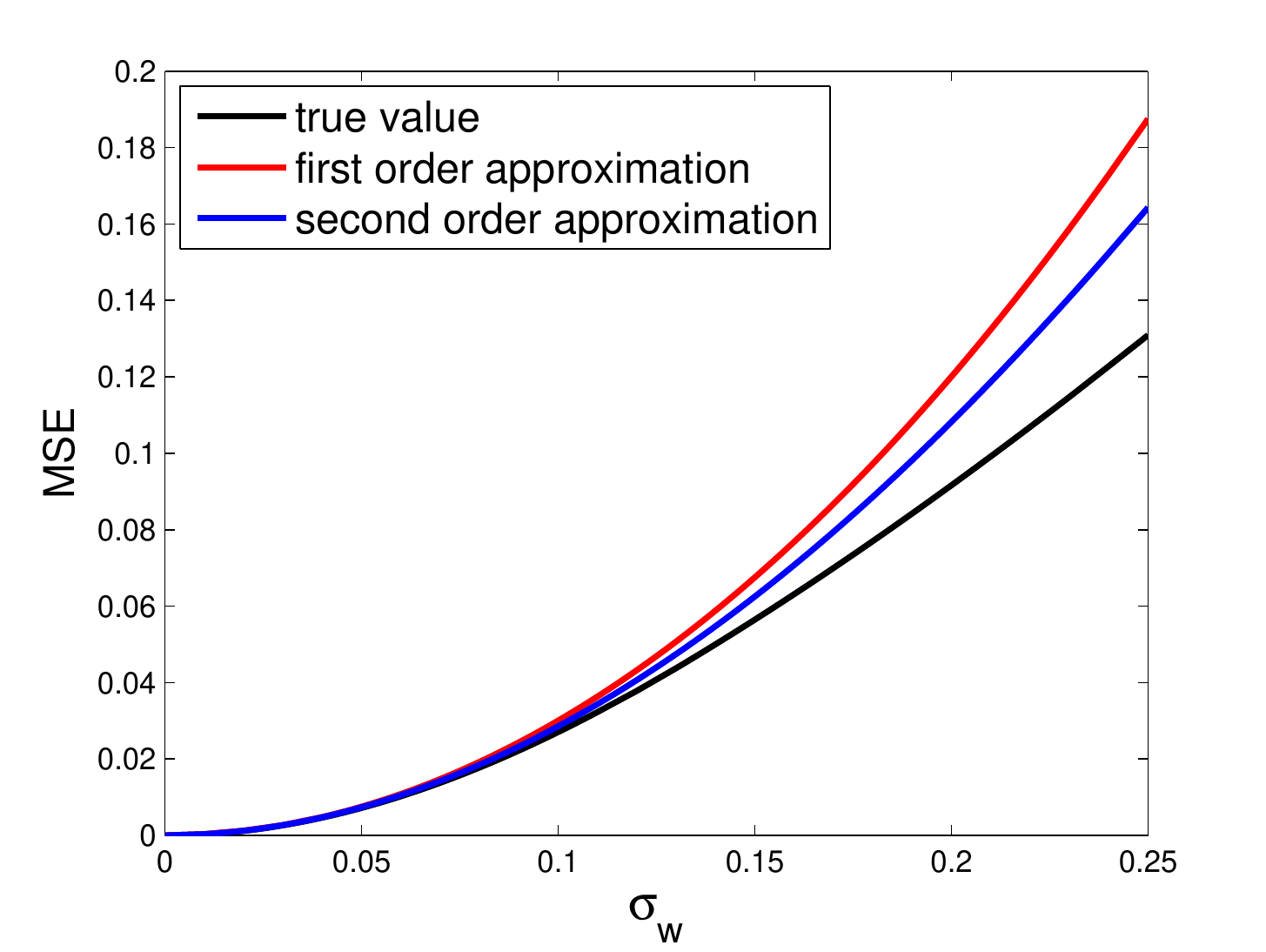}}
	\subfloat[][]{\includegraphics[width=1.9in]{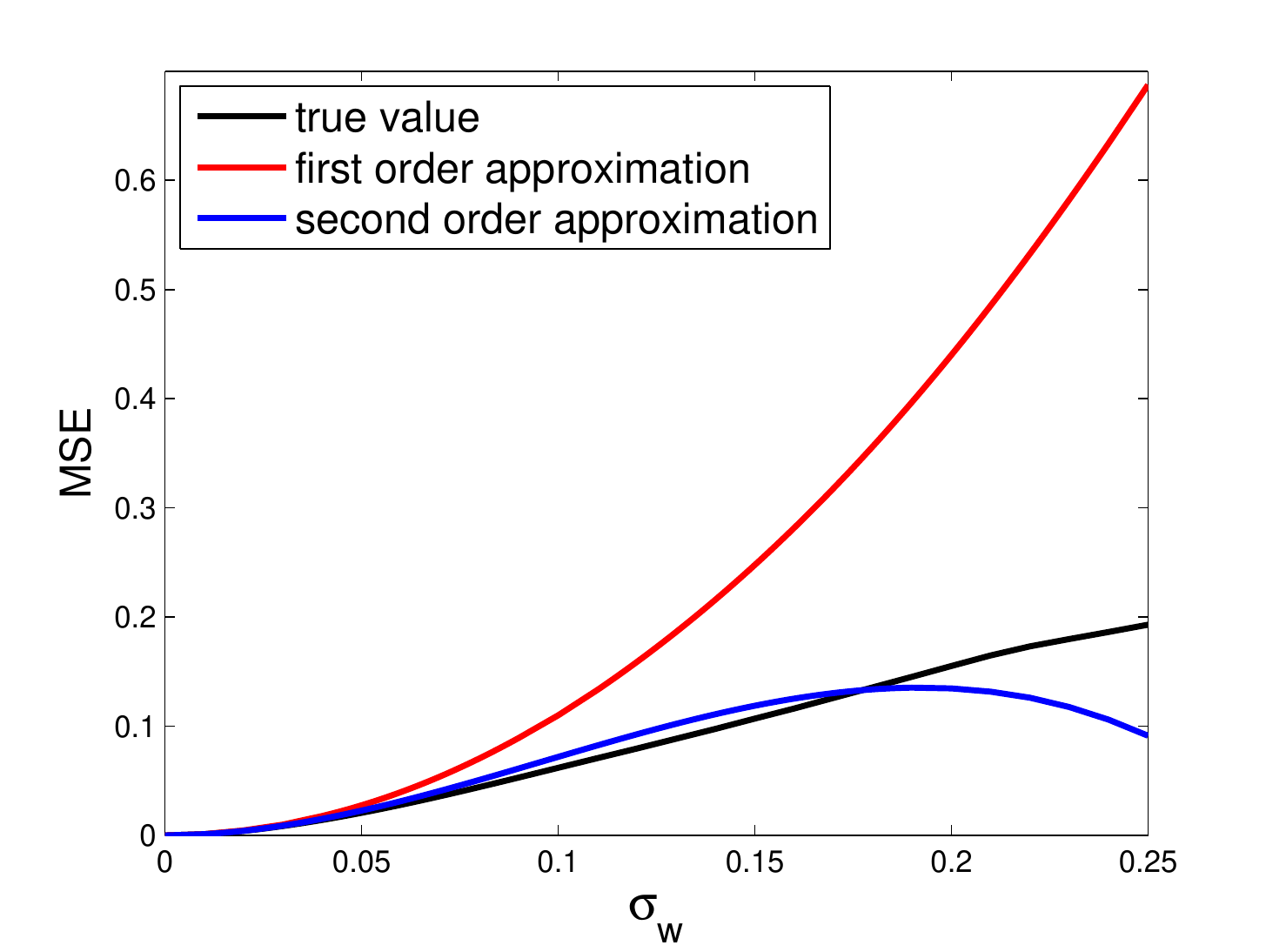}}

	\caption{Plots of actual AMSE and its approximations for $q=1.5, \epsilon =0.7$ with (a) $\delta = 5$, (b) $\delta= 2$, (c) $\delta= 1.5$ and (d) $\delta = 1.1$. }
	
	\label{fig:deltachange}
\end{figure}

\item In our second simulation, we fix $\delta =5, \epsilon=0.4$ and let $q \in \{1.8, 1.5, 1.1\}$. All the simulation results are summarized in Figure \ref{fig:qchange}. As we expected, the first order approximation becomes less accurate when $q$ decreases. Furthermore, we notice that when $q$ is very close to $1$ (check $q=1.1$ in the figure), even the second order approximation is not necessarily good. This again calls for higher order approximation of the AMSE.  

\begin{figure} 
	\centering
	\subfloat[][]{\includegraphics[width=1.7in]{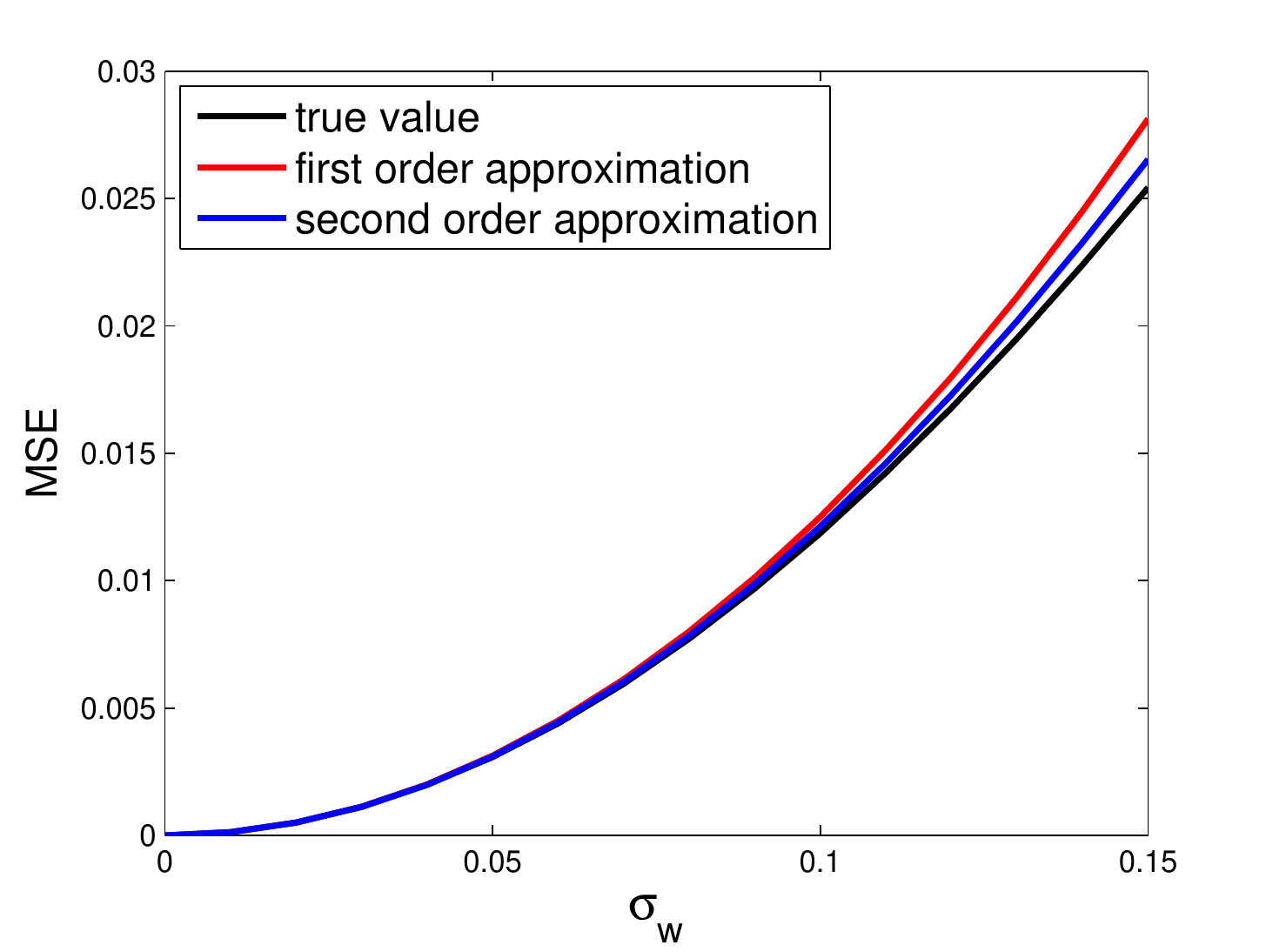}}
	\subfloat[][]{\includegraphics[width=1.7in]{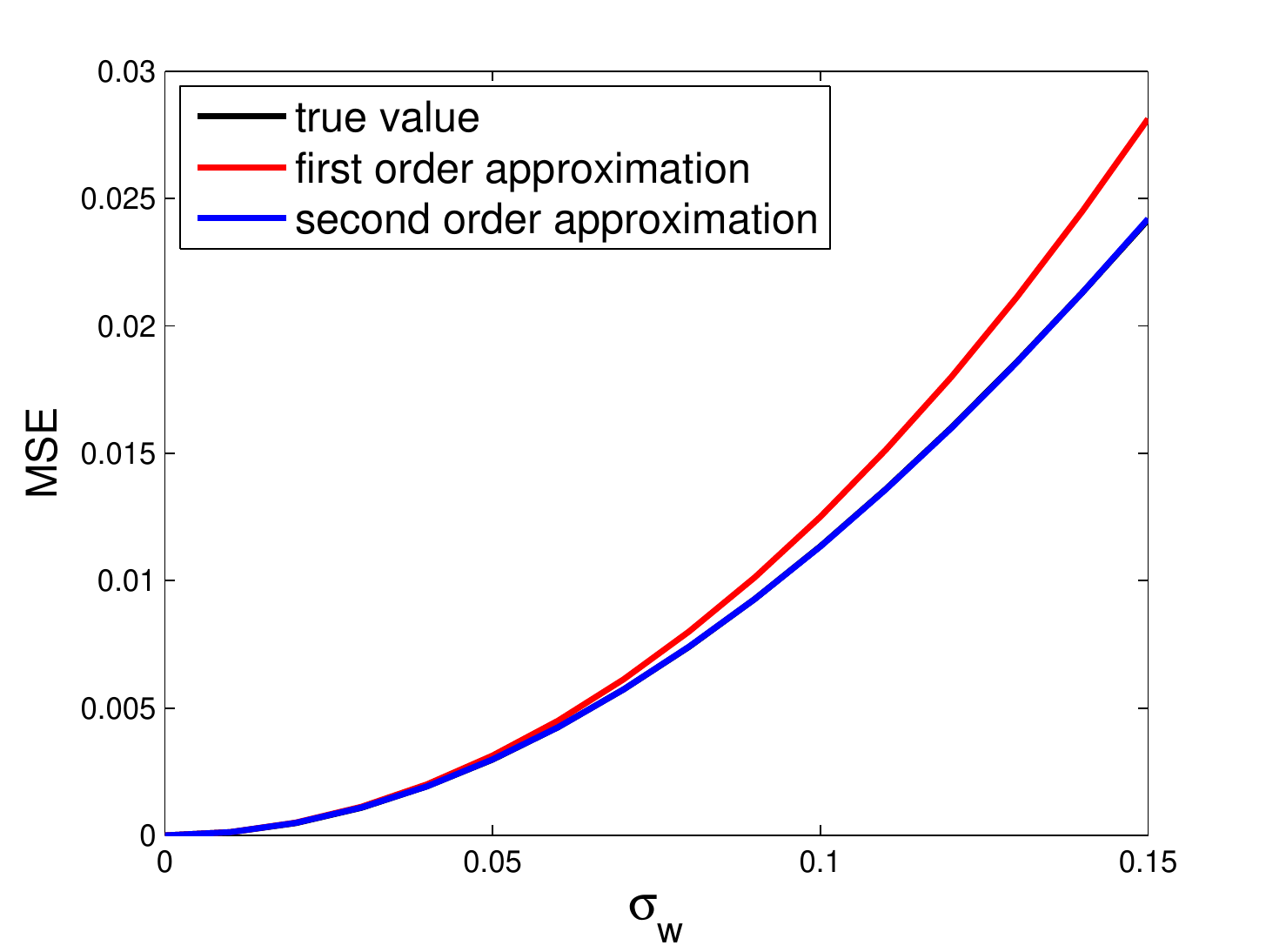}}
	\subfloat[][]{\includegraphics[width=1.7in]{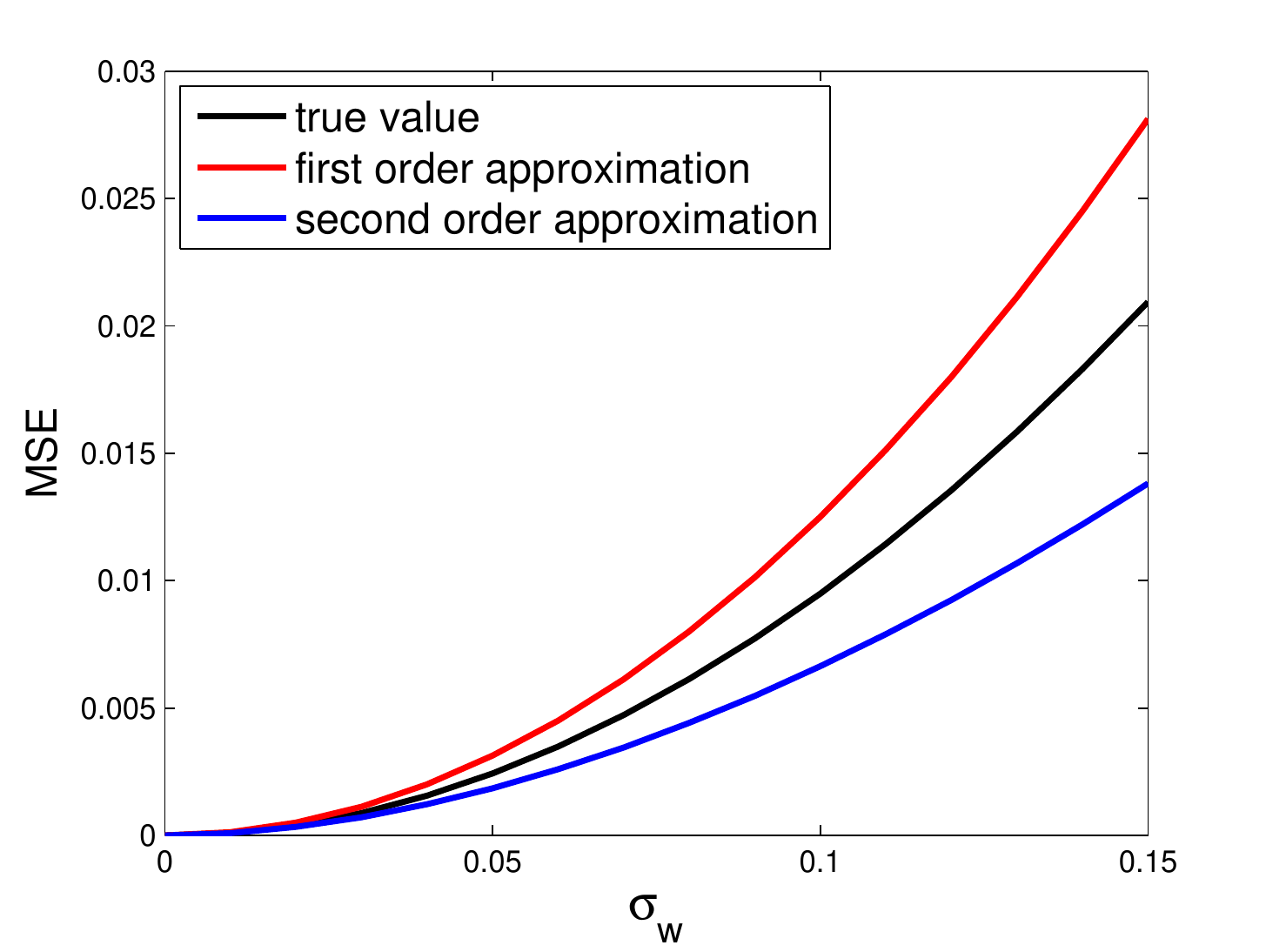}}
	\caption{Plots of actual AMSE and its approximations for $\delta=5, \epsilon =0.4$ with (a) $q = 1.8$, (b) $q= 1.5$, and (c) $q= 1.1$.  }
	\label{fig:qchange}
\end{figure}

\item For the last simulation, we fix $\delta=5$, $q=1.8$, and let $\epsilon \in \{0.7,0.5, 0.3, 0.1\}$. Our simulation results are presented in Figure \ref{fig:epschange}. We see that as $\epsilon$ decreases the first order approximation becomes less accurate. The second order approximation is always better than the first one. Moreover, we observe that when $\epsilon$ is very close to $0$ (check $\epsilon=0.1$ in the figure), even the second order approximation is not necessarily sufficient. As we discussed in the previous two simulations, we might need higher order approximation of the AMSE in such cases. 
   
\end{enumerate}
\begin{figure} 
	\centering
	\subfloat[][]{\includegraphics[width=1.7in]{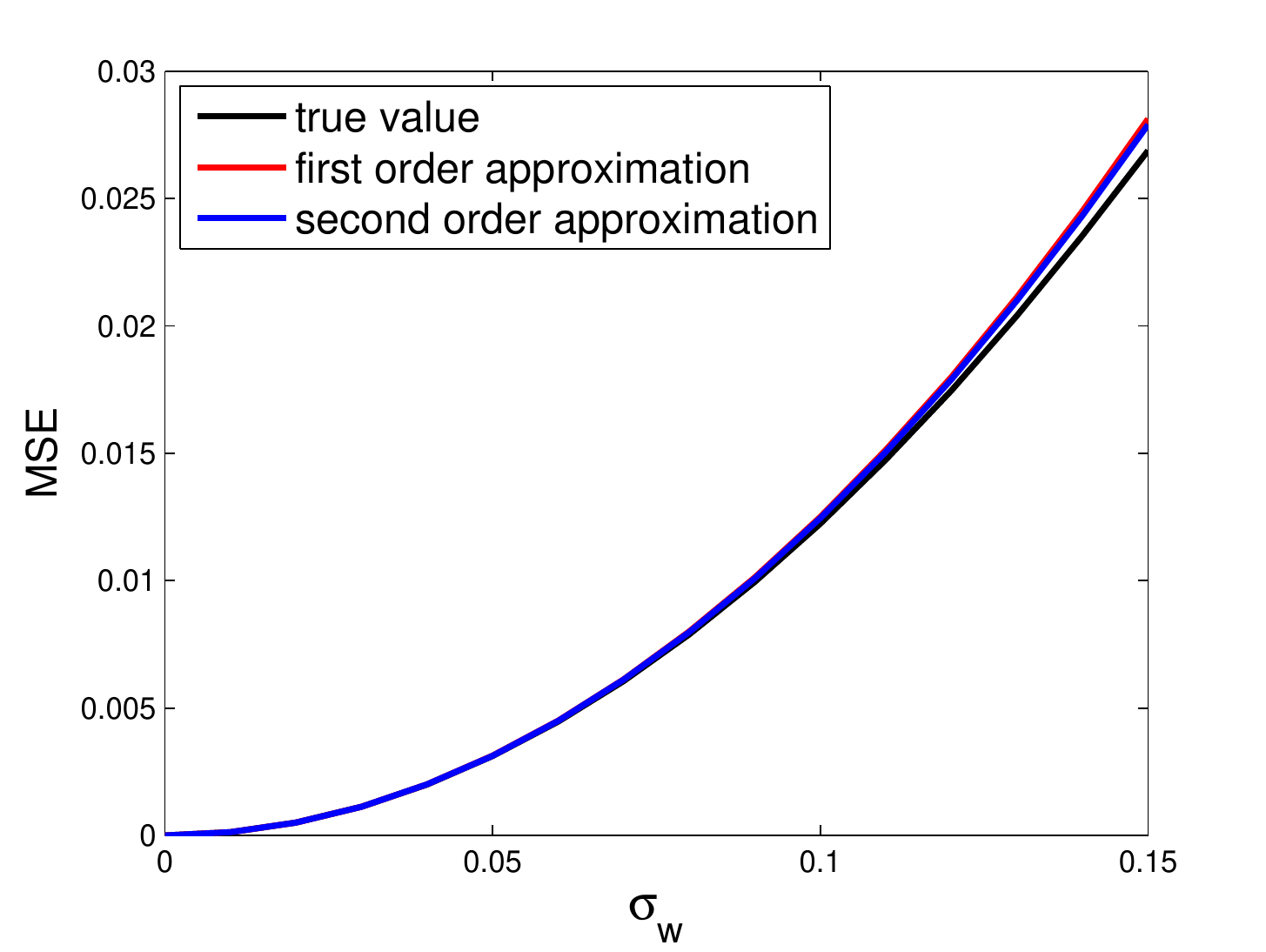}}
	\subfloat[][]{\includegraphics[width=1.7in]{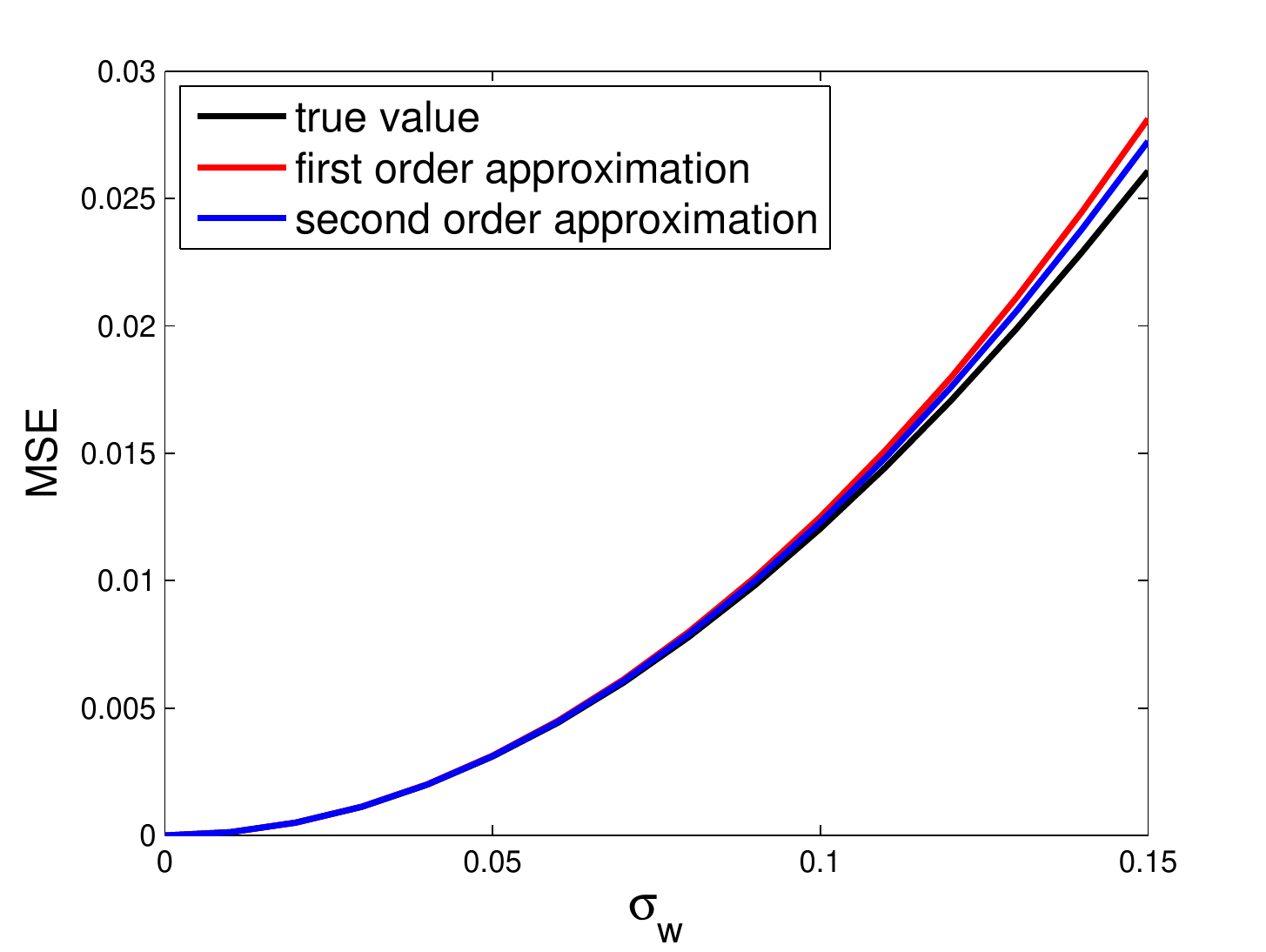}}

	\subfloat[][]{\includegraphics[width=1.7in]{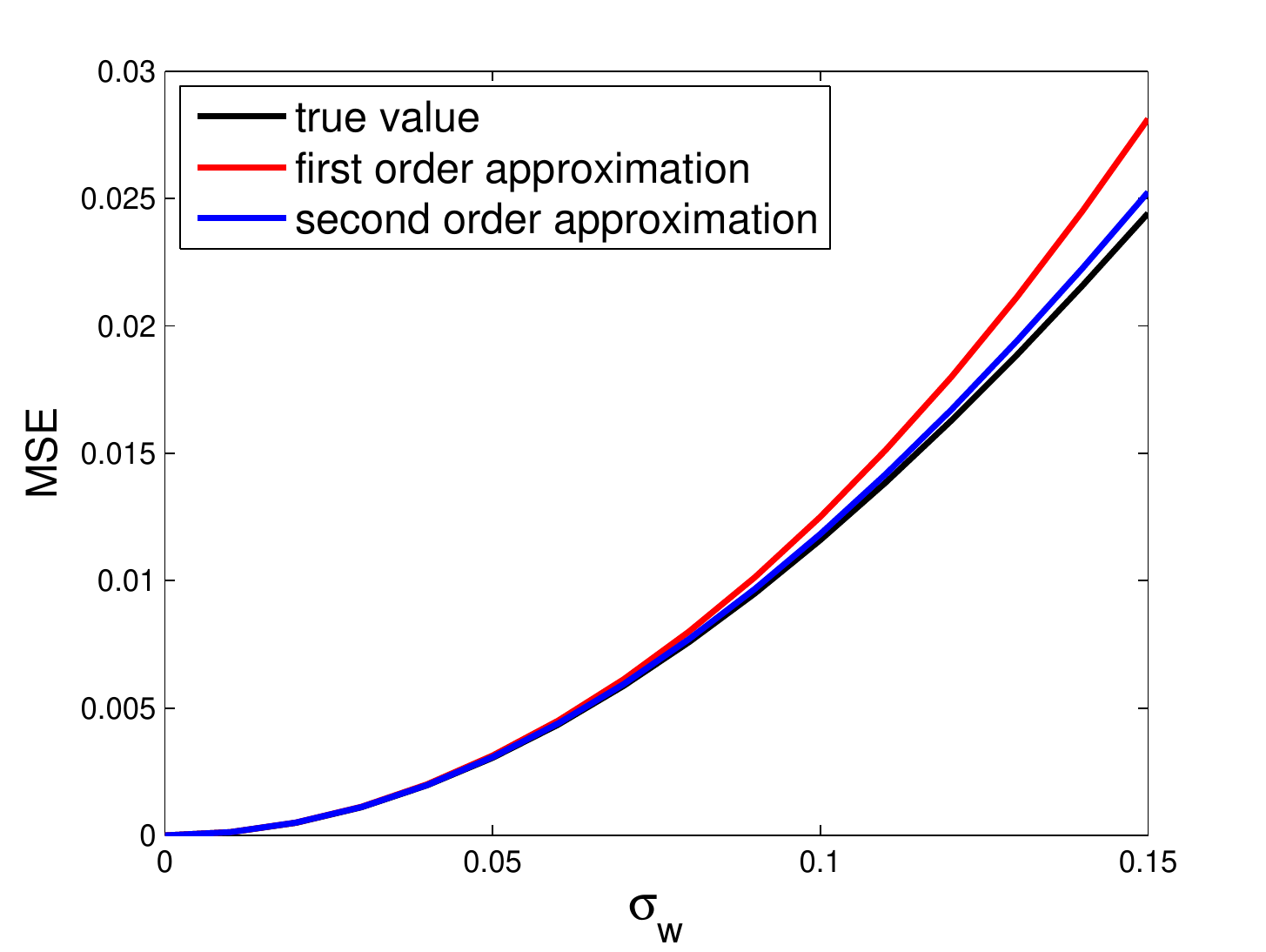}}
	\subfloat[][]{\includegraphics[width=1.7in]{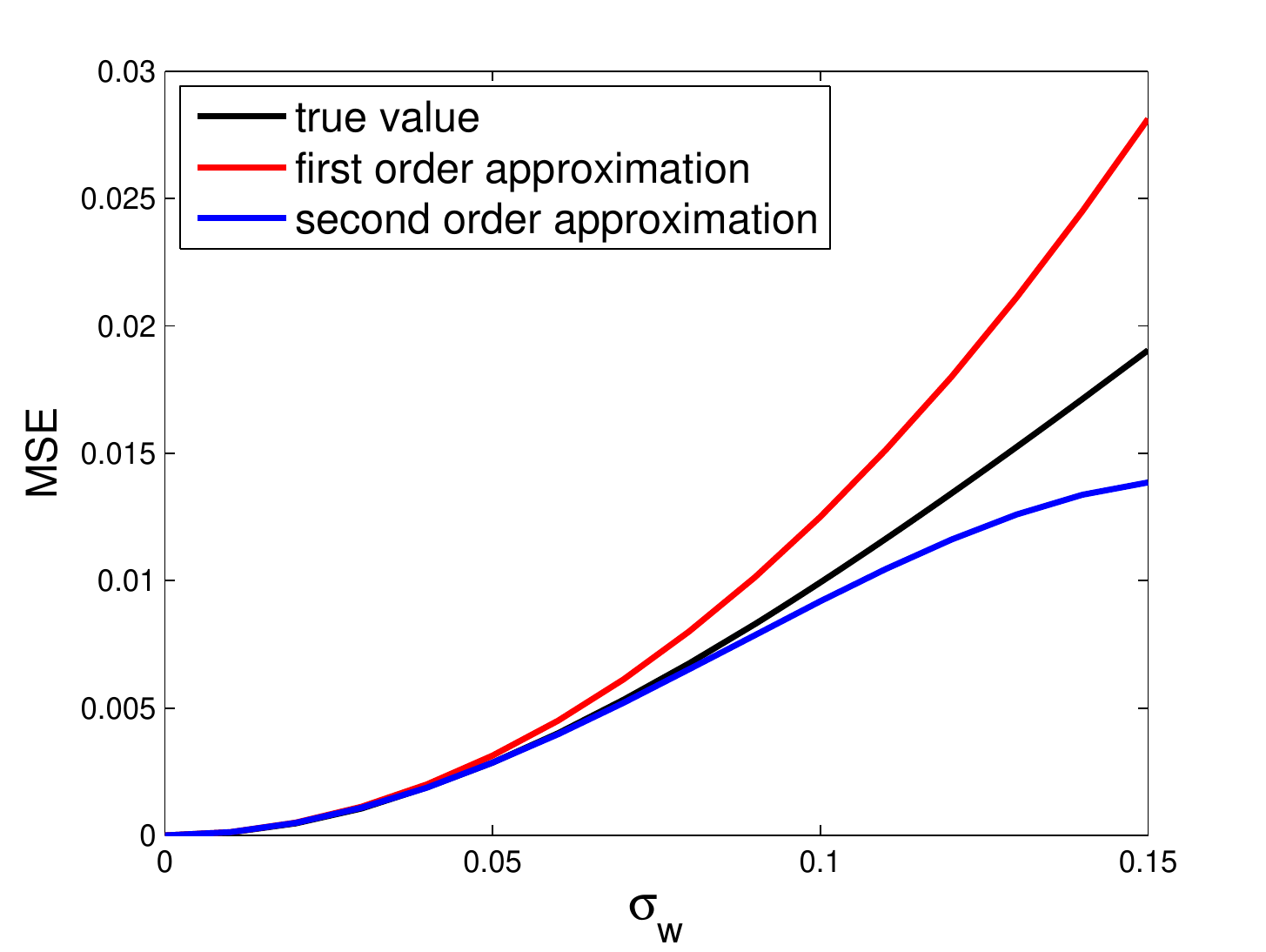}}
	\caption{Plots of actual AMSE and its approximations for $\delta=5$ and $q=1.8$ with (a) $\epsilon =0.7$, (b) $\epsilon= 0.5$,  (c) $\epsilon= 0.3$, and (d) $\epsilon=0.1$.}
	\label{fig:epschange}
\end{figure}
\subsection{Discussion}
Firstly, our numerical studies confirm that the first order term gives good approximations of AMSE for LASSO in the case where the distribution of non-zero elements of $\beta$ is bounded away from zero. Secondly, as the numerical results for $q>1$ demonstrate, while the second order approximation always improves over the first order term and works well in many cases, in the following situations it may not provide very accurate evaluation of AMSE: (i) when $\delta$ is close to $1$, (ii) $\epsilon$ is close to zero, and (iii) $q$ is close to $1$. In such cases, the value of the second order term becomes large and hence the approximation is only accurate for very small value of $\sigma_w$. The remedy that one can propose is to derive higher order expansions. Such terms can be calculated with the same strategy that we used to obtain the second dominant term.

\section{Proof of Theorem \ref{asymp:sparsel1_1}} \label{sec:proofthm4full}

Due to the limited space in the main text, we only present the proof of Theorem \ref{asymp:sparsel1_1}, one of our main results for LASSO, in this section. The proofs of all the other results are deferred to the supplementary material. Although some techniques used in the proofs  are quite different for LASSO and LQLS with $q \in (1,2]$, the roadmap remains the same. Hence we suggest readers to first read the proof in this section. Once this relatively simple proof is clear, the other more complicated proofs will be easier to read.

\subsection{Roadmap of the proof}

Since the proof of this result has several steps and is long, we lay out the roadmap of the proof here to help readers navigate through the details. According to Corollary \ref{thm:mseoptimalasymptot} (let us accept Corollary \ref{thm:mseoptimalasymptot} for the moment; its proof will be fully presented in Appendix \ref{sec:optimaltuning1} of the supplement), in order to evaluate AMSE$(\lambda_{*,1}, 1, \sigma_w)$ as $\sigma_w \rightarrow 0$, the crucial step is to characterize $\bar{\sigma}$ from the following equation
\begin{eqnarray}\label{roadmap:eq1}
\bar{\sigma}^2 = \sigma_{\omega}^2+\frac{1}{\delta} \min_{\chi \geq 0} \mathbb{E}_{B, Z} [(\eta_1(B +\bar{\sigma} Z; \chi) -B)^2].
\end{eqnarray} 
To study \eqref{roadmap:eq1}, the key part is to analyze the term $\min_{\chi \geq 0} \mathbb{E}_{B, Z} [(\eta_1(B +\bar{\sigma} Z; \chi) -B)^2]$. A useful fact that we will prove in Section \ref{sec:lassofixedpointproof} can simplify the analysis of  \eqref{roadmap:eq1}: The condition $\delta > M_1(\epsilon)$ implies that $\bar{\sigma} \rightarrow 0$, as $\sigma_w \rightarrow 0$. Hence one of the main steps of this proof is to derive the convergence rate of $\min_{\chi \geq 0}  \mathbb{E}_{B, Z} [(\eta_1(B +\sigma Z; \chi) -B)^2]$, as $\sigma \rightarrow 0$. Once we obtain that rate, we then characterize the convergence rate for $\bar{\sigma}$ as $\sigma_w \rightarrow 0$ from \eqref{roadmap:eq1}. Finally we connect $\bar{\sigma}$ to AMSE$(\lambda_{*,1}, 1, \sigma_w)$ based on Corollary \ref{thm:mseoptimalasymptot}, and derive the expansion for AMSE$(\lambda_{*,1}, 1, \sigma_w)$ as $\sigma_w \rightarrow 0$. We introduce the following notations:
\begin{eqnarray*}
R(\chi, \sigma)= \mathbb{E}_{B,Z} [(\eta_1(B/\sigma + Z; \chi) -B/\sigma)^2],  \quad \chi^*(\sigma)=\argmin_{\chi \geq 0} R(\chi, \sigma),
\end{eqnarray*}
where we have suppressed the subscript $B, Z$ in $\mathbb{E}$ for notational simplicity. According to \cite{mousavi2015consistent}, $R(\chi, \sigma)$ is a quasi-convex function of $\chi$ and has a unique global minimizer. Hence $\chi^*(\sigma)$ is well defined. It is straightforward to confirm 
\[
\min_{\chi \geq 0} \mathbb{E}_{B, Z} [(\eta_1(B +\sigma Z; \chi) -B)^2]=\sigma^2 R(\chi^*(\sigma),\sigma). 
\]
Throughout the proof, we may write $\chi^*$ for $\chi^*(\sigma)$ when no confusion is caused, and we use $F(g)$ to denote the distribution function of $|G|$. The rest of the proof of Theorem \ref{asymp:sparsel1_1} is organized in the following way:
\begin{enumerate}
\item We first prove $R(\chi^*(\sigma),\sigma) \rightarrow M_1(\epsilon)$, as $\sigma \rightarrow 0$ in Section \ref{ssec:zeroconv1}.
\item We further bound the convergence rate of $R(\chi^*(\sigma),\sigma)$ in Section \ref{ssec:zeroconv2}.
\item We finally utilize the convergence rate bound derived in Section \ref{ssec:zeroconv2} to characterize the convergence rate of $\bar{\sigma}$ and then derive the expansion for AMSE$(\lambda_{*,1}, 1, \sigma_w)$ in Section \ref{sec:lassofixedpointproof}.
\end{enumerate}

\subsection{Proof of $R(\chi^*(\sigma),\sigma) \rightarrow M_1(\epsilon)$, as $\sigma \rightarrow 0$} \label{ssec:zeroconv1}

Our goal in this section is to prove the following lemma. 
\begin{lemma}\label{lassotune}
Suppose $\mathbb{E}|G|^2<\infty$, then $\lim_{\sigma \rightarrow 0}\chi^*(\sigma)=\chi^{**}$ and 
\begin{eqnarray*}
\lim_{\sigma \rightarrow 0}R(\chi^*(\sigma),\sigma)= (1-\epsilon)\mathbb{E} (\eta_1(Z; \chi^{**}))^2+ \epsilon(1+ (\chi^{**})^2),
\end{eqnarray*}
where $\chi^{**}$ is the unique minimizer of $(1-\epsilon)\mathbb{E} (\eta_1(Z; \chi))^2+ \epsilon(1+ \chi^2)$ over $[0, \infty)$, and $Z \sim N(0,1)$.
\end{lemma}
\begin{proof}
By taking derivatives, it is straightforward to verify that $(1-\epsilon)\mathbb{E} (\eta_1(Z; \chi))^2+ \epsilon(1+ \chi^2)$, as a function of $\chi$ over $[0,\infty)$, is strongly convex and has a unique minimizer. Hence $\chi^{**}$ is well defined. 

We first claim that $\chi^*(\sigma_n)$ is bounded for any given sequence $\sigma_n \rightarrow 0$. Otherwise there exists an unbounded subsequence $\chi^*(\sigma_{n_k}) \rightarrow +\infty$ with $\sigma_{n_k}\rightarrow 0$. Since the distribution of $G$ does not have point mass at zero and
\[
\eta_1(G/\sigma_{n_k}+Z;\chi^*(\sigma_{n_k}))={\rm sign}(G/\sigma_{n_k}+Z)(|G/\sigma_{n_k}+Z|-\chi^*(\sigma_{n_k}))_+,
\] 
 it is not hard to conclude that 
 \[
 |\eta_1(G/\sigma_{n_k}+Z;\chi^*(\sigma_{n_k}))-G/\sigma_{n_k}| \rightarrow +\infty, a.s. 
 \]
 By Fatou's lemma, we then have
 \begin{eqnarray}
 \hspace{1.cm} R(\chi^*(\sigma_{n_k}),\sigma_{n_k}) \geq \epsilon \mathbb{E}(\eta_1(G/\sigma_{n_k}+Z;\chi^*(\sigma_{n_k}))-G/\sigma_{n_k})^2  \rightarrow +\infty.  \label{contra:one}
 \end{eqnarray}
 On the other hand, the optimality of $\chi^*(\sigma_{n_k})$ implies
 \[
 R(\chi^*(\sigma_{n_k}),\sigma_{n_k}) \leq R(0, \sigma_{n_k})=1,
 \]
contradicting the unboundedness in \eqref{contra:one}. 

We next show the sequence $\chi^*(\sigma_n)$ converges to a finite constant, for any $\sigma_n \rightarrow 0$. Taking a convergent subsequence $\chi^*(\sigma_{n_k})$, due to the boundedness of $\chi^*(\sigma_n)$, the limit of the subsequence is finite. Call it $\tilde{\chi}$. Note that
\begin{eqnarray}
\lefteqn{\mathbb{E}  (\eta_1(G/\sigma_{n_k} +Z; \chi^*(\sigma_{n_k})) -G/\sigma_{n_k})^2} \nonumber \\ 
&=& 1+ \mathbb{E} (\eta_1(G/\sigma_{n_k} +Z; \chi^*(\sigma_{n_k})) -G/\sigma_{n_k}-Z)^2+ \nonumber \\
&& 2\mathbb{E} Z(\eta_1(G/\sigma_{n_k} +Z; \chi^*(\sigma_{n_k})) -G/\sigma_{n_k}-Z). \nonumber
\end{eqnarray}
Since $\eta_1(u;\chi)={\rm sign}(u)(|u|-\chi)_+$, we have the following three inequalities:
\begin{eqnarray}
&&|\eta_1(Z;\chi^*(\sigma_{n_k}))|^2 \leq |Z|^2, \nonumber \\
&&(\eta_1(G/\sigma_{n_k} +Z; \chi^*(\sigma_{n_k})) -G/\sigma_{n_k}-Z)^2 \leq (\chi^*(\sigma_{n_k}))^2, \nonumber \\
&&| Z( \eta_1(G/\sigma_{n_k} +Z; \chi^*(\sigma_{n_k})) -G/\sigma_{n_k}-Z)| \leq |Z| \chi^*(\sigma_{n_k}). \nonumber
\end{eqnarray}
Furthermore, all the terms on the right hand side of the above inequalities are integrable. Therefore we can apply the Dominated Convergence Theorem (DCT) to obtain
\begin{eqnarray*}
&&\lim_{n_k \rightarrow \infty} R(\chi^*(\sigma_{n_k}),\sigma_{n_k}) \\
&&=\lim_{n_k \rightarrow \infty} (1- \epsilon) \mathbb{E} (\eta_1 (Z; \chi^*(\sigma_{n_k})))^2 + \epsilon \mathbb{E}(\eta_1(G/\sigma_{n_k} +Z; \chi^*(\sigma_{n_k})) -G/\sigma_{n_k})^2 \\
&&= (1-\epsilon)\mathbb{E} (\eta_1(Z; \tilde{\chi}))^2+ \epsilon(1+ \tilde{\chi}^2).
\end{eqnarray*}
Moreover, since $\chi^*(\sigma_{n_k})$ is the optimal threshold value for $R(\chi, \sigma_{n_k})$, 
\[
\lim_{n_k \rightarrow \infty} R(\chi^*(\sigma_{n_k}),\sigma_{n_k}) \leq \lim_{n_k \rightarrow \infty}R(\chi^{**},\sigma_{n_k})= (1-\epsilon)\mathbb{E} (\eta_1(Z; \chi^{**}))^2+ \epsilon(1+ (\chi^{**})^2)
\] 
Combining the last two limiting results, we can conclude $\tilde{\chi}=\chi^{**}$.
  Since $\chi^*(\sigma_{n_k})$ is an arbitrary convergent subsequence, this implies that the sequence $\chi^*(\sigma_n)$ converges to $\chi^{**}$ as well. This is true for any $\sigma_n \rightarrow 0$, hence $\chi^*(\sigma) \rightarrow \chi^{**}$, as $\sigma \rightarrow 0$. $\lim_{\sigma \rightarrow 0} R(\chi^*(\sigma), \sigma)$ can then be directly derived.
\end{proof}

\subsection{Bounding the convergence rate of $R(\chi^*(\sigma), \sigma)$} \label{ssec:zeroconv2}

In Section \ref{ssec:zeroconv1} we have shown $R(\chi^*(\sigma), \sigma) \rightarrow M_1(\epsilon)$ as $\sigma \rightarrow 0$. Our goal in this section is to bound the difference $R(\chi^*(\sigma), \sigma)- M_1(\epsilon)$. For that purpose, we first bound the convergence rate of $\chi^*(\sigma)$.

\begin{lemma}\label{l1:optimal_rate}
Suppose $\mathbb{P}(|G| \geq \mu) = 1$ with $\mu$ being a positive constant and $\mathbb{E}|G|^2<\infty$, then as $\sigma \rightarrow 0$
\[
|\chi^*(\sigma) - \chi^{**}| = O(\phi(-\mu/\sigma + \chi^{**})),
\]
where $\phi(\cdot)$ is the density function of the standard normal.
\end{lemma}
\begin{proof}
Since $\chi^*(\sigma)$ minimizes $R(\chi, \sigma)$, we have $\frac{\partial R(\chi^*(\sigma),\sigma)}{\partial \chi}=0$, which gives the following expression for $\chi^*(\sigma)$: 
\begin{eqnarray*}
\chi^*(\sigma)=\frac{2(1- \epsilon) \phi(\chi^*) + \epsilon \mathbb{E}\phi(\chi^*-G/\sigma)+ \epsilon \mathbb{E}\phi(\chi^*+G/\sigma)}{2(1- \epsilon)\int_{\chi^*}^{\infty}\phi(z)dz + \epsilon \mathbb{E} \int_{\chi^*-G/\sigma}^{\infty}\phi(z)dz+ \epsilon \mathbb{E} \int_{-\infty}^{-\chi^*-G/\sigma}\phi(z)dz }.   
\end{eqnarray*}
Letting $\sigma$ go to zero on both sides in the above equation, we then obtain
\begin{eqnarray*}
\hspace{-0cm} \chi^{**} =\frac{ 2(1- \epsilon) \phi(\chi^{**}) }{2(1- \epsilon)\int_{\chi^{**}}^{\infty}\phi(z)dz + \epsilon}, \hspace{-0cm}
\end{eqnarray*}
where we have applied Dominated Convergence Theorem (DCT). To bound $|\chi^*(\sigma) - \chi^{**}|$, we first bound  the convergence rate of the terms in the expression of $\chi^*(\sigma)$. A direct application of the mean value theorem leads to
\begin{eqnarray}
&&\phi(\chi^*)-\phi(\chi^{**})= (\chi^{**}- \chi^{*}) \tilde{\chi} \phi(\tilde{\chi}), \label{mean:one}\\
&&\int_{\chi^*}^{\infty}\phi(z)dz-\int_{\chi^{**}}^{\infty}\phi(z)dz= (\chi^{**}- \chi^{*})  \phi(\tilde{\tilde{\chi}}),\label{mean:two}
\end{eqnarray}
with $\tilde{\chi}, \tilde{\tilde{\chi}}$ being two numbers between  $\chi^*$ and $\chi^{**}$. We now consider the other four terms. By the condition $\mathbb{P}(|G| \geq \mu)=1$, we can conclude that for sufficiently small $\sigma$
\begin{eqnarray}
&& \mathbb{E}\phi(\chi^*-G/\sigma)\leq \mathbb{E}\phi(\chi^*-|G|/\sigma) \leq \phi(\mu/\sigma - \chi^*), \label{eq:e2_1} \\
&& \mathbb{E}\phi(\chi^*+G/\sigma)\leq \mathbb{E}\phi(\chi^*-|G|/\sigma) \leq \phi(\mu/\sigma - \chi^*).\label{eq:e2_2}
\end{eqnarray}
Moreover, it is not hard to derive 
\begin{eqnarray}\label{eq:e1firstlasso}
&&1-\mathbb{E} \int_{\chi^*-G/\sigma}^{\infty}\phi(z)dz- \mathbb{E} \int_{-\infty}^{-\chi^*-G/\sigma}\phi(z)dz  \\
&= &  \int_{0}^\infty \int_{ - \chi^*-g/\sigma}^{\chi^*-g/\sigma} \phi(z) dzdF(g)\leq  \int_{-\chi^*-\mu/\sigma}^{\chi^*-\mu/\sigma}\phi(z)dz \leq 2\chi^*\phi(\mu/\sigma-\chi^*), \nonumber
\end{eqnarray}
where to obtain the last two inequalities we have used the condition $\mathbb{P}(|G| \geq \mu)=1$ and the fact $\chi^*-\mu/\sigma <0$ for $\sigma$ small enough. We are now in the position to bound $|\chi^*(\sigma) - \chi^{**}|$. Define the following notations:
\begin{eqnarray*}
e_1 &  \triangleq& \epsilon \mathbb{E} \int_{\chi^*-G/\sigma}^{\infty}\phi(z)dz+ \epsilon \mathbb{E} \int_{-\infty}^{-\chi^*-G/\sigma}\phi(z)dz-\epsilon, \nonumber \\
e_2 &\triangleq& \epsilon\mathbb{E}\phi(\chi^*-G/\sigma)+ \epsilon \mathbb{E}\phi(\chi^*+G/\sigma), \\
S &\triangleq& 2(1- \epsilon) \phi(\chi^{**}), \quad T \triangleq 2(1- \epsilon)\int_{\chi^{**}}^{\infty}\phi(z)dz + \epsilon.
\end{eqnarray*}
Using the new notations and Equations \eqref{mean:one} and \eqref{mean:two}, we obtain
\begin{eqnarray*}
\chi^*(\sigma) = \frac{S+ 2(1-\epsilon) (\chi^{**} - \chi^* ) \tilde{\chi} \phi(\tilde{\chi}) +e_2 }{T+ 2(1-\epsilon) (\chi^{**} - \chi^* ) \phi(\tilde{\tilde{\chi}}) +e_1}, \quad \chi^{**} = \frac{S}{T}.
\end{eqnarray*}
Hence we can do the following calculations:
\begin{eqnarray}\label{eq:lastequationtauell_1}
\chi^*(\sigma)- \chi^{**} &=&  \frac{S+ 2(1-\epsilon) (\chi^{**} - \chi^* ) \tilde{\chi} \phi(\tilde{\chi}) + e_2 }{T+ 2(1-\epsilon) (\chi^{**} - \chi^* ) \phi(\tilde{\tilde{\chi}}) +e_1}- \frac{S}{T} \nonumber \\
&=& \frac{ 2(1-\epsilon) (\chi^{**} - \chi^* ) \tilde{\chi} \phi(\tilde{\chi}) +  e_2 }{T+ 2(1-\epsilon) (\chi^{**} - \chi^* ) \phi(\tilde{\tilde{\chi}}) + e_1}- \nonumber \\
&& \frac{S(2(1-\epsilon) (\chi^{**} - \chi^* ) \phi(\tilde{\tilde{\chi}}) + e_1) }{T(T+ 2(1-\epsilon) (\chi^{**} - \chi^* ) \phi(\tilde{\tilde{\chi}}) + e_1)} \nonumber \\
&=&  \frac{ 2(1-\epsilon) (\chi^{**} - \chi^* )( \tilde{\chi} \phi(\tilde{\chi}) -\chi^{**} \phi(\tilde{\tilde{\chi}})) }{T+ 2(1-\epsilon) (\chi^{**} - \chi^* ) \phi(\tilde{\tilde{\chi}}) +e_1} + \nonumber \\
&& \frac{e_2 - \chi^{**} e_1}{T+ 2(1-\epsilon) (\chi^{**} - \chi^* ) \phi(\tilde{\tilde{\chi}}) +e_1}. 
\end{eqnarray}
From \eqref{eq:lastequationtauell_1} we obtain
\begin{eqnarray} 
&&(\chi^*(\sigma) - \chi^{**}) \left(1+\frac{ 2(1-\epsilon)( \tilde{\chi} \phi(\tilde{\chi}) -\chi^{**} \phi(\tilde{\tilde{\chi}}))  }{T+ 2(1-\epsilon) (\chi^{**} - \chi^*(\sigma) ) \phi(\tilde{\tilde{\chi}}) +e_1} \right ) \label{diff} \\
&=&  \frac{e_2 - \chi^{**} e_1}{T+ 2(1-\epsilon) (\chi^{**} - \chi^*(\sigma) ) \phi(\tilde{\tilde{\chi}}) +e_1}.  \nonumber
\end{eqnarray}
Note that in the above expression we have $\tilde{\chi} \rightarrow \chi^{**}$ and $\tilde{\tilde{\chi}} \rightarrow \chi^{**}$ since $\chi^*(\sigma) \rightarrow \chi^{**}$. Therefore, we conclude that $\tilde{\chi} \phi(\tilde{\chi}) -\chi^{**} \phi(\tilde{\tilde{\chi}}) \rightarrow 0$ and $(\chi^{**} - \chi^*(\sigma) ) \phi(\tilde{\tilde{\chi}}) \rightarrow 0$. Moreover, since \eqref{eq:e2_1}, \eqref{eq:e2_2} and \eqref{eq:e1firstlasso} together show both $e_1$ and $e_2$ go to 0 exponentially fast, we conclude from \eqref{diff} that $(\chi^*(\sigma) - \chi^{**}) /\sigma \rightarrow 0$. This enables us to proceed
\begin{eqnarray*} \label{keydiff}
&&\lim_{\sigma \rightarrow 0} \frac{|\chi^*(\sigma) - \chi^{**}|}{\phi(\mu/\sigma- \chi^{**})} =\lim_{\sigma \rightarrow 0}\frac{|\chi^*(\sigma) - \chi^{**}|}{\phi(\mu/\sigma- \chi^{*})}  \overset{(a)}{=} \lim_{\sigma \rightarrow 0} \frac{|e_2 - \chi^{**} e_1|}{T\phi(\mu/\sigma- \chi^{*})}  \nonumber \\
&\overset{(b)}{\leq}& \lim_{\sigma \rightarrow 0} \frac{2\epsilon (1+\chi^*(\sigma)\chi^{**})\phi(\mu/\sigma- \chi^*)}{T\phi(\mu/\sigma- \chi^{*})} = \frac{2\epsilon(1+(\chi^{**})^2)}{T}. 
\end{eqnarray*}
We have used \eqref{diff} to obtain (a). We derived (b) by the following steps:
\begin{enumerate}
\item According to \eqref{eq:e1firstlasso}, $|e_1| \leq 2\epsilon\chi^*\phi(\mu/\sigma-\chi^*)$. 
\item According to \eqref{eq:e2_1} and \eqref{eq:e2_2}, $|e_2| \leq 2\epsilon \phi (\mu/\sigma - \chi^*)$. 
\end{enumerate} 
This completes the proof of Lemma \ref{l1:optimal_rate}.
\end{proof}

The next step is to bound the convergence rate of $R(\chi^*(\sigma), \sigma)$ based on the convergence rate of $\chi^*(\sigma)$ we have derived in Lemma \ref{l1:optimal_rate}. 

\begin{lemma} \label{risk:lq}
Suppose $\mathbb{P}(|G|\geq \mu)=1$ with $\mu$ being a positive constant and $\mathbb{E}|G|^2<\infty$, then as $\sigma \rightarrow 0$
\[
|R(\chi^*(\sigma),\sigma) - M_1(\epsilon)| = O(\phi(\mu/\sigma- \chi^{**})),
\]
where $\phi(\cdot)$ is the density function of the standard normal.
\end{lemma}
\begin{proof}
We recall the two quantities:
\begin{eqnarray}
M_1(\epsilon)&=&(1-\epsilon)\mathbb{E} (\eta_1(Z; \chi^{**}))^2+ \epsilon(1+ (\chi^{**})^2), \label{risk:ellone_1}\\
R(\chi^*(\sigma),\sigma)&=&(1-\epsilon)\mathbb{E}(\eta_1(Z;\chi^*))^2+ \nonumber \\
&& \epsilon[1+ \mathbb{E} (\eta_1(G/\sigma +Z; \chi^*) -G/\sigma-Z)^2] \nonumber \\
&&+ 2\epsilon \mathbb{E} Z(\eta_1(G/\sigma +Z; \chi^*) -G/\sigma-Z). \label{eq:m1eps1proof}
\end{eqnarray}
We bound $|R(\chi^*(\sigma),\sigma) - M_1(\epsilon)|$ by bounding the difference between the corresponding terms in \eqref{eq:m1eps1proof} and \eqref{risk:ellone_1}. From the proof of Lemma \ref{l1:optimal_rate} we know $e_1<0$ and $e_2>0$. Hence \eqref{diff} implies $\chi^*(\sigma)>\chi^{**}$ for small enough $\sigma$. We start with
\begin{eqnarray}
|\mathbb{E} (\eta_1 (Z; \chi^*)) ^2- \mathbb{E} (\eta_1(Z; \chi^{**}))^2| \label{risk:ell1term1firstcase}\hspace{-8cm} \\
&=& |\mathbb{E} (\eta_1(Z; \chi^*) - \eta_1(Z; \chi^{**}))(\eta_1(Z; \chi^*) + \eta_1(Z; \chi^{**}))| \nonumber \\
&\leq& \mathbb{E} [|  \chi^*- \chi^{**}+ \chi^{*}\mathbb{I}(|Z| \in (\chi^{**}, \chi^*))|\cdot |\eta_1(Z; \chi^*) + \eta_1(Z; \chi^{**})| ] \nonumber \\
&\overset{(a)}{\leq}& 2(\chi^*- \chi^{**}) \cdot \mathbb{E}|Z|+2\chi^{*}\mathbb{E}[\mathbb{I}(|Z| \in (\chi^{**}, \chi^*))|Z|] \nonumber \\
&\leq& 2(\chi^*- \chi^{**}) \cdot \mathbb{E}|Z|+4\chi^{*}(\chi^*- \chi^{**})\tilde{\chi}\phi(\tilde{\chi})  =O(\phi(\mu/\sigma - \chi^{**})), \nonumber
\end{eqnarray}
where we have used the fact $|\eta_1(u;\chi)|\leq |u|$ to obtain $(a)$; $\tilde{\chi}$ is a number between $\chi^*(\sigma)$ and $\chi^{**}$; and the last equality is due to Lemma \ref{l1:optimal_rate}. We next bound the difference between $\mathbb{E} (\eta_1 (G/\sigma + Z ; \chi^*) - G/ \sigma-Z)^2$ and $(\chi^{**})^2$:
\begin{eqnarray}
 |(\chi^{**})^2- \mathbb{E} (\eta_1 (G/\sigma + Z ; \chi^*) - G/ \sigma-Z)^2| \label{eq:riskl1term2firstcase} \hspace{-8cm} \\
&\leq& |(\chi^{*})^2- \mathbb{E} (\eta_1 (G/\sigma + Z ; \chi^*) - G/ \sigma-Z)^2| + |(\chi^{**})^2- (\chi^{*})^2|.  \nonumber
\end{eqnarray}
To bound the two terms on the right hand side of \eqref{eq:riskl1term2firstcase}, first note that 
\begin{eqnarray}\label{eq:riskl1term2firstcase1}
0 &\leq& (\chi^{*})^2- \mathbb{E} (\eta_1 (G/\sigma + Z ; \chi^*) - G/ \sigma-Z)^2 \nonumber \\
& = &\mathbb{E} [\mathbb{I}(|G/\sigma+Z|\leq \chi^*)\cdot ((\chi^*)^2-(G/\sigma+Z)^2)] \nonumber \\
&\leq &(\chi^*)^2 \int_{0}^{\infty} \int_{-g/\sigma-\chi^{*}}^{-g/\sigma+\chi^{*}} \phi(z)dz dF(g) \nonumber \\
&\overset{(b)}{\leq}& (\chi^*)^2 \int_{-\mu/\sigma-\chi^{*}}^{-\mu/\sigma+\chi^{*}} \phi(z)dz\leq 2(\chi^*)^3\phi(\mu/\sigma-\chi^*) \nonumber \\
&=&O(\phi(\mu/\sigma - \chi^{**})),
\end{eqnarray}
where $(b)$ is due to the condition $\mathbb{P}(|G|\geq \mu)=1$, and the last equality holds since $(\chi^*-\chi^{**})/\sigma \rightarrow 0$ implied by Lemma \ref{l1:optimal_rate}. Furthermore, Lemma \ref{l1:optimal_rate} yields
\begin{eqnarray} \label{eq:riskl1term2firstcase2}
(\chi^*)^2-(\chi^{**})^2=O(\phi(\mu/\sigma - \chi^{**})).
\end{eqnarray}
Combining \eqref{eq:riskl1term2firstcase}, \eqref{eq:riskl1term2firstcase1}, and  \eqref{eq:riskl1term2firstcase2}, we obtain
\begin{equation} \label{eq:riskl1term2firstcasefinal}
|(\chi^{**})^2- \mathbb{E} (\eta_1 (G/\sigma + Z ; \chi^*) - G/ \sigma-Z)^2| = O(\phi(\mu/\sigma - \chi^{**})). 
\end{equation}

Regarding the remaining term in $R(\chi^*(\sigma),\sigma)$, we can derive
\begin{eqnarray}\label{eq:riskthirdtermell1firstcase}
0 &\leq & \mathbb{E} Z(G/\sigma+Z - \eta_1(G/\sigma +Z; \chi^*) )\nonumber \\
&\overset{(c)}{=}&\mathbb{E}(1-\partial_1\eta_1(G/\sigma+Z;\chi^*)) \nonumber \\
&=& \mathbb{P}(|G/\sigma+Z|\leq \chi^*)\overset{(d)}{=}O(\phi(\mu/\sigma - \chi^{**})).
\end{eqnarray}
We have employed Stein's lemma (see Lemma \ref{lem:steins} in the supplement) to obtain $(c)$. Equality $(d)$ holds due to \eqref{eq:e1firstlasso}.
Putting the results \eqref{risk:ell1term1firstcase}, \eqref{eq:riskl1term2firstcasefinal}, and \eqref{eq:riskthirdtermell1firstcase} together finishes the proof.
\end{proof}

\subsection{Deriving the expansion of AMSE$(\lambda_{*,1},1,\sigma_w)$}\label{sec:lassofixedpointproof}

In this section we utilize the convergence rate result of $R(\chi^*(\sigma), \sigma)$ from Section \ref{ssec:zeroconv2} to derive the expansion of AMSE$(\lambda_{*,1},1,\sigma_w)$ in \eqref{asymp:l1}, and thus finish the proof of Theorem \ref{asymp:sparsel1_1}. Towards that goal, we first prove a useful lemma.

\begin{lemma} \label{finallemma}
Let $\bar{\sigma}$ be the solution to the following equation:
\begin{equation}  \label{fixedpoint:eq}
\bar{\sigma}^2 = \sigma_{\omega}^2+\frac{1}{\delta} \min_{\chi \geq 0} \mathbb{E}_{B, Z} [(\eta_1(B +\bar{\sigma} Z; \chi) -B)^2].
\end{equation}
Suppose $\delta > M_1(\epsilon)$, then
\begin{eqnarray*}
\lim_{\sigma_w\rightarrow 0}\frac{\sigma_w^2}{\bar{\sigma}^2}=\frac{\delta-M_1(\epsilon)}{\delta}.
\end{eqnarray*}
\end{lemma}

\begin{proof}
We first claim that $\mathbb{E} (\eta_1(\alpha+Z; \chi)- \alpha)^2$ is an increasing function of $\alpha$, because
\[
\frac{d}{d\alpha} \mathbb{E} (\eta_1(\alpha+Z; \chi)- \alpha)^2 = 2 \mathbb{E}(\alpha \mathbb{I} (|\alpha+Z| \leq \chi) )\geq 0. 
\]
Hence we obtain
\begin{eqnarray}
\hspace{0.6cm} \mathbb{E} (\eta_1(\alpha+Z; \chi)- \alpha)^2 \leq \lim_{\alpha \rightarrow \infty} \mathbb{E} (\eta_1(\alpha+Z; \chi)- \alpha)^2 = 1+\chi^2.  \label{devamse:one}
\end{eqnarray}
Inequality \eqref{devamse:one} then yields
\begin{eqnarray*}
R(\chi,\bar{\sigma})&=&(1-\epsilon)\mathbb{E}(\eta_1(Z;\chi))^2+\epsilon \mathbb{E}(\eta_1(G/\bar{\sigma} + Z;\chi)-G/\bar{\sigma})^2 \nonumber \\
&\leq& (1-\epsilon)\mathbb{E}(\eta_1(Z;\chi))^2+\epsilon (1+\chi^2). 
\end{eqnarray*}
Taking minimum over $\chi$ on both sides above gives us
\begin{eqnarray}
R(\chi^*(\bar{\sigma}), \bar{\sigma})\leq M_1(\epsilon). \label{key:oneone}
\end{eqnarray}
Moreover, since $\bar{\sigma}$ is the solution of \eqref{fixedpoint:eq}, it satisfies
\begin{eqnarray}\label{eq:fixedpointlass1firstcase}
\bar{\sigma}^2=\sigma^2_w+\frac{\bar{\sigma}^2}{\delta}R(\chi^*(\bar{\sigma}),\bar{\sigma}).
\end{eqnarray}
Combining \eqref{key:oneone} and \eqref{eq:fixedpointlass1firstcase} with the condition $\delta > M_1(\epsilon)$, we have
\[
\bar{\sigma}^2\leq \frac{\sigma_w^2}{ 1- M_1(\epsilon)/\delta},
\]
which leads to $\bar{\sigma} \rightarrow 0$, as $\sigma_w \rightarrow 0$. Then applying Lemma \ref{lassotune} shows
\[
\lim_{\sigma_w \rightarrow 0}R(\chi^*(\bar{\sigma}),\bar{\sigma})=\lim_{\bar{\sigma} \rightarrow 0}R(\chi^*(\bar{\sigma}),\bar{\sigma})=M_1(\epsilon).
\]
Diving both sides of \eqref{eq:fixedpointlass1firstcase} by $\bar{\sigma}^2$ and letting $\sigma_w \rightarrow 0$ finishes the proof.
\end{proof}

To complete the proof of Theorem \ref{asymp:sparsel1_1}, first note that Corollary \ref{thm:mseoptimalasymptot} tells us
\begin{eqnarray*}
{\rm AMSE}(\lambda_{*,1},1,\sigma_w)=\bar{\sigma}^2R(\chi^*(\bar{\sigma}), \bar{\sigma}), \quad \sigma_w^2=\bar{\sigma}^2-\frac{\bar{\sigma}^2}{\delta}R(\chi^*(\bar{\sigma}), \bar{\sigma}).
\end{eqnarray*}
We then have
\begin{eqnarray}
&&{\rm AMSE}(\lambda_{*,1},1,\sigma_w)-\frac{\delta M_1(\epsilon)}{\delta-M_1(\epsilon)}\sigma_w^2 \label{term:one}\\
&=&\bar{\sigma}^2R(\chi^*(\bar{\sigma}),\bar{\sigma})- \frac{\delta M_1(\epsilon)}{\delta-M_1(\epsilon)} \cdot \bigg [ \bar{\sigma}^2-\frac{\bar{\sigma}^2}{\delta}R(\chi^*(\bar{\sigma}), \bar{\sigma})\bigg ] \nonumber \\
&=&\frac{\delta(R(\chi^*(\bar{\sigma}),\bar{\sigma})-M_1(\epsilon))}{\delta-M_1(\epsilon)}\bar{\sigma}^2\overset{(a)}{=} O(\bar{\sigma}^2\phi(\mu/\bar{\sigma}-\chi^{**})),  \nonumber 
\end{eqnarray}
where $(a)$ is due to Lemma \ref{risk:lq}. Finally, since $\lim_{\sigma_w\rightarrow 0}\frac{\sigma_w^2}{\bar{\sigma}^2}=\frac{\delta-M_1(\epsilon)}{\delta}$ according to Lemma \ref{finallemma}, it is not hard to see
\begin{eqnarray}\label{smallo:chain}
\hspace{0.4cm} O(\bar{\sigma}^2\phi(\mu/\bar{\sigma}-\chi^{**}))=o(\phi(\bar{\mu}/\bar{\sigma}))=o\Big(\phi \Big(\sqrt{\frac{\delta-M_1(\epsilon)}{\delta}}\frac{\tilde{\mu}}{\sigma_w}\Big)\Big),
\end{eqnarray}
where $\bar{\mu}$ and $\tilde{\mu}$ are any constants satisfying $0 \leq \tilde{\mu}<\bar{\mu}<\mu$. Results \eqref{term:one} and \eqref{smallo:chain} together close the proof of Theorem \ref{asymp:sparsel1_1}. 

\textbf{Remark}: \eqref{key:oneone} and \eqref{term:one} together imply that the second dominant term of AMSE$(\lambda_{*,1},1,\sigma_w)$ is in fact negative.


\section*{Acknowledgements}
This work is supported by NSF grant CCF-1420328.

\bibliographystyle{unsrt}
\bibliography{myrefs}

\input{p_greater_one_supp.tex}

\end{document}

%% file: p_greater_one_supp.tex
\newpage

\appendix
\begin{center}
\textbf{\large{Supplementary material}}
\end{center}

\section{Organization}

This supplement contains simulations and all the proof of lemmas and theorems that have not been covered in the main text. We outline the structure of the supplement to help readers find the materials they are interested in. The organization is as follows:

\begin{enumerate}

\item Appendix \ref{add:simulations} contains additional simulation results. These numerical experiments evaluate the accuracy of AMSE$(\lambda_{*,q},q,\sigma_w)$ as an asymptotic prediction for finite-sample mean square error (MSE). Along this line, the convergence rate of MSE is empirically verified.

\vspace{0.3cm}
\item Appendix \ref{ssec:etaq:summary} covers several important properties of the proximal operator function $\eta_q(u;\chi)$. These properties will later be used in our proofs. It includes Lemmas \ref{lem:toobasicpropprox}, \ref{prox:smooth}, \ref{lem:continuitysecondder}, \ref{lem:toobasicpropproxder} and their proofs. 
\vspace{0.3cm}
\item Appendix \ref{subsec:solutionfixedpoint} proves Lemma \ref{lem:eq7and8solutionexists} from the main text. This lemma is concerned with the uniquesness of the solution of \eqref{eq:fixedpoint11} and \eqref{eq:fixedpoint21}. Lemmas \ref{lem:toobasicpropprox}, \ref{prox:smooth}, \ref{lem:continuitysecondder}, \ref{lem:toobasicpropproxder} from Appendix \ref{ssec:etaq:summary} are extensively used here. 

\vspace{0.3cm}
\item Appendix \ref{sec:optimaltuning1} proves Corollary \ref{thm:mseoptimalasymptot} in the main text. This corollary characterizes AMSE when the regularization parameter $\lambda$ is optimally tuned. This appendix uses Lemma \ref{lem:posfirstlemma} from Appendix \ref{subsec:solutionfixedpoint}. 

\vspace{0.3cm}
\item Appendix \ref{sec:prooftheorem2full} includes the proof of  Theorem \ref{asymp:lqbelowpt}, one of the main results in this paper. The theorem derives the second order expansion of AMSE$(\lambda_{*,q},q,\sigma_w)$ for $q\in (1,2]$. We recommend interested readers to read the proof in Section \ref{sec:proofthm4full} before Appendix \ref{sec:prooftheorem2full}. 

\vspace{0.3cm}
\item Appendix \ref{sec:proofthm3full} contains the proof of Theorem \ref{asymp:lqabovept}. This theorem identifies the necessary condition for successful recovery with $q\in (1,2]$. Phase transition is implied by this theorem together with Theorem \ref{asymp:lqbelowpt}.

\vspace{0.3cm}
\item Appendix \ref{sec:prooflassogeralepsless1} proves Theorem \ref{asymp:sparsel1_2}. The proof of this theorem is along the same lines as the proof of Theorem \ref{asymp:sparsel1_1} presented in Section \ref{sec:proofthm4full} of the main text. We suggest the reader to study Section \ref{sec:proofthm4full} before studying this section. 

\vspace{0.3cm}
\item Appendix \ref{sec:proofthm6full} proves Theorem \ref{them:ellabovpt}. The proof is essentially the same as the proof of Theorem \ref{asymp:lqabovept}. Since we do not repeat the detailed arguments, readers may want to study Appendix \ref{sec:proofthm3full} first.

\vspace{0.3cm}
\item Appendix \ref{ssec:App:firstresult} includes the proof of Theorem \ref{thm:eqpseudolip}. Though Theorem \ref{thm:eqpseudolip} lays the first step towards analyzing the AMSE, we postpone the proof to the end of the supplement. That is because the focus of this paper is on the second order dominant term of AMSE and we do not want to move the proof of this theorem ahead to potentially digress from the highlights. Moreover, since the key ideas of this proof are similar to those of \cite{BaMo11}, we will not detail out the entire proof and rather describe several key steps. Note that in this proof we will use some results we have developed in Appendices \ref{ssec:etaq:summary} and \ref{subsec:solutionfixedpoint}.


\end{enumerate}

\section{Additional numerical results} \label{add:simulations}

This section contains additional numerical studies to examine the relation between AMSE and its finite-sample version. Throughout the simulations, we fix the sparsity level $\epsilon=0.4$ and the distribution of $G$ to $g(b)=0.5\delta_1(b)+0.5\delta_{-1}(b)$. We also let the elements of $X$ be i.i.d from $N(0, 1/n)$, and the noise $w_i \overset{i.i.d}{\sim} N(0, \sigma_w^2)$. The other parameters of our problem instances are set in the following way:
\begin{enumerate}
\item $\delta$ takes values in $\{1.5, 3, 5\}$
\item $p$ takes any value in $\{20, 50, 200, 400\}$
\item $q$ is chosen to be one of $\{1, 1.5, 1.8, 2\}$
\item $\sigma_w$ ranges over $[0, 0.15]$
\end{enumerate}
Given the values of $\epsilon, G, \delta, \sigma_w$ and $q$, we compute AMSE$(\lambda_{*,q},q,\sigma_w)$ according to the formulas in Corollary \ref{thm:mseoptimalasymptot} from the main text. We also generate samples from linear models by setting $n=\delta p$, $\{\beta_i\}_{i=1}^p \overset{i.i.d}{\sim} (1-\epsilon)\delta_0+\epsilon G$. We then calculate the finite-sample MSE (under optimal tuning): $\inf_{\lambda \geq 0}\|\hat{\beta}(\lambda,q)-\beta\|_2^2/p$. The same experiment is repeated $200$ times. For completeness, we have included the first-order and second-order expansions of AMSE. The results are shown in Figures \ref{addfig:delta5one}-\ref{addfig:delta15three}. We summarize our observations and findings below.
\begin{itemize}
\item Figures \ref{addfig:delta5one}-\ref{addfig:delta5three} are about the case $\delta=5$. Figure \ref{addfig:delta5one} shows that the averaged MSE is well matched with AMSE even when $p$ is around $200$ or $400$. This is true for different values of $q$ as $\sigma_w$ varies over $[0,0.15]$. Even for $p=20$, the averaged MSE is already quite close to AMSE when $\sigma_w<0.075$. As confirmed by the simulations of Section \ref{sec:simanddis} in the main text, we see again that the first-order approximation for LASSO is very accurate, and so is the second-order approximation for other values of $q$.
\vspace{0.2cm}
\item Figure \ref{addfig:delta5two} draws the standard deviation of MSE against $\sigma^2_w$ over the range $[0,0.15^2]$. The plots with different values of $q$ and $p$ all indicate that the standard deviation scales linearly with the noise variance $\sigma_w^2$. Intuitively speaking, due to the quadratic form of MSE, its linear scaling with $\sigma_w^2$ is somewhat expected.
\vspace{0.2cm}
\item In Figure \ref{addfig:delta5three}, we multiply the standard deviation (SD) by $\sqrt{p/20}$ and plot the adjusted SD against $\sigma_w^2$. Interestingly, all the adjusted SD's (with the same value of $q$) are aligned around the same straight line. It implies that the scaling of SD on the dimension $p$ is $1/\sqrt{p}$, and the convergence rate of MSE is $\sqrt{n}$ (recall $n=\delta p$). We leave a rigorous proof of this empirical result for future research. 
\vspace{0.2cm}
\item Figures \ref{addfig:delta3one}-\ref{addfig:delta3three} are concerned with the case $\delta=3$. Compared with Figures \ref{addfig:delta5one}-\ref{addfig:delta5three}, it is seen that the observations we have made for $\delta=5$ still hold in $\delta=3$. We hence do not repeat the details. 
\vspace{0.2cm}
\item Figures \ref{addfig:delta15one}-\ref{addfig:delta15three} are for the scenario $\delta=1.5$. As the value of $\delta$ is decreased to $1.5$, the main messages we reveal in the case $\delta=5$ carry over. But some details may change. For instance, only when $p=400$ the difference between the averaged MSE and AMSE is negligible throughout all the values of $q$ and $\sigma_w$. Also, the linear pattern between SD (or adjusted SD) and $\sigma_w^2$ is less significant than when $\delta=5$. Given that for a fixed $p$ smaller values of $\delta$ lead to smaller number of observations (and hence less number of random elements in the design matrix), these changes should not come as a surprise. In the language of statistical physics, since the the number of random elements in the system decreases self averaging has not happened yet. 
\end{itemize}

\begin{figure}[htb]
\centering
\setlength\tabcolsep{1.5pt}
\begin{tabular}{cc}
\includegraphics[width=6.2cm, height=5.8cm]{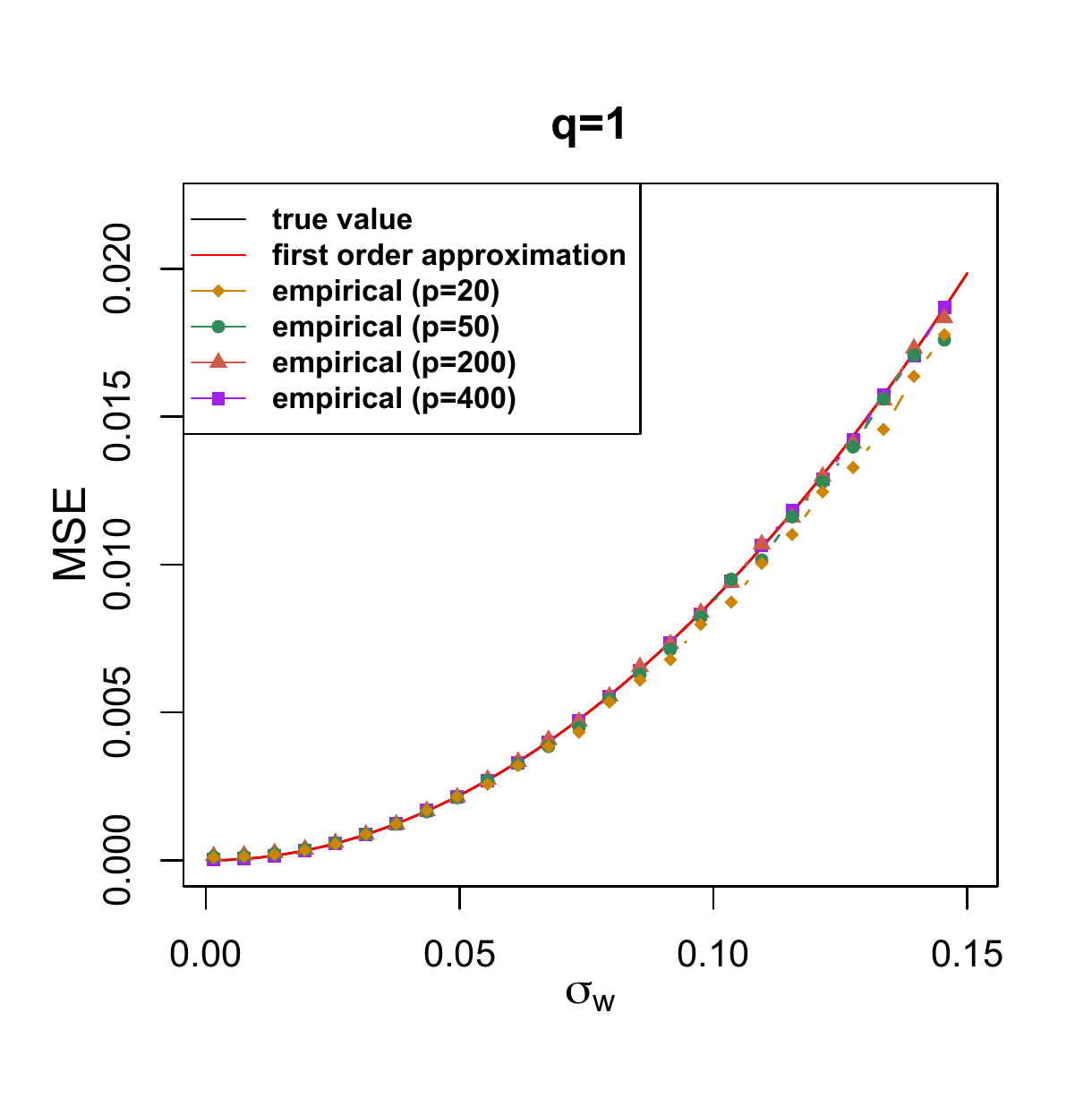} &
 \includegraphics[width=6.2cm, height=5.8cm]{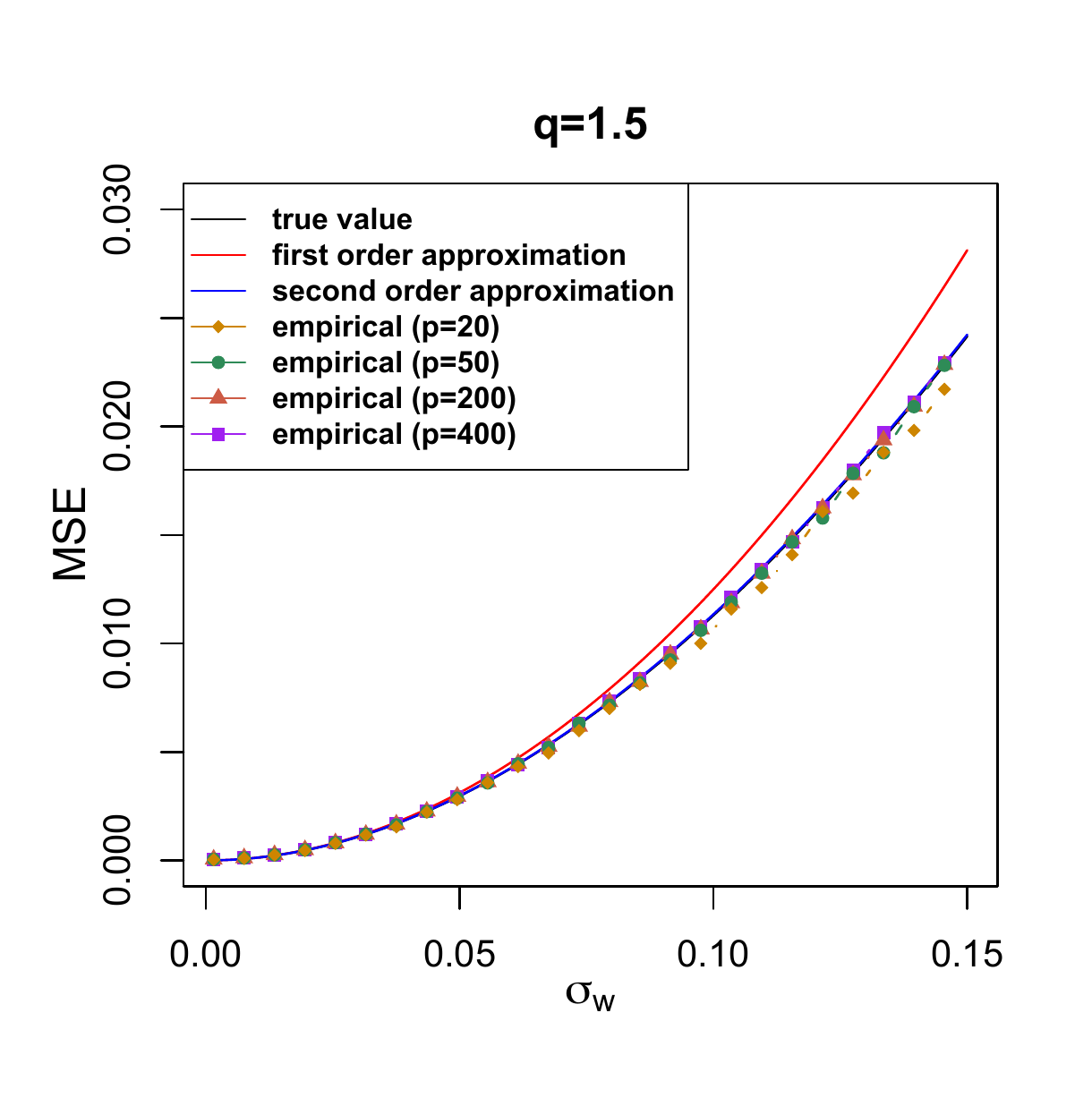} \\ 
\includegraphics[width=6.2cm, height=5.8cm]{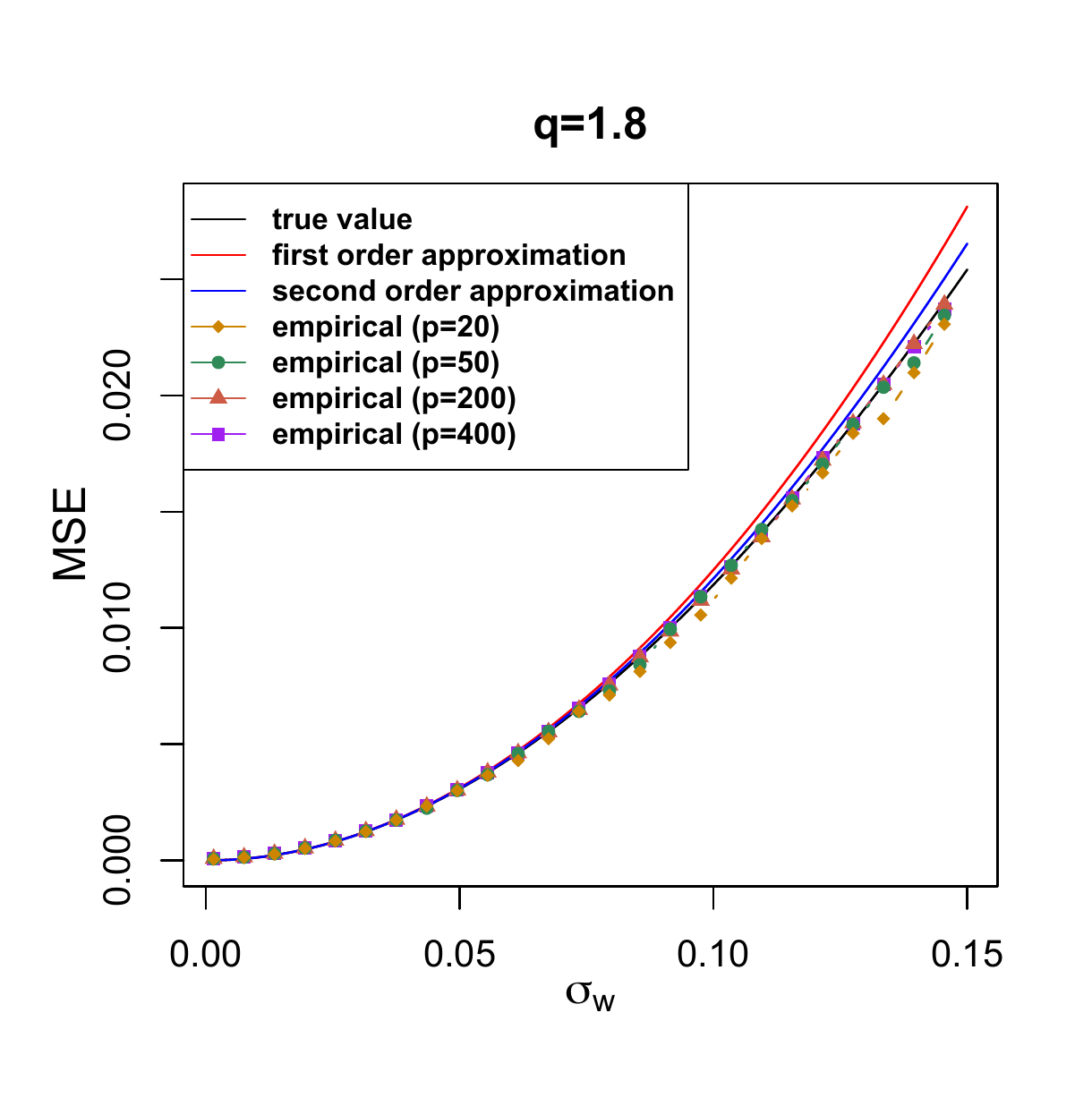} &
\includegraphics[width=6.2cm, height=5.8cm]{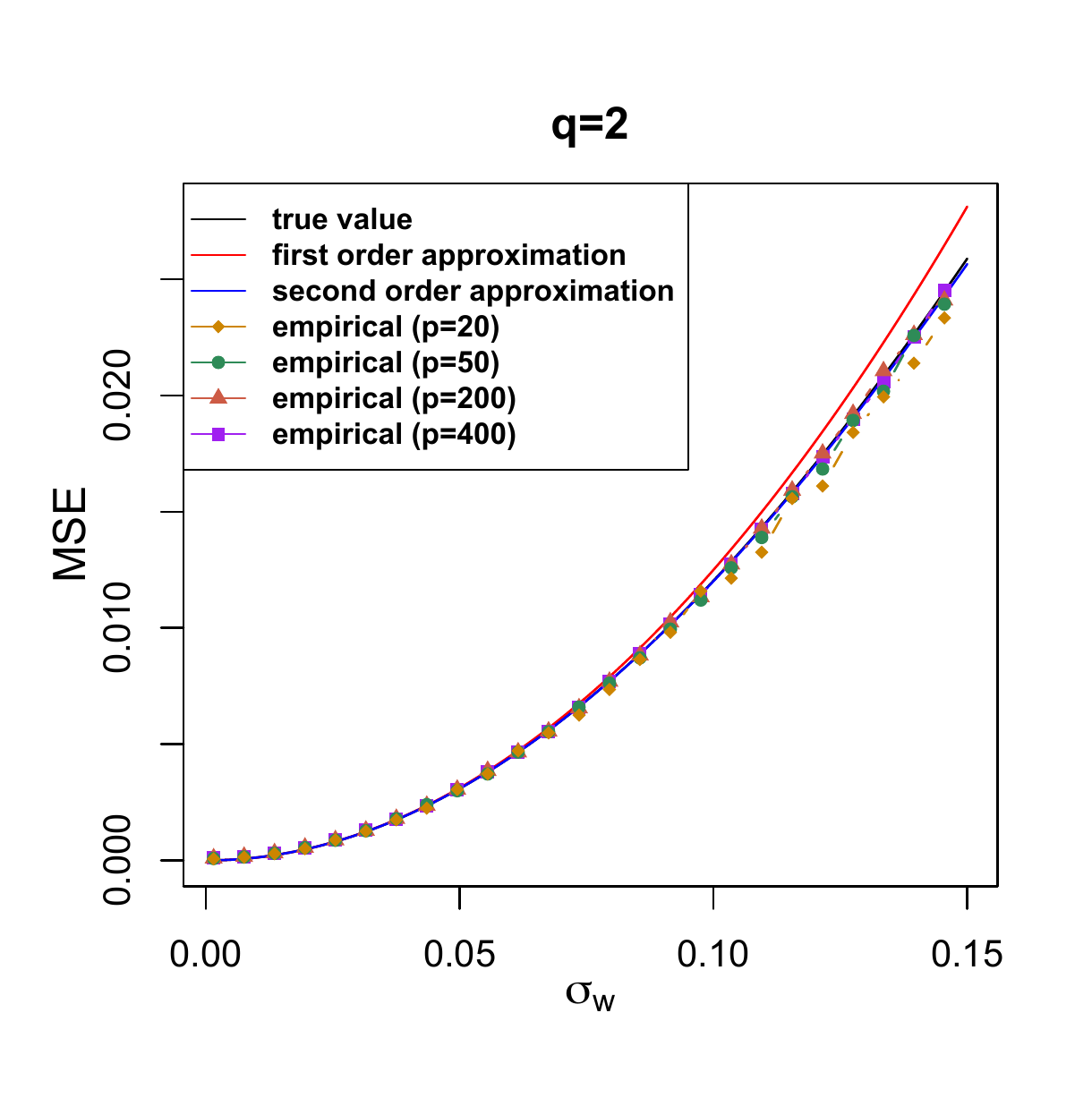} 
\end{tabular}
\caption{Plots of actual AMSE, its approximations, and finite-sample MSE. The MSE is averaged over 200 times. The setup is $\delta=5, \epsilon=0.4, g(b)=0.5\delta_1(b)+0.5\delta_{-1}(b)$.} \label{addfig:delta5one}
\end{figure}

\begin{figure}[htb]
\centering
\setlength\tabcolsep{1.5pt}
\begin{tabular}{cc}
\includegraphics[width=6.2cm, height=5.8cm]{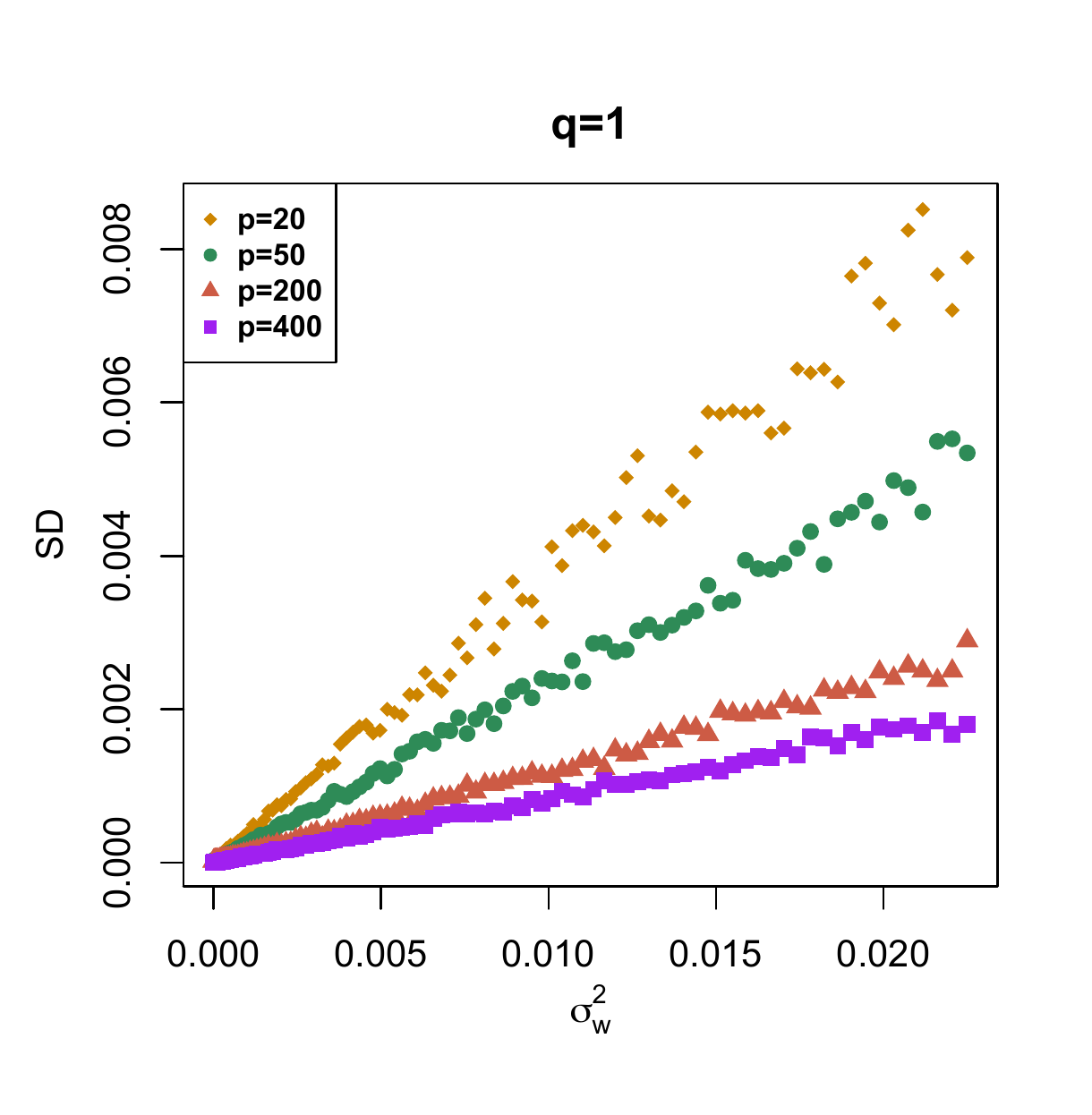} &
 \includegraphics[width=6.2cm, height=5.8cm]{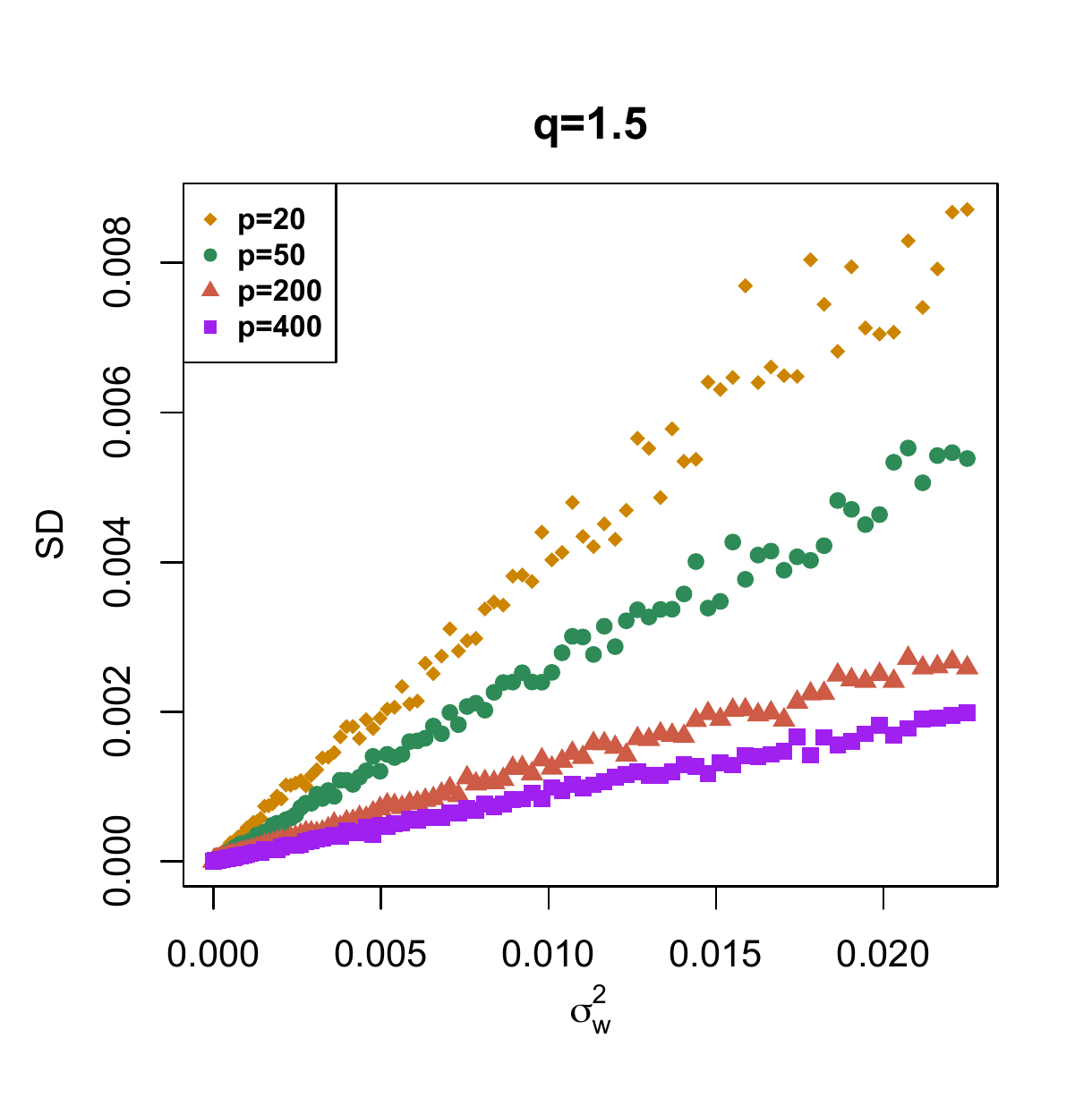} \\ 
\includegraphics[width=6.2cm, height=5.8cm]{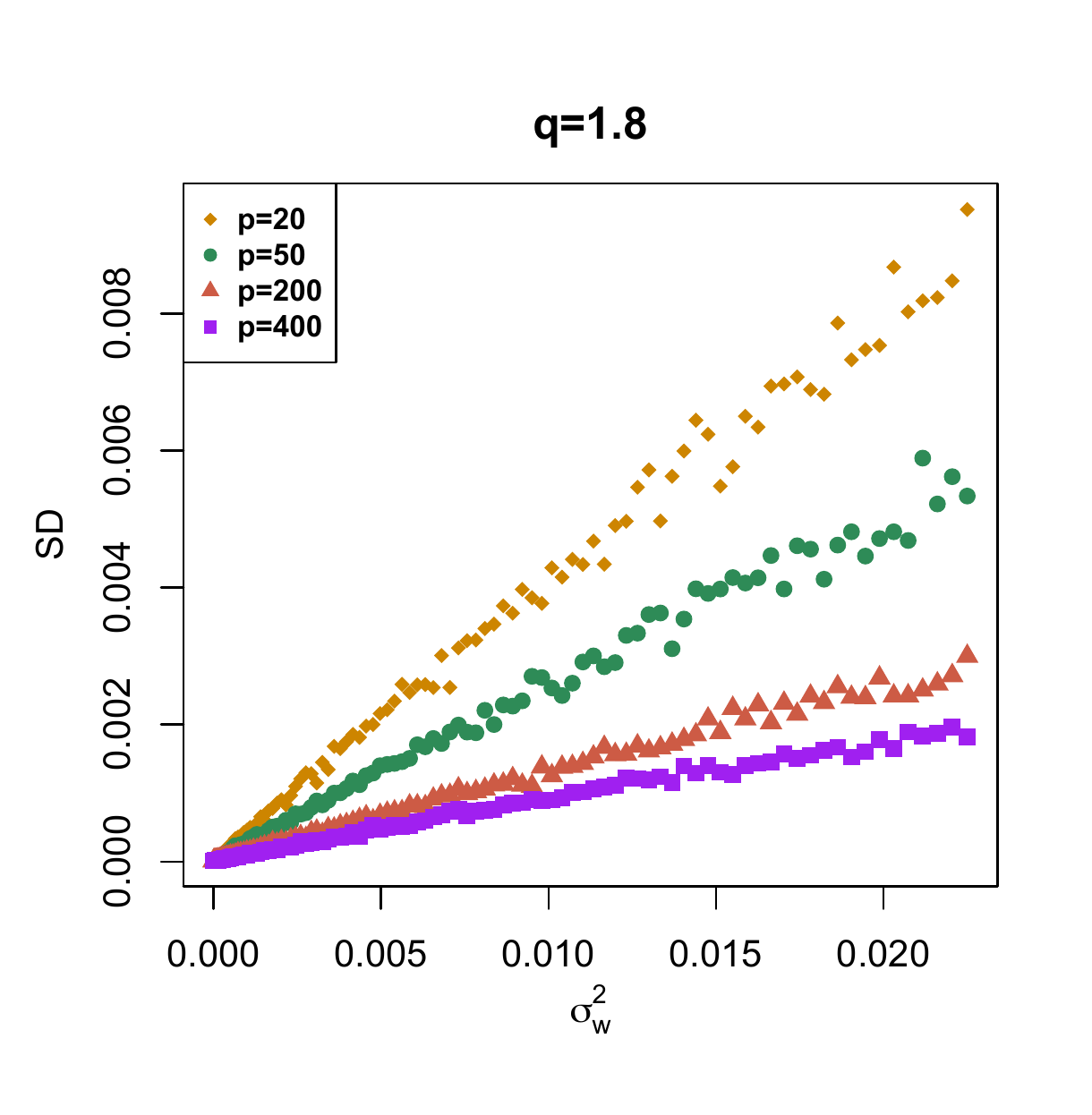} &
\includegraphics[width=6.2cm, height=5.8cm]{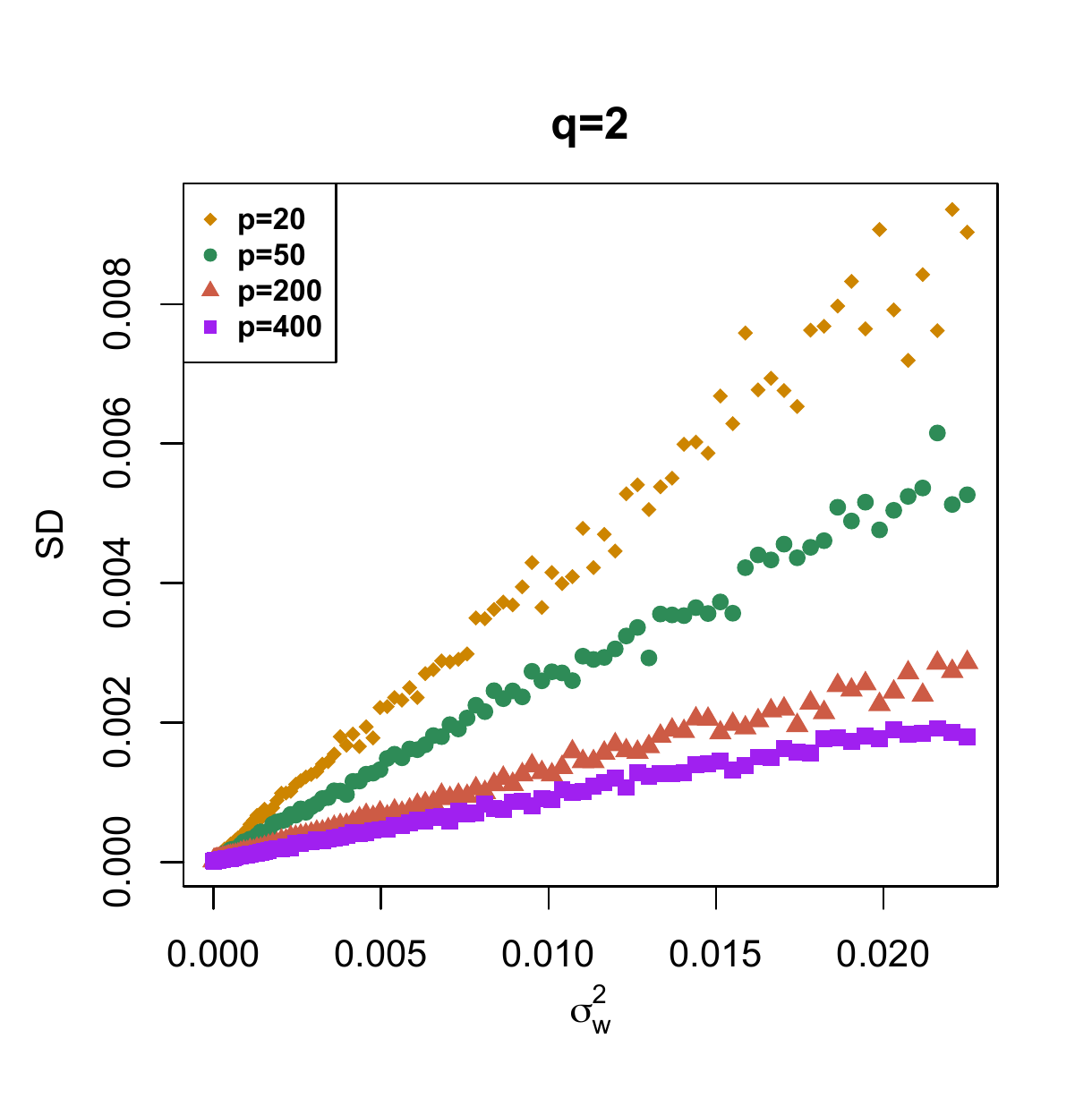} 
\end{tabular}
\caption{Plots of the standard deviation of finite-sample MSE. The same experiment is repeated 200 times. The setup is $\delta=5, \epsilon=0.4, g(b)=0.5\delta_1(b)+0.5\delta_{-1}(b)$.} \label{addfig:delta5two}
\end{figure}

\begin{figure}[htb]
\centering
\setlength\tabcolsep{1.5pt}
\begin{tabular}{cc}
\includegraphics[width=6.2cm, height=5.8cm]{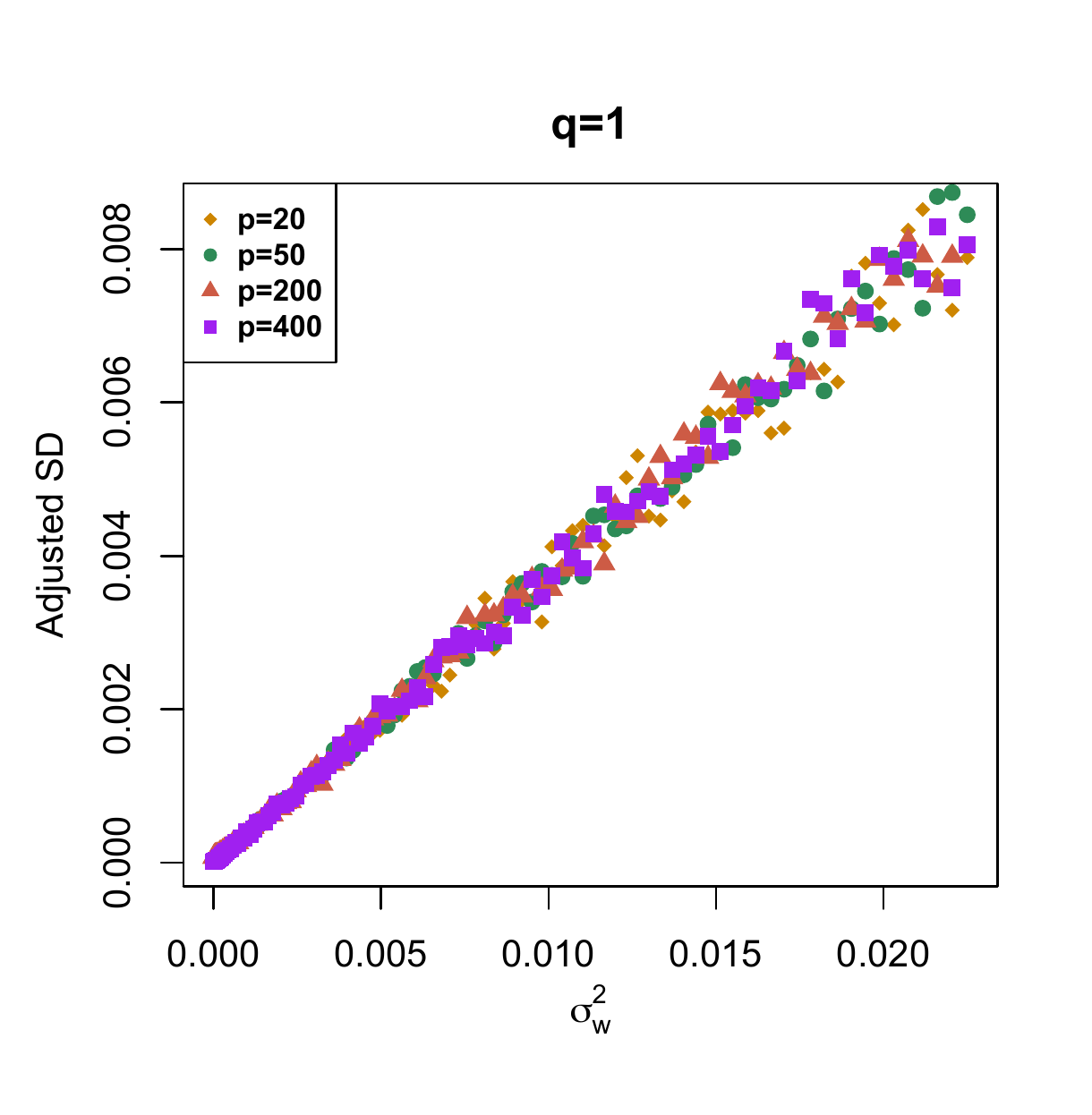} &
 \includegraphics[width=6.2cm, height=5.8cm]{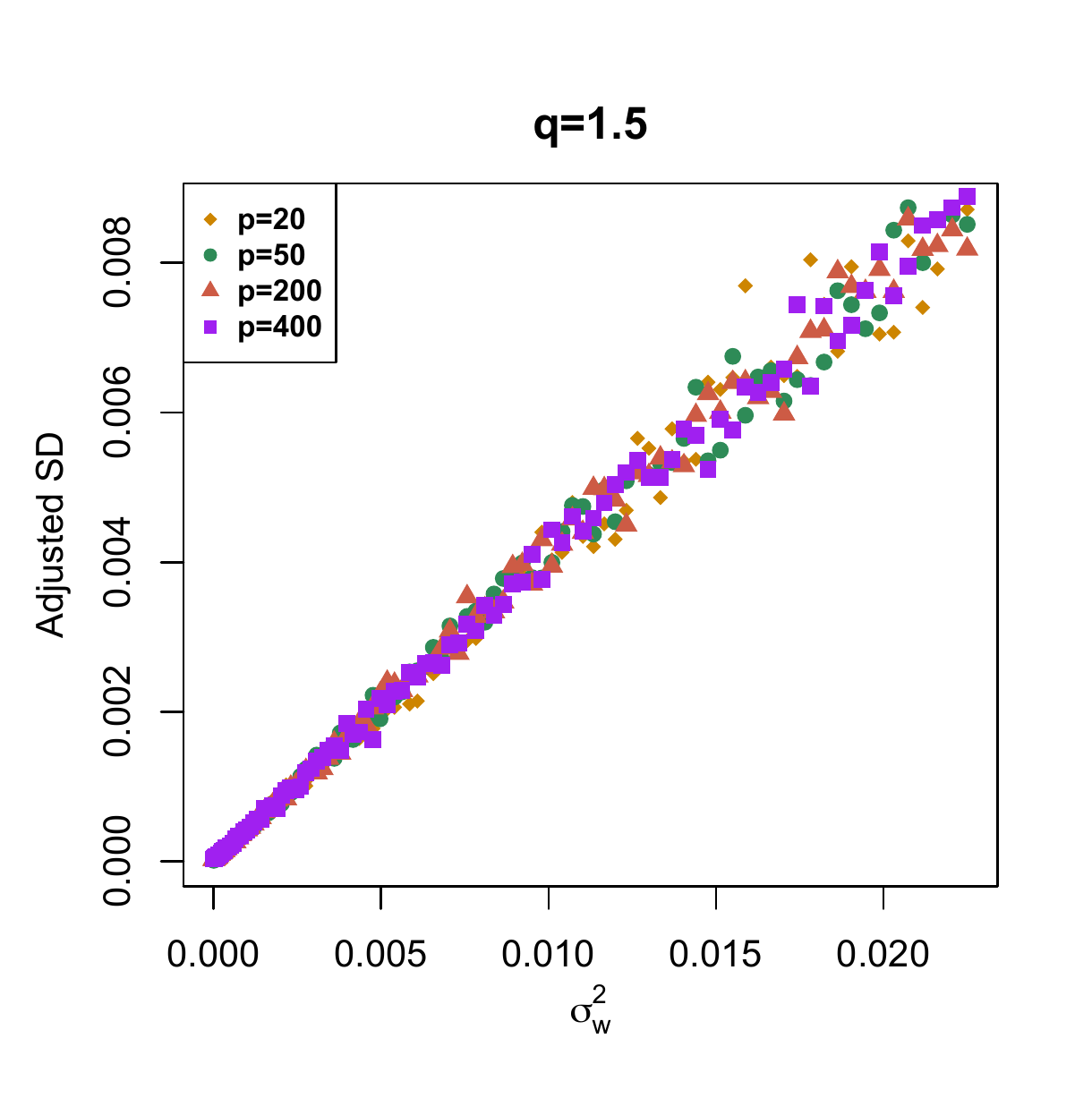} \\ 
\includegraphics[width=6.2cm, height=5.8cm]{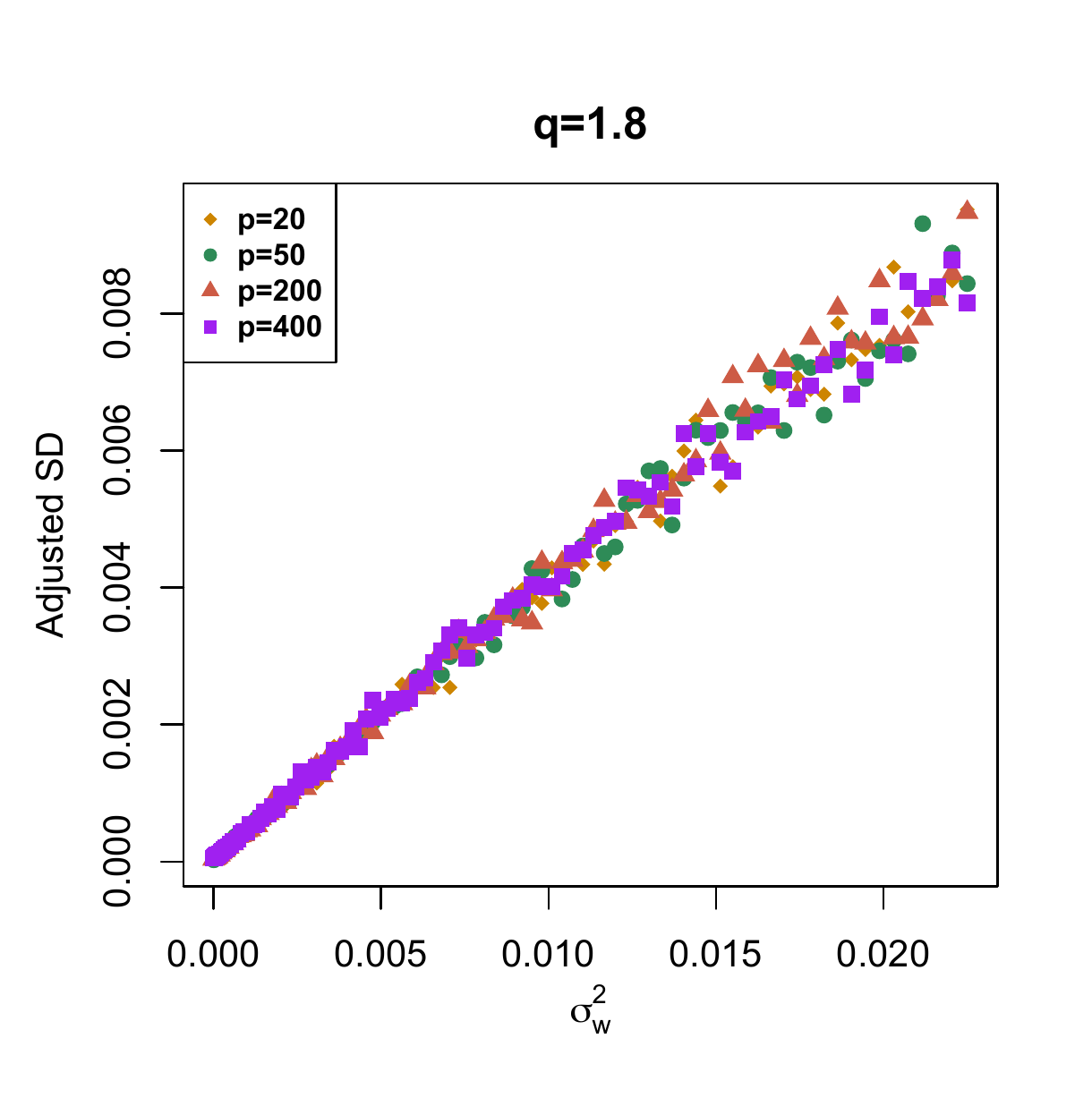} &
\includegraphics[width=6.2cm, height=5.8cm]{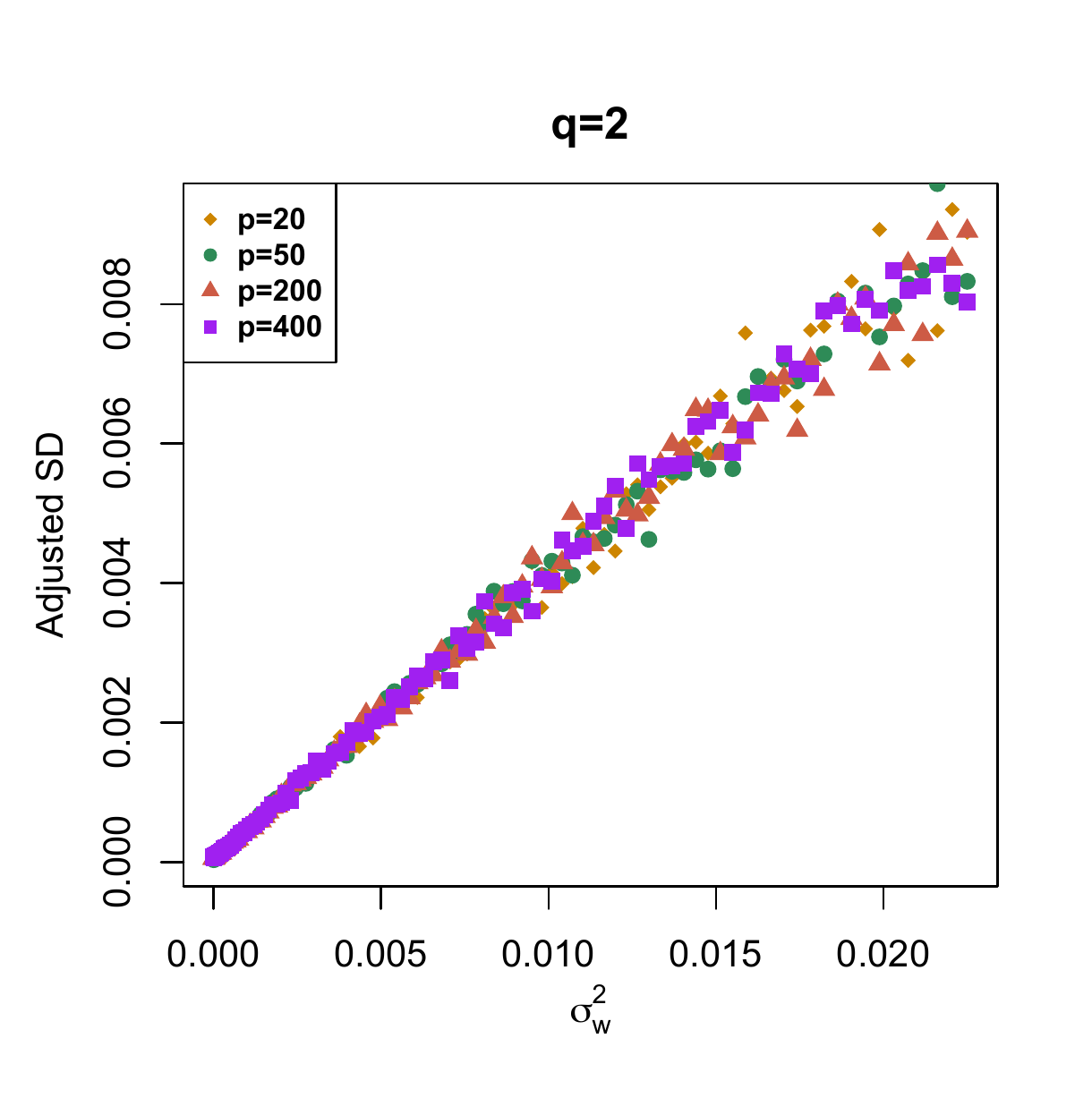} 
\end{tabular}
\caption{Plots of the adjusted standard deviation of finite-sample MSE. The same experiment is repeated 200 times. The setup is $\delta=5, \epsilon=0.4, g(b)=0.5\delta_1(b)+0.5\delta_{-1}(b)$.} \label{addfig:delta5three}
\end{figure}

\begin{figure}[htb]
\centering
\setlength\tabcolsep{1.5pt}
\begin{tabular}{cc}
\includegraphics[width=6.2cm, height=5.8cm]{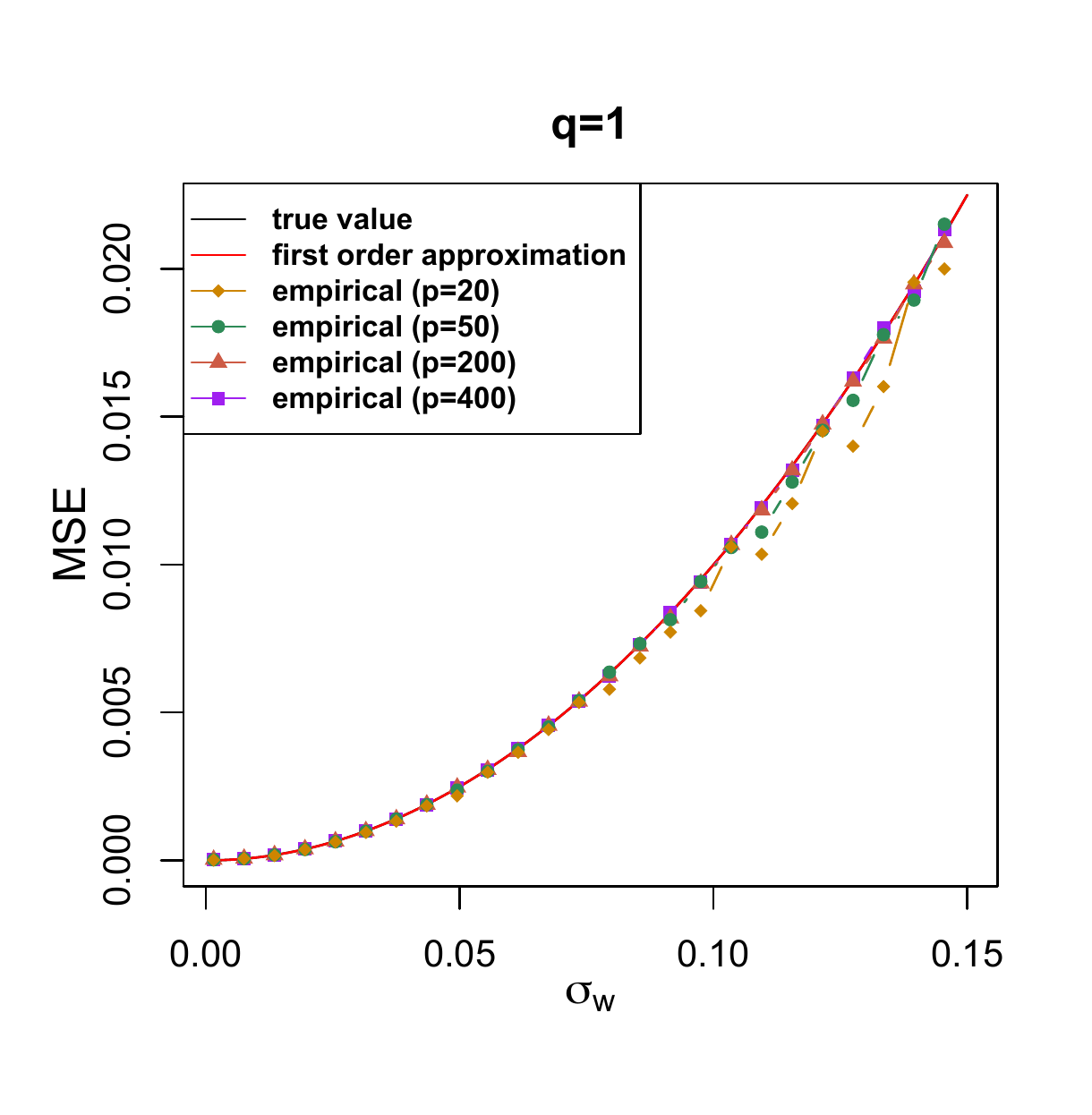} &
 \includegraphics[width=6.2cm, height=5.8cm]{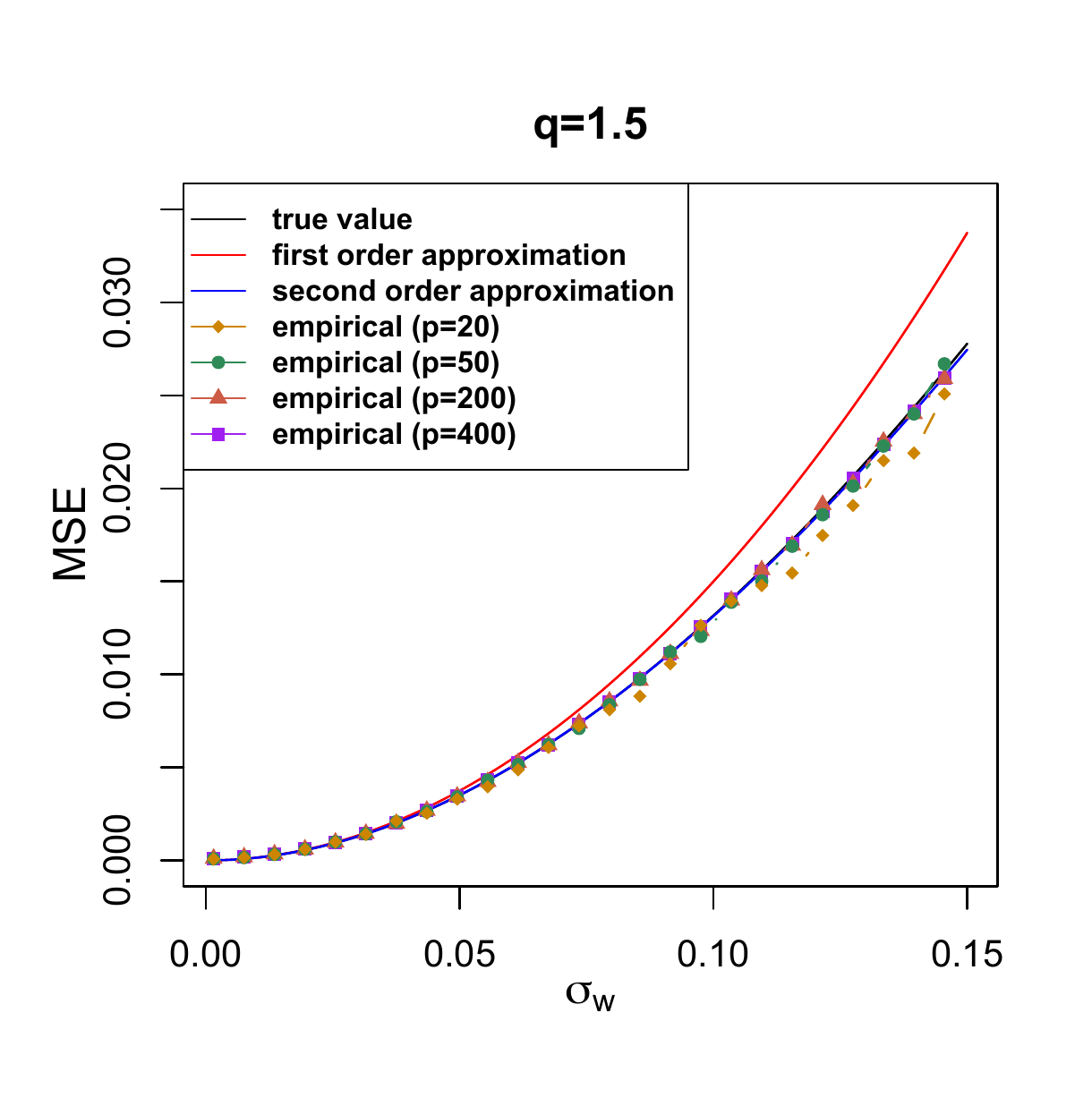} \\ 
\includegraphics[width=6.2cm, height=5.8cm]{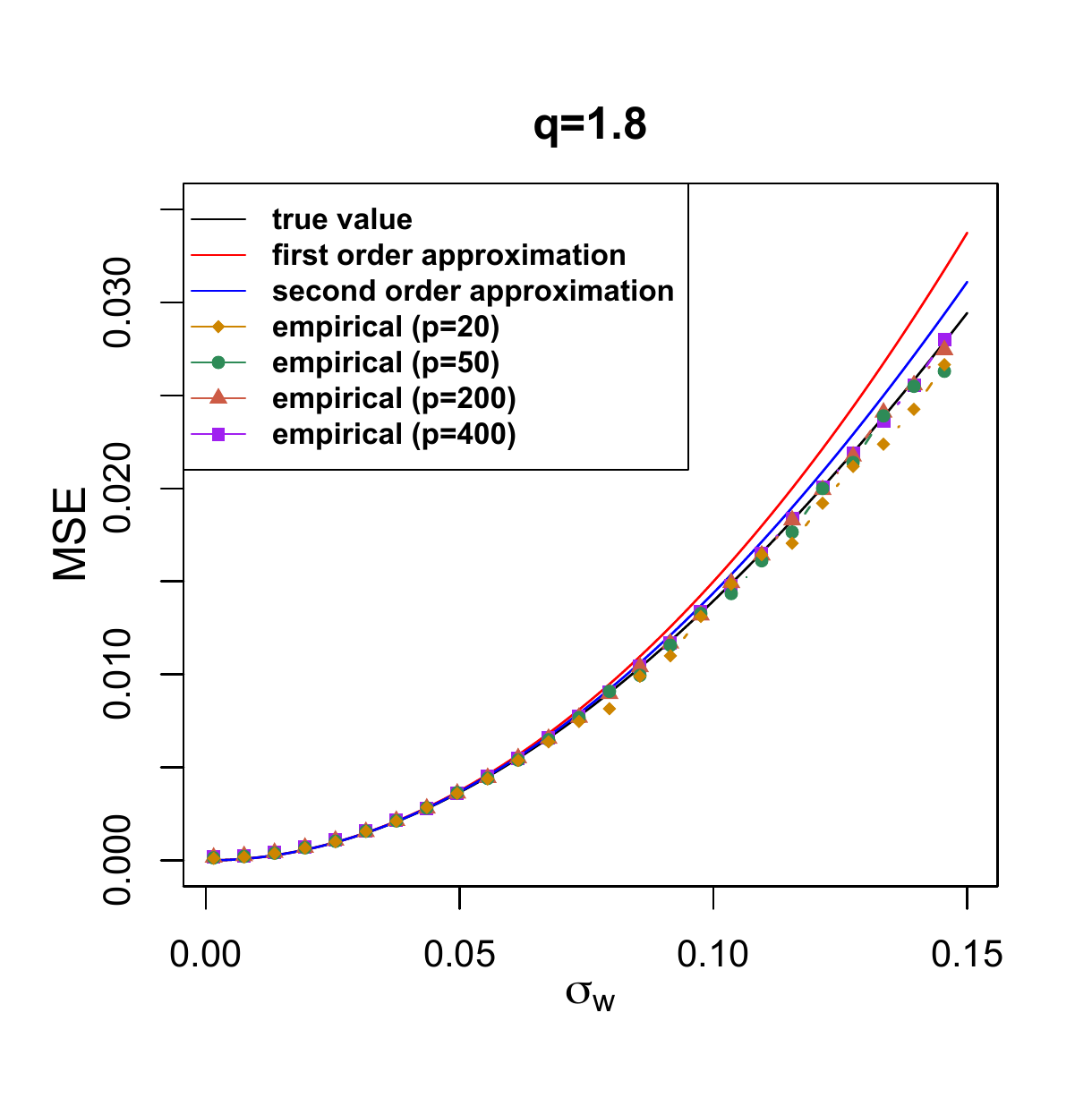} &
\includegraphics[width=6.2cm, height=5.8cm]{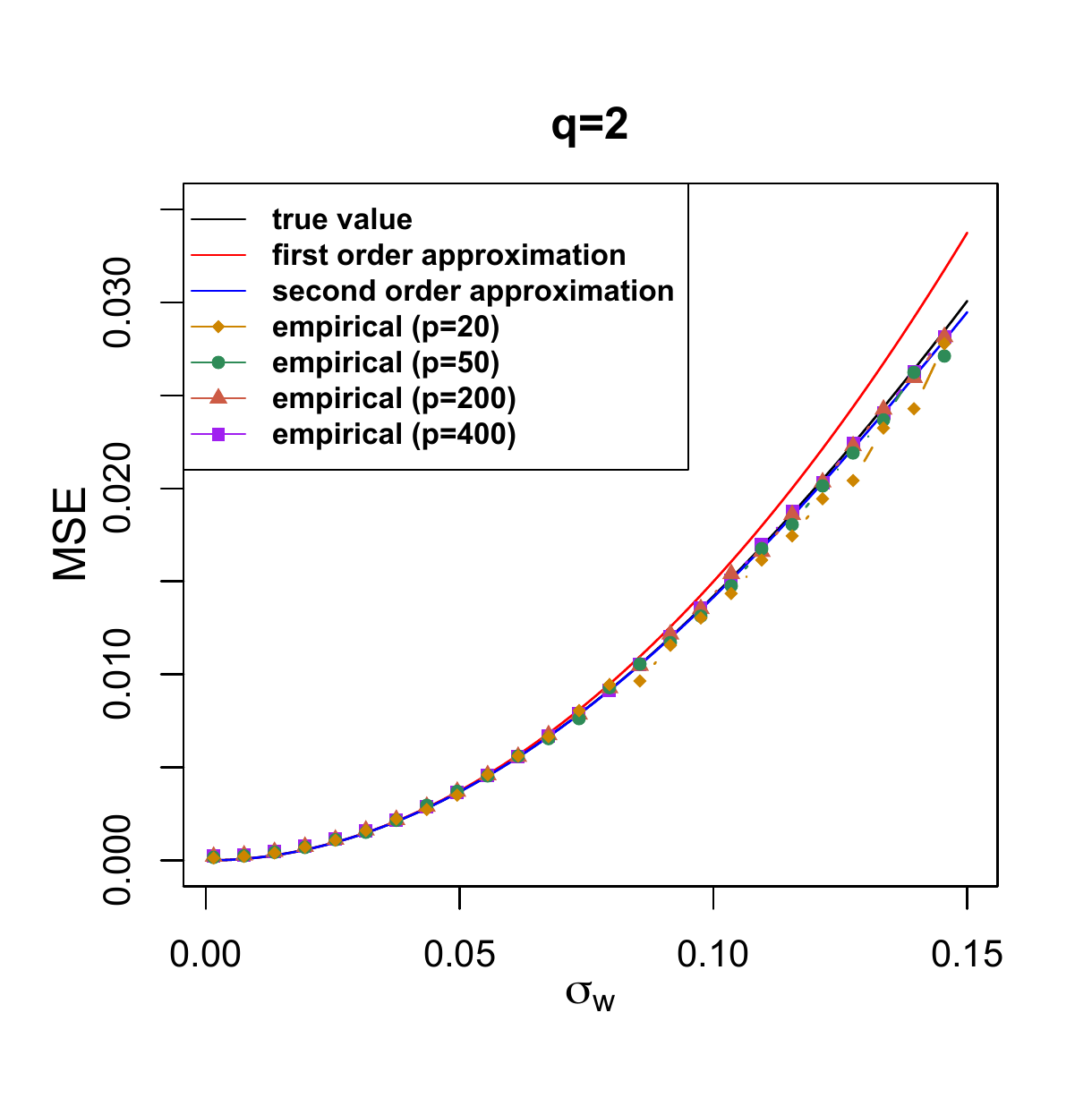} 
\end{tabular}
\caption{Plots of actual AMSE, its approximations, and finite-sample MSE. The MSE is averaged over 200 times. The setup is $\delta=3, \epsilon=0.4, g(b)=0.5\delta_1(b)+0.5\delta_{-1}(b)$.} \label{addfig:delta3one}
\end{figure}

\begin{figure}[htb]
\centering
\setlength\tabcolsep{1.5pt}
\begin{tabular}{cc}
\includegraphics[width=6.2cm, height=5.8cm]{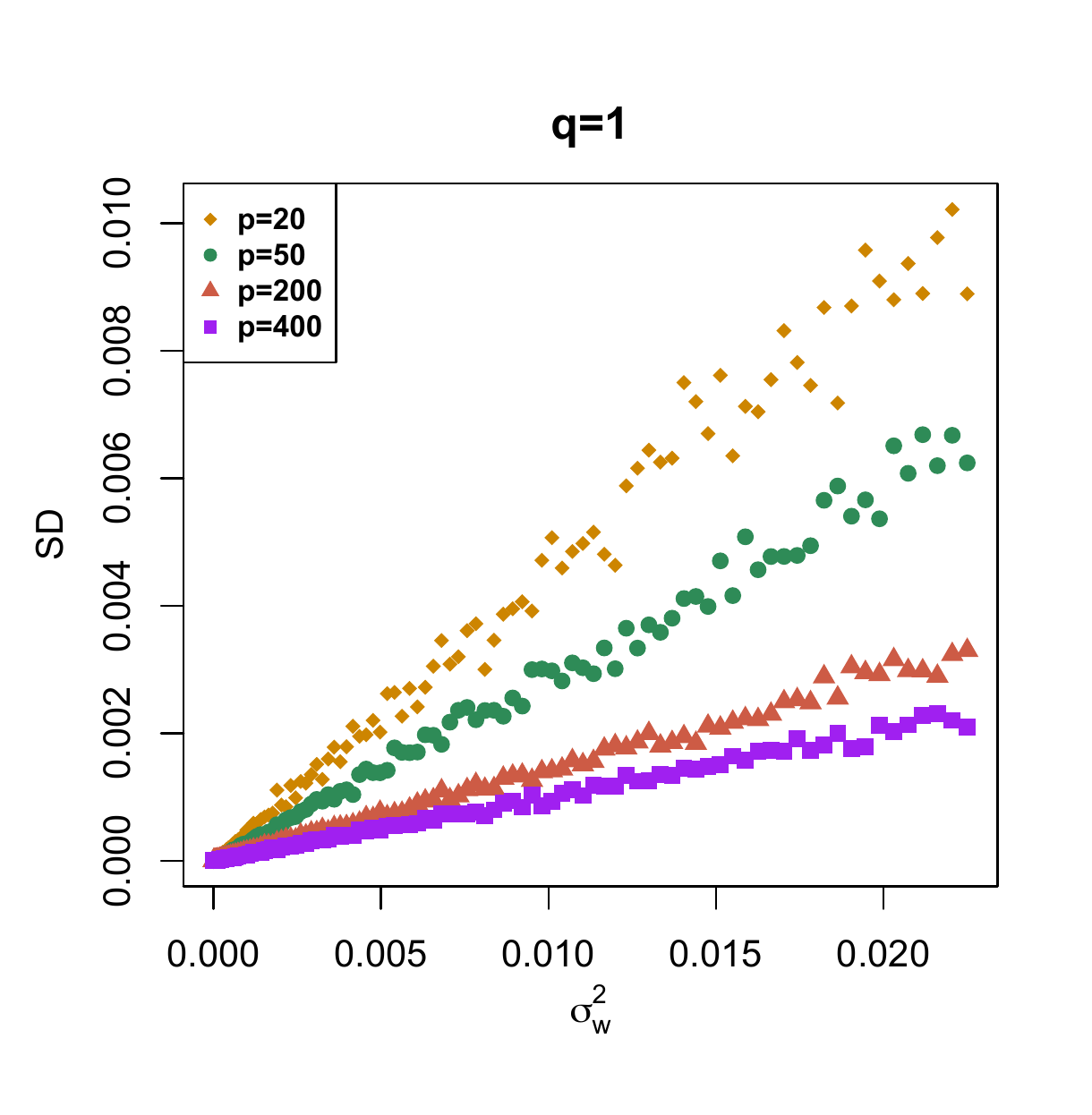} &
 \includegraphics[width=6.2cm, height=5.8cm]{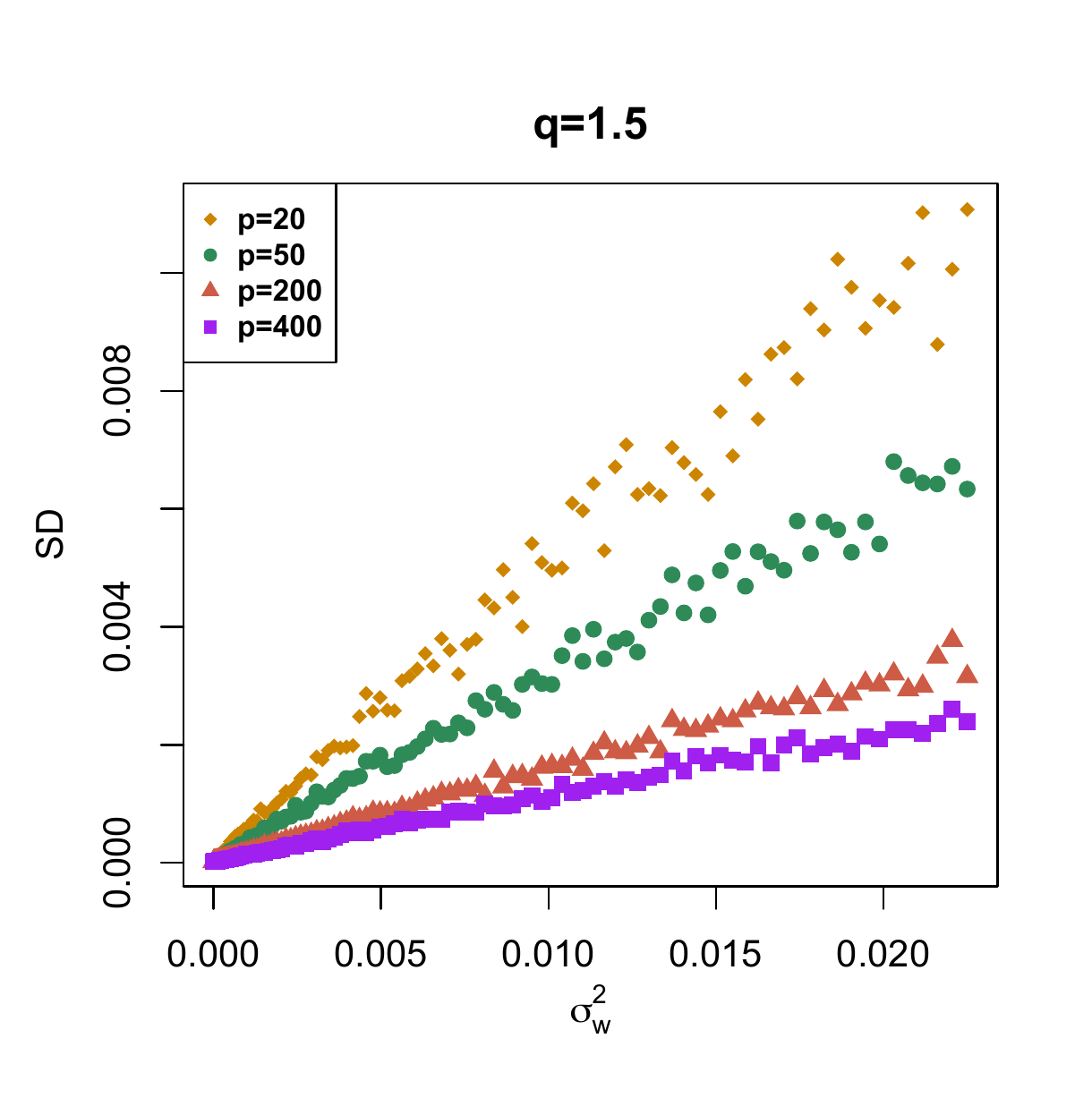} \\ 
\includegraphics[width=6.2cm, height=5.8cm]{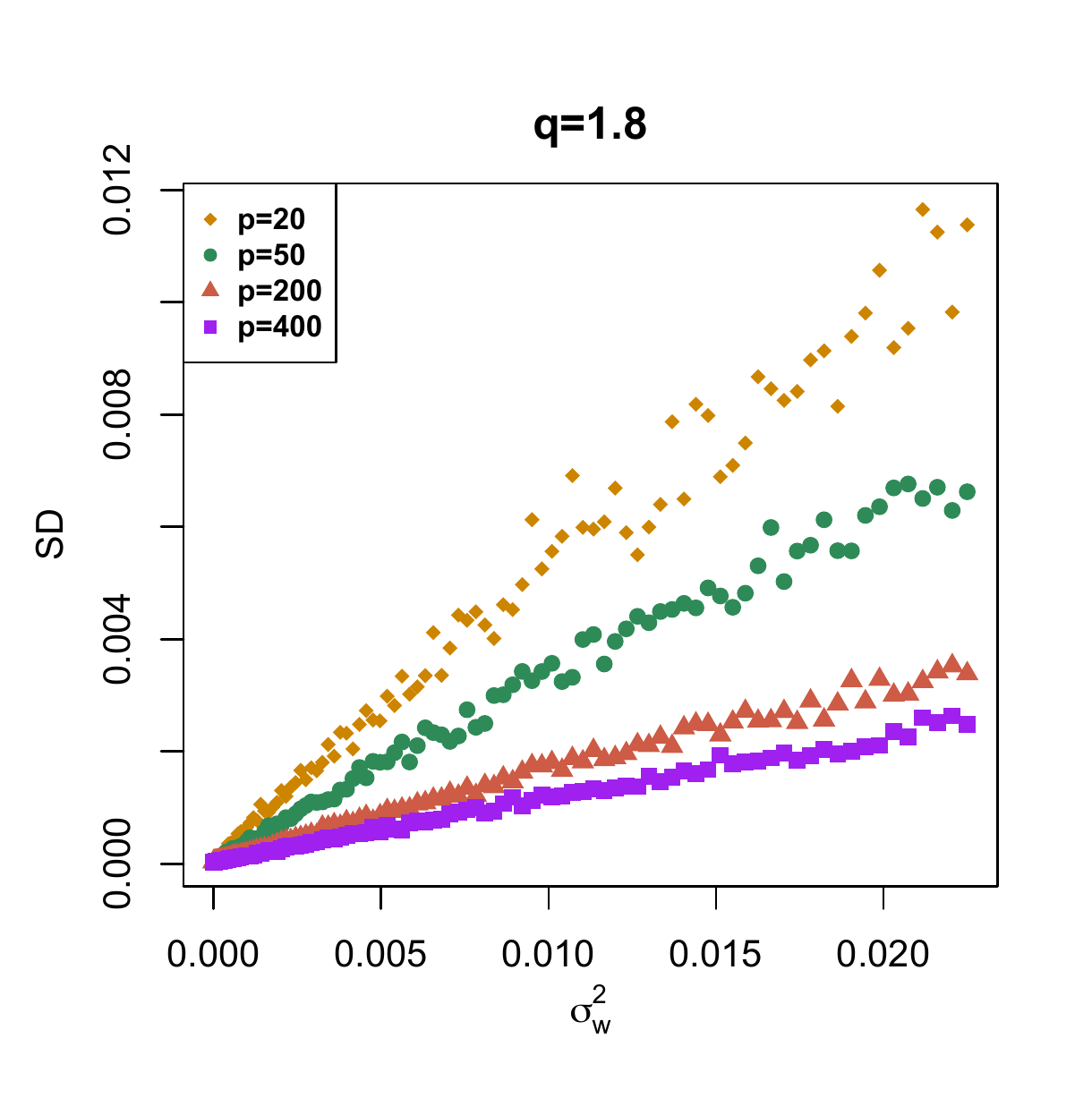} &
\includegraphics[width=6.2cm, height=5.8cm]{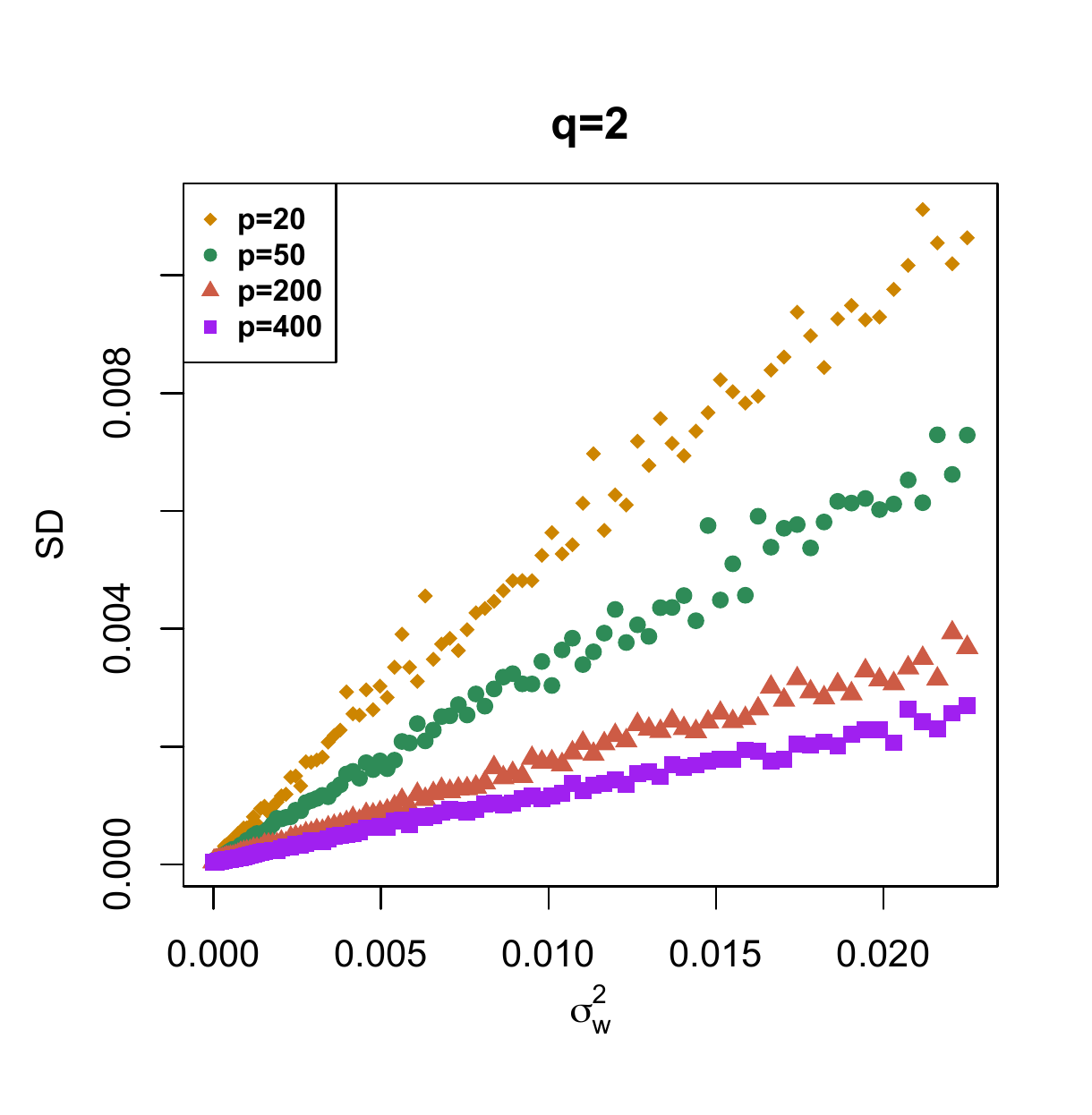} 
\end{tabular}
\caption{Plots of the standard deviation of finite-sample MSE. The same experiment is repeated 200 times. The setup is $\delta=3, \epsilon=0.4, g(b)=0.5\delta_1(b)+0.5\delta_{-1}(b)$.} \label{addfig:delta3two}
\end{figure}

\begin{figure}[htb]
\centering
\setlength\tabcolsep{1.5pt}
\begin{tabular}{cc}
\includegraphics[width=6.2cm, height=5.8cm]{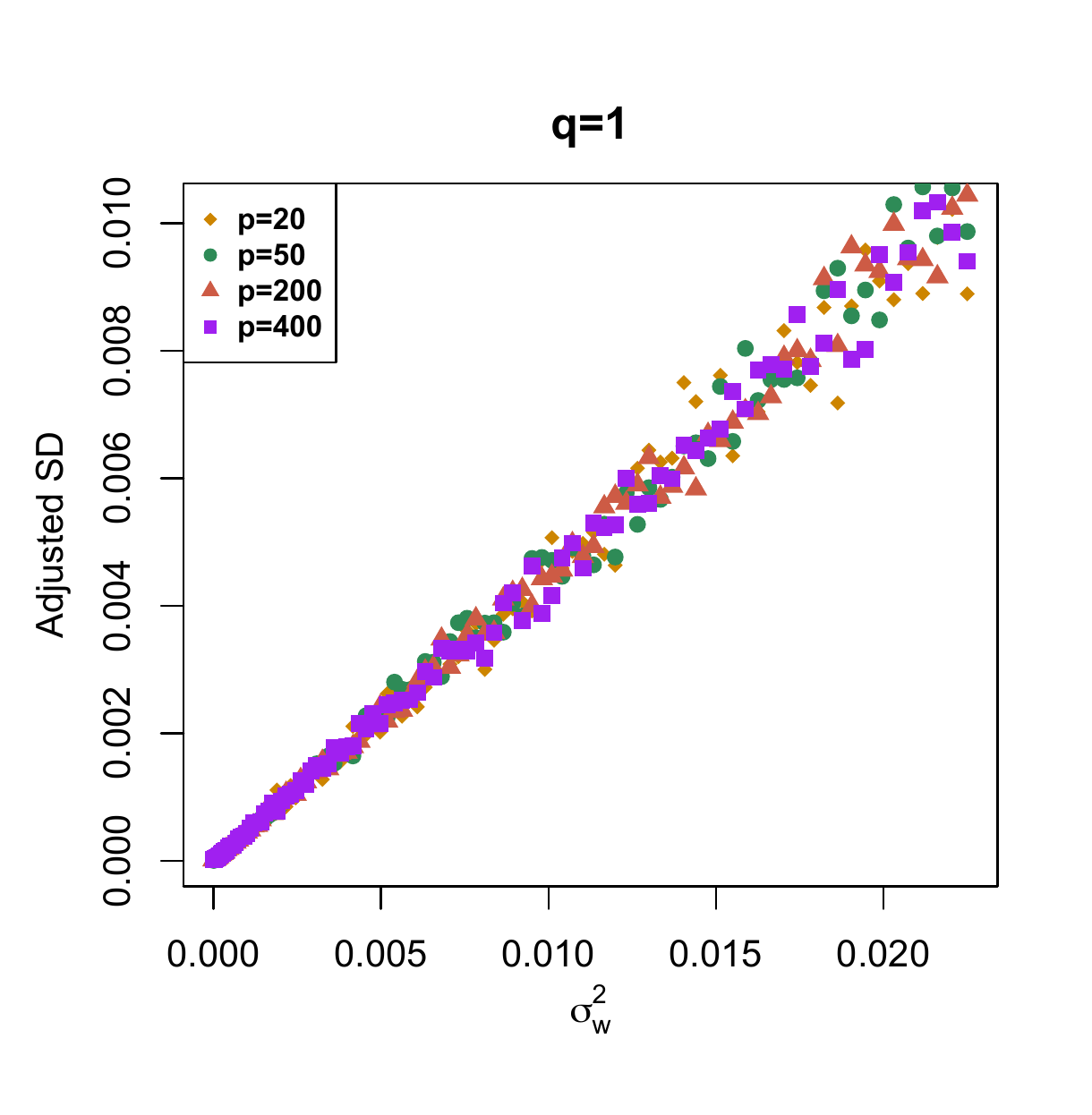} &
 \includegraphics[width=6.2cm, height=5.8cm]{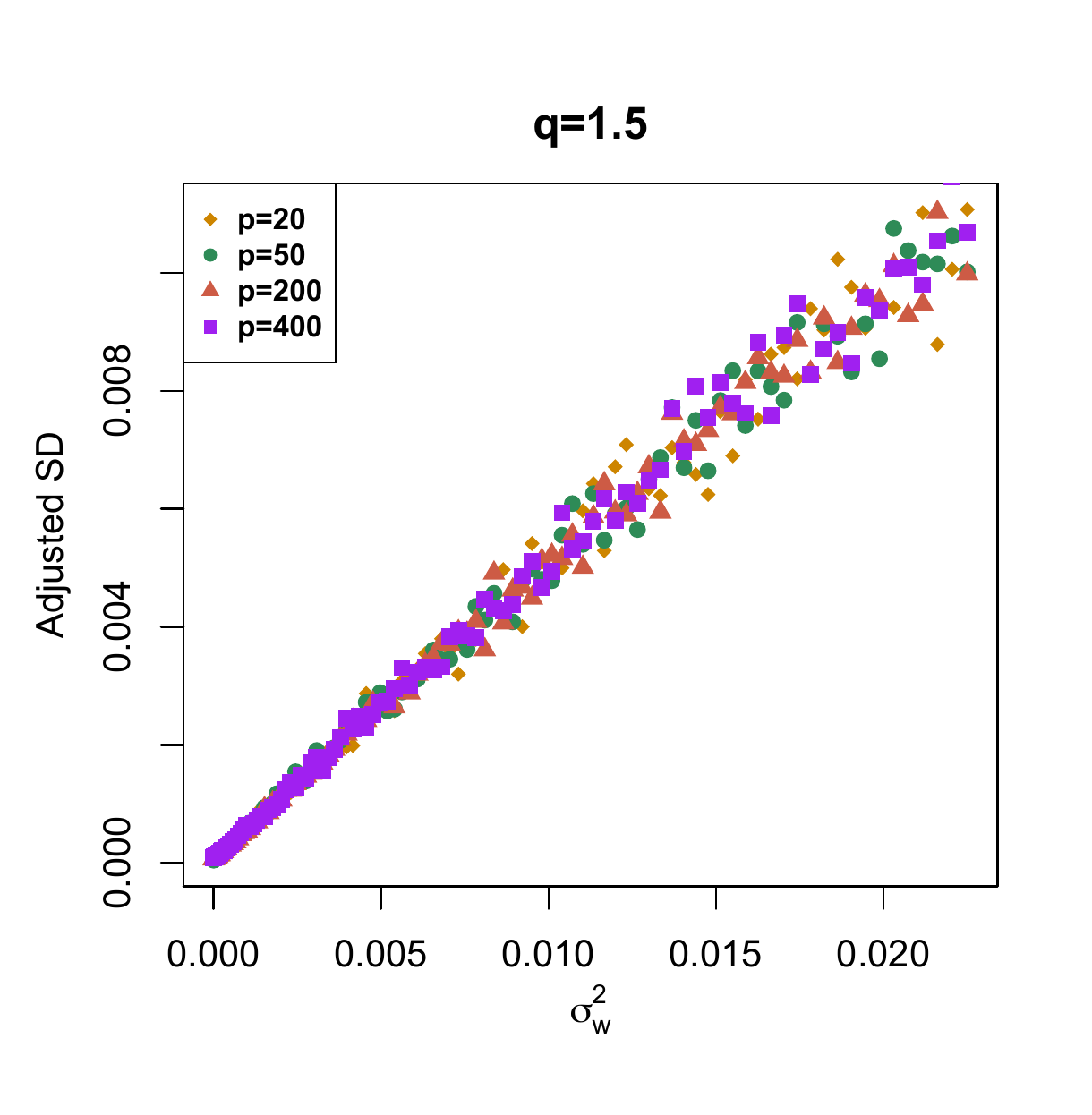} \\ 
\includegraphics[width=6.2cm, height=5.8cm]{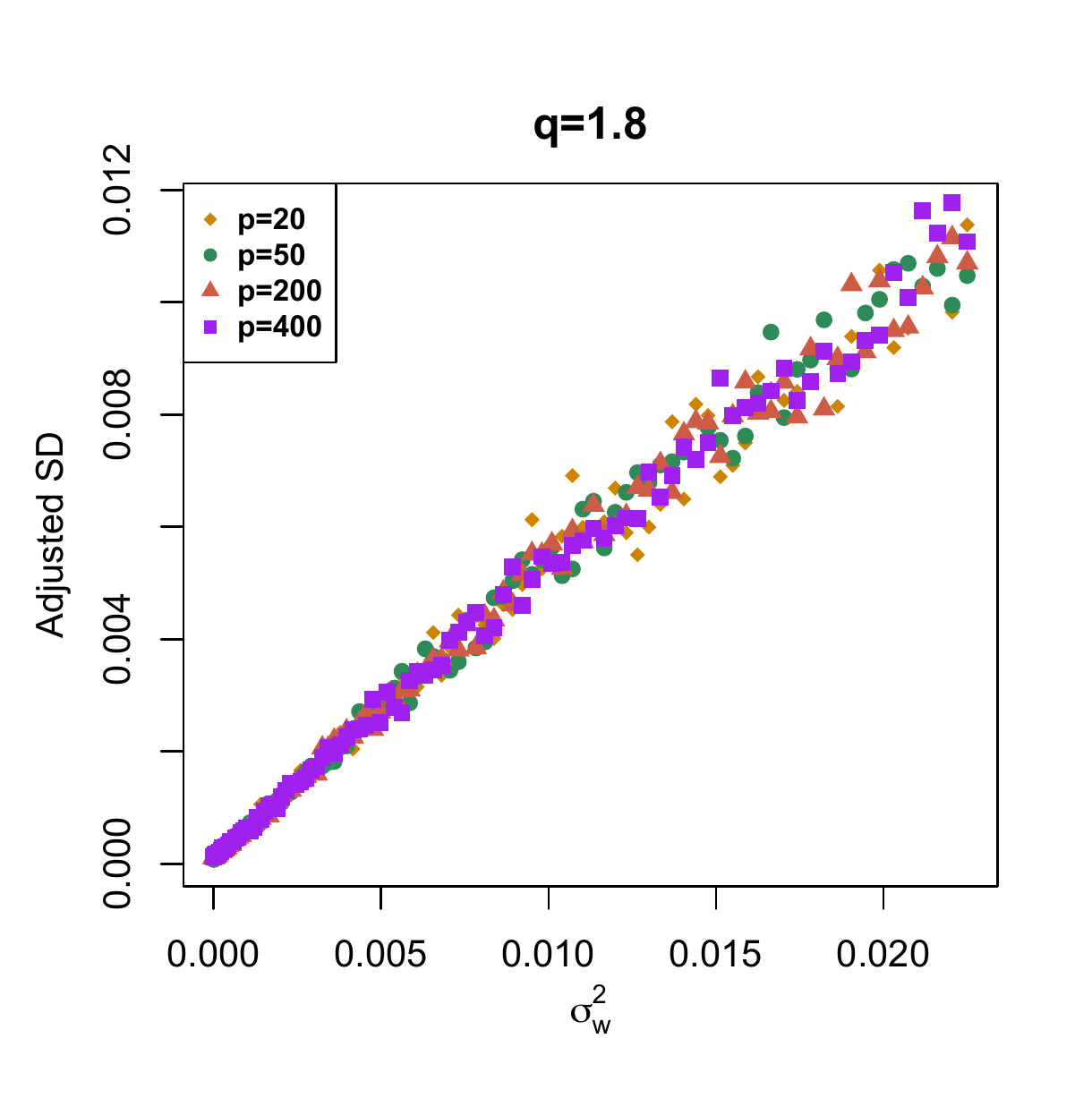} &
\includegraphics[width=6.2cm, height=5.8cm]{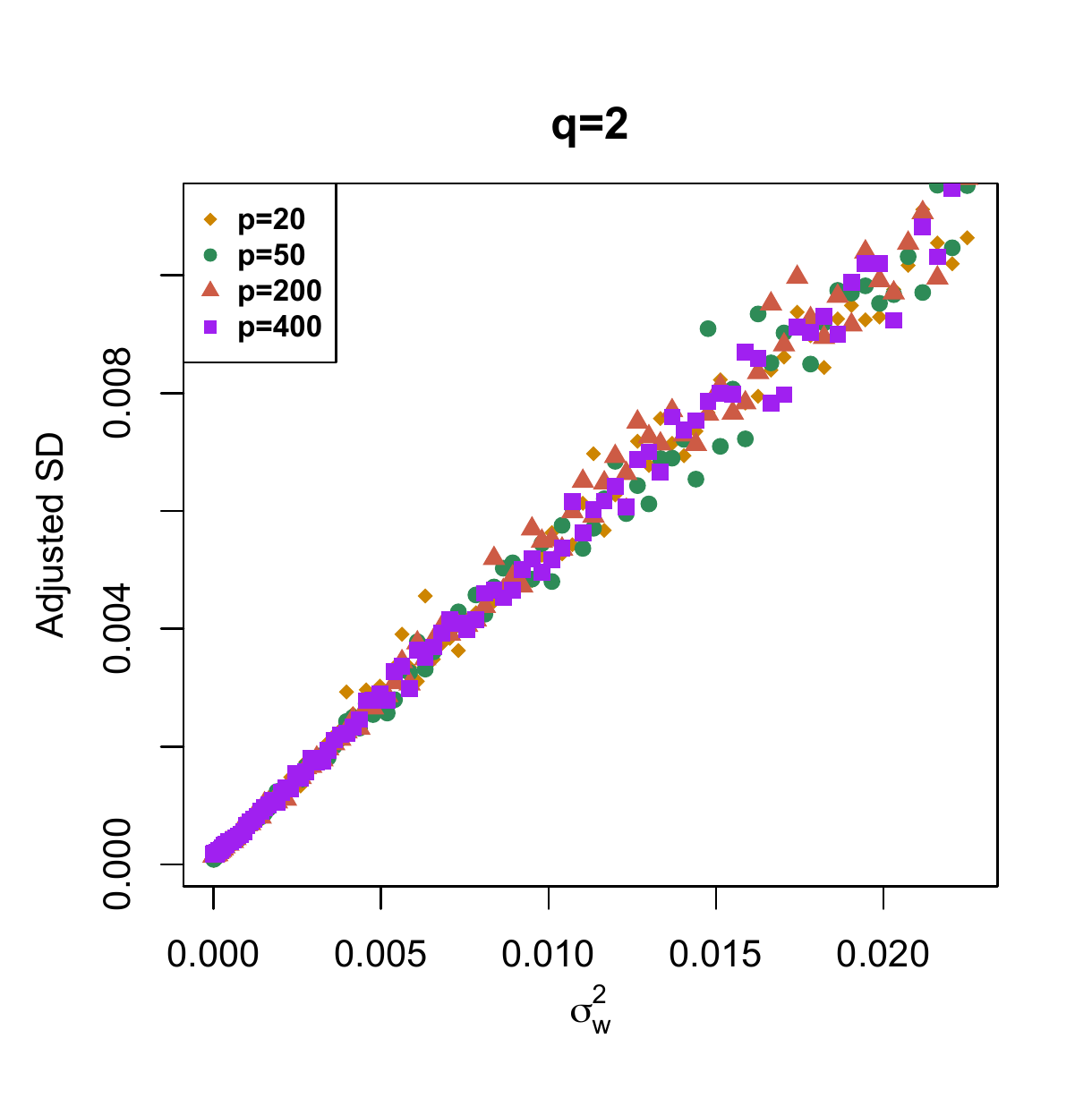} 
\end{tabular}
\caption{Plots of the adjusted standard deviation of finite-sample MSE. The same experiment is repeated 200 times. The setup is $\delta=3, \epsilon=0.4, g(b)=0.5\delta_1(b)+0.5\delta_{-1}(b)$.} \label{addfig:delta3three}
\end{figure}

\begin{figure}[htb]
\centering
\setlength\tabcolsep{1.5pt}
\begin{tabular}{cc}
\includegraphics[width=6.2cm, height=5.8cm]{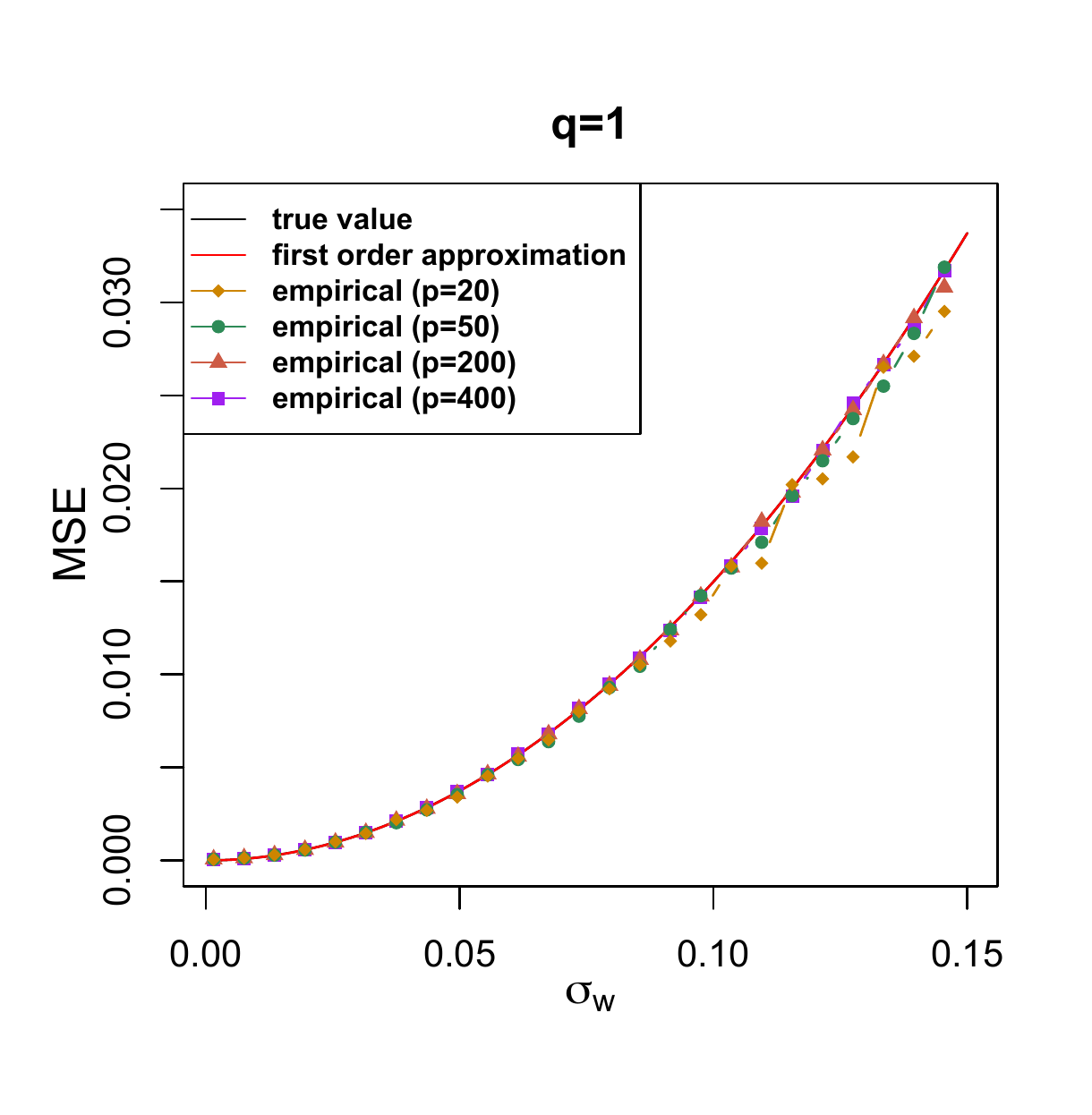} &
 \includegraphics[width=6.2cm, height=5.8cm]{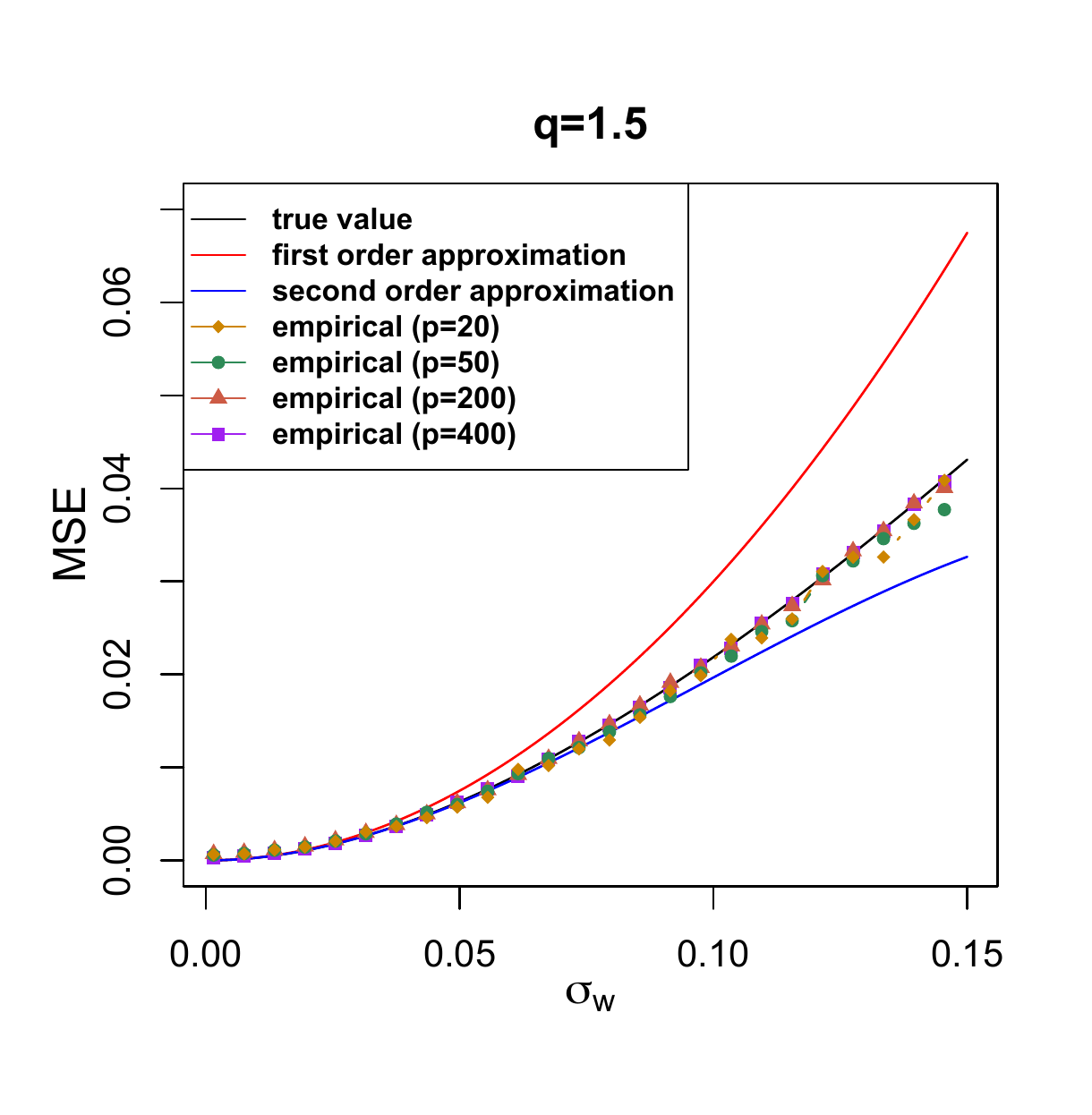} \\ 
\includegraphics[width=6.2cm, height=5.8cm]{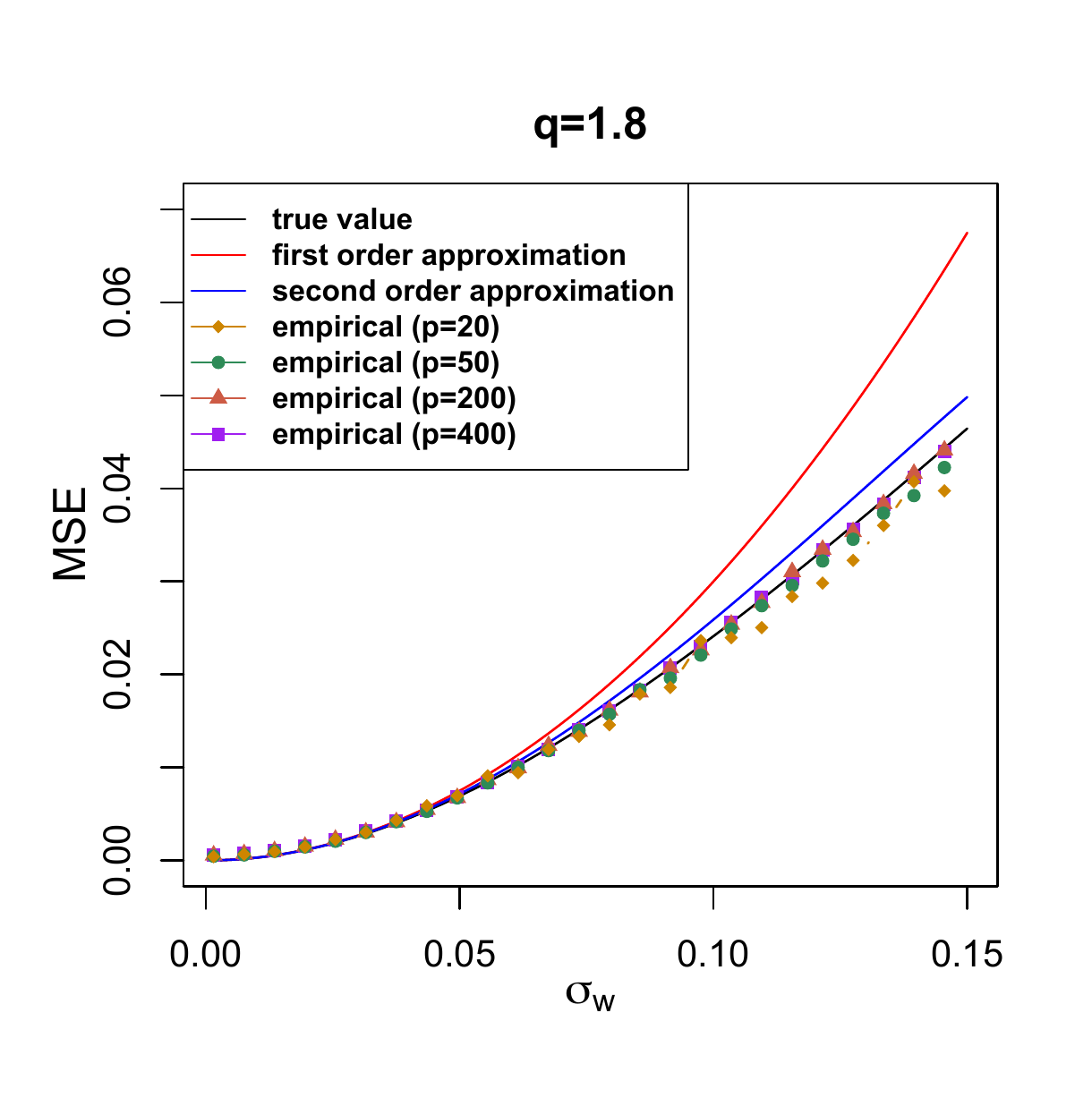} &
\includegraphics[width=6.2cm, height=5.8cm]{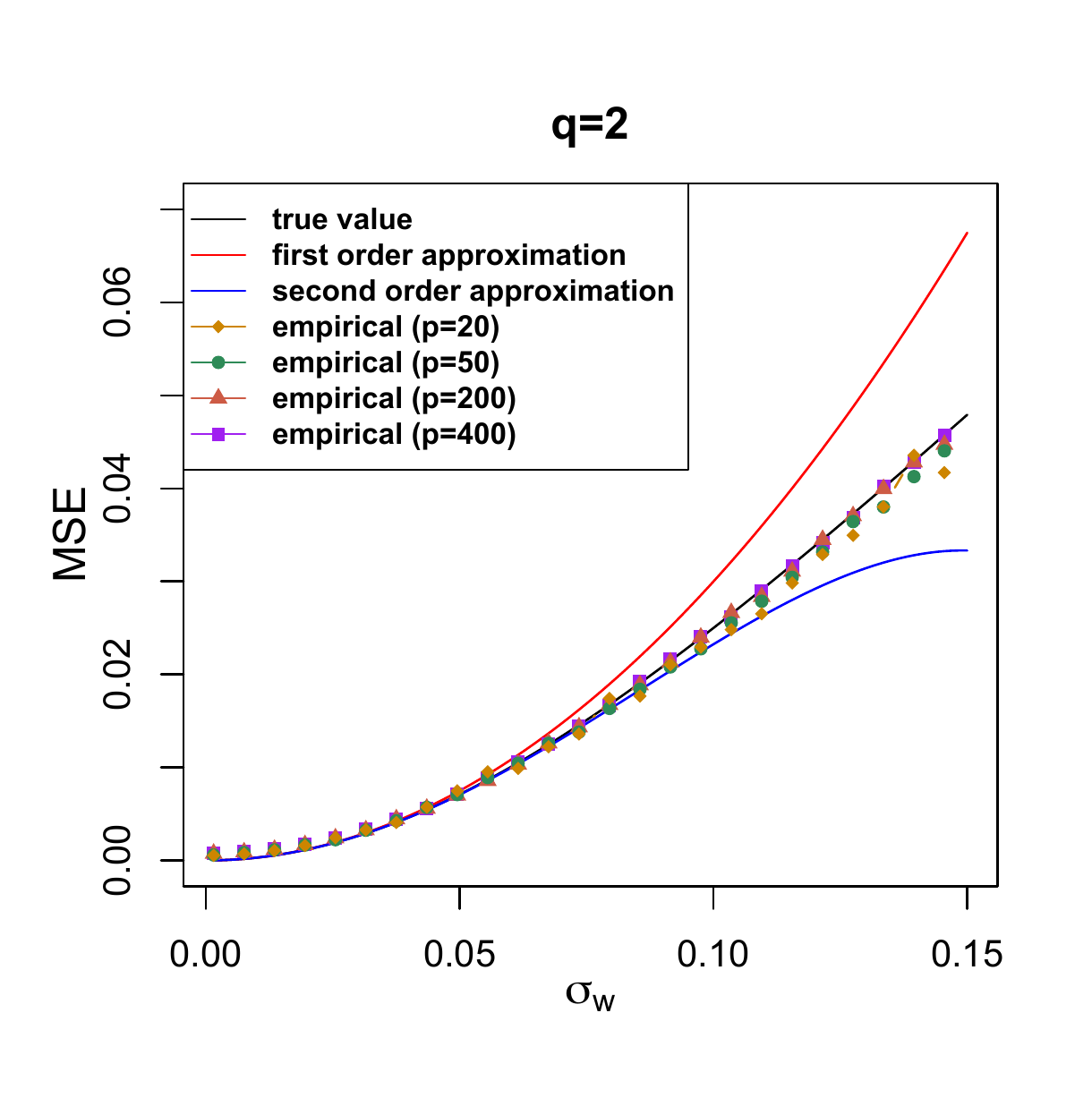} 
\end{tabular}
\caption{Plots of actual AMSE, its approximations, and finite-sample MSE. The MSE is averaged over 200 times. The setup is $\delta=1.5, \epsilon=0.4, g(b)=0.5\delta_1(b)+0.5\delta_{-1}(b)$.} \label{addfig:delta15one}
\end{figure}

\begin{figure}[htb]
\centering
\setlength\tabcolsep{1.5pt}
\begin{tabular}{cc}
\includegraphics[width=6.2cm, height=5.8cm]{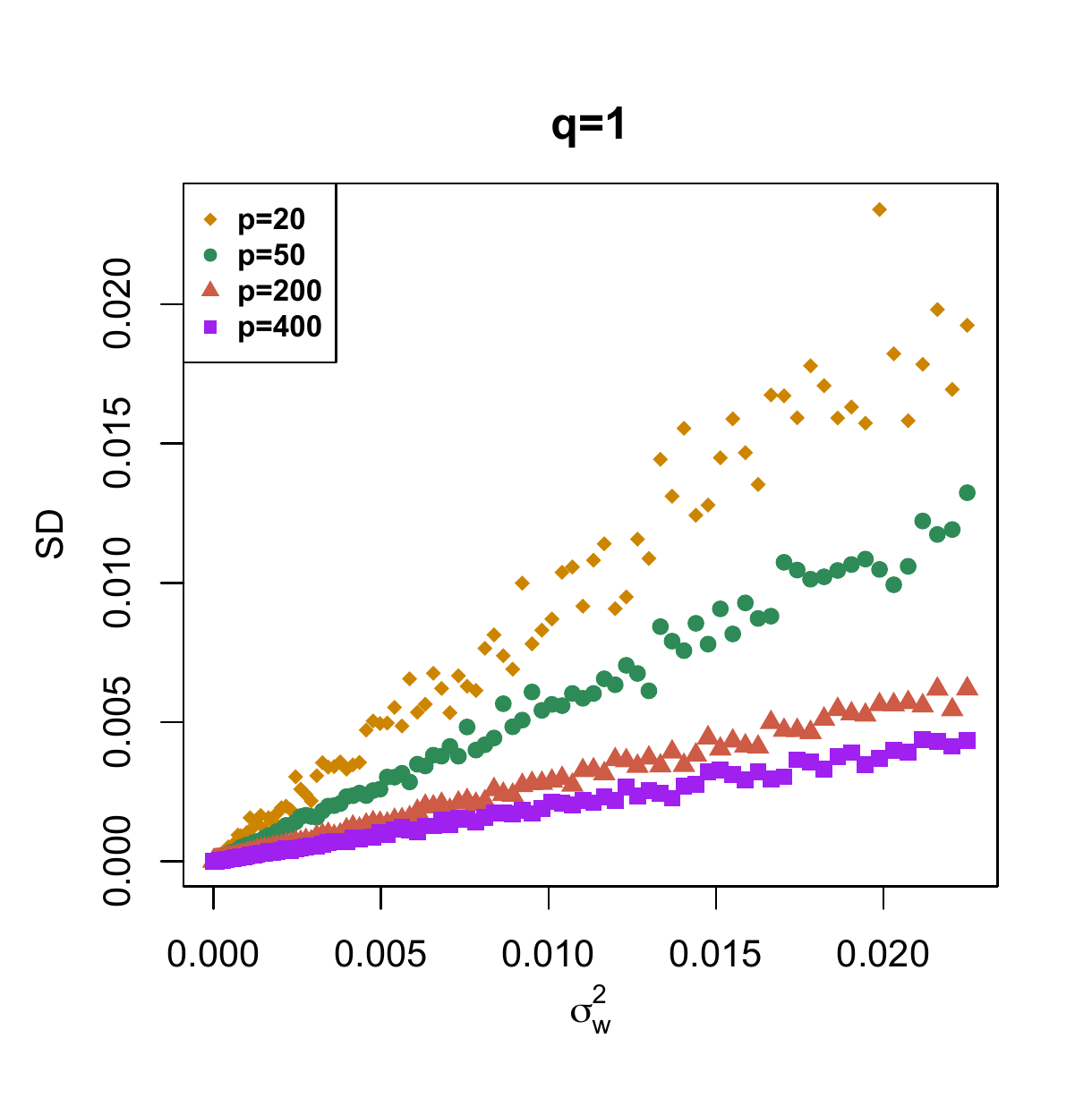} &
 \includegraphics[width=6.2cm, height=5.8cm]{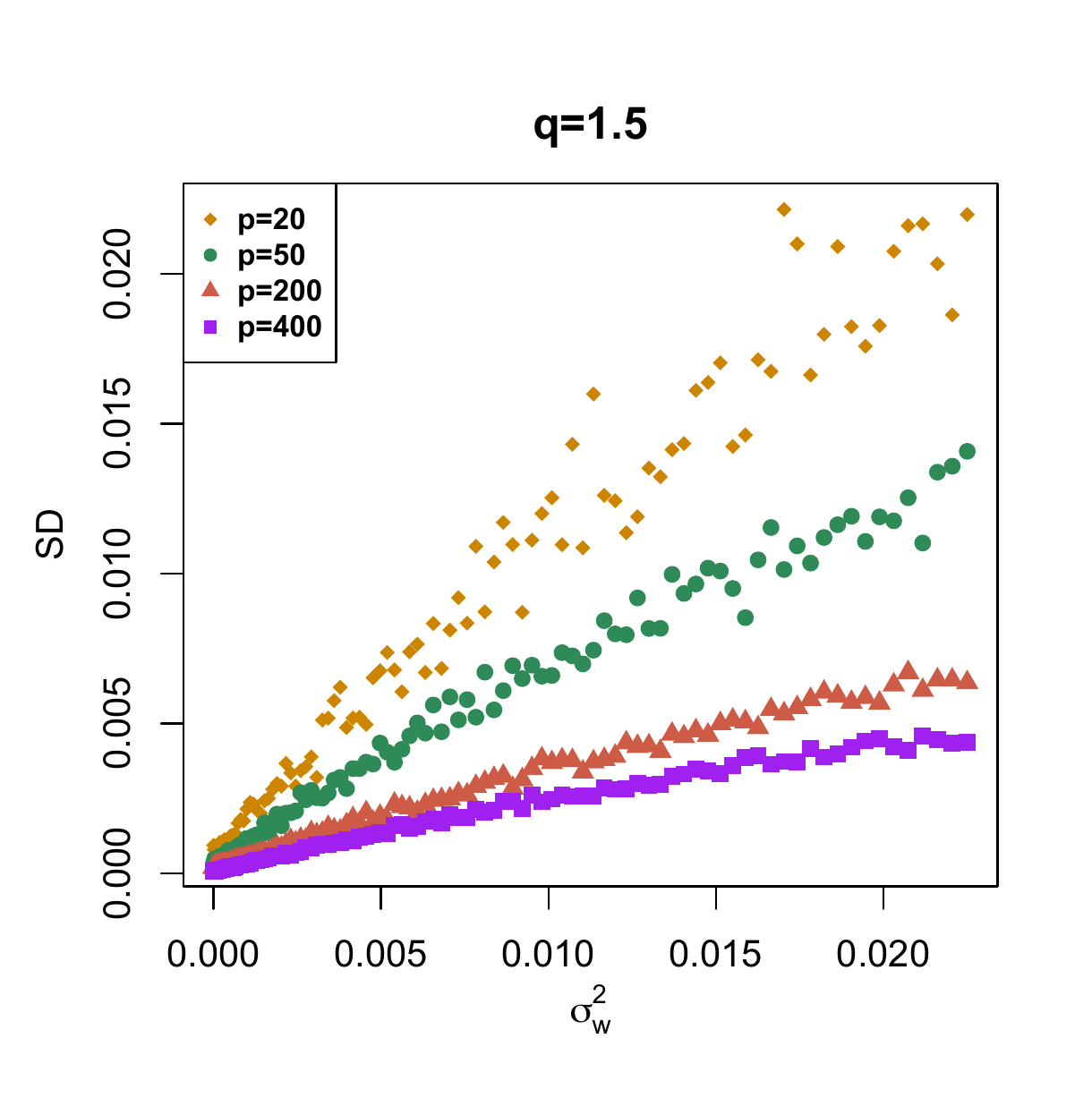} \\ 
\includegraphics[width=6.2cm, height=5.8cm]{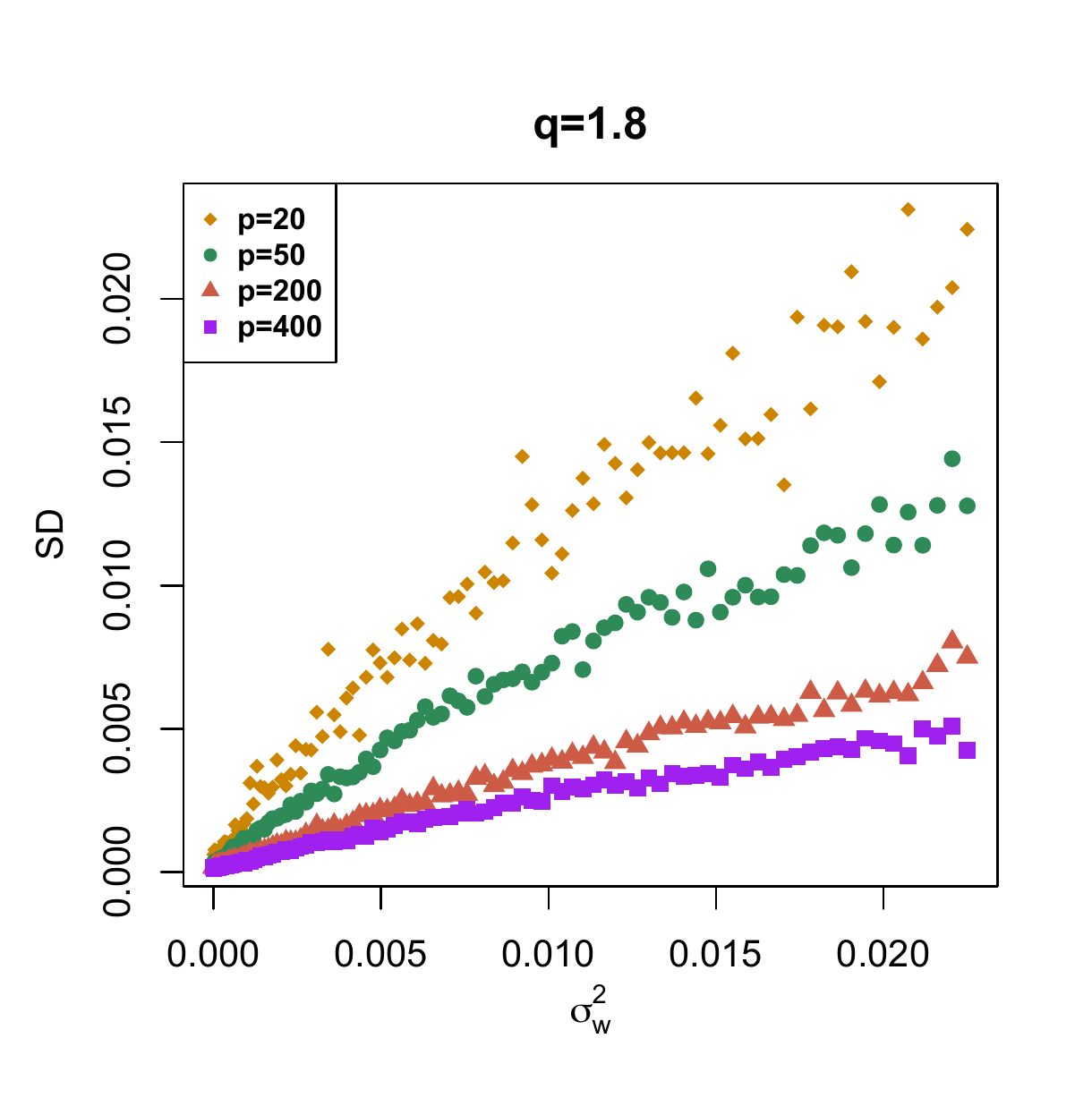} &
\includegraphics[width=6.2cm, height=5.8cm]{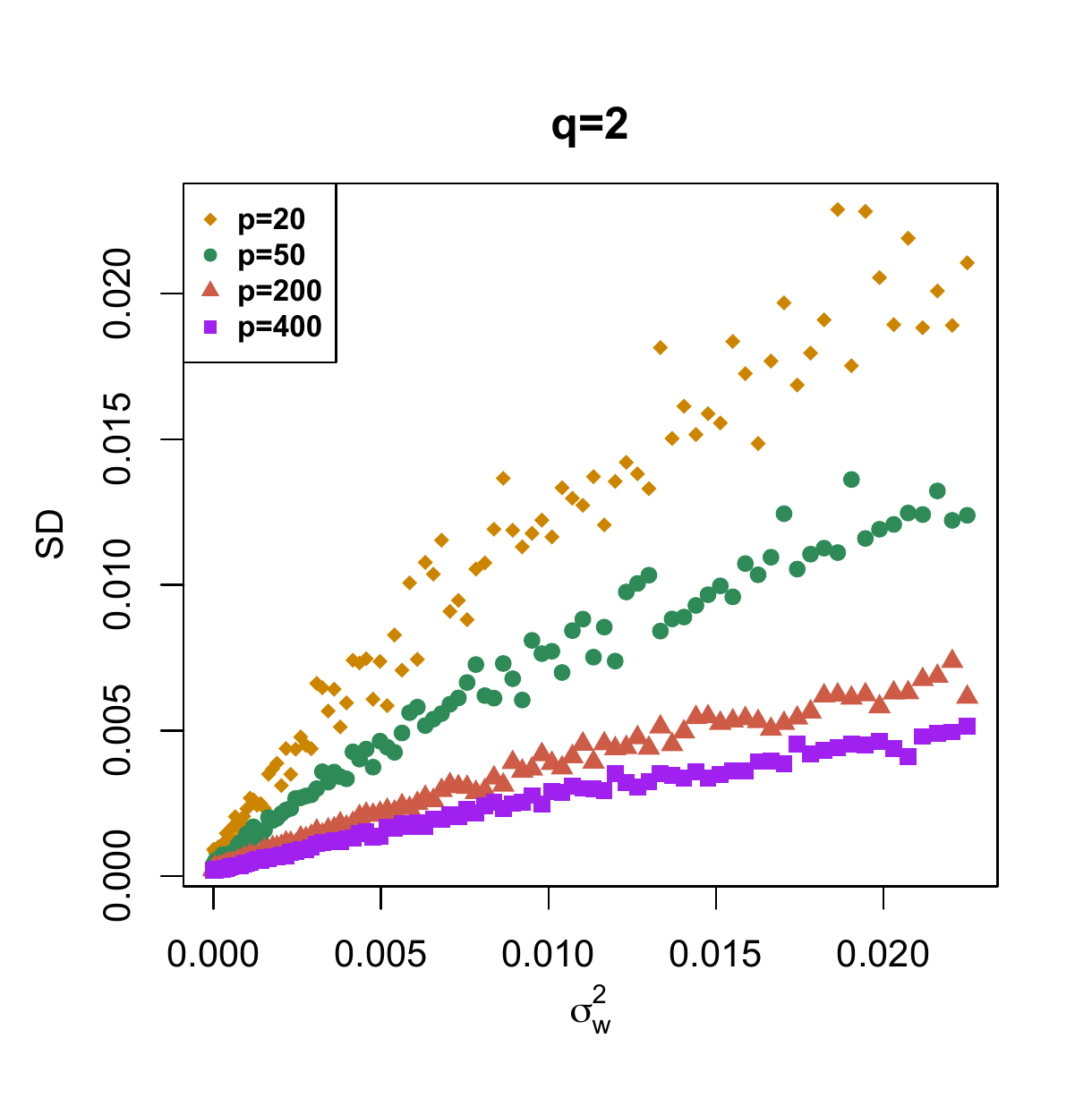} 
\end{tabular}
\caption{Plots of the standard deviation of finite-sample MSE. The same experiment is repeated 200 times. The setup is $\delta=1.5, \epsilon=0.4, g(b)=0.5\delta_1(b)+0.5\delta_{-1}(b)$.} \label{addfig:delta15two}
\end{figure}

\begin{figure}[htb]
\centering
\setlength\tabcolsep{1.5pt}
\begin{tabular}{cc}
\includegraphics[width=6.2cm, height=5.8cm]{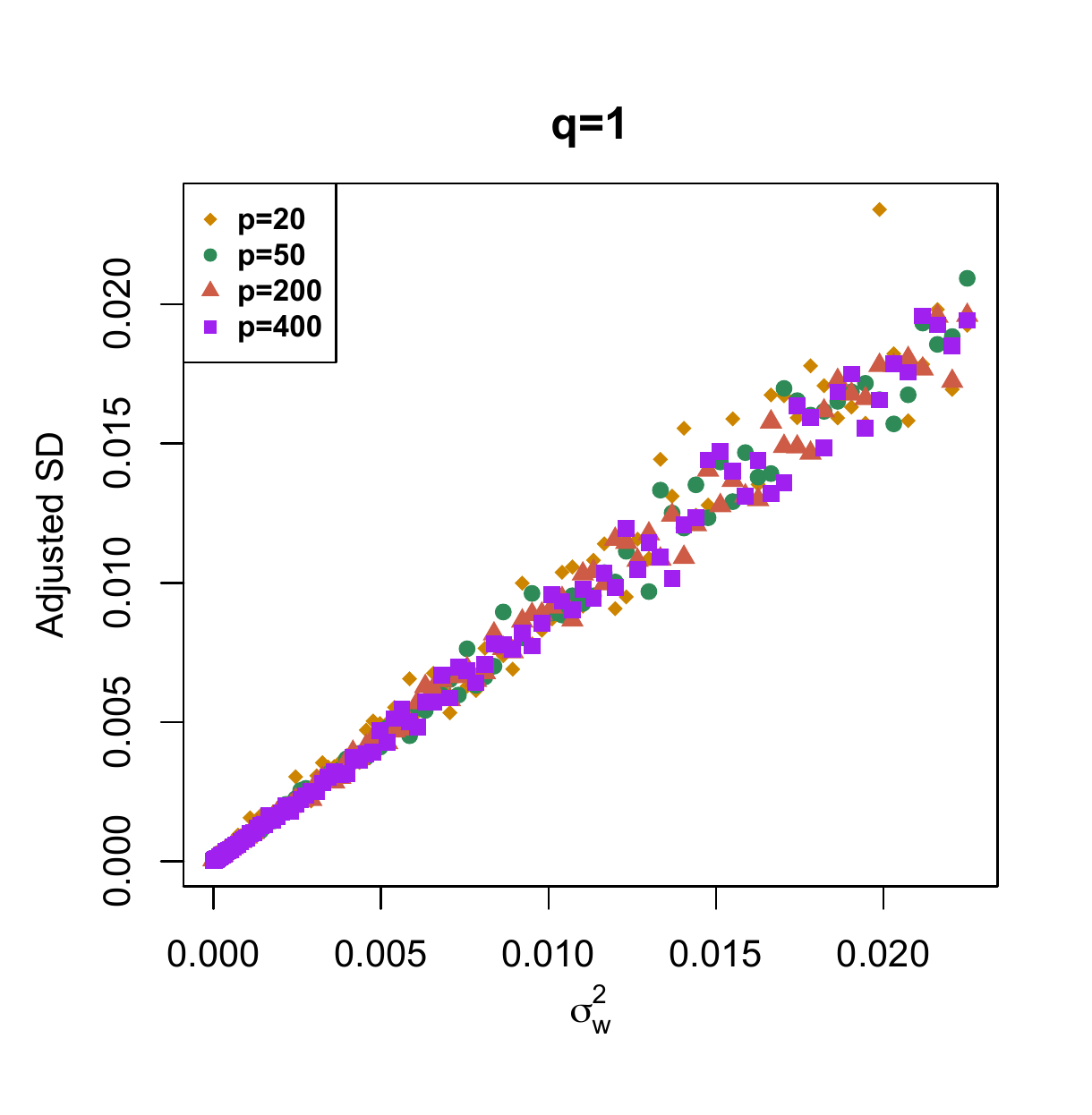} &
 \includegraphics[width=6.2cm, height=5.8cm]{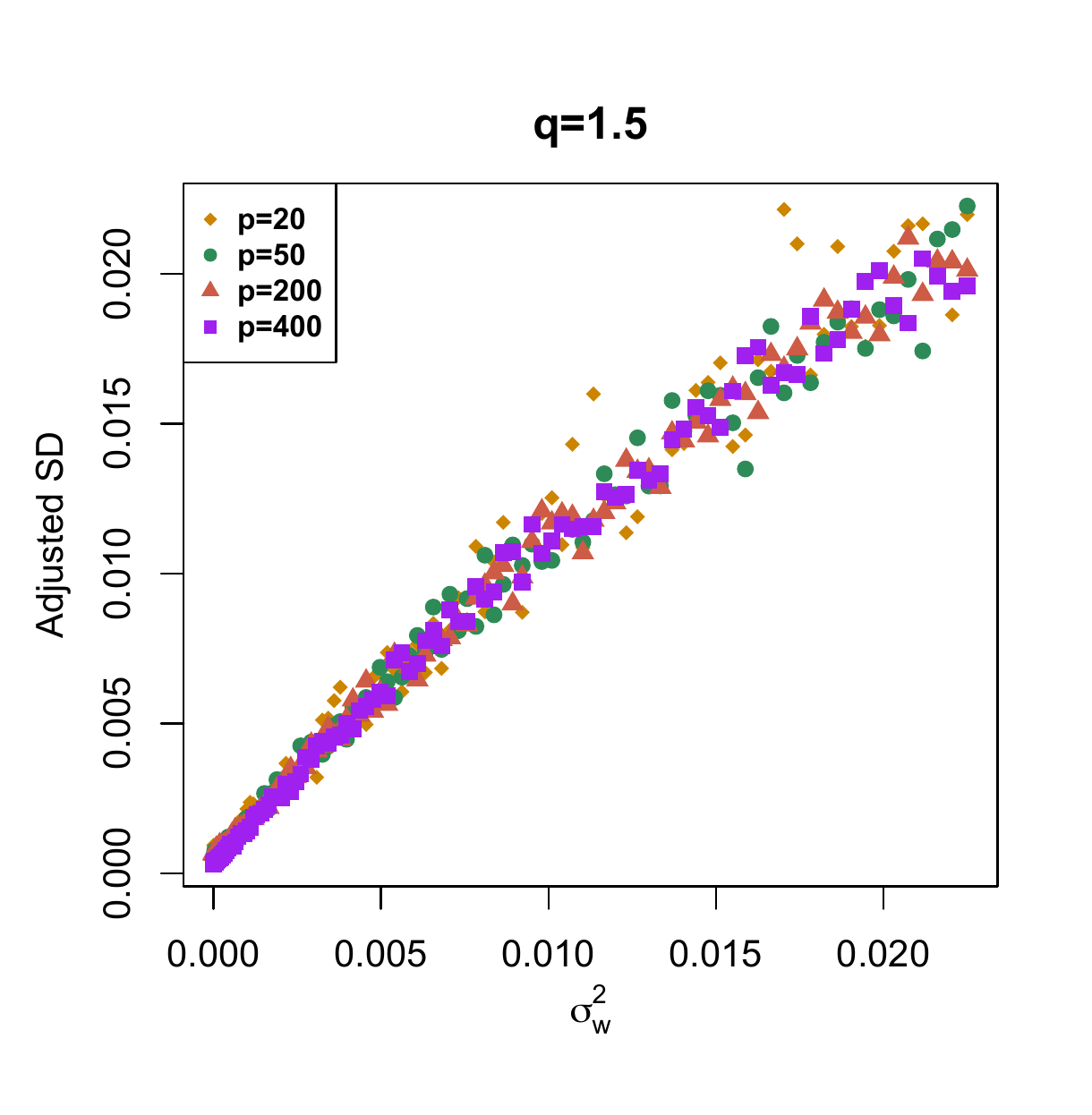} \\ 
\includegraphics[width=6.2cm, height=5.8cm]{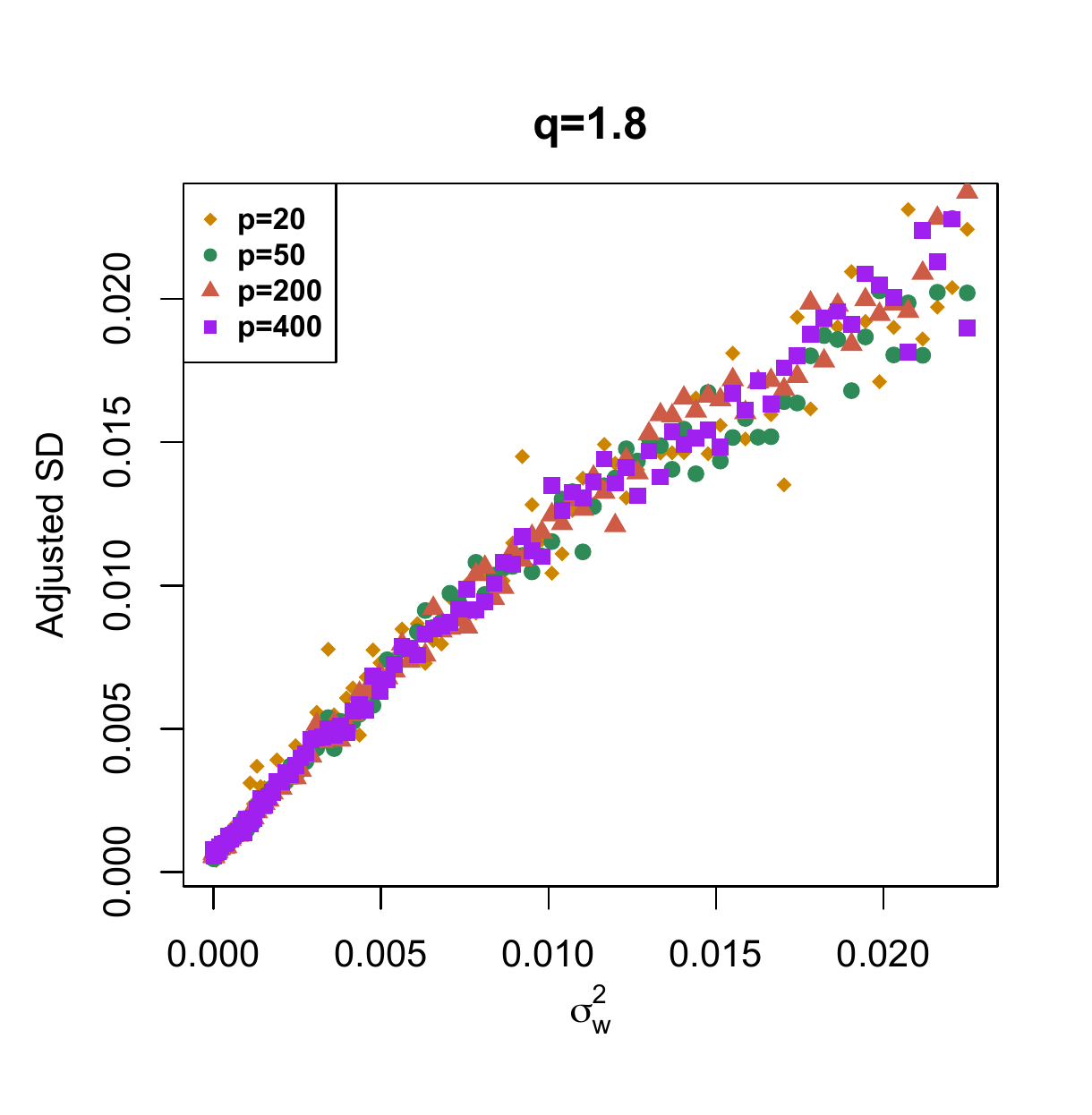} &
\includegraphics[width=6.2cm, height=5.8cm]{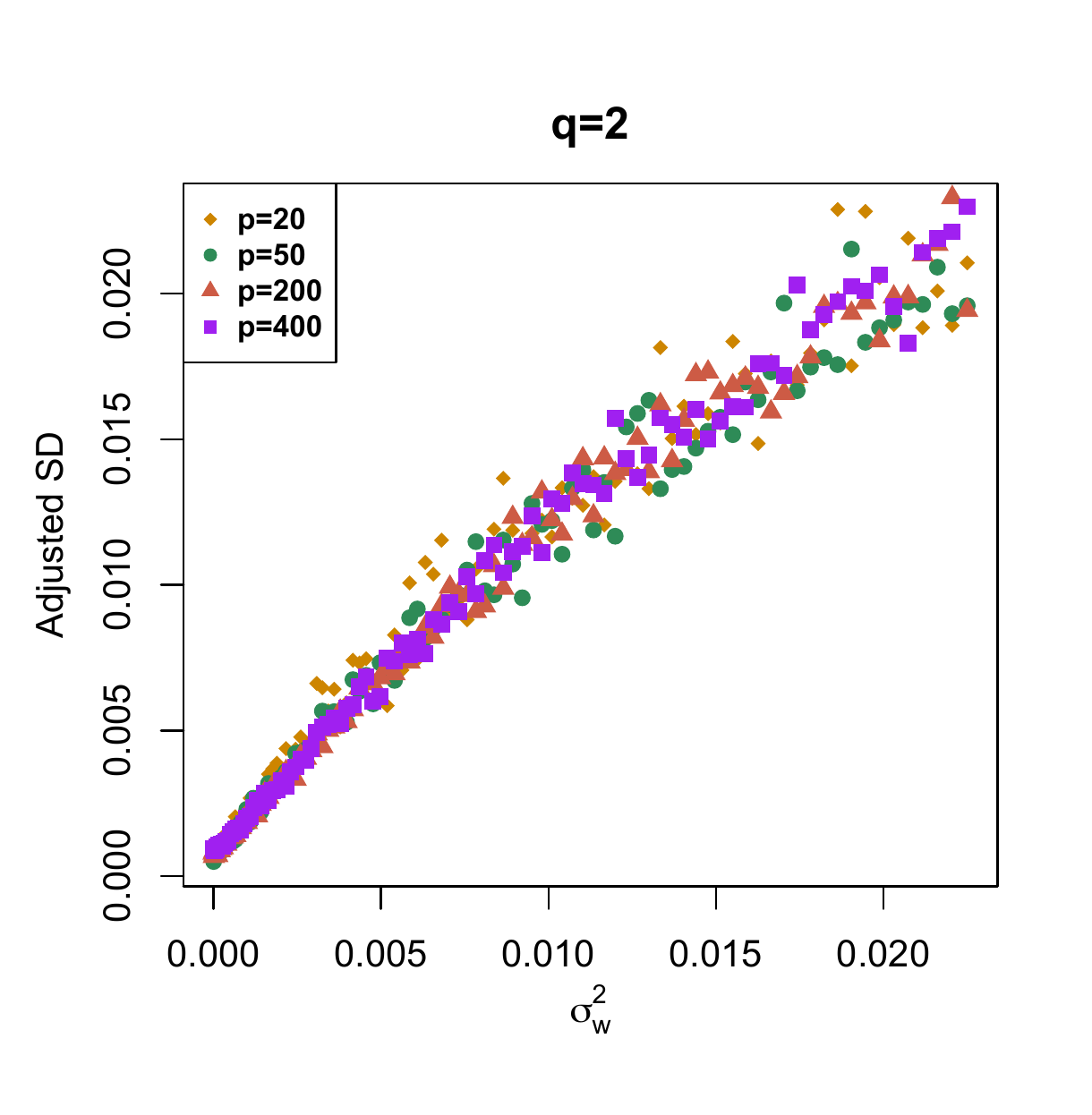} 
\end{tabular}
\caption{Plots of the adjusted standard deviation of finite-sample MSE. The same experiment is repeated 200 times. The setup is $\delta=1.5, \epsilon=0.4, g(b)=0.5\delta_1(b)+0.5\delta_{-1}(b)$.} \label{addfig:delta15three}
\end{figure}

\section{Preliminaries on ${\lowercase{\eta}}_{\lowercase{q}}\lowercase{(u;\chi)}$}\label{ssec:etaq:summary}

This section is devoted to the properties of $\eta_q(u;\chi)$ defined as 
\begin{equation*}\label{eq:defetaqproof}
\eta_q (u; \chi) \triangleq \argmin_z \frac{1}{2} (u-z)^2 + \chi |z|^q. 
\end{equation*}
We start with some basic properties of these functions. Since the explicit forms of $\eta_q(u;\chi)$ for $q=1$ and $2$ are known: $\eta_1(u; \chi) = (|u| -\chi) \rm{sign}(u) \mathbb{I} (|u|> \chi), \eta_2(u; \chi) = \frac{u}{1+ 2\chi}$, we focus our study on the case $1<q<2$.

\begin{lemma}\label{lem:toobasicpropprox}
$\eta_q(u; \chi)$ satisfies the following properties:
\begin{enumerate}
\item[(i)] $u- \eta_q(u; \chi) = \chi q {\rm sign}(u) |\eta_q(u; \chi)|^{q-1}$.
\item[(ii)] $|\eta_q(u; \chi)| \leq  |u|$. 
\item[(iii)] $\lim_{\chi \rightarrow 0} \eta_q (u; \chi) = u$ and $\lim_{\chi \rightarrow \infty}  \eta_q(u; \chi) = 0$. 
\item[(iv)] $\eta_q(-u; \chi) = - \eta_q(u;\chi)$.
\item[(v)] For $\alpha>0$, we have $\eta_q(\alpha u ; \alpha^{2-q} \chi) = \alpha \eta_q (u; \chi)$. 
\item[(vi)] $|\eta_q(u; \chi) - \eta_q(\tilde{u}, \chi)| \leq |u- \tilde{u}|$. 
\end{enumerate}
\end{lemma}

\begin{proof}
To prove (i), we should take the derivative of $ \frac{1}{2} (u-z)^2 + \chi |z|^q$ and set it to zero. Proofs of parts (ii), (iii) and (iv) are straightforward and are hence skipped. To prove (v), note that
\begin{eqnarray}
\eta_q(\alpha u ; \alpha^{2-q} \chi) &=& \arg\min_z \frac{1}{2} (\alpha u - z)^2 + \chi \alpha^{2-q} | z|^q \nonumber \\
&=& \arg\min_{z} \frac{\alpha^2}{2} (u- z/\alpha)^2+ \chi \alpha^2 |z/\alpha|^q \nonumber \\
&=& \alpha \arg\min_{\tilde{z}} \frac{1}{2} (u- \tilde{z})^2+ \chi |\tilde{z}|^q = \alpha \eta_q( u ; \chi).
\end{eqnarray}
(vi) is a standard property of proximal operators of convex functions \cite{parikh2014proximal}.
\end{proof}

In many proofs, we will be dealing with derivatives of $\eta_q(u;\chi)$. For notational simplicity, we may use $\partial_1 \eta_q(u;\chi), \partial^2_1 \eta_q(u;\chi), \partial_2 \eta_q(u;\chi), \partial^2_2 \eta_q(u;\chi)$ to represent $\frac{\partial \eta_q(u,\chi)}{\partial u}, \frac{\partial^2 \eta_q(u,\chi)}{\partial u^2}, \frac{\partial \eta_q(u,\chi)}{\partial \chi}, \frac{\partial^2 \eta_q(u,\chi)}{\partial \chi^2}$, respectively. Our next two lemmas are concerned with differentiability of $\eta_q(u;\chi)$ and its derivatives.

\begin{lemma}\label{prox:smooth}
For every $ 1 < q < 2$, $\eta_q(u; \chi)$ is a differentiable function of $(u, \chi)$ for $u \in \mathbb{R}$ and $\chi>0$ with continuous partial derivatives. Moreover, $\partial_2 \eta_q(u;\chi)$ is differentiable with respect to $u$, for any given $\chi >0$.
\end{lemma}

\begin{proof}
We start with the case $u_0, \chi_0>0$. The goal is to prove that $\eta_q(u; \chi)$ is differentiable at $(u_0, \chi_0)$.  Since $u_0>0$, $\eta_q(u_0;\chi_0)$ will be positive. Then Lemma \ref{lem:toobasicpropprox} part (i) shows $\eta_q(u_0; \chi_0)$ must satisfy 
\begin{equation}\label{eq:fpproxop1}
\eta_q(u_0; \chi_0) + \chi_0 q \eta_q^{q-1}(u_0; \chi_0) = u_0. 
\end{equation}
Define the function $F(u,\chi, v) = u- v -\chi q v^{q-1}$. Equation \eqref{eq:fpproxop1} says $F(u, \chi, v)$ is equal to zero at $(u_0, \chi_0, \eta_q (u_0;\chi_0))$. It is straightforward to confirm that the derivative of $F(u,\chi,v)$ with respect to $v$ is nonzero at $(u_0, \chi_0, \eta_q (u_0;\chi_0))$. By implicit function theorem, we can conclude $\eta_q(u ; \chi)$ is differentiable at $(u_0, \chi_0)$. Lemma \ref{lem:toobasicpropprox} part (iv) implies that the same result holds when $u_0<0$. We now focus on the point $(0, \chi_0)$. Since $\eta_q(0, \chi_0) = 0$, we obtain
\[
\partial_1 \eta_q(0;\chi_0)  =\lim_{u \rightarrow 0} \frac{ |\eta_q(u; \chi_0)|}{|u|}  \leq \lim_{u \rightarrow 0} \frac{|u|^{1/(q-1)}}{(\chi_0 q)^{1/(q-1)} |u|} =0. 
\] 
where the last inequality comes from \eqref{eq:fpproxop1}. It is straightforward to see that the partial derivative of $\eta_q(u;\chi)$ with respect to $\chi$ at $(0,\chi_0)$ exists and is equal to zero as well. So far we have proved that $\eta_q(u,\chi)$ has partial derivatives with respect to both $u$ and $\chi$ for every $u\in \mathbb{R}, \chi>0$. We next show the partial derivatives are continuous. For $u\neq 0$, the result comes directly from the implicit function theorem, because $F(u,\chi,v)$ is a smooth function when $v\neq 0$. We now turn to the proof when $u=0$.
By taking derivative with respect to $u$ on both sides of  \eqref{eq:fpproxop1}, we obtain
\begin{eqnarray}
\partial_1 \eta_q(u; \chi)  + \chi q (q-1) \eta_q^{q-2}(u; \chi) \partial_1 \eta_q(u; \chi)  = 1,  \label{revisionadd:one}
\end{eqnarray}
for any $u, \chi>0$.
Moreover, it is clear from \eqref{eq:fpproxop1} that $\eta_q(u;\chi)\rightarrow 0,$ as $(u,\chi) \rightarrow (0^+,\chi_0)$. This fact combined with \eqref{revisionadd:one} yields
\[
\lim_{(u,\chi) \rightarrow (0^+,\chi_0)} \partial_1 \eta_q(u; \chi)  =\lim_{(u,\chi) \rightarrow (0^+,\chi_0)}\frac{1}{1+ \chi q (q-1) \eta_q^{q-2}(u; \chi)} =0.
\]
Since $\partial_1 \eta_q(u;\chi) =\partial_1 \eta_q(-u;\chi)$ implied by Lemma \ref{lem:toobasicpropprox} part (iv), we conclude
\[
\lim_{(u,\chi) \rightarrow (0,\chi_0)}  \partial_1 \eta_q(u; \chi)  =0. 
\] 
The same approach can prove that the partial derivative $\partial_2 \eta_q(u; \chi) $ is continuous at $(0, \chi_0)$. For simplicity we do not repeat the arguments.

We now prove the second part of the lemma. Because $F(u,\chi,v)$ is infinitely many times differentiable in any open set with $v\neq 0$, implicit function theorem further implies $\partial_2 \eta_q(u;\chi)$ is differentiable at any $u\neq 0$. The rest of the proof is to show its differentiability at $u=0$. This follows by noting $\partial_2 \eta_q(0;\chi)=0$, and
\begin{eqnarray*}
\lim_{u\rightarrow 0}\frac{\partial_2 \eta_q(u;\chi)}{u}  &\overset{(a)}{=}&\lim_{u\rightarrow 0} \frac{-q|\eta_q(u;\chi)|^{q-1}}{|u|(1+\chi q(q-1)|\eta_q(u;\chi)|^{q-2})}  \\
&=&\lim_{u\rightarrow 0} \frac{-(|u|-|\eta_q(u;\chi)|)}{\chi |u|(1+\chi q(q-1)|\eta_q(u;\chi)|^{q-2})}=0,
\end{eqnarray*} 
where $(a)$ is by taking derivative with respect to $\chi$ on both sides of \eqref{eq:fpproxop1}, and the last two equalities above are due to Lemma \ref{lem:toobasicpropprox} part (i) and (iii). 
\end{proof}

\begin{lemma}\label{lem:continuitysecondder}
Consider a given $\chi>0$, then for every $ 1 < q < 3/2$, $\partial_1 \eta_q(u; \chi)$ is a differentiable function of $u$ for $u \in \mathbb{R}$ with continuous derivative; for $q=3/2$, it is a weakly differentiable function of $u$; for $3/2<q<2$, $\partial_1 \eta_q(u; \chi)$ is differentiable at $u \neq 0$, but is not differentiable at zero.
\end{lemma}

\begin{proof}

As is clear from the proof of Lemma \ref{prox:smooth},  the implicit function theorem guarantees that $\partial_1\eta_q(u;\chi)$ is differentiable at $u\neq 0$ with continuous derivative for $1<q<2$.  Hence we will be focused on $u=0$. In the proof of Lemma \ref{prox:smooth}, we have derived
\begin{equation}\label{eq:formfirstderarg1}
\partial_1 \eta_q(u;\chi) = \frac{1}{1+ \chi q (q-1) |\eta_q(u; \chi)|^{q-2}}, \quad \mbox{for~} u \neq 0,
\end{equation}
and $\partial_1 \eta_q(0;\chi)=0$. We thus know
\begin{eqnarray}
\partial_1^2\eta_q(0;\chi)=\lim_{u\rightarrow 0}\frac{1}{u+\chi q(q-1)u|\eta_q(u;\chi)|^{q-2}}.  \label{smooth:lq:one} \hspace{-0.cm}
\end{eqnarray}
Moreover, Lemma \ref{lem:toobasicpropprox} part (i) implies
\begin{eqnarray}
\lim_{u\rightarrow 0}\frac{u}{\chi q |\eta_q(u;\chi)|^{q-1}\mbox{sign}(u)}=1+ \lim_{u\rightarrow 0}\frac{|\eta_q(u;\chi)|^{2-q}}{\chi q}=1. \label{addrevision:two}
\end{eqnarray}
For $1<q<3/2$, \eqref{smooth:lq:one} and \eqref{addrevision:two} together give us
\begin{eqnarray*}
\partial_1^2\eta_q(0;\chi)=\lim_{u\rightarrow 0}\frac{1}{u+\chi q(q-1)u(|u|/(\chi q))^{\frac{q-2}{q-1}}}=0.
\end{eqnarray*}
We can also calculate the limit of $\partial_1^2 \eta_q(u;\chi)$ (this second derivative can be obtained from \eqref{eq:formfirstderarg1}) as follows.
\begin{eqnarray*}
\lim_{u\rightarrow 0}\partial_1^2 \eta_q(u;\chi)=\lim_{u \rightarrow 0}\frac{-\chi q(q-1)(q-2)|\eta_q(u;\chi)|^{q-3}\mbox{sign}(u)}{(1+\chi q(q-1)|\eta_q(u;\chi)|^{q-2})^3}=0.
\end{eqnarray*}
Therefore, $\partial_1 \eta_q(u; \chi)$ is continuously differentiable on $(-\infty,+\infty)$ for $1<q<3/2$. Regarding $3/2<q<2$, Similar calculations yield
\[
\lim_{u\rightarrow 0^+}\partial_1^2\eta_q(0;\chi)=+\infty, \quad \lim_{u\rightarrow 0^-}\partial_1^2\eta_q(0;\chi)=-\infty.
\] 

Finally to prove the weak differentiability for $q=3/2$, we show $\partial_1 \eta_{q}(u;\chi)$ is a Lipschitz continuous function on $(-\infty,+\infty)$. Note that for $u\neq 0$,
\begin{eqnarray*}
|\partial_1^2 \eta_q(u;\chi)|= \frac{\chi q(q-1)(2-q)|\eta_q(u;\chi)|^{q-3}}{(1+\chi q(q-1)|\eta_q(u;\chi)|^{q-2})^3}  \leq \frac{8}{9\chi^2},
\end{eqnarray*}
and $\partial^2_1\eta_q(0^+;\chi)=-\partial^2_1(0^-;\chi)=\frac{8}{9\chi^{2}}$. Mean value theorem leads to
\[
|\partial_1\eta_{q}(u;\chi)-\partial_1 \eta_{q}(\tilde{u};\chi)| \leq \frac{8}{9\chi^2}|u-\tilde{u}|, ~~~\mbox{for~}u\tilde{u}\geq 0. 
\]
When $u\tilde{u}<0$, we can have 
\begin{eqnarray*}
&&|\partial_1\eta_{q}(u;\chi)-\partial_1 \eta_{q}(\tilde{u};\chi)|=|\partial_1\eta_{q}(u;\chi)-\partial_1 \eta_{q}(-\tilde{u};\chi)|  \\
&\leq& \frac{8}{9\chi^2}|u+\tilde{u}| \leq \frac{8}{9\chi^2}|u-\tilde{u}|.
\end{eqnarray*}
This completes the proof of the lemma.
\end{proof}

The last lemma in this section presents some additional properties regarding the derivatives of $\eta_q(u;\chi)$.
\begin{lemma}\label{lem:toobasicpropproxder}
The derivatives of $\eta_q(u; \chi)$ satisfy the following properties:
\begin{enumerate}
\item[(i)] $\partial_1 \eta_q(u;\chi) = \frac{1}{1+ \chi q (q-1) |\eta_q(u; \chi)|^{q-2}}$.  
\item[(ii)] $\partial_2 \eta_q (u; \chi) = \frac{-q |\eta_q (u; \chi)|^{q-1}  {\rm sign} (u) }{ 1+ \chi q (q-1) |\eta_q (u ;\chi)|^{q-2}}$. 
\item[(iii)] $0\leq \partial_1 \eta_q(u;\chi) \leq 1$.
\item[(iv)] For $u>0$, $\partial_1^2\eta_q(u;\chi) >0$. 
\item[(v)] $|\eta_q(u; \chi)|$ is a decreasing function of $\chi$. 
\item [(vi)] $\lim_{\chi \rightarrow \infty} \partial_1 \eta_q(u;\chi) =0$. 
\end{enumerate}
\end{lemma}

\begin{proof}
Parts (i) (ii) have been derived in the proof of Lemma \ref{prox:smooth}. Part (iii) is a simple conclusion of part (i). Part (iv) is clear from the proof of Lemma \ref{lem:continuitysecondder}. Part (v) is a simple application of part (ii). Finally, part (vi) is an application of part(i) of Lemma \ref{lem:toobasicpropproxder} and part (iii) of Lemma \ref{lem:toobasicpropprox}. 
\end{proof}

\section{Proof of lemma 2} \label{subsec:solutionfixedpoint}

\subsection{Roadmap}
The goal of this section is to show that Equations \eqref{eq:fixedpoint11} and \eqref{eq:fixedpoint21} (rewritten below) have a unique solution $(\bar{\sigma},\bar{\chi})$:
\begin{eqnarray*}
\bar{\sigma}^2 &=& \sigma_{\omega}^2+\frac{1}{\delta} \mathbb{E}_{B, Z} [(\eta_q(B +\bar{\sigma} Z; \bar{\chi} \bar{\sigma}^{2-q}) -B)^2],   \\
\lambda &=& \bar{\chi} \bar{\sigma}^{2-q} \left(1-\frac{1}{\delta} \mathbb{E}_{B,Z}[\eta_q'(B +\bar{\sigma} Z;  \bar{\chi} \bar{\sigma}^{2-q})] \right).
\end{eqnarray*}

Towards the goal we pursue the following two main steps:
\begin{enumerate}
\item We show the existence of the solution in Section \ref{revision:lemma2:unique}. In order to do that, we first study the solution of Equation \eqref{eq:fixedpoint11}, and demonstrate that for any $\chi \in (\chi_{\min}, \infty)$ ($\chi_{\min}$ is a constant we will clarify later), there exists a unique $\sigma_{\chi}$ such that $(\sigma_{\chi},\chi)$ satisfies  \eqref{eq:fixedpoint11}. The proof will be presented in Section \ref{ssec:uniquenessfirsteqfp}. We then show that by varying $\chi$ over $(\chi_{\min},\infty)$, the range of the value of the following term
\[
\chi{\sigma_{\chi}}^{2-q} \big(1-\frac{1}{\delta} \mathbb{E}_{B,Z}[\eta_q'(B +\sigma_{\chi} Z;  \chi \sigma_{\chi}^{2-q})] \big)
\]
covers the number $\lambda$ from Equation \eqref{eq:fixedpoint21}. That means Equations \eqref{eq:fixedpoint11} and \eqref{eq:fixedpoint21} share at least one common solution pair $(\sigma_{\chi}, \chi)$. We detail out the proof in Section \ref{revision:lemma2:unique2}.

\item We prove the uniqueness of the solution in Section \ref{revision:lemma2:three}. The key idea is to apply Theorem \ref{thm:eqpseudolip} to evaluate the asymptotic loss of the LQLS estimates under two different pseudo-Lipschitz functions. These two quantities determine the uniqueness of both $\sigma_{\chi}$ and $\chi$ in the common solution pair $(\sigma_{\chi}, \chi)$. Note that we have denoted this unique pair by $(\bar{\sigma}, \bar{\chi})$.

\end{enumerate}

Before we start the details of the proof, we present Stein's lemma \cite{stein1981estimation} that will be used several times in this paper.

\begin{lemma}\label{lem:steins}
Let $g: \mathbb{R} \rightarrow \mathbb{R}$ denote a weakly differentiable function. If $Z \sim N(0,1)$ and $\mathbb{E}|g'(Z)|<\infty$, we have
\[
\mathbb{E} (Z g(Z)) = \mathbb{E} (g'(Z)),
\]
where $g'$ denotes the weak-derivative of $g$. 
\end{lemma}

\subsection{Proving the existence of the solution of \eqref{eq:fixedpoint11} and \eqref{eq:fixedpoint21}} \label{revision:lemma2:unique}

\subsubsection{Studying the solution of \eqref{eq:fixedpoint11}} \label{ssec:uniquenessfirsteqfp}

We first define a function that is closely related to Equation \eqref{eq:fixedpoint11}:
 \begin{equation}\label{keyquantity:def}
R_q(\chi, \sigma) \triangleq \mathbb{E}_{B,Z} [\eta_q(B/ \sigma +Z; \chi) - B/\sigma]^2.
\end{equation}
Note that we have used the same definition for LASSO in Section \ref{sec:proofthm4full}. Here we adopt a general notation $R_q(\chi, \sigma)$ to represent the function defined above for any $q\in [1,2]$.

\begin{lemma}\label{lem:posfirstlemma}
For $1\leq q \leq 2$, $R_q(\chi, \sigma)$ is a decreasing function of $\sigma>0$.
\end{lemma}
 \begin{proof}
We consider four different cases: (i) $q=1$, (ii) $q=2$, (iii) $1<q \leq 3/2$, (iv) $3/2<q<2$. 
\begin{itemize}
\item[(i)] $q=1$: Since $R_1(\chi, \sigma)$ is a differentiable function of $\sigma$, we will prove this case by showing that $\frac{\partial R_1(\chi, \sigma)}{\partial  \sigma} <0$. We have
\begin{eqnarray}
\frac{\partial R_1(\chi,\sigma)}{\partial  \sigma} &=& -\frac{2}{\sigma^2} \mathbb{E} \left[ B(\mathbb{I} (|B/ \sigma +Z|> \chi)-1)(\eta_1(B/ \sigma +Z; \chi) - B/\sigma)  \right] \nonumber \\ &=& - \frac{2}{\sigma^2} \mathbb{E} \left[ \mathbb{I} (|B/ \sigma +Z|\leq  \chi)(B^2/\sigma)  \right]<0. \nonumber
\end{eqnarray}
The first equality above is due to Dominated Convergence Theorem.

\item[(ii)] $q=2$: Since $\eta_2(u; \chi) = \frac{u}{1+2\chi}$, we have
\begin{eqnarray}
\frac{\partial R_2(\chi, \sigma) }{\partial \sigma} =  - \frac{8\chi^2 \mathbb{E}|B|^2}{(1+2\chi)^2 \sigma^3}<0. \nonumber 
\end{eqnarray}

\item[(iii)] $1<q \leq 3/2$: The strategy for this case is similar to that of the last two cases. We show that the derivative $\frac{\partial R_q(\chi,\sigma)}{\partial \sigma}<0$.
\begin{eqnarray}
 \hspace{-0.4cm}\frac{ \partial R_q(\chi, \sigma)}{\partial \sigma} &\overset{(a)}{=}& 2\mathbb{E} \left[ (\eta_q(B/ \sigma +Z; \chi) - B/\sigma )(\partial_1\eta_q(B/ \sigma +Z; \chi) - 1 ) (-B/\sigma^2) \right] \nonumber \\
&&\hspace{-1.7cm} =  2\mathbb{E} \left[ (\eta_q(B/ \sigma +Z; \chi) - B/\sigma-Z )(\partial_1\eta_q(B/ \sigma +Z; \chi) - 1 ) (-B/\sigma^2) \right] \nonumber \\
&&    + 2 \mathbb{E} \left[ Z(\partial_1\eta_q(B/ \sigma +Z; \chi) - 1 ) (-B/\sigma^2) \right].  \label{eq:uniquefp1} 
\end{eqnarray}
 To obtain Equality (a), we have used Dominated Convergence Theorem (DCT); We employed Lemma \ref{lem:toobasicpropprox} part (vi) to confirm the conditions of DCT. Our goal is to show that the two terms in \eqref{eq:uniquefp1} are both negative. Regarding the first term,  we first evaluate it by conditioning on $B=b$ for a given constant $b>0$ (note that $B$ and $Z$ are independent):
 \begin{eqnarray}
 \lefteqn{ \mathbb{E}_Z \left[ (\eta_q(b/ \sigma +Z; \chi) - b/\sigma-Z )(\partial_1\eta_q(b/ \sigma +Z; \chi) - 1 )  \right]} \nonumber \\
  & \overset{(b)}{=}& \int_{-\infty}^{+\infty} (\eta_q(z; \chi) -z ) (\partial_1 \eta_q (z;\chi)-1) \phi(z -b/\sigma)dz  \nonumber \\
  &=&  \int_0^\infty (\eta_q(z; \chi) -z ) (\partial_1 \eta_q (z;\chi)-1) \phi(z -b/\sigma)dz +\nonumber \\
 &&  \int_{-\infty}^0 (\eta_q(z; \chi) -z ) (\partial_1 \eta_q (z;\chi)-1) \phi(z -b/\sigma)dz \nonumber \\
  &\overset{(c)}{=} &\hspace{-0.2cm} \int_0^\infty (\eta_q(z; \chi) -z ) (\partial_1 \eta_q (z;\chi)-1) (\phi(z -b/\sigma)- \phi (z + b/ \sigma))dz\overset{(d)}{>}0, \nonumber 
 \end{eqnarray}
where $\phi(\cdot)$ is the density function of standard normal; $(b)$ is obtained by a change of variables; $(c)$ is due to the fact $\partial_1 \eta_q(-z; \chi) = \partial_1 \eta_q(z; \chi)$ implied by Lemma \ref{lem:toobasicpropprox} part (iv); $(d)$ is based on the following arguments: According to Lemmas \ref{lem:toobasicpropprox} part (ii) and Lemma \ref{lem:toobasicpropproxder} part (iii), $\eta_q(z; \chi)<z$ and $\partial_1\eta_q(z; \chi) <1$ for $z>0$. Moreover, $\phi(z -b/\sigma)- \phi (z + b/ \sigma)>0$ for $z, b/\sigma>0$. Hence we have
\[
 \mathbb{E}_Z \left[ (\eta_q(b/ \sigma +Z; \chi) - b/\sigma-Z )(\partial_1\eta_q(b/ \sigma +Z; \chi) - 1 )(-b/\sigma^2)  \right] <0.
 \]
 Similarly we can show the above inequality holds for $b<0$. It is clear that the term on the left hand side equals zero when $b=0$. Thus we have proved the first term in \eqref{eq:uniquefp1} is negative. Now we should discuss the second term. Again we condition on $B=b$ for a given $b>0$:
  \begin{eqnarray}\label{eq:lasttermpsotivity}
  \mathbb{E}_Z \left[ Z(\partial_1\eta_q(b/ \sigma +Z; \chi) - 1 ) \right] \overset{(e)}{=} \mathbb{E} (\partial_1^2\eta_q(b/ \sigma +Z; \chi) ) \nonumber \\
 = \int_0^\infty  [\partial_1^2\eta_q(z; \chi)  (\phi(z- b/\sigma) - \phi (z+ b/\sigma))]dz>0.
  \end{eqnarray}
  Equality (e) is the result of Stein's lemma, i.e. Lemma \ref{lem:steins}. Note that the weak differentiability condition required in Stein's lemma is guaranteed by Lemma \ref{lem:continuitysecondder}. To obtain the last inequality, we have used Lemma \ref{lem:toobasicpropproxder} part (iv) and the fact that $\phi(z- b/\sigma) - \phi (z+ b/\sigma)>0$ for $z, b/\sigma>0$. Hence we obtain that
  \[
 \mathbb{E}_Z \left[ Z(\partial_1\eta_q(b/ \sigma +Z; \tau) - 1 )(-b/\sigma^2)  \right] <0.
   \]
The same approach would work for $b<0$, and clearly the left hand side term of the above inequality equals zero for $b=0$. We can therefore conclude the second term in \eqref{eq:uniquefp1} is negative as well.
   
\item[(iv)] $3/2 <q< 2$: The proof of this case is similar to the last one. The only difference is that the proof steps we presented in \eqref{eq:lasttermpsotivity} may not work, due to the non-differentiability of $\partial_1\eta_q(u; \chi)$ for $q>3/2$ as shown in Lemma \ref{lem:continuitysecondder}. Our goal here is to use an alternative approach to prove: $\mathbb{E}_Z\left[ Z(\partial_1\eta_q(b/ \sigma +Z; \chi) - 1 ) \right] >0$ for $b>0$. We have
   \begin{eqnarray}
&&  \mathbb{E}_Z \left[ Z(\partial_1\eta_q(b/ \sigma +Z; \chi) - 1 ) \right]  = \int_{-\infty}^{\infty} z (\partial_1 \eta_q(b/ \sigma +z; \chi) - 1 ) \phi(z) dz \nonumber \\
  &&\hspace{0.5cm}=  \int_{0}^{\infty} z( \partial_1 \eta_q(b/ \sigma +z; \chi) -  \partial_1 \eta_q(b/ \sigma -z; \chi)) \phi(z) dz \nonumber \\
  &&\hspace{0.5cm}=  \int_{0}^{\infty} z( \partial_1 \eta_q(|b/ \sigma +z|; \chi) -  \partial_1 \eta_q(|b/ \sigma -z|; \chi)) \phi(z) dz,  \label{newrevision:mono1}
   \end{eqnarray}
 where the last equality is due to the fact $\partial_1 \eta_q(u;\chi)=\partial_1 \eta_q(|u|;\chi)$ for any $u\in\mathbb{R}$. Since $|b/\sigma -z|< |b/\sigma+z|$ for $z, b/\sigma>0$ and according to Lemma \ref{lem:toobasicpropproxder} part (iv),  we obtain 
 \begin{eqnarray}
 \partial_1 \eta_q(|b/ \sigma +z|; \chi) -  \partial_1 \eta_q(|b/ \sigma -z|; \chi)>0.  \label{newrevision:mono2}
 \end{eqnarray}
 Combining \eqref{newrevision:mono1} and \eqref{newrevision:mono2} completes the proof.
 \end{itemize}
 \end{proof}
  
  Lemma \ref{lem:posfirstlemma} paves our way in the study of the solution of \eqref{eq:fixedpoint11}. Define
  \begin{eqnarray} \label{def:chimin}
  \chi_{\min}= \inf \big \{\chi \geq 0: \frac{1}{\delta}\mathbb{E}(\eta^2_q(Z;\chi)) \leq 1  \big \},
  \end{eqnarray}
 where $Z \sim N(0,1)$. The following corollary is a conclusion from Lemma \ref{lem:posfirstlemma}. 
  
  \begin{corollary}\label{cor:fixedpointeqanalysis}
  For a given $1\leq q \leq 2$, Equation \eqref{eq:fixedpoint11}:
  \[
  \sigma^2 = \sigma_{\omega}^2+\frac{1}{\delta} \mathbb{E}_{B, Z} [(\eta_q(B +\sigma Z; \chi \sigma^{2-q}) -B)^2], \quad \sigma_w>0
 \]
  has a unique solution $\sigma=\sigma_{\chi}$ for any $\chi \in (\chi_{\min},\infty)$, and does not have any solution if $\chi \in (0, \chi_{\min})$.
  \end{corollary}
  \begin{proof}
  First note that since $\sigma_w>0$, $\sigma=0$ is not a solution of \eqref{eq:fixedpoint11}. Hence we can equivalently write Equation \eqref{eq:fixedpoint11} in the following form: 
\begin{eqnarray} \label{eq:fix:point:equation}
1= \frac{\sigma_w^2}{ \sigma^2} + \frac{1}{\delta}R_q(\chi, \sigma) \triangleq F(\sigma, \chi).
\end{eqnarray}
According to Lemma \ref{lem:posfirstlemma}, $F(\sigma,\chi)$ is a strictly decreasing function of $\sigma$ over $(0,\infty)$. We also know that $F(\sigma, \chi)$ is a continuous function of $\sigma$ from the proof of Lemma \ref{lem:posfirstlemma}. Moreover, it is straightforward to confirm that
   \begin{eqnarray}\label{eq:cond0andinf}
\lim_{\sigma \rightarrow 0} F(\sigma, \chi)= \infty,  \quad \lim_{\sigma \rightarrow \infty} F(\sigma,\chi)= \frac{1}{\delta} \mathbb{E} (\eta^2_q(Z; \chi)).
   \end{eqnarray}
Thus Equation \eqref{eq:fixedpoint11} has a solution (the uniqueness is automatically guaranteed by the monotonicity of $F(\sigma,\chi)$) if and only if $ \frac{1}{\delta} \mathbb{E} (\eta^2_q(Z; \chi))<1$. Recall the definition of $\chi_{\min}$ given in \eqref{def:chimin}. Since $ \mathbb{E} (\eta^2_q(Z; \chi))$ is a strictly decreasing and continuous function of $\chi$, $\frac{1}{\delta} \mathbb{E} (\eta^2_q(Z; \chi))<1$ holds if $\chi \in (\chi_{\min},\infty)$ and fails when $\chi \in (0,\chi_{\min})$.
\end{proof}

\subsubsection{Studying the common solution of \eqref{eq:fixedpoint11} and \eqref{eq:fixedpoint21}} \label{revision:lemma2:unique2}

Corollary \ref{cor:fixedpointeqanalysis} from Section \ref{ssec:uniquenessfirsteqfp} characterizes the existence and uniqueness of solution for Equation \eqref{eq:fixedpoint11}. Our next goal is to prove that \eqref{eq:fixedpoint11} and \eqref{eq:fixedpoint21} share at least one common solution. Our strategy is: among all the pairs $(\sigma_{\chi}, \chi)$ that satisfy \eqref{eq:fixedpoint11}, we show that at least one of them satisfies \eqref{eq:fixedpoint21}. We do this in the next few lemmas.

\begin{lemma}\label{lem:limitoflambdas}
Let $\delta<1$. For each value of $\chi \in (\chi_{min}, \infty)$, define $\sigma_\chi$ as the value of $\sigma$ that satisfies \eqref{eq:fixedpoint11}. Then,
\begin{eqnarray}
\lim_{\chi \rightarrow \infty} \chi \sigma_\chi^{2-q} \left(1- \frac{1}{\delta} \mathbb{E}[\partial_1 \eta_q (B +  \sigma_\chi Z; \chi\sigma_\chi^{2-q})] \right)  &=& \infty, \label{delta:smaller1:one} \\
\lim_{\chi \rightarrow \chi^+_{\min} } \chi \sigma_\chi^{2-q} \left(1- \frac{1}{\delta} \mathbb{E}[\partial_1 \eta_q (B +  \sigma_\chi Z; \chi\sigma_\chi^{2-q})] \right)  &=& -\infty.  \nonumber 
\end{eqnarray}
\end{lemma}
\begin{proof}
We first show that 
\begin{eqnarray}\label{common:one}
\lim_{\chi\rightarrow \infty} \sigma^2_{\chi}= \sigma_w^2+\frac{\mathbb{E}|B|^2}{\delta}, \quad \lim_{\chi \rightarrow \chi^+_{\min}} \sigma_{\chi}= \infty.
\end{eqnarray}

For the first part, we only need to show $\sigma^2_{\chi_k}\rightarrow \sigma_w^2+\frac{\mathbb{E}|B|^2}{\delta}$ for any sequence $\chi_k \rightarrow \infty$. For that purpose, we first prove $\sigma_{\chi}=O(1)$. Otherwise, there exists a sequence $\chi_n \rightarrow \infty$ such that $\sigma_{\chi_n}\rightarrow \infty$. Because 
\[
(\eta_q(B/\sigma_{\chi_n}+Z;\chi_n)-B/\sigma_{\chi_n})^2 \leq 2(B/\sigma_{\chi_n}+Z)^2+2B^2/\sigma^2_{\chi_n}\leq 6B^2+4Z^2
\]
 for large enough $n$, we can apply Dominated Convergence Theorem (DCT) to conclude 
 \[
 \lim_{n\rightarrow \infty} R_q(\chi_n, \sigma_{\chi_n})=\mathbb{E}\lim_{n\rightarrow \infty} [\eta_q(B/ \sigma_{\chi_{n}} +Z; \chi_n) - B/\sigma_{\chi_n}]^2=0.
 \]
 On the other hand, since the pair $(\sigma_{\chi_{n}},\chi_n)$ satisfies \eqref{eq:fixedpoint11} we obtain
 \[
 \lim_{n\rightarrow \infty} R_q(\chi_n, \sigma_{\chi_n})= \lim_{n\rightarrow \infty} \delta\big (1-\frac{\sigma_w^2}{\sigma^2_{\chi_n}}\big)=\delta.
 \]
This is a contradiction. We next consider any convergent subsequence $\{\sigma_{\chi_{k_n}}\}$ of $\{\sigma_{\chi_k}\}$. The facts $\sigma_{x_k}\geq \sigma_w$ and $\sigma_{\chi_k}=O(1)$ imply $\sigma_{\chi_{k_n}} \rightarrow \sigma^* \in (0,\infty)$. Moreover, since
\[
(\eta_q(B+\sigma_{\chi_{k_n}} Z; \chi_{k_n} \sigma^{2-q}_{\chi_{k_n}})-B)^2  \leq  6B^2+5(\sigma^*)^2Z^2,
\]
when $n$ is large enough. We can apply DCT to obtain, 
\[
(\sigma^*)^2 =\lim_{n\rightarrow \infty} \sigma^2_{\chi_{k_n}}=\sigma_w^2 +\frac{1}{\delta}\mathbb{E}\lim_{n\rightarrow \infty} (\eta_q(B+\sigma_{\chi_{k_n}} Z; \chi_{k_n}\sigma_{\chi_{k_n}}^{2-q})-B)^2=\sigma_w^2+\frac{\mathbb{E}B^2}{\delta}.
\]
Thus, we have showed any convergent subsequence of $\{\sigma^2_{\chi_k}\}$ converges to the same limit $\sigma_w^2+\frac{\mathbb{E}B^2}{\delta}$. Hence the sequence converges to that limit as well. 

Regarding the second part in \eqref{common:one}, if it is not the case, then there exists a sequence $\chi_{n}\rightarrow \chi^+_{\min}$ such that $\sigma_{\chi_n}=O(1)$. Equation \eqref{eq:fix:point:equation} shows,
\[
1= \frac{\sigma_w^2}{ \sigma^2_{\chi_n}} + \frac{1}{\delta}R_q(\chi_n, \sigma_{\chi_n}) \overset{(a)}{\geq} \frac{\sigma_w^2}{ \sigma^2_{\chi_n}}+ \frac{1}{\delta}R_q(\chi_n, \infty)= \frac{\sigma_w^2}{ \sigma^2_{\chi_n}}+\frac{1}{\delta}\mathbb{E}(\eta_q^2(Z;\chi_n)),
\] 
where $(a)$ is due to Lemma \ref{lem:posfirstlemma}. From the definition of $\chi_{\min}$ in \eqref{def:chimin}, it is clear that $\frac{1}{\delta}\mathbb{E}(\eta_q^2(Z;\chi_{\min}))=1$ when $\delta <1$. Hence letting $n \rightarrow \infty$ on the both sides of the above inequaitiy leads to $1\geq \Omega(1)+1$, which is a contradiction. 

We are in position to derive the two limiting results in \eqref{delta:smaller1:one}. To obtain the first one, note that $\sigma^2_{\chi}\rightarrow \sigma_w^2+\frac{\mathbb{E}|B|^2}{\delta}$, as $\chi \rightarrow \infty$. Therefore, Lemma \ref{lem:toobasicpropproxder} part (vi) combined with DCT gives us 
\[
\lim_{\chi \rightarrow \infty} \mathbb{E}\partial_1 \eta_q (B +  \sigma_{\chi} Z; \chi\sigma_{\chi}^{2-q}) = 0.
\]
The first result of \eqref{delta:smaller1:one} can then be trivially derived. Regarding the second result, we have showed that as $\chi \rightarrow \chi^+_{\min}$, $\sigma_\chi \rightarrow \infty$. We also have
\begin{eqnarray*}
&&\mathbb{E}\partial_1 \eta_q (B +  \sigma_\chi Z; \chi\sigma_\chi^{2-q}) \overset{(b)}{=} \frac{1}{\sigma_\chi} \mathbb{E}  (Z \eta_q (B +  \sigma_\chi Z; \chi\sigma_\chi^{2-q}))\\
&& \hspace{0.8cm} \overset{(c)}{=} \mathbb{E}  (Z \eta_q (B/ \sigma_\chi +  Z; \chi)),
\end{eqnarray*}
where $(b)$ holds by Lemma \ref{lem:steins} and $(c)$ is due to Lemma \ref{lem:toobasicpropprox} part (v). Hence
\begin{eqnarray}\label{eq:lambdaspan1}
&& \lim_{\chi \rightarrow \chi^+_{\min} } \mathbb{E}\partial_1 \eta_q (B +  \sigma_\chi Z; \chi\sigma_\chi^{2-q}) \nonumber \\
&=& \lim_{\chi \rightarrow \chi^+_{\min} } \mathbb{E}  (Z \eta_q (B/ \sigma_\chi +  Z; \chi)) = \mathbb{E}  (Z \eta_q ( Z; \chi_{\min}))\nonumber \\ 
&\overset{(d)}{=}& \mathbb{E}  ( \eta_q^2 ( Z; \chi_{\min})) + \chi_{\min}q \mathbb{E}  ( |\eta_q ( Z; \chi_{\min})|^q) \nonumber \\
&=& \delta + \chi_{\min}q \mathbb{E}  ( |\eta_q ( Z; \chi_{\min})|^q), \nonumber
\end{eqnarray}
where $(d)$ is the result of Lemma \ref{lem:toobasicpropprox} part (i). We thus obtain
 \begin{equation}\label{eq:lambdaspan3}
 \hspace{-0.5cm} \lim_{\chi \rightarrow \chi^+_{min} }  \left(1- \frac{1}{\delta} \mathbb{E}\partial_1 \eta_q (B +  \sigma_\chi Z; \chi\sigma_\chi^{2-q}) \right) = -\frac{1}{\delta}\chi_{\min} q \mathbb{E}   |\eta_q (Z; \chi_{\min})|^q. 
 \end{equation}
Combining \eqref{eq:lambdaspan3} and the fact that $\chi_{\min}>0, \sigma_\chi \rightarrow \infty$ finishes the proof. 
\end{proof}

\begin{lemma} \label{lem:limitoflambdascopy}
Let $\delta \geq 1$. For each value of $\chi \in (\chi_{\min}, \infty)$, define $\sigma_\chi$ as the value of $\sigma$ that satisfies \eqref{eq:fixedpoint11}. Then,
\begin{eqnarray}
\lim_{\chi \rightarrow \infty} \chi \sigma_\chi^{2-q} \left(1- \frac{1}{\delta} \mathbb{E}[\partial_1 \eta_q (B +  \sigma_\chi Z; \chi\sigma_\chi^{2-q})] \right)  &=& \infty, \label{delta:large1:one} \\
\lim_{\chi \rightarrow \chi^+_{\min} } \chi \sigma_\chi^{2-q} \left(1- \frac{1}{\delta} \mathbb{E}[\partial_1 \eta_q (B +  \sigma_\chi Z; \chi\sigma_\chi^{2-q})] \right)  &=& 0.  \nonumber 
\end{eqnarray}
\end{lemma}

\begin{proof}
The exactly same arguments presented in the proof of Lemma \ref{lem:limitoflambdas} can be applied to prove the first result in \eqref{delta:large1:one}. We now focus on the proof of the second one. Since $\mathbb{E}|\partial_1 \eta_q (B +  \sigma_\chi Z; \chi\sigma_\chi^{2-q})|\leq 1$, our goal will be to show $\chi \sigma^{2-q}_{\chi}=o(1)$, as $\chi \rightarrow 0^+$ (note that $\chi_{\min}=0$ when $\delta \geq 1$). 

We first consider the case $\delta >1$. To prove $\chi \sigma^{2-q}_{\chi}=o(1)$, it is sufficient to show $\sigma_{\chi}=O(1)$. Suppose this is not true, then there exists a sequence $\chi_n \rightarrow 0$ such that $\sigma_{\chi_n} \rightarrow \infty$. Recall that $(\chi_{n},\sigma_{\chi_n})$ satisfies \eqref{eq:fixedpoint11}:
\begin{eqnarray}\label{thekey:eq}
\sigma^2_{\chi_n}=\sigma_w^2+ \frac{1}{\delta}\mathbb{E}(\eta_q(B+\sigma_{\chi_n}Z;\chi_n\sigma^{2-q}_{\chi_n})-B)^2.
\end{eqnarray}
Dividing both sides of the above equation by $\sigma^2_{\chi_n}$ and letting $n\rightarrow \infty$ yields $1=\frac{1}{\delta}<1$, which is a contradiction.  

Regarding the case $\delta=1$, we first claim that $\sigma_{\chi} \rightarrow \infty$, as $\chi \rightarrow 0$. Otherwise, there exists a sequence $\chi_n\rightarrow 0$ such that $\sigma_{\chi_n} \rightarrow \sigma^*\in (0, \infty)$. However, taking the limit $n\rightarrow \infty$ on both sides of \eqref{thekey:eq} gives us $(\sigma^*)^2=\sigma_w^2+(\sigma^*)^2$ where contradiction arises. Hence, if we can show $\chi \sigma^2_{\chi} =O(1)$, then $\chi \sigma^{2-q}_{\chi} =o(1)$ will be proved. Starting from \eqref{thekey:eq} (replacing $\chi_n$ by $\chi$) with $\delta=1$, we can have for $q\in (1,2]$
\begin{eqnarray*}
&&0=\sigma^2_w+\sigma^2_{\chi} \mathbb{E}(\eta_q(B/\sigma_{\chi}+Z;\chi)-B/\sigma_{\chi}-Z)^2+  \\
&&\hspace{1.5cm} 2\sigma^2_{\chi}\mathbb{E}Z(\eta_q(B/\sigma_{\chi}+Z;\chi)-B/\sigma_{\chi}-Z) \\
&&\overset{(a)}{=}\sigma^2_w +\chi \sigma^2_{\chi} \cdot \underbrace{\mathbb{E} (\chi q^2|\eta_q(B/\sigma_{\chi}+Z;\chi)|^{2q-2})}_{A}+ \\
&&\hspace{0.4cm} \chi \sigma^2_{\chi}\cdot \underbrace{\mathbb{E} \frac{-2q(q-1)|\eta_q(B/\sigma_{\chi}+Z;\chi)|^{q-2}}{1+\chi q(q-1)|\eta_q(B/\sigma_{\chi}+Z;\chi)|^{q-2}}}_{B},
\end{eqnarray*}
where to obtain $(a)$ we have used Lemma \ref{lem:toobasicpropprox} part (i), Lemma \ref{lem:toobasicpropproxder} part (i) and Lemma \ref{lem:steins}. Therefore we obtain
\begin{eqnarray}\label{finalevaluate}
\chi \sigma^2_{\chi}=-\sigma^2_w\cdot (A+B)^{-1}.
\end{eqnarray}
Because $\sigma_{\chi} \rightarrow \infty$ as $\chi \rightarrow 0$, it is easily seen that
\begin{eqnarray}\label{easilyseen}
\lim_{\chi \rightarrow 0^+}A = 0, \quad  \liminf_{\chi \rightarrow 0^+}|B|\geq 2q(q-1) \mathbb{E}|Z|^{q-2}. 
\end{eqnarray}
Combining results \eqref{finalevaluate} and \eqref{easilyseen} we can conclude that $\chi\sigma^2_{\chi}=O(1)$. Finally for the case $q=1$, we do similar calculations and have
\begin{eqnarray*}
&&|A|= \frac{1}{\chi}\mathbb{E}(\eta_q(B/\sigma_{\chi}+Z;\chi)-B/\sigma_{\chi}-Z)^2\leq \chi \rightarrow 0,\\
&&|B|=\frac{2}{\chi}P(|B/\sigma_{\chi}+Z| \leq \chi)  \rightarrow \frac{4}{\sqrt{2\pi}}.
\end{eqnarray*}
This completes the proof.
\end{proof}

According to the results presented in Lemmas \ref{lem:limitoflambdas} and \ref{lem:limitoflambdascopy}, if the function $\chi \sigma_\chi^{2-q} \big (1- \frac{1}{\delta} \mathbb{E}[\partial_1 \eta_q (B +  \sigma_\chi Z; \chi\sigma_\chi^{2-q})] \big)$, is continuous with respect to $\chi \in (\chi_{\min},\infty)$, then we can conclude that Equations \eqref{eq:fixedpoint11} and \eqref{eq:fixedpoint21} share at least one common solution pair. To confirm the continuity, it is straightforward to employ implicit function theorem to show $\sigma_{\chi}$ is continuous about $\chi$. Moreover, According to Lemma \ref{prox:smooth}, $\partial_1 \eta_q(u;\chi)$ is also a continuous function of its arguments.

\subsection{Proving the uniqueness of the solution of \eqref{eq:fixedpoint11} and \eqref{eq:fixedpoint21}} \label{revision:lemma2:three}

The proof of uniqueness is motivated by the idea presented in \cite{BaMo11}. Suppose there are two different solutions denoted by $(\sigma_{\chi_1},\chi_1)$ and $(\sigma_{\chi_2},\chi_2)$, respectively. By applying Theorem \ref{thm:eqpseudolip} with $\psi(a, b)=(a-b)^2$, we have
\begin{eqnarray*}
{\rm AMSE}(\lambda,q, \sigma_w)&=&\mathbb{E}[\eta_q(B+\sigma_{\chi_1}Z;\chi_1 \sigma^{2-q}_{\chi_1})-B]^2 \\
&\overset{(a)}{=}& \delta(\sigma^2_{\chi_1}-\sigma^2_w),
\end{eqnarray*}
where $(a)$ is due to \eqref{eq:fixedpoint11}. The same equations hold for the other solution pair $(\sigma_{\chi_2},\chi_2)$. Since they have the same AMSE, it follows that $\sigma_{\chi_1}=\sigma_{\chi_2}$. Next we choose a different pseudo-Lipschitz function $\psi(a,b)=|a|$ in Theorem \ref{thm:eqpseudolip} to obtain
\begin{eqnarray*}
&&\lim_{p\rightarrow \infty}\frac{1}{p}\sum_{i=1}^p | \hat{\beta}_i(\lambda,q, p)|=\mathbb{E} |\eta_q(B+\sigma_{\chi_1}Z;\chi_1\sigma^{2-q}_{\chi_1})| \\
&=&\mathbb{E} |\eta_q(B+\sigma_{\chi_2}Z;\chi_2 \sigma^{2-q}_{\chi_2})|=\mathbb{E} |\eta_q(B+\sigma_{\chi_1}Z;\chi_2 \sigma^{2-q}_{\chi_1})|
\end{eqnarray*}
Since $\mathbb{E}|\eta_q(B+\sigma_{\chi_1}Z; \chi)|$, as a function of $\chi \in (0, \infty)$, is strictly decreasing based on Lemma \ref{lem:toobasicpropproxder} part (v), we conclude $\chi_1=\chi_2$.

\section{Proof of Corollary 1}  \label{sec:optimaltuning1}

According to Theorem \ref{thm:eqpseudolip}, the key of proving Corollary \ref{thm:mseoptimalasymptot} is to analyze the following equations:

\begin{eqnarray}
\sigma^2 &=& \sigma_{\omega}^2+\frac{1}{\delta} \mathbb{E} [(\eta_q(B +\sigma Z; \chi \sigma^{2-q}) -B)^2], \label{optimaleq:fixedpoint11}  \\
\lambda_{*,q} &=& \chi \sigma^{2-q} \big (1-\frac{1}{\delta} \mathbb{E}[\eta_q'(B +\sigma Z;  \chi \sigma^{2-q})] \big), \label{optimaleq:fixedpoint21} 
\end{eqnarray}
where $\lambda_{*,q}=\argmin_{\lambda \geq 0} {\rm AMSE}(\lambda,q,\sigma_w)$. We present the main result in the following lemma.



 \begin{lemma}\label{lem:chi*meanslambda*}
For every $1 \leq q \leq 2$ and a given optimal tuning $\lambda_{*,q}$,  there exists a unique solution pair $(\bar{\sigma},\bar{\chi})$ that satisfies \eqref{optimaleq:fixedpoint11} and \eqref{optimaleq:fixedpoint21}. Furthermore, $\bar{\sigma}$ is the unique solution of 
\begin{equation}\label{uniquesolution:optimal}
\sigma^2=  \sigma_w^2 + \frac{1}{\delta} \min_{\chi \geq 0} \mathbb{E} (\eta_q(B+ \sigma Z; \chi) -B)^2,
\end{equation}
and 
 \begin{equation}\label{eq:hsigfixed1chi}
 \bar{\chi}  \in \argmin_{\chi\geq 0} \mathbb{E} [(\eta_q(B+ \bar{\sigma} Z; \chi \bar{\sigma}^{2- q}) -B)^2].
 \end{equation} 
 \end{lemma}
 
\begin{proof}
The first part of this lemma directly comes from Lemma \ref{lem:eq7and8solutionexists}. We focus on the proof of the second part. Recall the definition of $R_q(\chi, \sigma)$ in \eqref{keyquantity:def}. Define 
\[
G_q(\sigma)\triangleq \frac{\sigma_w^2}{\sigma^2}+\frac{1}{\delta} \min_{\chi \geq 0} R_q(\chi, \sigma).
\]
Then \eqref{uniquesolution:optimal} is equivalent to $G_q(\sigma)=1$. We first show that $G_q(\sigma)$ is a strictly decreasing function of $\sigma$ over $(0, \infty)$. For any given $\sigma_1>\sigma_2>0$, we can choose $\chi_1, \chi_2$ such that 
\begin{eqnarray*}
\chi_1 &=& \arg \min_{\chi\geq 0} \mathbb{E} (\eta_q(B/\sigma_1+ Z; \chi) -B/\sigma_1)^2,  \\
\chi_2 &=& \arg \min_{\chi \geq 0} \mathbb{E} (\eta_q(B/ \sigma_{2} +Z; \chi ) -B/\sigma_2)^2.
\end{eqnarray*}
Applying Lemma \ref{lem:posfirstlemma} we have
\begin{eqnarray*}
R_q(\chi_1, \sigma_1)=\min_{\chi \geq 0} R_q(\chi, \sigma_1) \leq R_q(\chi_2, \sigma_1)\leq R_q(\chi_2, \sigma_2)=\min_{\chi \geq 0}R_q(\chi,\sigma_2).
\end{eqnarray*}
Hence we obtain that $G_q(\sigma_1) < G_q(\sigma_2)$. We next show 
\begin{eqnarray}
\lim_{\sigma \rightarrow 0} G_q(\sigma)=\infty, \quad \lim_{\sigma \rightarrow \infty} G_q(\sigma)<1. \label{fix:point:unique}
\end{eqnarray}
The first result in \eqref{fix:point:unique} is obvious. To prove the second one, we know
\begin{eqnarray*}
 \lim_{\sigma \rightarrow \infty} G_q(\sigma) \leq \lim_{\sigma \rightarrow \infty}  \frac{\sigma^2_w}{\sigma^2}+\frac{1}{\delta}R_q(\chi, \sigma)=\frac{1}{\delta}\mathbb{E}\eta^2_q(Z;\chi),
\end{eqnarray*}
for any given $\chi \geq 0$. Choosing a sufficiently large $\chi$ completes the proof for the second inequality in \eqref{fix:point:unique}. Based on \eqref{fix:point:unique} and the fact that $G_q(\sigma)$ is a strictly decreasing and continuous function of $\sigma$ over $(0,\infty)$, we can conclude \eqref{uniquesolution:optimal} has a unique solution. Call it $\sigma^*$, and denote
\[
\chi^*  \in \argmin_{\chi\geq 0} \mathbb{E} [(\eta_q(B+ \sigma^* Z; \chi (\sigma^*)^{2- q}) -B)^2].
\]
Note that $\chi^*$ may not be unique. Further define 
\[
\lambda^* =  \chi^* (\sigma^*)^{2-q} \Big(1- \frac{1}{\delta} \mathbb{E}[\partial_1 \eta_q (B +  \sigma^* Z; \chi^* (\sigma^*)^{2-q})] \Big). 
\]
It is straightforward to see that the pair $(\sigma^*,\chi^*)$ satisfies \eqref{eq:fixedpoint11} and \eqref{eq:fixedpoint21} with $\lambda=\lambda^*$. According to Theorem \ref{thm:eqpseudolip}, we obtain 
\begin{eqnarray}
{\rm AMSE}(\lambda^*, q, \sigma_w)= \delta ((\sigma^*)^2- \sigma_w^2). 
\end{eqnarray}
Also we already know for the optimal tuning $\lambda_{*,q}$
\begin{eqnarray*}
{\rm AMSE}(\lambda_{*, q}, q, \sigma_w)= \delta (\bar{\sigma}^2- \sigma_w^2)  \leq {\rm AMSE}(\lambda^*, q, \sigma_w).
\end{eqnarray*}
Therefore $\bar{\sigma} \leq \sigma^*$. On the other hand, 
\begin{eqnarray*}
\bar{\sigma}^2 &=& \sigma_w^2 + \frac{1}{\delta} \mathbb{E} (\eta_q(B+ \bar{\sigma} Z ; \bar{\chi} \bar{\sigma}^{2-q}) - B)^2 \\
&\geq&  \sigma_w^2 +  \frac{1}{\delta} \min_{\chi \geq 0} \mathbb{E} (\eta_q(B+ \bar{\sigma} Z ; \chi ) - B)^2= \bar{\sigma}^2 \cdot G_q(\bar{\sigma}).
\end{eqnarray*}
Thus $G_q(\bar{\sigma})\leq 1 =G_q(\sigma^*)$. Since $G_q(\sigma)$ is a strictly decreasing function, we then obtain $\bar{\sigma}\geq \sigma^*$. Consequently, we conclude $\bar{\sigma}=\sigma^*$. Finally we claim \eqref{eq:hsigfixed1chi} has to hold. Otherwise,
\begin{eqnarray*}
&&{\rm AMSE}(\lambda_{*, q}, q, \sigma_w)=\mathbb{E}(\eta_q(B+\bar{\sigma}Z;\bar{\chi}\bar{\sigma}^{2-q})-B)^2 \\
&>& \min_{\chi\geq 0}\mathbb{E}(\eta_q(B+\bar{\sigma}Z;\chi)-B)^2=\mathbb{E}(\eta_q(B+\sigma^*Z;\chi^*)-B)^2 \\
&=&{\rm AMSE}(\lambda^*, q, \sigma_w),
\end{eqnarray*}
contradicts the fact that $\lambda_{*,q}$ is the optimal tuning. 
\end{proof}

\textbf{Remark:} Lemma \ref{lem:chi*meanslambda*} leads directly to the result of Corollary \ref{thm:mseoptimalasymptot}. Furthermore, from the proof of Lemma \ref{lem:chi*meanslambda*}, we see that if $\mathbb{E} [(\eta_q(B+ \sigma Z; \chi \sigma^{2-q}) -B)^2]$, as a function of $\chi$, has a unique minimizer for any given $\sigma>0$, then $\bar{\chi}=\chi^*$. That means the optimal tuning value $\lambda_{*,q}$ is unique. \cite{mousavi2015consistent} has proved that $\mathbb{E} [(\eta_q(B+ \sigma Z; \chi \sigma^{2-q}) -B)^2]$ is quasi-convex and has a unique minimizer for $q=1$. We conjecture it is true for $q\in (1,2]$ as well and leave it for future research.


\section{Proof of Theorem 3.1} \label{sec:prooftheorem2full}

\subsection{Roadmap of the proof}\label{sec:roadmapqless2}

Different from the result of Theorem \ref{asymp:sparsel1_1} that bounds the second order term in AMSE$(\lambda_{*,1},1,\sigma_w)$, Theorem \ref{asymp:lqbelowpt} characterizes the precise analytical expression of the second dominant term for AMSE$(\lambda_{*,q},q,\sigma_w)$ with $q\in (1,2]$. However, the idea of this proof is similar to the one for Theorem \ref{asymp:sparsel1_1} presented in Section \ref{sec:proofthm4full} of the main text, though the detailed proof steps are more involved here. We suggest interested readers first going over the proof of Theorem \ref{asymp:sparsel1_1} and then this section so that both the proof idea and technical details are smoothly understood. Recall the definition we introduced in \eqref{keyquantity:def}:
\[
R_q(\chi,\sigma)=\mathbb{E}(\eta_q(B/\sigma +Z;\chi)-B/\sigma)^2,
\]
where $Z \sim N(0,1)$ and $B$ with the distribution $p_{\beta}(b)=(1-\epsilon)\delta_0(b)+\epsilon g(b)$ are independent. Define
\begin{eqnarray} \label{optimal:tuningforq}
\chi^*_q(\sigma)=\argmin_{\chi \geq 0} R_q(\chi, \sigma).
\end{eqnarray}
Based on Lemma \ref{prox:smooth} of Appendix \ref{ssec:etaq:summary}, it is straightforward to show $R_q(\chi,\sigma)$ is a differentiable function of $\chi$. It is also easily seen that 
\[
\lim_{\chi\rightarrow \infty}R_q(\chi,\sigma)=\frac{\mathbb{E}|B|^2}{\sigma^2}, \quad \lim_{\chi \rightarrow 0} R_q(\chi,\sigma)=1.
\]
Therefore, the minimizer $\chi^*_q(\sigma)$ exists at least for sufficiently small $\sigma$. If it is not unique, we will consider the one having smallest value itself.  As like the proof of Theorem \ref{asymp:sparsel1_1}, the key is to characterize the convergence rate for $R_q(\chi^*_q(\sigma),\sigma)$ as $\sigma \rightarrow 0$. After having that convergence rate result, we can then obtain the convergence rate for $\bar{\sigma}$ from Equation \eqref{eq:fixedpointoptimalnew} and finally derive the expansion of AMSE$(\lambda_{*,q},q,\sigma_w)$ based on Corollary \ref{thm:mseoptimalasymptot}. We organize our proof steps as follows:
\begin{enumerate}
\item We first characterize the convergence rate of $\chi_q^*(\sigma)$ in Section \ref{sec:optiamltau:zero}.
\item We then obtain the convergence rate of $R_q(\chi_q^*(\sigma),\sigma)$ in Section \ref{sec:generaldistfixqbigger1epsless1}.
\item We finally derive the second order expansion for AMSE$(\lambda_{*,q},q,\sigma_w)$ in Section \ref{sec:AMSEellqgeneraldistepsless1}.
\end{enumerate}

As a final remark, once we show the proof for $q\in (1,2)$, since $\eta_2(u;\chi)=\frac{u}{1+2\chi}$ has a nice explicit form, the proof for $q=2$ can be easily derived. We hence skip it for simplicity.

\subsection{Characterizing the convergence rate of $\chi^*_q(\sigma)$} \label{sec:optiamltau:zero}

The goal of this section is to derive the convergence rate of $\chi^*_q(\sigma)$ as $\sigma \rightarrow 0$. We will make use of the fact that $\chi^*_q(\sigma)$ is the minimizer of $R_q(\chi,\sigma)$, to first show $\chi^*_q(\sigma)\rightarrow 0$ and then obtain the rate $\chi^*_q(\sigma) \propto \sigma^{2q-2}$. This is done in the following three lemmas.

\begin{lemma}\label{eq:infinity}
Let $\chi^*_q(\sigma)$ denote the minimizer of $R_q(\chi,\sigma)$ as defined in \eqref{optimal:tuningforq}. Then for every $b \neq 0$ and $z \in \mathbb{R}$,
\[
|\eta_q(b/\sigma +z ; \chi^*_q(\sigma)) | \rightarrow \infty, \quad \mbox{as~} \sigma \rightarrow 0.
\] 
\end{lemma}

\begin{proof}
Suppose this is not the case. Then there exist a value of $b \neq 0, z \in \mathbb{R}$ and a sequence $\sigma_k \rightarrow 0$, such that $|\eta_q(b/\sigma_k+z; \chi_q^*(\sigma_k))|$ is bounded.  Combined with Lemma \ref{lem:toobasicpropprox} part (i) we obtain 
\begin{eqnarray}\label{contradiction:eqone}
\chi^*_q(\sigma_k)=\frac{|b/\sigma_k+z|-|\eta_q(b/\sigma_k+z; \chi^*_q(\sigma_k))|}{q|\eta_q(b/\sigma_k+z; \chi^*_q(\sigma_k))|^{q-1}}=\Omega \Big(\frac{1}{\sigma_k} \Big).
\end{eqnarray}
We next show that the result \eqref{contradiction:eqone} implies for any other $\tilde{b}\neq 0$ and $\tilde{z}\in \mathbb{R}$, $ |\eta_q(\tilde{b}/\sigma_k+\tilde{z}; \chi^*_q(\sigma_k))|$ is bounded as well. From Lemma \ref{lem:toobasicpropprox} part (i) we know
\begin{eqnarray} \label{eq:equalityproxtildeb}
|\tilde{b}/\sigma_k+\tilde{z}| =|\eta_q(\tilde{b}/\sigma_k+\tilde{z}; \chi^*_q(\sigma_k))| + \chi^*_q(\sigma_k) q |\eta_q( \tilde{b}/\sigma_k+\tilde{z} ; \chi^*_q(\sigma_k)) |^{q-1}.  \hspace{-2cm}
\end{eqnarray}
If $ |\eta_q(\tilde{b}/\sigma_k+\tilde{z}; \chi^*_q(\sigma_k))|$ is unbounded, then the right hand side of equation \eqref{eq:equalityproxtildeb} (take a subsequence if necessary) has the order larger than $1/\sigma_k$. Hence \eqref{eq:equalityproxtildeb} can not hold for all the values of $k$. We thus have reached the conclusion that $|\eta_q(b/\sigma_k+ z; \chi_q^*(\sigma_k))|$ is bounded for every $b\neq 0$ and $z\in \mathbb{R}$. Therefore, 
\[
|\eta_q(G/\sigma_k+ Z; \chi^*_q(\sigma_k)) - G/ \sigma_k| \rightarrow \infty ~a.s.,  \mbox{~~as~} k \rightarrow \infty,
\]
where $G$ has the distribution $g(\cdot)$. We then use Fatou's lemma to obtain 
\[
R_q(\chi^*_q(\sigma_k),\sigma_k) \geq \epsilon \mathbb{E} |\eta_q(G/\sigma_k+ Z; \chi^*_q(\sigma_k)) - G/ \sigma_k|^2 \rightarrow \infty.
\]
On the other hand, since $\chi^*_q(\sigma_k)$ minimizes $R_q(\chi, \sigma_k)$,
\[
R_q(\chi^*_q(\sigma_k),\sigma_k) \leq R_q(0, \sigma_k)=1.
\]
Such contradiction completes the proof.
\end{proof}

Lemma \ref{eq:infinity} enables us to derive $\chi^*_q(\sigma)\rightarrow 0$ as $\sigma \rightarrow 0$. We present it in the next lemma.

\begin{lemma}\label{lem:optimaltaugoestozero}
Let $\chi^*_q(\sigma)$ denote the minimizer of $R_q(\chi,\sigma)$ as defined in \eqref{optimal:tuningforq}. Then $\chi^*_q(\sigma) \rightarrow 0$ as $\sigma \rightarrow 0$. 
\end{lemma}
\begin{proof}
First note that 
\begin{eqnarray}\label{eq:riskexpansion1}
R_q(\chi, \sigma)&=&(1-\epsilon)\mathbb{E}(\eta_q(Z;\chi))^2 + \epsilon \mathbb{E}(\eta_q(G/\sigma + Z;\chi)-G/\sigma-Z)^2 + \nonumber \\
&&2\epsilon \mathbb{E}Z(\eta_q(G/\sigma + Z;\chi)-G/\sigma-Z) +\epsilon \nonumber \\
&\overset{(a)}{=}&(1-\epsilon)\mathbb{E}(\eta_q(Z;\chi))^2+\epsilon \chi^2 q^2\mathbb{E}|\eta_q(G/\sigma+Z;\chi)|^{2q-2} + \nonumber\\
&&2\epsilon \mathbb{E}(\partial_1 \eta_q(G/\sigma+Z;\chi)-1) +\epsilon \nonumber \\
&\overset{(b)}{=}&(1-\epsilon)\mathbb{E}(\eta_q(Z;\chi))^2+\epsilon \chi^2 q^2\mathbb{E}|\eta_q(G/\sigma+Z;\chi)|^{2q-2} + \nonumber\\
&&2\epsilon \mathbb{E}\left(\frac{1}{1+\chi q(q-1)|\eta_q(G/\sigma+Z;\chi)|^{q-2}}-1\right) +\epsilon. \label{uu1}
\end{eqnarray}
We have employed Lemma \ref{lem:toobasicpropprox} part (i) and Lemma \ref{prox:smooth} to obtain $(a)$; $(b)$ is due to Lemma \ref{lem:toobasicpropproxder} part (i). According to Lemma \ref{eq:infinity}, $|\eta_q(G/\sigma+Z;\chi^*_q(\sigma))| \rightarrow \infty ~a.s.,$ as $\sigma \rightarrow 0$. Hence, if $\chi^*_q(\sigma) \nrightarrow 0$, the second term in \eqref{eq:riskexpansion1} (with $\chi=\chi_q^*(\sigma)$) goes off to infinity, while the other terms remain finite, and consequently $R_q(\chi^*_q(\sigma) , \sigma) \rightarrow \infty$. This is a contradiction with the fact $R_q(\chi^*_q(\sigma) , \sigma)\leq R_q(0,\sigma)=1$.
\end{proof}

So far we have shown $\chi^*_q(\sigma) \rightarrow 0$ as $\sigma \rightarrow 0$. Our next lemma further characterizes the convergence rate of $\chi^*_q(\sigma)$.

\begin{lemma}\label{thm:riskell_qgeneralqless2}
Suppose $\mathbb{P}(|G|\leq t)=O(t)$ (as $t\rightarrow 0$) and $\mathbb{E}|G|^2<\infty$. Then for $q\in (1,2)$ we have
\begin{eqnarray*}
\lim_{\sigma \rightarrow 0} \frac{\chi^*_q(\sigma)}{\sigma^{2q-2}} = \frac{(1- \epsilon) \mathbb{E} |Z|^q}{\epsilon q \mathbb{E} |G|^{2q-2}}.
\end{eqnarray*}
\end{lemma}

\begin{proof}
We first claim that $\chi_q^*(\sigma)=\Omega(\sigma^{2q-2})$. Otherwise there exists a sequence $\sigma_k \rightarrow 0$ such that $\chi^*_q(\sigma_k)=o(\sigma^{2q-2}_k)$. According to Lemma \ref{general:lemma} (we postpone Lemma \ref{general:lemma} since it deals with $R_q(\chi,\sigma)$), 
\[
\lim_{k\rightarrow \infty} \frac{R_q(\chi^*_q(\sigma_k),\sigma_k)-1}{\sigma_k^{2q-2}}=0. 
\]
On the other hand, by choosing $\chi(\sigma_k)=C\sigma_k^{2q-2}$ with $C=\frac{(1- \epsilon) \mathbb{E} |Z|^q}{\epsilon q \mathbb{E} |G|^{2q-2}}$, Lemma \ref{general:lemma} implies that 
\[
\lim_{k\rightarrow \infty} \frac{R_q(\chi(\sigma_k),\sigma_k)-1}{\sigma_k^{2q-2}}<0,
\]
which contradicts with the fact that $\chi_q^*(\sigma_k)$ is the minimizer of $R_q(\chi,\sigma_k)$. Moreover, this choice of $C$ shows that for sufficiently small $\sigma$ there exists $\chi(\sigma)$ such that
\[
R_q(\chi^*_q(\sigma),\sigma)\leq R_q(\chi(\sigma),\sigma)<R_q(0,\sigma)=1.
\]
That means $\chi^*_q(\sigma)$ is a non-zero finite value. Hence it satisfies $\frac{\partial R_q(\chi^*_q(\sigma), \sigma)}{\partial \chi}=0$. From now on we use $\chi^*$ to denote $\chi^*_q(\sigma)$ for notational simplicity. That equation can be detailed out as follows:
\begin{eqnarray}
&& \hspace{-0.8cm} 0=(1-\epsilon)\mathbb{E}\eta_q(Z;\chi^{\ast})\partial_2 \eta_q(Z;\chi^{\ast})+\epsilon \mathbb{E}(\eta_q(G/\sigma+Z;\chi^{\ast})-G/\sigma)\partial_2 \eta_q(G/\sigma+Z;\chi^{\ast})   \nonumber \\
&&\hspace{-0.6cm} \overset{(a)}{=}(1-\epsilon) \underbrace{\mathbb{E}\frac{-q|\eta_q(Z;\chi^{\ast})|^q}{1+\chi^{\ast} q(q-1)|\eta_q(Z;\chi^{\ast})|^{q-2}}}_{H_1} +  \epsilon \underbrace{\mathbb{E} (Z \partial_2 \eta_q (G/\sigma + Z; \chi^*)) }_{H_2}\nonumber  \\
&& \hspace{1.5cm} + \epsilon\chi^{\ast} \underbrace{\mathbb{E}\frac{ q^2|\eta_q(G/\sigma+Z;\chi^{\ast})|^{2q-2}}{1+\chi^{\ast} q(q-1)|\eta_q(G/\sigma+Z;\chi^{\ast})|^{q-2}}}_{H_3}, \label{fromthebeginning} 
\end{eqnarray}
where we have used Lemma \ref{lem:toobasicpropprox} part (i) and Lemma \ref{lem:toobasicpropproxder} part (ii) to obtain $(a)$. We now analyze the three terms $H_1,H_2$ and $H_3$, respectively. According to Lemma \ref{lem:optimaltaugoestozero}, we have that $\eta_q(Z;\chi^{\ast})\rightarrow Z$ as $\sigma \rightarrow 0$. Lemma \ref{lem:toobasicpropprox} part (ii) enables us to bound the expression inside the expectation of $H_1$ by $q|Z|^q$. Hence we can employ Dominated Convergence Theorem (DCT) to obtain
\begin{eqnarray}
\lim_{\sigma \rightarrow 0}H_1 =-q\mathbb{E}|Z|^q  \label{H1}.
\end{eqnarray}
For the term $H_3$, we first note that 
\begin{eqnarray*}
&&\lim_{\sigma \rightarrow 0} \eta_q(G+\sigma Z;\sigma^{2-q}\chi^{\ast})= G, \\
&& |\eta_q(G+\sigma Z;\sigma^{2-q}\chi^{\ast})|\leq |B|+\sigma| Z|.
\end{eqnarray*}
We also know that $|\eta_q(G/\sigma+Z;\chi^{\ast})|\rightarrow \infty, a.s.$ by Lemma \ref{eq:infinity}. We therefore can apply DCT to conclude 
\begin{eqnarray}\label{H2}
\lim_{\sigma \rightarrow 0} \frac{H_3}{\sigma^{2-2q}}=\mathbb{E} \lim_{\sigma \rightarrow 0} \frac{ q^2|\eta_q(G+\sigma Z;\sigma^{2-q}\chi^{\ast})|^{2q-2}}{1+\chi^{\ast} q(q-1)|\eta_q(G/\sigma+Z;\chi^{\ast})|^{q-2}}=q^2\mathbb{E}|G|^{2q-2} . \hspace{-2cm}
\end{eqnarray}
We now study the remaining term $H_2$. According to Lemma \ref{prox:smooth}, $\partial_2 \eta_q (G/\sigma+Z; \chi^*)$ is differentiable with respect to its first argument. So we can apply Lemma \ref{lem:steins} to get
\begin{eqnarray*}
H_2 &=& q(1-q) \underbrace{\mathbb{E} \frac{  |\eta_q (G/\sigma+Z; \chi^*)| ^{q-2}}{ (1+ \chi^* q(q-1) |\eta_q (G/\sigma+Z; \chi^*)|^{q-2})^3}}_{J_1} \nonumber \\
&&+ q^2(1-q)\underbrace{\mathbb{E} \frac{\chi^*  |\eta_q (G/\sigma+Z; \chi^*)|^{2q-4}}{ (1+ \chi^* q(q-1) |\eta_q (G/\sigma+Z; \chi^*)|^{q-2})^3}}_{J_2}. \nonumber 
\end{eqnarray*}
It is straightforward to see that 
\begin{eqnarray*}
J_1 &\leq& \mathbb{E} \frac{1}{ |\eta_q (G/\sigma+Z; \chi^*)|^{2-q}+ \chi^* q(q-1) }, \nonumber \\
 J_2 &\leq&  \mathbb{E} \frac{1/(q(q-1))}{ |\eta_q (G/\sigma+Z; \chi^*)|^{2-q}+ \chi^* q(q-1) }.
\end{eqnarray*}
We would like to prove $H_2 \rightarrow 0$ by showing
\begin{eqnarray}\label{eq:etaqpless1laststep}
\lim_{\sigma \rightarrow 0}\mathbb{E} \frac{1}{ |\eta_q (G/\sigma+Z; \chi^*)|^{2-q}+ \chi^* q(q-1) }=0.
\end{eqnarray}
Note that DCT may not be directly applied here, because the function inside the expectation cannot be easily bounded. Alternatively, we prove \eqref{eq:etaqpless1laststep} by breaking the expectation into different parts and showing each part converges to zero. Let $\alpha_1$ be a number that satisfies $\eta_q(\alpha_1; \chi^*) = (\chi^*)^{\frac{1}{2-q}}, \alpha_2= (\chi^*)^{\frac{1}{2}},$ and $\alpha_3$ a fixed positive constant that does not depend on $\sigma$. Denote the distribution of $|G|$ by $F(g)$. Note the following simple fact about $\alpha_1$ according to Lemma \ref{lem:toobasicpropprox} part (i):
\begin{equation*}\label{eq:alpha1first}
\alpha_1=\eta_q(\alpha_1; \chi^*) + \chi^* q \eta_q^{q-1}(\alpha_1; \chi^*)  = (q+1) (\chi^*)^{\frac{1}{2-q}}. 
\end{equation*}
So $\alpha_1<\alpha_2<\alpha_3$ when $\sigma$ is small. Define the following three nested intervals: 
\[
\mathcal{I}_{i} (x) \triangleq [-x-\alpha_i, -x+\alpha_i] ,  \quad i=1,2,3. 
\]
 With these definitions, we start the proof of \eqref{eq:etaqpless1laststep}. We have
\begin{eqnarray*}
\mathbb{E} \frac{1}{ |\eta_q (G/ \sigma +Z ; \chi^*)|^{2-q} + \chi^* q (q-1) }  \label{general:fzero}  \hspace{-8cm}\\
&&  =\int_{0}^\infty \int_{z \in \mathcal{I}_1(g/ \sigma)}  \frac{1}{ |\eta_q (g/ \sigma +z ; \chi^*)|^{2-q} + \chi^* q (q-1) } \phi(z) dz dF(g) + \nonumber \\
&&  \int_{0}^\infty \int_{z \in \mathcal{I}_2(g/ \sigma) \backslash  \mathcal{I}_1(g/ \sigma)}  \frac{1}{ |\eta_q (g/ \sigma +z ; \chi^*)|^{2-q} + \chi^* q (q-1) } \phi(z) dz dF(g)+ \nonumber \\
&&  \int_{0}^\infty \int_{z \in \mathcal{I}_3(b/ \sigma) \backslash  \mathcal{I}_2(g/ \sigma)}  \frac{1}{ |\eta_q (g/ \sigma +z ; \chi^*)|^{2-q} + \chi^* q (q-1) } \phi(z) dz dF(g)+ \nonumber \\
&&   \int_{0}^\infty \int_{\mathbb{R} \backslash \mathcal{I}_3(b/ \sigma)} \frac{1}{ |\eta_q (g/ \sigma +z ; \chi^*)|^{2-q} + \chi^* q (q-1) } \phi(z) dz dF(g) \nonumber \\ 
&\triangleq& G_1+G_2+G_3+G_4, \nonumber
\end{eqnarray*}
where $\phi(\cdot)$ is the density function of standard normal. We now bound each of the four integrals in \eqref{general:fzero} respectively.
\begin{eqnarray*}
G_1 &\leq& \int_{0}^\infty \int_{z \in \mathcal{I}_1(g/ \sigma)}  \frac{1}{ \chi^* q (q-1) } \phi(z) dz dF(g) \nonumber \\
&\leq& \frac{2 \alpha_1 \phi(0)}{\chi^* q (q-1)} \leq \frac{2(q+1) (\chi^*)^{\frac{1}{2-q}}\phi(0)}{\chi^* q (q-1)} \rightarrow 0,  \mbox{~as~} \sigma \rightarrow 0. \label{general:fone}
\end{eqnarray*}
The last step is due to the fact that $\chi^* \rightarrow 0$ by Lemma \ref{lem:optimaltaugoestozero}. For $G_2$ we have
\allowdisplaybreaks[1]
\begin{eqnarray*}
G_2&\leq & \int_{0}^\infty \int_{z \in \mathcal{I}_2(g/ \sigma) \backslash  \mathcal{I}_1(g/ \sigma)}  \frac{1}{ |\eta_q (g/ \sigma +z ; \chi^*)|^{2-q} } \phi(z) dz dF(g) \nonumber \\
&\overset{(b)}{\leq}& \int_{0}^\infty \int_{z \in \mathcal{I}_2(g/ \sigma) \backslash  \mathcal{I}_1(g/ \sigma)}  \frac{1}{ \chi^* } \phi(z) dz dF(g) \nonumber \\
&\leq& \int_{0}^{ \sigma \log(1/\sigma) } \int_{z \in \mathcal{I}_2(g/ \sigma) \backslash  \mathcal{I}_1(g/ \sigma)}  \frac{1}{ \chi^* } \phi(z) dz dF(g) + \nonumber \\
&&\int_{ \sigma \log(1/\sigma)}^{\infty } \int_{z \in \mathcal{I}_2(g/ \sigma) \backslash  \mathcal{I}_1(g/ \sigma)}  \frac{1}{ \chi^* } \phi(z) dz dF(g) \nonumber \\
&\leq& \mathbb{P} (|G| \leq \sigma \log(1/\sigma)) \frac{2\phi(0) \alpha_2}{\chi^*} + \frac{2\phi( \log(1/\sigma) - \alpha_2) \alpha_2}{\chi^*} \nonumber \\
 &=&\mathbb{P} (|G| \leq \sigma \log(1/\sigma)) \frac{2\phi(0)}{(\chi^*)^{1/2}} + \frac{2\phi( \log(1/\sigma) - \alpha_2) }{(\chi^*)^{1/2}} \nonumber \\
& \overset{(c)}{\leq}& O(1)\cdot  \sigma^{2-q}\log(1/\sigma) + O(1)\cdot \frac{\phi( (\log(1/\sigma))/2 )}{\sigma^{q-1}}  \rightarrow 0, {\rm ~as~}\sigma \rightarrow 0.  \label{general:ftwo} 
\end{eqnarray*}
In the above derivations, $(b)$ is because  
\[
|\eta_q (g/ \sigma +z ; \chi^*)| \geq \eta_q(\alpha_1; \chi^*) = (\chi^*)^{1/(2-q)}  \mbox{~~for~} z \notin \mathcal{I}_1(g/ \sigma).
\]
To obtain $(c)$, we have used the condition $\mathbb{P} (|G| \leq \sigma \log(1/\sigma))=O(\sigma\log(1/\sigma))$ and the result $\chi^*=\Omega(\sigma^{2q-2})$ we proved at the beginning. Regarding $G_3$,
\begin{eqnarray*}
G_3&\leq& \int_{0}^\infty \int_{z \in \mathcal{I}_3(g/ \sigma) \backslash  \mathcal{I}_2(g/ \sigma)}  \frac{1}{ |\eta_q (\alpha_2 ; \chi^*)|^{2-q} } \phi(z) dz dF(g) \nonumber \\
&\leq& \int_{0}^{\sigma \log 1/\sigma} \int_{z \in \mathcal{I}_3(g/ \sigma) \backslash  \mathcal{I}_2(g/ \sigma)}  \frac{1}{ |\eta_q (\alpha_2 ; \chi^*)|^{2-q} } \phi(z) dz dF(g) \nonumber \\
&&+  \int_{\sigma \log 1/\sigma}^{\infty} \int_{z \in \mathcal{I}_3(g/ \sigma) \backslash  \mathcal{I}_2(g/ \sigma)}  \frac{1}{ |\eta_q (\alpha_2 ; \chi^*)|^{2-q} } \phi(z) dz dF(g) \nonumber \\
&\leq &\mathbb{P} (|G| \leq \sigma \log(1/\sigma)) \frac{2\phi(0) \alpha_3}{|\eta_q (\alpha_2 ; \chi^*)|^{2-q}} + \frac{2\phi( \log(1/\sigma) - \alpha_3) \alpha_3}{|\eta_q (\alpha_2 ; \chi^*)|^{2-q}} \nonumber \\
&\overset{(d)}{\leq}& O(1)\cdot \sigma^{q^2-3q+3}\log(1/\sigma)+O(1)\cdot \frac{\phi((\log(1/\sigma))/2)}{\sigma^{(q-1)(2-q)}}\rightarrow 0,  {\rm ~as~}\sigma \rightarrow 0.  \label{general:fthree} \hspace{-0.8cm}
\end{eqnarray*}
The calculations above are similar to those for $G_2$. In $(d)$ we have used the following result:
\[
\lim_{\sigma \rightarrow 0} \frac{\eta_q (\alpha_2; \chi^*)}{(\chi^*)^{1/2}}= \lim_{\sigma \rightarrow 0} \frac{\eta_q ((\chi^*)^{1/2}; \chi^*)}{(\chi^*)^{1/2}} = \lim_{\sigma \rightarrow 0} \eta_q (1; (\chi^*)^{q/2})=1.
\]
Finally we can apply DCT to obtain
\begin{eqnarray*}
\lim_{\sigma \rightarrow 0}G_4=\mathbb{E} \lim_{\sigma \rightarrow 0}\frac{\mathbbm{I}(|G/\sigma+Z|>\alpha_3)}{ |\eta_q (G/ \sigma +Z ; \chi^*)|^{2-q} + \chi^* q (q-1) }=0.    \label{general:ffour}
\end{eqnarray*}
We have finished the proof of $\lim_{\sigma \rightarrow 0} H_2 =0$. This fact together with \eqref{fromthebeginning}, \eqref{H1} and \eqref{H2} gives us
\begin{eqnarray*}
\lim_{\sigma \rightarrow 0}\frac{\chi^*}{\sigma^{2q-2}}=\frac{-(1-\epsilon)H_1-\epsilon H_2}{\epsilon H_3/\sigma^{2-2q}}=\frac{(1-\epsilon)\mathbb{E}|Z|^q}{\epsilon q\mathbb{E}|G|^{2q-2}}.
\end{eqnarray*}
\end{proof}


\subsection{Characterizing the convergence rate of $R_q(\chi^*_q(\sigma), \sigma)$} \label{sec:generaldistfixqbigger1epsless1}
Having derived the convergence rate of $\chi^*_q(\sigma)$ in Section \ref{sec:optiamltau:zero}, we aim to obtain the convergence rate for $R_q(\chi^*_q(\sigma), \sigma)$ in this section. Towards that goal, we first present a useful lemma. 

\begin{lemma}\label{general:lemma}
Suppose $\mathbb{P}(|G|\leq t)=O(t)$ (as $t \rightarrow 0$) and $\mathbb{E}|G|^2<\infty$. If $\chi(\sigma) = C{\sigma^{2q-2}}$ with a fixed number $C>0$, then for $1<q<2$ we have 
\begin{equation}\label{eq:riskatrightorder}
\lim_{\sigma \rightarrow 0}\frac{R_q(\chi(\sigma), \sigma)-1}{ \sigma^{2q-2}} = -2C(1-\epsilon)q\mathbb{E} |Z|^q+ \epsilon C^2 q^2 \mathbb{E} |G|^{2q-2}. 
\end{equation}
Moreover, if $\chi(\sigma)=o(\sigma^{2q-2})$ then
\begin{equation}\label{eq:wrongorder}
\lim_{\sigma \rightarrow 0}\frac{R_q(\chi(\sigma), \sigma)-1}{ \sigma^{2q-2}} =0.
\end{equation}
\end{lemma}

\begin{proof}
We first focus on the case $\chi(\sigma) = C \sigma^{2q-2}$. According to \eqref{eq:riskexpansion1},
\begin{eqnarray}\label{eq:risktotaleq1}
R_q(\chi, \sigma) -1 &=& \underbrace{(1- \epsilon) \mathbb{E} (\eta_q^2 (Z; \chi) -Z^2)}_{R_1} + \underbrace{\epsilon \chi^2 q^2 \mathbb{E} |\eta_q(G / \sigma +Z ; \chi)  |^{2q-2}}_{R_2}  \nonumber \\
&&\underbrace{- 2 \epsilon \chi q(q-1) \mathbb{E} \frac{|\eta_q (G/ \sigma + Z; \chi)|^{q-2}}{ 1+ \chi q (q-1) |\eta_q (G/ \sigma +Z ; \chi)|^{q-2}}}_{R_3}. \nonumber 
\end{eqnarray}
Now we calculate the limit of each of the terms individually. We have
\begin{eqnarray}
R_1 &=& (1- \epsilon) \mathbb{E} (\eta_q (Z ; \chi)+Z) (\eta_q (Z; \chi) - Z) \nonumber \\ 
&\overset{(a)}{=}& - (1- \epsilon) \mathbb{E} (\eta_q (Z ; \chi)+Z) ( \chi q |\eta_q (Z; \chi)|^{q-1}{\rm sign} (Z)) \nonumber \\
&=& -(1- \epsilon) \chi q (\mathbb{E} |\eta_q (Z; \chi)|^{q} + \mathbb{E} |Z||\eta_q (Z; \chi)|^{q-1}), \nonumber
\end{eqnarray}
where $(a)$ is due to Lemma \ref{lem:toobasicpropprox} part (i). Hence we obtain
\begin{eqnarray}\label{eq:riskR1}
\lim_{\sigma \rightarrow 0} \frac{R_1}{\sigma^{2q-2}} &=& - C(1- \epsilon) q \lim_{\sigma \rightarrow 0}  (\mathbb{E} |\eta_q (Z; \chi)|^{q}+ \mathbb{E} |Z||\eta_q (Z; \chi)|^{q-1} ) \hspace{-1.5cm} \\
&=& -2C(1- \epsilon ) q \mathbb{E} |Z|^q. \nonumber 
\end{eqnarray}
The last equality is by Dominated Convergence Theorem (DCT). For $R_2$,
\begin{eqnarray}\label{eq:riskR2}
\lim_{\sigma \rightarrow 0}  \frac{R_2}{\sigma^{2q-2}} &=& \lim_{\sigma \rightarrow 0} \frac{\epsilon \chi^2 q^2 \mathbb{E} |\eta_q (G/\sigma + Z; \chi)|^{2q-2} }{\sigma^{2q-2}} \\
&=& \epsilon C^2 q^2 \lim_{\sigma \rightarrow 0} \mathbb{E} |\eta_q (G +\sigma Z; \chi \sigma^{2-q})|^{2q-2} \nonumber \\
&=& \epsilon C^2 q^2  \mathbb{E} |G|^{2q-2}.  \nonumber
\end{eqnarray}
Regarding the term $R_3$, we would like to show that if $\chi(\sigma)=O(\sigma^{2q-2})$, then
\begin{eqnarray} \label{general:r3}
\lim_{\sigma \rightarrow 0} \frac{\chi}{\sigma^{2q-2}}\mathbb{E} \frac{1}{|\eta_q(G/\sigma +Z; \chi)|^{2-q} + \chi q (q-1)}=0. 
\end{eqnarray}
Define $\alpha_1=1$ if $1<q<3/2$ and $\alpha_1=\sigma^{2q-3+c}$ if $3/2 \leq q <2$, where $c>0$ is a sufficiently small constant that we will specify later. Let $F(g)$ be the distribution function of $|G|$ and $\alpha_2>1$ a fixed constant. So $\alpha_1<\alpha_2$ for small $\sigma$. Define the following two nested intervals
\[
\mathcal{I}_{i} (x) \triangleq [-x-\alpha_i, -x+\alpha_i], \quad  i=1, 2.
\]

Then the expression in \eqref{general:r3} can be written as
\begin{eqnarray*}
 \frac{\chi}{\sigma^{2q-2}}\mathbb{E} \frac{1}{|\eta_q(G/\sigma +Z; \chi)|^{2-q} + \chi q (q-1)}\label{eq:entiretermtozero} \hspace{-10cm}\\
&=& \frac{\chi}{\sigma^{2q-2}} \int_{0}^\infty \int_{z \in \mathcal{I}_1(g/ \sigma)}  \frac{1}{ |\eta_q (g/ \sigma +z ; \chi)|^{2-q} + \chi q (q-1) } \phi(z) dz dF(g)+  \nonumber \\
&&\frac{\chi}{\sigma^{2q-2}}  \int_{0}^\infty \int_{z \in \mathcal{I}_2(g/ \sigma) \backslash  \mathcal{I}_1(g/ \sigma)}  \frac{1}{ |\eta_q (g/ \sigma +z ; \chi)|^{2-q} + \chi q (q-1) } \phi(z) dz dF(g)+ \nonumber \\
&&  \frac{\chi}{\sigma^{2q-2}}  \int_{0}^\infty \int_{\mathbb{R} \backslash \mathcal{I}_2(g/ \sigma)} \frac{1}{ |\eta_q (g/ \sigma +z ; \chi)|^{2-q} + \chi q (q-1) } \phi(z) dz dF(g) \nonumber \\
&\triangleq& G_1+G_2+G_3. \nonumber
\end{eqnarray*}
We will bound each of the three integrals above. The idea is similar as the one presented in the proof of Lemma \ref{thm:riskell_qgeneralqless2}. For the first integral,
\begin{eqnarray*}\label{eq:G1checkequality}
G_1&\leq &\frac{\chi}{\sigma^{2q-2}} \int_{0}^\infty  \int_{-g/\sigma - \alpha_1}^{-g/\sigma+ \alpha_1} \frac{1}{ \chi q (q-1)} \phi(z)dz dF(g)  \nonumber \\
&\leq& \frac{\chi}{\sigma^{2q-2}} \int_{0}^{\sigma \log1/\sigma}  \int_{-g/\sigma - \alpha_1}^{-g/\sigma+ \alpha_1} \frac{1}{ \chi q (q-1)} \phi(z)dz dF(g) \nonumber \\
&&+ \frac{\chi}{\sigma^{2q-2}} \int_{\sigma \log1/\sigma}^{\infty}  \int_{-g/\sigma - \alpha_1}^{-g/\sigma+ \alpha_1} \frac{1}{ \chi q (q-1)} \phi(z)dz dF(g) \nonumber \\
&\overset{(b)}{\leq}&  \mathbb{P}(|G| \leq  \sigma \log 1/\sigma) \frac{2\alpha_1 \phi(0)}{q(q-1)\sigma^{2q-2}}+ \frac{2\alpha_1 \phi( \log 1/\sigma- \alpha_1)}{q(q-1)\sigma^{2q-2}} \nonumber \\
&\leq& O(1)\cdot \sigma^c\log(1/\sigma) \frac{\alpha_1}{\sigma^{2q-3+c}}+ O(1)\cdot \frac{\alpha_1 \phi( (\log 1/\sigma)/2)}{\sigma^{2q-2}} \nonumber \\
& \overset{(c)}{\rightarrow}& 0, {\rm ~as~} \sigma \rightarrow 0.   \label{general:g1}
\end{eqnarray*}
To obtain $(b)$, we have used the following inequalities when $\sigma$ is small:
\begin{eqnarray*}
\int_{-g/\sigma - \alpha_1}^{-g/\sigma+ \alpha_1}  \phi(z)dz &\leq& 2\phi(0) \alpha_1,  \ \ \ {\rm for \ } g \leq \sigma \log(1/\sigma),   \nonumber \\
\int_{-g/\sigma - \alpha_1}^{-g/\sigma+ \alpha_1}  \phi(z)dz &\leq& 2\alpha_1 \phi(\log(1/\sigma)- \alpha_1),  \ \ \ {\rm for \ } g> \sigma \log(1/\sigma).
\end{eqnarray*}
The limit $(c)$ holds due to the choice of $\alpha_1$. Regarding the second term $G_2$,
\begin{eqnarray*}
G_2&\leq& \frac{\chi}{\sigma^{2q-2}} \int_{0}^\infty  \int_{z \in \mathcal{I}_2(g/ \sigma) \backslash  \mathcal{I}_1(g/ \sigma)} \frac{1}{|\eta_q(g/\sigma +z; \chi)|^{2-q} } \phi(z)dz dF(g)  \nonumber \\
&=&  \frac{\chi}{\sigma^{2q-2}} \int_{0}^{ \sigma \log 1/\sigma}  \int_{z \in \mathcal{I}_2(g/ \sigma) \backslash  \mathcal{I}_1(g/ \sigma)} \frac{1}{|\eta_q(g/\sigma +z; \chi)|^{2-q} } \phi(z)dz dF(g) \nonumber \\
&&+ \frac{\chi}{\sigma^{2q-2}} \int_{ \sigma \log 1/\sigma}^\infty  \int_{z \in \mathcal{I}_2(g/ \sigma) \backslash  \mathcal{I}_1(g/ \sigma)} \frac{1}{|\eta_q(g/\sigma +z; \chi)|^{2-q} } \phi(z)dz dF(g) \nonumber \\
&\overset{(d)}{\leq}&  \frac{\chi}{\sigma^{2q-2}} \int_{0}^{ \sigma \log 1/\sigma}  \int_{z \in \mathcal{I}_2(g/ \sigma) \backslash  \mathcal{I}_1(g/ \sigma)} \frac{1}{|\eta_q(\alpha_1; \chi)|^{2-q} } \phi(z)dz dF(g) \nonumber \\
&&+ \frac{\chi}{\sigma^{2q-2}} \int_{ \sigma \log 1/\sigma}^\infty  \int_{z \in \mathcal{I}_2(g/ \sigma) \backslash  \mathcal{I}_1(g/ \sigma)} \frac{1}{|\eta_q(\alpha_1; \chi)|^{2-q} } \phi(z)dz dF(g) \nonumber \\
&\overset{(e)}{\leq}& \mathbb{P}(|G| \leq \sigma \log 1/\sigma) \frac{2\alpha_2\phi(0)\chi }{\sigma^{2q-2} |\eta_q(\alpha_1; \chi)|^{2-q} } + \frac{2\alpha_2\phi(\log 1/\sigma- \alpha_2 )  \chi }{\sigma^{2q-2}|\eta_q(\alpha_1; \chi)|^{2-q} }  \nonumber \\
& \overset{(f)}{\leq} & O(1)\cdot \sigma^c\log(1/\sigma)\frac{1 }{ \alpha_1^{2-q} \sigma^{c-1}}+O(1)\cdot \frac{\phi((\log1/\sigma)/2)}{\alpha_1^{2-q}}\overset{(g)}{\rightarrow} 0, {\rm ~as~} \sigma \rightarrow 0,  \label{general:g2} \hspace{-1.8cm}
\end{eqnarray*}
where $(d)$ is due to the fact $|\eta_q(g/\sigma+z;\chi)| \geq \eta_q(\alpha_1;\chi)$ for $z \notin \mathcal{I}_1(g/\sigma)$; The argument for $(e)$ is similar to that for $(b)$; $(f)$ holds based on two facts:
\begin{enumerate}
\item $\lim_{\sigma\rightarrow 0}\frac{\eta_q(\alpha_1;\chi)}{\alpha_1}=\lim_{\sigma \rightarrow 0}\eta_q(1;\alpha_1^{q-2}\chi)=1$, since $\alpha_1^{q-2}\chi\rightarrow 0$. This is obvious for the case $\alpha_1=1$. When $\alpha_1=\sigma^{2q-3+c}$, we have $\alpha_1^{q-2}\chi=O(1) \cdot \sigma^{2q^2+(c-5)q+4-2c}$ and $2q^2+(c-5)q+4-2c>0$ if $c$ is chosen small enough. 
\item $\mathbb{P}(|G| \leq \sigma \log 1/\sigma) = O(\sigma \log 1/\sigma)$. This is one of the conditions.
\end{enumerate}
And finally $(g)$ works as follows: it is clear that $\sigma^c\log(1/\sigma)\frac{1 }{ \alpha_1^{2-q} \sigma^{c-1}}$ goes to zero when $\alpha_1=1$; when $\alpha_1=\sigma^{2q-3+c}$, we can sufficiently small $c$ such that $\alpha_1^{q-2}\sigma^{1-c}=\sigma^{2q^2+(c-7)q+7-3c} =o(1)$. For the third integral $G_3$, we are able to invoke DCT to obtain
\begin{eqnarray*}\label{eq:lasttermqless2rightorder}
\lim_{\sigma \rightarrow 0}G_3=O(1) \cdot \lim_{\sigma \rightarrow 0}\mathbb{E}\frac{\mathbbm{I}(|G/\sigma+Z|>\alpha_2)}{ |\eta_q (G/ \sigma +Z ; \chi)|^{2-q} + \chi q (q-1) }=0.  \label{general:g3} \hspace{-1.5cm}
\end{eqnarray*}
Note that DCT works because for small $\sigma$
\begin{eqnarray*}
&&\frac{\mathbbm{I}(|G/\sigma+Z|>\alpha_2)}{ |\eta_q (G/ \sigma +Z ; \chi)|^{2-q} + \chi q (q-1) } \\
&\leq& \frac{\mathbbm{I}(|G/\sigma+Z|>\alpha_2)}{ |\eta_q (G/ \sigma +Z ; \chi)|^{2-q} } \leq   \frac{1}{ |\eta_q (\alpha_2 ; \chi)|^{2-q} } \leq  \frac{1}{ |\eta_q (\alpha_2 ; 1)|^{2-q} }.
\end{eqnarray*}
We hence have showned that
\begin{equation}\label{eq:R3calcugeneraldistqless2}
\lim_{\sigma\rightarrow 0}\frac{R_3}{\sigma^{2q-2}}=0.
\end{equation}
Combining the results \eqref{eq:riskR1}, \eqref{eq:riskR2}, and \eqref{eq:R3calcugeneraldistqless2} establishes \eqref{eq:riskatrightorder}. 

To prove \eqref{eq:wrongorder}, first note that \eqref{eq:R3calcugeneraldistqless2} has been derived in the general setting $\chi(\sigma) = O(\sigma^{2q-2})$. Moreover, we can use similar arguments to show for $\chi(\sigma) = o(\sigma^{2q-2})$,
\begin{eqnarray}
\lim_{\sigma \rightarrow 0} \frac{R_1}{\sigma^{2q-2}} = \lim_{\sigma \rightarrow 0} \frac{R_2}{\sigma^{2q-2}} = 0. \nonumber
\end{eqnarray}
This completes the proof.
\end{proof}

We are in position to derive the convergence rate of $R_q(\chi^*_q(\sigma),\sigma)$.

\begin{lemma} \label{cor:upperbountau*}
Suppose $\mathbb{P}(|G|\leq t)=O(t)$ (as $t\rightarrow 0$) and $\mathbb{E}|G|^2<\infty$. Then for $q \in (1,2)$ we have
\[
\lim_{\sigma \rightarrow 0} \frac{R_q(\chi^*_q(\sigma), \sigma) -1}{\sigma^{2q-2}} = - \frac{(1- \epsilon)^2(\mathbb{E} |Z|^q)^2}{\epsilon  \mathbb{E} |G|^{2q-2}}. 
\]
\end{lemma}

\begin{proof}
According to Lemma \ref{thm:riskell_qgeneralqless2}, choosing $C = \frac{(1-\epsilon) \mathbb{E} |Z|^q}{\epsilon q \mathbb{E} |G|^{2q-2}}$ in Lemma \ref{general:lemma} finishes the proof. 
\end{proof}

\subsection{Deriving the expansion of AMSE$(\lambda_{*,q},q,\sigma_w)$}\label{sec:AMSEellqgeneraldistepsless1}

According to Corollary \ref{thm:mseoptimalasymptot} we know
\begin{eqnarray}\label{amse:revisit}
{\rm AMSE}(\lambda_{*,q},q,\sigma_w)=\bar{\sigma}^2 \cdot R_q(\chi^*_q(\bar{\sigma}),\bar{\sigma}),
\end{eqnarray}
where $\bar{\sigma}$ satisfies the following equation:
\begin{eqnarray}\label{fixedpoint:revisit}
\bar{\sigma}^2=\sigma^2_w+\frac{\bar{\sigma}^2}{\delta}R_q(\chi^*_q(\bar{\sigma}),\bar{\sigma}).
\end{eqnarray}
Since $\chi^*_q(\bar{\sigma})$ minimizes $R_q(\chi,\bar{\sigma})$ we obtain
\[
R_q(\chi^*_q(\bar{\sigma}), \bar{\sigma}) \leq R_q(0, \bar{\sigma}) =1. 
\]
Therefore, under the condition $\delta >1$,
\begin{equation}
\bar{\sigma}^2 \leq \sigma_w^2 + \frac{1}{\delta} \bar{\sigma}^2 \Rightarrow \bar{\sigma}^2 \leq \frac{\delta}{\delta-1} \sigma_w^2,
\end{equation}
which implies that $\bar{\sigma} \rightarrow 0$ as $\sigma_w \rightarrow 0$. Accordingly, we combine Equation \eqref{fixedpoint:revisit} with the fact $R_q(\chi^*_q(\bar{\sigma}),\bar{\sigma}) \rightarrow 1$ as $\bar{\sigma}\rightarrow 0$ from Lemma \ref{cor:upperbountau*} to conclude 
\begin{eqnarray}\label{finalration}
\lim_{\sigma_w\rightarrow 0}\frac{\sigma_w^2}{\bar{\sigma}^2}=\lim_{\bar{\sigma}\rightarrow 0}\frac{\sigma_w^2}{\bar{\sigma}^2}=\lim_{\bar{\sigma}\rightarrow 0}\Big(1-\frac{1}{\delta}R_q(\chi^*_q(\bar{\sigma}),\bar{\sigma})\Big)= \frac{\delta-1}{\delta}.
\end{eqnarray}
We now derive the expansion of ${\rm AMSE}(\lambda_{*,q},q,\sigma_w)$ presented in \eqref{mse:lp}. From \eqref{amse:revisit} and \eqref{fixedpoint:revisit} we can compute that
\begin{eqnarray*}
&&\hspace{-0.8cm} \frac{{\rm AMSE}(\lambda_{*,q},q,\sigma_w)-\frac{\sigma_w^2}{1-1/\delta}}{\sigma_w^{2q}}=\frac{\bar{\sigma}^2R_q(\chi^*_q(\bar{\sigma}),\bar{\sigma})-\frac{1}{1-1/\delta}\cdot (\bar{\sigma}^2-\frac{\bar{\sigma}^2}{\delta}R_q(\chi^*_q(\bar{\sigma}),\bar{\sigma}))}{\sigma_w^{2q}} \\
&=&\frac{\bar{\sigma}^2\delta(R_q(\chi^*_q(\bar{\sigma}),\bar{\sigma})-1)}{\sigma_w^{2q}(\delta-1)}=\frac{\bar{\sigma}^{2q}}{\sigma_w^{2q}}\cdot \frac{\delta}{\delta-1}\cdot \frac{R_q(\chi^*_q(\bar{\sigma}),\bar{\sigma})-1}{\bar{\sigma}^{2q-2}}.
\end{eqnarray*}
Letting $\sigma_w\rightarrow 0$ on both sides of the above equation and using the results from \eqref{finalration} and Lemma \ref{cor:upperbountau*} completes the proof.


\section{Proof of Theorem 3.2}\label{sec:proofthm3full}

From \eqref{amse:revisit} and \eqref{fixedpoint:revisit} we see that 
\[
{\rm AMSE}(\lambda_{*,q},q,\sigma_w)=\delta(\bar{\sigma}^2-\sigma_w^2).
\]
Hence Theorem \ref{asymp:lqabovept} can be proved by showing 
\[
\lim_{\sigma_w\rightarrow 0}\bar{\sigma}>0.
\]
For that purpose we first prove the following under the condition $\mathbb{E}|G|^2<\infty$.
\begin{eqnarray}\label{exist:limit:one}
\lim_{\sigma \rightarrow 0}R_q(\chi_q^*(\sigma),\sigma)=1.
\end{eqnarray}
When $q=2$, $R_q(\chi^*_q(\sigma);\sigma)$ admits a nice explicit expression and can be easily shown to converge to $1$. For $1<q<2$, since $\chi_q^*(\sigma)$ is the minimizer of $R_q(\chi,\sigma)$ we know
\[
R_q(\chi^*_q(\sigma),\sigma)\leq R_q(0,\sigma)=1,
\]
hence $\limsup_{\sigma \rightarrow 0}R_q(\chi^*_q(\sigma),\sigma)\leq 1$. On the other hand, \eqref{uu1} gives us
\begin{eqnarray*}
R_q(\chi^*_q(\sigma),\sigma)&\geq& (1-\epsilon)\mathbb{E}(\eta_q(Z;\chi^*))^2-\epsilon \\
&&\hspace{-0.4cm} +2\epsilon \mathbb{E}\left(\frac{1}{1+\chi^* q(q-1)|\eta_q(G/\sigma+Z;\chi^*)|^{q-2}}\right)
\end{eqnarray*}
where we have used $\chi^*$ to denote $\chi^*_q(\sigma)$ for simplicity. Based on Lemmas \ref{eq:infinity} and \ref{lem:optimaltaugoestozero}, we can apply Fatou's lemma to the above inequality to obtain $\liminf_{\sigma\rightarrow 0}R_q(\chi^*_q(\sigma),\sigma) \geq 1$. 

Next it is clear that 
\begin{eqnarray} \label{exist:limit:two}
\lim_{\sigma \rightarrow \infty} R_q(\chi^*_q(\sigma),\sigma)\leq \lim_{\sigma \rightarrow \infty} \lim_{\chi\rightarrow \infty} R_q(\chi,\sigma)=0.
\end{eqnarray}
We now consider an arbitrary convergent sequence $\bar{\sigma}_n \rightarrow \sigma^*$. We claim $\sigma^* \neq 0$. Otherwise Equation \eqref{fixedpoint:revisit} tells us
\[
R_q(\chi^*_q(\bar{\sigma}_n),\bar{\sigma}_n) <\delta <1,
\]
and letting $n\rightarrow \infty$ above contradicts \eqref{exist:limit:one}. Now that $\sigma^*>0$ we can take $n\rightarrow \infty$ in \eqref{fixedpoint:revisit} to obtain
\[
R_q(\chi^*_q(\sigma^*),\sigma^*)=\delta <1.
\]
According to Lemma \ref{lem:posfirstlemma}, it is not hard to confirm $R_q(\chi_q^*(\sigma),\sigma)$ is a strictly decreasing and continuous function of $\sigma$. Results \eqref{exist:limit:one} and \eqref{exist:limit:two} then imply that $\sigma^*$ is the unique solution to $R_q(\chi^*_q(\sigma),\sigma)=\delta$. Since this is true for any sequence, we have proved $\lim_{\sigma_w \rightarrow 0}\bar{\sigma}$ exists and larger than zero.


\section{Proof of Theorem 3.4}\label{sec:prooflassogeralepsless1}

\subsection{Roadmap}

Theorem \ref{asymp:sparsel1_2} differs from Theorem \ref{asymp:sparsel1_1} in that the order of the second dominant term of AMSE$(\lambda_{*,1},1,\sigma_w)$ becomes polynomial (ignore the logarithm term) when the distribution of $G$ has mass around zero. However, the proof outline remains the same. We hence stick to the same notations used in the proof of Theorem \ref{asymp:sparsel1_1}. In particular, $R(\chi, \sigma), \chi^*(\sigma)$ represent $R_q(\chi,\sigma),\chi^*_q(\sigma)$ with $q=1$, respectively. We characterize the convergence rate of $\chi^*(\sigma)$ in Section \ref{splitnew:section1}, and bound the convergence rate of $R(\chi^*(\sigma),\sigma)$ in Section \ref{splitnew:section2}. After we characterize $R(\chi^*(\sigma),\sigma)$, the rest of the proof is similar to that in Section \ref{sec:lassofixedpointproof} from the main text. We therefore do not repeat it here.

\subsection{Bounding the convergence rate of $\chi^*(\sigma)$} \label{splitnew:section1}

\begin{lemma}\label{l1:optimal_rate_general}
Suppose $\mathbb{P}(|G| \leq t) = \Theta(t^{\ell})$ with $\ell >0$ (as $t\rightarrow 0$) and $\mathbb{E}|G|^2<\infty$, then for sufficiently small $\sigma$
\[
\alpha_m \sigma^{\ell}  \leq \chi^*(\sigma) - \chi^{**} \leq \beta_m  \sigma^{\ell} \cdot \left( \sqrt{\underbrace{\log \log \ldots \log}_{m\  \rm times} \left(\frac{1}{\sigma}\right)} \right)^{\ell},
\]
where $m> 0$ is an arbitrary integer number, $\alpha_m, \beta_m>0$ are two constants depending on $m$, and $\chi^{**}$ is the unique minimizer of $(1-\epsilon)\mathbb{E} (\eta_1(Z; \chi))^2+ \epsilon(1+ \chi^2)$ over $[0, \infty)$.
\end{lemma}

\begin{proof}
According to Lemma \ref{lassotune}, $\chi^*(\sigma) \rightarrow \chi^{**}$ as $\sigma \rightarrow 0$. To characterize the convergence rate, we follow the same line of proof that we presented for Lemma \ref{l1:optimal_rate} and adopt the same notations. For simplicity we do not detail out the entire proof and instead highlight the differences. The key difference is that neither $e_1$ or $e_2$ are exponentially small in the current setting. We now start by bounding $e_2$. Let $F(g)$ be the distribution function of $|G|$ and define
\[
\log_m(a)\triangleq \underbrace{\log \log \ldots \log}_{m\  \rm times}(a).
\]
 Given an integer $m>0$ and a constant $c>0$, we then have
\begin{eqnarray}
\lefteqn{\mathbb{E}\phi(\chi^*-|G|/\sigma)} \nonumber \\
 &=& \sum_{i=1}^{m-1}  \int_{c\sigma (\log_{m-i+1} (1/\sigma))^{1/2}}^{c\sigma (\log_{m-i} (1/\sigma))^{1/2}} \phi(\chi^*-g/\sigma) dF(g)+   \nonumber \\
&&\int_0^{c \sigma (\log_m(1/\sigma))^{1/2}} \phi(\chi^*-g/\sigma) dF(g)  +  \int_{c\sigma(\log(1/\sigma))^{1/2}}^{\infty} \phi(\chi^*-g/\sigma) dF(g)  \nonumber \\
&\leq& \sum_{i=1}^{m-1} \phi(c(\log_{m-i+1}(1/\sigma))^{1/2} - \chi^*) \cdot \mathbb{P} (|G| \leq c\sigma (\log_{m-i} (1/\sigma))^{1/2})+   \nonumber \\
&& \hspace{1cm} \phi(0) \cdot \mathbb{P} (|G| \leq c\sigma (\log_m (1/\sigma))^{1/2})  + \phi(c(\log(1/\sigma))^{1/2}- \chi^*).   \label{eq:taufirsttermupper1} 
\end{eqnarray}
The condition $\mathbb{P}(|G| \leq t) = \Theta(t^{\ell})$ leads to
\begin{eqnarray}\label{eq:taufirsttermupper2}
\lefteqn{ \sum_{i=1}^{m-1} \phi(c(\log_{m-i+1}(1/\sigma))^{1/2} - \chi^*) \cdot  \mathbb{P} (|G| \leq c\sigma (\log_{m-i} (1/\sigma))^{1/2}) } \nonumber \\
&\leq& \frac{e^{(\chi^*)^2/2}}{\sqrt{2\pi}} \sum_{i=1}^{m-1} {\rm e}^{-(c^2\log_{m-i+1}(1/\sigma))/4 } \cdot \Theta(\sigma^{\ell}(\log_{m-i} (1/\sigma))^{\ell/2}) \nonumber \\
&=&  O(1) \cdot   \sum_{i=1}^{m-1} \sigma^{\ell}(\log_{m-i}(1/\sigma))^{\ell/2-c^2/4}, \nonumber
\end{eqnarray}
where we have used the simple inequality $e^{-(a-b)^2/2} \leq e^{-b^2/4}\cdot e^{a^2/2}$. It is also clear that 
\[
\phi(c(\log(1/\sigma))^{1/2}- \chi^*)\leq \frac{1}{\sqrt{2\pi}}e^{(\chi^*)^2/2}\cdot \sigma^{c^2/4}.
\]
Therefore, by choosing a sufficiently large $c$ we can conclude that the dominant term in \eqref{eq:taufirsttermupper1} is $\phi(0) \mathbb{P} (|G| \leq c\sigma (\log_m (1/\sigma))^{1/2})=\Theta(\sigma^{\ell}(\log_m(1/\sigma))^{\ell/2})$.  Furthermore, choosing a fixed constant $C>0$ we have the following lower bound 
\[
\mathbb{E}\phi(\chi^*+|G|/\sigma) \geq \int_{0}^{C \sigma} \phi(\chi^*+g/\sigma) dF(g) \geq \phi(C+ \chi^*) \cdot \mathbb{P}(|G|\leq C\sigma)=\Theta(\sigma^{\ell}).
\]
Because
\[
\mathbb{E}\phi(\chi^*+|G|/\sigma) \leq  \mathbb{E}\phi(\chi^* \pm G/\sigma) \leq \mathbb{E}\phi(\chi^*-|G|/\sigma),
\]
We are able to derive 
\[
 \Theta(\sigma^{\ell}) \leq \mathbb{E}\phi(\chi^*\pm G/\sigma) \leq \Theta(\sigma^{\ell}(\log_m(1/\sigma))^{\ell/2}). 
 \]
As a result we obtain the bound for $e_2$:
\begin{eqnarray}\label{e_2:bound}
\Theta(\sigma^{\ell}) \leq e_2\leq  \Theta(\sigma^{\ell}(\log_m(1/\sigma))^{\ell/2})
\end{eqnarray}
To bound $e_1$, recall that 
\begin{eqnarray*}
e_1=-\epsilon \mathbb{E}\int_{-G/\sigma-\chi^*}^{-G/\sigma+\chi^*}\phi(z)dz=-2\epsilon \chi^*\mathbb{E}\phi(a\chi^*-G/\sigma),
\end{eqnarray*}
where $|a|\leq 1$ depends on $G$. We can find two positive constants $C_1, C_2>0$ such that for small $\sigma$
\begin{eqnarray*}
C_1\mathbb{E}\phi(\chi^*+|G|/\sigma)\leq \mathbb{E}\phi(a\chi^*-G/\sigma) \leq C_2 \mathbb{E} \phi(\chi^* -|G|/\sigma).
\end{eqnarray*}
Hence we have
\begin{eqnarray}\label{e_1:bound}
\Theta(\sigma^{\ell}) \leq -e_1\leq  \Theta(\sigma^{\ell}(\log_m(1/\sigma))^{\ell/2}).
\end{eqnarray}
Based on the results from \eqref{e_2:bound} and \eqref{e_1:bound} and Equation \eqref{diff}, we can use similar arguments as in the proof of Lemma \ref{l1:optimal_rate} to conclude
\begin{eqnarray*}
\Theta(\sigma^{\ell})\leq  \chi^*(\sigma)-\chi^{**} \leq  \Theta(\sigma^{\ell}(\log_m(1/\sigma))^{\ell/2}).
\end{eqnarray*}
\end{proof}

\subsection{Bounding the convergence rate of $R(\chi^*(\sigma),\sigma)$}\label{splitnew:section2}

\begin{lemma} \label{risk:lq:general}
Suppose $\mathbb{P}(|G| \leq t) = \Theta(t^{\ell})$ with $\ell >0$ (as $t\rightarrow 0$) and $\mathbb{E}|G|^2<\infty$, then for sufficiently small $\sigma$
\[
-\beta_m  \sigma^{\ell} \cdot \left( \sqrt{\underbrace{\log \log \ldots \log}_{m\  \rm times} \left(\frac{1}{\sigma}\right)} \right)^{\ell}
 \leq R(\chi^*(\sigma),\sigma) - M_1(\epsilon) \leq  -\alpha_m \sigma^{\ell},
\]
where $m> 0$ is an arbitrary integer number and $\alpha_m, \beta_m>0$ are two constants depending on $m$.
\end{lemma}
\begin{proof}
We recall the two quantities:
\begin{eqnarray*}
M_1(\epsilon)&=&(1-\epsilon)\mathbb{E} (\eta_1(Z; \chi^{**}))^2+ \epsilon(1+ (\chi^{**})^2) \label{risk:one}\\
R(\chi^*(\sigma),\sigma)&=&(1-\epsilon)\mathbb{E}(\eta_1(Z;\chi^*))^2+\epsilon(1+ \mathbb{E} (\eta_1(G/\sigma +Z; \chi^*) -G/\sigma-Z)^2) \nonumber \\
&&+ 2\epsilon \mathbb{E} Z(\eta_1(G/\sigma +Z; \chi^*) -G/\sigma-Z) \label{risk:two}.
\end{eqnarray*}
Since $\chi^*(\sigma)$ is the minimizer of $R(\chi, \sigma)$,
\begin{eqnarray}
&&R(\chi^*(\sigma),\sigma)-M_1(\epsilon) \leq R(\chi^{**},\sigma)-M_1(\epsilon) \nonumber \\
&=&\epsilon[ \mathbb{E} (\eta_1(G/\sigma +Z; \chi^{**}) -G/\sigma-Z)^2-(\chi^{**})^2] + \nonumber \\
&&2\epsilon \mathbb{E} Z(\eta_1(G/\sigma +Z; \chi^{**}) -G/\sigma-Z)  \nonumber \\
&\overset{(a)}{\leq}& -2\epsilon \mathbb{E}\mathbbm{I}(|G/\sigma+Z|\leq \chi^{**}) \overset{(b)}{\leq} -\Theta(\sigma^{\ell}). \label{risk:general:1}
\end{eqnarray}
To obtain $(a)$, we have used Lemma \ref{lem:steins} and the fact $|\eta_1(u;\chi)-u|\leq \chi$. $(b)$ is due to the similar arguments for bounding $e_1$ in Lemma \ref{l1:optimal_rate_general}. To derive the lower bound for $R(\chi^*(\sigma),\sigma)-M_1(\epsilon)$, we can follow the same reasoning steps as in the proof of Lemma \ref{risk:lq} and utilize the bound we derived for $|\chi^*(\sigma)-\chi^{**}|$ in Lemma \ref{l1:optimal_rate_general}. We will obtain
\begin{eqnarray}
|R(\chi^*(\sigma),\sigma)-M_1(\epsilon)|\leq \Theta(\sigma^{\ell}(\log_m(1/\sigma))^{\ell/2}). \label{risk:general:2}  
\end{eqnarray}
Putting \eqref{risk:general:1} and \eqref{risk:general:2} together completes the proof.
\end{proof}

\section{Proof of Theorem 3.5}\label{sec:proofthm6full}

The proof of Theorem \ref{them:ellabovpt} is essentially the same as that of Theorem \ref{asymp:lqabovept}. We do not repeat the details. Note the key argument $\lim_{\sigma \rightarrow 0}R(\chi^*(\sigma),\sigma)=M_1(\epsilon)$ has been shown in Lemma \ref{lassotune}.


\section{Proof of Theorem 2.1}\label{ssec:App:firstresult}

\subsection{Roadmap of the proof}

This appendix contains the proof of Theorem \ref{thm:eqpseudolip}. The proof for LASSO ($q=1$) has been shown in \cite{BaMo11}. We aim to extend the results to $1<q\leq 2$. We will follow similar proof strategy as the one proposed in  \cite{BaMo11}. However, as will be described later some of the steps are more challenging for $1<q \leq 2$ (and some are easier). Motivated by \cite{BaMo11} we construct an approximate message passing (AMP) algorithm for solving LQLS. We then establish an asymptotic equivalence between the output of AMP and the bridge regression estimates. We finally utilize the existing asymptotic results from AMP framework to prove Theorem \ref{thm:eqpseudolip}. The rest of the material is organized as follows. In Section \ref{sec:one}, we briefly review approximate message passing algorithms and state some relevant results that will be used later in our proof. Section \ref{sec:three} collects two useful results to be applied in the later proof. We describe the main proof steps in Section \ref{sec:two}.

\subsection{Approximate message passing algorithms}\label{sec:one}

For a function $f:\mathbb{R}\rightarrow \mathbb{R}$ and a vector $v \in \mathbb{R}^m$, we use $f(v) \in \mathbb{R}^m$ to denote the vector $(f(v_1),\ldots, f(v_m))$. Recall $\eta_q(u;\chi)$ is the proximal operator for the function $\|\cdot\|^q_q$. We are in the linear regression model setting: $y=X\beta+w$. To estimate $\beta$, we adapt the AMP algorithm in \cite{maleki2010approximate} to generate a sequence of estimates $\beta^t \in \mathbb{R}^p$, based on the following iterations (initialized at $\beta^0=0, z^0=y$):
\begin{eqnarray}
\beta^{t+1}&=&\eta_q(X^Tz^t+\beta^t; \theta_t), \nonumber \\
z^t&=& y- X\beta^t +  \frac{1}{\delta}z^{t-1}\langle \partial_1 \eta_q(X^Tz^{t-1}+\beta^{t-1}; \theta_{t-1}) \rangle, \label{iter:beta}
\end{eqnarray}
where $\langle v \rangle = \frac{1}{p}\sum_{i=1}^p v_i$ denotes the average of a vector's components and $\{\theta_t\}$ is a sequence of tuning parameters specified during the iterations. A remarkable phenomenon about AMP is that the asymptotics of the sequence $\{\beta^t\}$ can be characterized by one dimensional parameter, known as the state of the system. The following theorem clarifies this claim.

\begin{theorem}\label{amp:th}
Let $\{\beta(p),X(p),w(p)\}$ be a converging sequence and $\psi : \mathbb{R}^2\rightarrow \mathbb{R}$ be a pseudo-Lipschitz function. For any iteration number $t>0$,
\begin{eqnarray*}
\lim_{p\rightarrow \infty} \frac{1}{p}\sum_{i=1}^p\psi(\beta^{t+1}_i,\beta_i)=\mathbb{E}[ \psi(\eta_q(B+\tau_tZ;\theta_t),B)], \quad a.s.,
\end{eqnarray*}
where $B\sim p_{\beta}$ and $Z\sim N(0,1)$ are independent and $\{\tau_t\}_{t=0}^{\infty}$ can be tracked through the following recursion ($\tau_0^2=\sigma^2+\frac{1}{\delta}\mathbb{E}|B|^2$): 
\begin{eqnarray}\label{eq:stateevoAMPellq}
\tau^2_{t+1}=\sigma_w^2+\frac{1}{\delta}\mathbb{E}[\eta_q(B+\tau_t Z;\theta_t)-B]^2, ~~t \geq 0. \label{iter:tau}
\end{eqnarray}
\end{theorem}
\begin{proof}
According to Lemma \ref{lem:toobasicpropprox} part (vi), $\eta_q(u;\chi)$ is a Lipschitz continuous function of $u$. We can then directly apply Theorem 1 in \cite{BaMo10} to complete the proof.
\end{proof}
Equation \eqref{eq:stateevoAMPellq} is called state evolution. Theorem \ref{amp:th} demonstrates that the general asymptotic performance of $\{\beta^t\}$ is sharply predicted by the state evolution. From now on, we will consider the AMP estimates $\{\beta^t\}$ with $\theta_t=\chi \tau_t^{2-q}$ in \eqref{iter:beta}. The positive constant $\chi$ is the solution of \eqref{eq:fixedpoint11} and \eqref{eq:fixedpoint21}. Note we have proved in Section \ref{revision:lemma2:unique} that the solution exists. We next present a useful lemma that characterizes the convergence of $\{\tau_t\}$. Recall the definition in \eqref{def:chimin}:
 \[
  \chi_{\min}= \inf \big \{\chi \geq 0: \frac{1}{\delta}\mathbb{E}(\eta^2_q(Z;\chi)) \leq 1  \big \}.
  \]
\begin{lemma}\label{cab:unique}
For any given $\chi \in (\chi_{\min},\infty)$, the sequence $\{\tau_t\}_{t=0}^{\infty}$ generated from \eqref{iter:tau} with $\theta_t=\chi \tau_t^{2-q}$ converges to a finite number as $t\rightarrow \infty$.
\end{lemma}
\begin{proof}
Denote $\mathcal{H}(\tau)=\sigma_w^2+\frac{1}{\delta}\mathbb{E}[\eta_q(B+\tau Z; \chi \tau^{2-q} )-B]^2$. According to Corollary \ref{cor:fixedpointeqanalysis}, we know $\mathcal{H}(\tau)=\tau^2$ has a unique solution. Furthermore, since $\mathcal{H}(0)>0$ and $\mathcal{H}(\tau) < \tau^2 $ when $\tau $ is large enough, it is straightforward to confirm the result stated in the above lemma. 
\end{proof}

Denote $\tau_t \rightarrow \tau_*$ as $t \rightarrow \infty$. Lemma \ref{cab:unique} and \eqref{eq:stateevoAMPellq} together yield
\begin{eqnarray}
\tau^2_{*}=\sigma_w^2+\frac{1}{\delta}\mathbb{E}[\eta_q(B+\tau_* Z; \chi \tau_*^{2-q})-B]^2. 
\end{eqnarray}
This is the same as Equation \eqref{eq:fixedpoint11}. We hence see the connection between AMP estimates and bridge regression. The main part of the proof for Theorem \ref{thm:eqpseudolip} is to rigorously establish such connection. In particular we will show the sequence $\{\beta^t\}$ converges (in certain asymptotic sense) to $\hat{\beta}( \lambda, q)$ as $t\rightarrow \infty$ in Section \ref{sec:two}. Towards that goal, we present the next theorem that shows asymptotic characterization of other quantities in the AMP algorithm.

\begin{theorem}\label{amp:th2}
Define $w_t \triangleq \frac{1}{\delta}\langle \partial_1 \eta_q(X^Tz^{t-1}+\beta^{t-1}; \chi \tau^{2-q}_{t-1}) \rangle $. Under the conditions of Theorem \ref{amp:th}, we have almost surely
\begin{enumerate}
\item[(i)] $\underset{t \rightarrow \infty}{\lim }\underset{p \rightarrow \infty}{\lim } \frac{\|\beta^{t+1}-\beta^t \|_2^2}{p}=0.$
\item[(ii)] $\underset{t \rightarrow \infty}{\lim }\underset{p \rightarrow \infty}{\lim } \frac{\|z^{t+1}-z^t \|_2^2}{p}=0.$
\item[(iii)] $\underset{p \rightarrow \infty}{\lim } \frac{\|z^t\|^2_2}{n}=\tau^2_t.$
\item[(iv)] $\underset{p \rightarrow \infty}{\lim } w_t = \frac{1}{\delta} \mathbb{E} [\partial_1\eta_q(B+\tau_{t-1}Z;\chi \tau_{t-1}^{2-q})],$  where $B, Z$ are the same random variables as in Theorem \ref{amp:th}.
\end{enumerate}
\begin{proof}
All the results for $q=1$ have been derived in \cite{BaMo11}. We here generalize them to the case $1<q\leq 2$. Since the proof is mostly a direct modification of that in \cite{BaMo11}, we only highlight the differences and refer the reader to \cite{BaMo11} for detailed arguments. According to the proof of Lemma 4.3 in \cite{BaMo11}, we have almost surely
\begin{eqnarray*}
&&\lim_{p \rightarrow \infty} \frac{ \|z^{t+1}-z^t \|_2^2}{p}=\lim_{p\rightarrow \infty}\frac{\|\beta^{t+1}-\beta^t\|_2^2}{p} \\
&=& \mathbb{E}[\eta_q(B+Z_t;\tau_{t}^{2-q}\chi)-\eta_q(B+Z_{t-1};\tau_{t-1}^{2-q}\chi)]^2, 
\end{eqnarray*}
where $(Z_t,Z_{t-1})$ is jointly zero-mean gaussian, independent from $B \sim p_{\beta}$, with covariance matrix defined by the recursion (4.13) in \cite{BaMo11}. From Lemma \ref{prox:smooth}, we know $\eta_q(u;\chi)$ is a differentiable function over $(-\infty,+\infty) \times (0, \infty)$. Hence we can apply mean value theorem to obtain
\begin{eqnarray*}
&&\mathbb{E}[\eta_q(B+Z_t;\tau_{t}^{2-q}\chi)-\eta_q(B+Z_{t-1};\tau_{t-1}^{2-q}\chi)]^2 \\
&\leq& \mathbb{E}[\partial_1 \eta_q(a; b)\cdot (Z_t-Z_{t-1})+\partial_2 \eta_q(a;b)\cdot (\tau^{2-q}_t-\tau^{2-q}_{t-1})\chi]^2 \\
&\leq& 2\mathbb{E} [(\partial_1\eta_q(a;b))^2\cdot (Z_t-Z_{t-1})^2]+2\mathbb{E} [(\partial_2\eta_q(a;b))^2\cdot (\tau_t^{2-q}-\tau_{t-1}^{2-q})^2\chi^2] \\
&\overset{(a)}{\leq}& 2\mathbb{E}[(Z_t-Z_{t-1})^2]+2 (\tau_t^{2-q}-\tau_{t-1}^{2-q})^2\chi^2q^2\mathbb{E}|a|^{2q-2},
\end{eqnarray*}
where $(a,b)$ is a point on a line that connects the two points $(B+Z_t, \tau_{t}^{2-q}\chi)$ and $(B+Z_{t-1},\tau_{t-1}^{2-q}\chi)$; we have used Lemma \ref{lem:toobasicpropprox} part (ii) and Lemma \ref{lem:toobasicpropproxder} part (i)(ii) to obtain (a). Note that Lemma \ref{cab:unique} implies the second term on the right hand side of the last inequality goes to zero, as $t\rightarrow \infty$. Regarding the first term, we can follow similar proof steps as for Lemma 5.7 in \cite{BaMo11} to show $\mathbb{E}(Z_t-Z_{t-1})^2 \rightarrow 0$, as $t\rightarrow \infty$. 

The proof of part (iii) is the same as that of Lemma 4.1 in \cite{BaMo11}. We do not repeat the proof here. For (iv), Lemma F.3(b) in \cite{BaMo11} implies the empirical distribution of $\{((X^Tz^{t-1}+\beta^{t-1})_i, \beta_i)\}_{i=1}^p$ converges weakly to the distribution of $(B+\tau_{t-1}Z,B)$. Since the function $J(y,z)\triangleq\partial_1\eta_q(y;\chi\tau^{2-q}_{t-1})$ is bounded and continuous with respect to $(y,z)$ according to Lemma \ref{lem:toobasicpropprox} part (i) and Lemma \ref{prox:smooth}, (iv) follows directly from the Portmanteau theorem.
\end{proof}
\end{theorem}


\subsection{Two useful theorems}\label{sec:three}

In this section, we refer to two useful theorems that have also been applied and cited in \cite{BaMo11}. The first one is regarding the limit of the singular values of random matrices taken from \cite{bai1993limit}.

\begin{theorem}\label{BY:1993}
(Bai and Yin, 1993). Let $X\in \mathbb{R}^{n\times p}$ be a matrix having $i.i.d.$ entries with $\mathbb{E}X_{ij}=0, \mathbb{E}X^2_{ij}=1/n$. Denote by $\sigma_{\max}(X), \sigma_{\min}(X)$ the largest and smallest non-zero singular values of $X$, respectively. If $n/p \rightarrow \delta >0$, as $p \rightarrow \infty$, then
\begin{eqnarray*}
\lim_{p\rightarrow 0}\sigma_{\max}(X) = \frac{1}{\sqrt{\delta}}+1, \quad a.s.,\\
\lim_{p\rightarrow 0}\sigma_{\min}(X)= \Big |\frac{1}{\sqrt{\delta}}-1\Big|, \quad a.s.
\end{eqnarray*}
\end{theorem}

The second theorem establishes the relation between $\ell_1$ and $\ell_2$ norm for vectors from random subspace, showed in \cite{kashin1977diameters}.

\begin{theorem}\label{ka:1977}
(Kashin, 1977). For a given constant $0<v\leq 1$, there exists a universal constant $c_v$ such that for any $p\geq 1$ and a uniformly random subspace $V$ of dimension $p(1-v)$,
\[
\mathbb{P}\Big( \forall \beta \in V : c_v \|\beta\|_2 \leq \frac{1}{\sqrt{p}}\|\beta\|_1\Big)\geq 1- 2^{-p}.
\]
\end{theorem}


\subsection{The main proof steps}\label{sec:two}

As mentioned before we will use similar arguments as the ones shown in \cite{BaMo11}. To avoid redundancy, we will not present all the details and rather emphasize on the differences. We suggest interested readers going over the proof in \cite{BaMo11} before studying this section. Similar to \cite{BaMo11}, we start with a lemma that summarizes several structural properties of LQLS formulation. Define $\mathcal{F}(\beta)\triangleq \frac{1}{2}\|y-X\beta\|_2^2+\lambda \|\beta\|^q_q$. 

\begin{lemma}\label{corner:lemma}
 Suppose $\beta, r \in \mathbb{R}^p$ satisfy the following conditions:
\begin{enumerate}
\item[(i)] $\|r\|_2 \leq c_1 \sqrt{p}$
\item[(ii)] $\mathcal{F}(\beta+r)\leq \mathcal{F}(\beta)$
\item[(iii)] $\|\nabla F(\beta)\|_2 \leq \sqrt{p} \epsilon$
\item[(iv)] $\sup_{0\leq \mu_i\leq 1}\sum_{i=1}^p |\beta_i+\mu_i r_i|^{2-q}\leq p c_2$
\item[(v)] $0<c_3\leq \sigma_{\min}(X)$, where $\sigma_{\min}(X)$ is defined in Theorem \ref{BY:1993}
\item[(vi)] $\|r^{\parallel}\|^2_2\leq c_4 \frac{\|r^{\parallel}\|^2_1}{p}$. The vector $r^{\parallel}\in \mathbb{R}^p$ is the projection of $r$ onto \mbox{ker}(X)\footnote{It is the nullspace of $X$ defined as $\mbox{ker}(X)= \{\beta \in \mathbb{R}^p \mid X\beta=0\}$.}
\end{enumerate}
Then there exists a function $f(\epsilon, c_1, c_2, c_3, c_3, c_4, \lambda, q)$ such that 
\[
\|r\|_2 \leq \sqrt{p}f(\epsilon, c_1, c_2, c_3, c_4, \lambda, q).
\]
Moreover, $f(\epsilon, c_1,c_2, c_3, c_4, \lambda, q) \rightarrow 0$ as $\epsilon \rightarrow 0$. 
\end{lemma}

\begin{proof}
First note that 
\[
\nabla \mathcal{F}(\beta)=-X^T(y-X\beta)+ \lambda  q(|\beta_1|^{q-1}\mbox{sign}(\beta_1),\ldots, |\beta_p|^{q-1}\mbox{sign}(\beta_p))^T.
\]
Combining it with Condition (ii) we have
\begin{eqnarray}
 &0& \geq  \mathcal{F}(\beta+r)-\mathcal{F}(\beta)  \nonumber \\
& &= \frac{1}{2}\|y-X\beta-Xr\|^2_2+\lambda \|\beta+r\|^q_q-\frac{1}{2}\|y-X\beta\|^2_2-\lambda \|\beta\|^q_q \nonumber \\
&&=\frac{1}{2}\|Xr\|^2_2 -r^TX^T(y-X\beta) + \lambda(\|\beta+r\|^q_q-\|\beta\|^q_q) \nonumber \\
&&= \frac{1}{2}\|Xr\|^2_2+r^T\nabla \mathcal{F}(\beta)+\lambda \sum_{i=1}^p (|\beta_i+r_i|^q-|\beta_i|^q-qr_i|\beta_i|^{q-1}\mbox{sign}(\beta_i))  \nonumber \\
\hspace{0.5cm} &&\overset{(a)}{\geq} \frac{1}{2}\|Xr\|^2_2+r^T\nabla \mathcal{F}(\beta)+\frac{\lambda q(q-1)}{2}\sum_{i=1}^p |\beta_i+\mu_ir_i|^{q-2}r_i^2, \label{start:point} 
\end{eqnarray}
where $(a)$ is obtained by Lemma \ref{simple:key} that we will prove shortly and $\{\mu_i\}$ are numbers between $0$ and $1$. Note that we can decompose $r$ as $r=r^{\parallel}+r^{\perp}$ such that $r^{\parallel} \in \mbox{ker}(X), r^{\perp}\in \mbox{ker}(X)^{\perp}$. Accordingly Condition (v) yields $c_3^2 \|r^{\perp}\|^2_2 \leq \|Xr^{\perp}\|^2_2$. This fact combined with Inequality \eqref{start:point} implies
\begin{eqnarray*}
\frac{c_3^2}{2}\|r^{\perp}\|^2_2\leq \frac{1}{2}\|Xr^{\perp}\|^2_2=\frac{1}{2}\|Xr\|^2_2\leq -r^T\nabla \mathcal{F}(\beta)\leq \|r\|_2 \cdot \|\nabla \mathcal{F}(\beta)\|_2 \leq c_1p \epsilon,
\end{eqnarray*}
where the last inequality is derived from Condition (i) and (iii). Hence we can obtain 
\[
\|r^{\perp}\|^2_2\leq \frac{2c_1p \epsilon}{c^2_3}.
\]
 Our next step is to bound $\|r^{\parallel}\|_2^2$. By Cauchy-Schwarz inequality we know
\begin{eqnarray}
\sum_{i=1}^p|r_i|&=&\sum_{i=1}^p \sqrt{|\beta_i+\mu_i r_i|^{2-q}} \cdot \sqrt{r_i^2|\beta_i+\mu_ir_i|^{q-2}},  \nonumber \\
&\leq &\sqrt{ \sum_{i=1}^p |\beta_i+\mu_i r_i|^{2-q}} \cdot \sqrt{\sum_{i=1}^p r_i^2|\beta_i+\mu_ir_i|^{q-2}}. \nonumber
\end{eqnarray}
So
\begin{eqnarray} \label{kk:two}
\sum_{i=1}^p r_i^2|\beta_i+\mu_ir_i|^{q-2} \geq  \frac{ \|r\|_1^2}{\sum_{i=1}^p |\beta_i+\mu_ir_i|^{2-q}}.
\end{eqnarray}
Combining Inequality \eqref{start:point} and \eqref{kk:two} gives us
\begin{eqnarray*}
\|r\|_1^2 \leq  \frac{-2r^T\nabla \mathcal{F}(\beta)}{\lambda q(q-1)} \cdot \sum_{i=1}^p |\beta_i+\mu_ir_i|^{2-q} \overset{(b)}{\leq} \frac{2c_1c_2\epsilon}{\lambda q(q-1)}p^2, \label{r:norm}
\end{eqnarray*}
where we have used Conditions (i), (iii), and (iv) to derive $(b)$. Using the upper bounds we obtained for $\|r\|_1^2$ and $\|r^{\perp}\|^2_2$, together with Condition (vi), it is straightforward to verify the following chains of inequalities
\begin{eqnarray*}
\|r^{\parallel}\|^2_2 &\leq& \frac{c_4}{p} \|r^{\parallel}\|^2_1\leq \frac{2c_4}{p}(\| r\|_1^2 + \|r^{\perp}\|^2_1) \leq \frac{2c_4}{p}(\| r\|_1^2+p \|r^{\perp}\|_2^2) \\
&\leq & \frac{2c_4}{p} \cdot \Big(\frac{2c_1c_2\epsilon}{\lambda q(q-1)}p^2+\frac{2c_1\epsilon}{c^2_3} p^2 \Big)=\Big(\frac{4c_1c_2c_4}{\lambda q(q-1)}+\frac{4c_1c_4}{c^2_3} \Big)\cdot \epsilon p.
\end{eqnarray*}
We are finally able to derive
\[
\|r\|^2_2=\|r^{\parallel}\|^2_2+\|r^{\perp}\|^2_2\leq \Big(\frac{4c_1c_2c_4}{\lambda q(q-1)}+\frac{4c_1c_4}{c^2_3}+\frac{2c_1}{c_3^2} \Big)\cdot \epsilon p.
\]
This completes the proof.
\end{proof}

Note that Lemma \ref{corner:lemma} is a non-asymptotic and deterministic result. It sheds light on the behavior of the cost function $\mathcal{F}(\beta)$ around its global minimum. Suppose $\beta+r$ is the global minimizer (a reasonable assumption according to Condition (ii)), and if there is another point $\beta$ having small function value (indicated by its gradient from Condition (iii)), then the distance $\|r\|_2$ between $\beta$ and the optimal solution $\beta+r$ should also be small. This interpretation should not sound surprising, since we already know $\mathcal{F}(\beta)$ is a strictly convex function. However, Lemma \ref{corner:lemma} enables us to characterize this property in a precise way, which is crucial in the high dimensional asymptotic analysis. Based on Lemma \ref{corner:lemma}, we will set $\beta+r=\hat{\beta}(\lambda,q), \beta= \beta^t$ and then verify all the conditions in Lemma \ref{corner:lemma} to conclude $\|r\|_2=\|\hat{\beta}(\lambda,q)-\beta^t\|_2$ is small. In particular that small distance will vanish as $t \rightarrow \infty$, thus establishing the asymptotic equivalence between $\hat{\beta}(\lambda,q)$ and $\beta^t$. We perform the analysis in a sequel of lemmas and Proposition \ref{p:one}.


\begin{lemma}\label{simple:key}
Given a constant $q$ satisfying $1<q\leq 2$, for any $x, r \in \mathbb{R}$, there exists a number $0\leq \mu \leq 1$ such that
\begin{equation}
|x+r|^q-|x|^q-rq|x|^{q-1}{\rm sign}(x)\geq \frac{q(q-1)}{2}|x+\mu r|^{q-2}r^2. \label{eq:one}
\end{equation}
\end{lemma}

\begin{proof}
Denote $f_q(x)=|x|^q$. When $q=2$, since $f_2(x)$ is a smooth function over $(-\infty, +\infty)$, we can apply Taylor's theorem to obtain \eqref{eq:one}. For any $1<q<2$, note that $f''_q(0)=\infty$, hence Taylor's theorem is not applicable to all the values of $x \in \mathbb{R}$. We prove the inequality above in separate cases. First observe that if \eqref{eq:one} holds for any $x>0, r\in\mathbb{R}$, then it is true for any $x<0, r\in \mathbb{R}$ as well. It is also straightforward to confirm that when $x=0$, we can always choose $\mu=1$ to satisfy Inequality \eqref{eq:one} for any $r\in \mathbb{R}$. We therefore focus on the case $x>0, r\in \mathbb{R}$. 
\begin{enumerate}
\item[a.] When $x+r>0$, since $f_q(x)$ is a smooth function over $(0,\infty)$, we can apply Taylor's theorem to obtain \eqref{eq:one}.
\item[b.] If $x+r=0$, choosing $\mu=0$, Inequality \eqref{eq:one} is simplified to $(q-1)x^q\geq \frac{q(q-1)}{2}x^q$, which is clearly valid. 
\item[c.] When $x+r<0$, we consider two different scenarios. 
\begin{enumerate}
\item[i.]First suppose $-x-r\geq x$. We apply \eqref{eq:one} to the pair $-r-x$ and $x$. Then we know there exists $0 \leq \tilde{\mu}\leq 1$ such that 
\begin{eqnarray*}
&&|x+r|^q-|x|^q\geq \frac{q(q-1)}{2}|\tilde{\mu}(-x-r)+(1-\tilde{\mu})x|^{q-2}(2x+r)^2 \\
&&\hspace{3.5cm} -(2x+r)q|x|^{q-1}
\end{eqnarray*}
 It is also straightforward to verify that there is $0\leq \mu \leq 1$ so that $\mu(x+r)+(1-\mu)x=-\tilde{\mu}(-x-r)-(1-\tilde{\mu})x$. Denote 
 \[
 g(y)=\frac{q(q-1)}{2}|\tilde{\mu}(-x-r)+(1-\tilde{\mu})x|^{q-2}y^2+q|x|^{q-1}y. 
 \]
If we can show $g(-2x-r)\geq g(r)$, we can obtain the Inequality \eqref{eq:one}. It is easily seen that the quadratic function $g(y)$ achieves global minimum at 
\[
y_0=\frac{-1}{q-1}|x|^{q-1}|\tilde{\mu}(-x-r)+(1-\tilde{\mu})x|^{2-q}\leq \frac{-1}{q-1}|x|<-x. 
\]
Moreover, note that $-2x-r\geq 0, r<0$ and they are symmetric around $y=-x$, hence $g(-2x-r)\geq g(r)$.
\item[ii.] Consider $0<-x-r<x$. We again use \eqref{eq:one} for the pair $-x--r$ and $x$ to obtain 
\begin{eqnarray*}
&&\hspace{-0.7cm} |x+r|^q-|x|^q\geq \frac{q(q-1)}{2}|\tilde{\mu}(x+r)-(1-\tilde{\mu})x|^{q-2}(2x+r)^2\\
&&\hspace{-0.5cm} -(2x+r)q|x|^{q-1} \geq (-2x-r)q|x|^{q-1}+\frac{q(q-1)}{2}|x|^{q-2}(2x+r)^2. 
\end{eqnarray*}
Denote $h(y)=\frac{q(q-1)}{2}|x|^{q-2}y^2+q|x|^{q-1}y$. If we can show $h(-2x-r)\geq h(r)$, Inequality \eqref{eq:one} will be established with $\mu=0$. Since $h(x)$ achieves global minimum at $y_0=\frac{-1}{q-1}|x|<-x$ and $-2x-r>r$, we can get $h(-2x-r)\geq h(r)$. 
\end{enumerate}
\end{enumerate}
\end{proof}

The next lemma is similar to Lemma 3.2 in \cite{BaMo11}. The proof is adapted from there.

\begin{lemma}\label{verify:1}
Let $\{\beta(p), X(p),w(p)\}$ be a converging sequence. Denote the solution of LQLS by $\hat{\beta}(\lambda, q)$, and let $\{\beta^t\}_{t\geq0}$ be the sequence of estimates generated from the AMP algorithm. There exists a positive constant $C$ s.t.
\begin{eqnarray*}
&&\lim_{t\rightarrow \infty}\lim_{p\rightarrow \infty} \frac{\|\beta^t\|^2_2}{p}\leq C, \quad a.s., \\
&&\limsup_{p\rightarrow \infty} \frac{\|\hat{\beta}(\lambda,q)\|^2}{p} \leq C, \quad a.s.
\end{eqnarray*}
\end{lemma}
\begin{proof}
To show the first inequality, according to Theorem \ref{amp:th} and Lemma \ref{cab:unique}, choosing a particular pseudo-Lipschitz function $\psi(x,y)=x^2$ we obtain
\[
\lim_{t\rightarrow \infty} \lim_{p\rightarrow \infty}\frac{\|\beta^t\|^2_2}{p}=\mathbb{E}_{B,Z}[\eta_q(B+\tau_* Z;\chi \tau_*^{2-q})]^2<\infty , \quad a.s.,
\]
where $B \sim p_{\beta}$ and $Z\sim N(0,1)$ are independent. For the second inequality, first note that since $\hat{\beta}(\lambda,q)$ is the optimal solution we have
\begin{eqnarray}
&&\lambda \|\hat{\beta}(\lambda,q)\|^q_q \leq \mathcal{F}(\hat{\beta}(\lambda,q))\leq \mathcal{F}(0) =\frac{1}{2}\|y\|_2^2\nonumber \\
&=& \frac{1}{2}\|X\beta+w\|_2^2 \leq \|X\beta\|^2_2+\|w\|^2_2\leq  \sigma^2_{\max}(X) \|\beta\|_2^2+\|w\|^2_2. \label{chain:one}
\end{eqnarray}
We then consider the decomposition $\hat{\beta}(\lambda,q)=\hat{\beta}(\lambda,q)^{\perp}+\hat{\beta}(\lambda,q)^{\parallel}$, where $\hat{\beta}(\lambda,q)^{\perp}\in \mbox{ker}(X)^{\perp}$ and $\hat{\beta}(\lambda,q)^{\parallel}\in \mbox{ker}(X)$. Since $\mbox{ker}(X)$ is a uniformly random subspace with dimension $p(1-\delta)_+$, we can apply Theorem \ref{ka:1977} to conclude that, there exists a constant $c(\delta)>0$ depending on $\delta$ such that the following holds with high probability,
\begin{eqnarray}
\|\hat{\beta}(\lambda,q)\|_2^2&=&\|\hat{\beta}(\lambda,q)^{\parallel}\|_2^2+\|\hat{\beta}(\lambda,q)^{\perp}\|_2^2 \leq \frac{\|\hat{\beta}(\lambda,q)^{\parallel}\|^2_1}{c(\delta)p}+\|\hat{\beta}(\lambda,q)^{\perp}\|_2^2  \nonumber \\
&\leq& \frac{2\|\hat{\beta}(\lambda,q)^{\perp}\|^2_1+2\|\hat{\beta}(\lambda,q)\|^2_1}{c(\delta)p}+\|\hat{\beta}(\lambda,q)^{\perp}\|_2^2 \nonumber \\
& \leq &\frac{2\|\hat{\beta}(\lambda,q)\|^2_1}{c(\delta)p}+\frac{2+c(\delta)}{c(\delta)}\|\hat{\beta}(\lambda,q)^{\perp}\|_2^2 .  \label{chain:two}
\end{eqnarray}
Moreover, H$\ddot{o}$lder's inequality combined with Inequality \eqref{chain:one} yields
\begin{eqnarray}
\frac{\|\hat{\beta}(\lambda,q)\|_1}{p} \leq \Bigg(\frac{\|\hat{\beta}(\lambda,q)\|^q_q}{p}\Bigg)^{1/q}\leq \Bigg ( \frac{\sigma^2_{\max}(X)\|\beta\|_2^2+\|w\|^2_2}{\lambda p}  \Bigg )^{1/q}.  \label{chain:three} \hspace{-1.5cm}
\end{eqnarray}
Using the results from \eqref{chain:two} and \eqref{chain:three}, we can then upper bound $\|\hat{\beta}(\lambda,q)\|_2^2$:
\begin{eqnarray*}
 \frac{\|\hat{\beta}(\lambda,q)\|_2^2}{p} \overset{(a)}{\leq} \frac{2}{c(\delta)}\Bigg ( \frac{\sigma^2_{\max}(X) \|\beta\|_2^2+\|w\|^2_2}{\lambda p}  \Bigg )^{2/q}+\frac{2+c(\delta)}{pc(\delta)\sigma^2_{\min}(X)}\|X\hat{\beta}(\lambda,q)^{\perp} \|^2_2.
\end{eqnarray*}
To obtain $(a)$ we have used the fact $\|X\hat{\beta}(\lambda,q)^{\perp} \|^2_2\geq \sigma^2_{\min}(X)\|\hat{\beta}(\lambda,q)^{\perp} \|^2_2$. We can further bound 
\begin{eqnarray*}
\|X \hat{\beta}(\lambda,q)^{\perp} \|^2_2 \overset{(b)}{\leq} 2\|y-X\hat{\beta}(\lambda,q) \|_2^2+2\|y\|_2^2 \overset{(c)}{\leq} 4\|y\|^2_2  \overset{(d)}{\leq} 8\sigma^2_{\max}(X) \|\beta\|_2^2+ 8\|w\|^2_2.
\end{eqnarray*}
$(b)$ is due to the simple fact $X\hat{\beta}(\lambda,q)^{\perp} =X\hat{\beta}(\lambda,q)$; $(c)$ and $(d)$ hold since $\|y-X\hat{\beta}(\lambda,q) \|^2_2\leq 2\mathcal{F}(\hat{\beta}(\lambda,q))$ and inequalities in \eqref{chain:one}. Combining the last two chains of inequalities we obtain with probability larger than $1-2^{-p}$
\begin{eqnarray*}
&&\hspace{-1.5cm} \frac{\|\hat{\beta}(\lambda,q)\|_2^2}{p} \leq \frac{2}{c(\delta)}\Bigg ( \frac{\sigma^2_{\max}(X) \|\beta\|_2^2+\|w\|^2_2}{\lambda p}  \Bigg )^{2/q}  \\
 &&\hspace{0.8cm} + \frac{16+ 8c(\delta)}{c(\delta)\sigma^2_{\min}(X)} \cdot \frac{\sigma^2_{\max}(X) \|\beta\|_2^2+ \|w\|^2_2}{p}.
\end{eqnarray*}

Finally, because both $\sigma_{\min}(X)$ and $\sigma_{\max}(X)$ converge to non-zero constants by Theorem \ref{BY:1993} and $(\beta, X, w)$ is a converging sequence, the right hand side of the above inequality converges to a finite number. 
\end{proof}

\begin{lemma}\label{verify:4}
Let $\{\beta(p), X(p),w(p)\}$ be a converging sequence. Denote the solution of LQLS by $\hat{\beta}(\lambda, q)$, and let $\{\beta^t\}_{t\geq0}$ be the sequence of estimates generated from the AMP algorithm.There exists a positive constant $\tilde{C}$ s.t.
\[
\limsup_{t\rightarrow \infty}\limsup_{p \rightarrow \infty} \sup_{0\leq \mu_i \leq 1} \frac{\sum_{i=1}^p |\mu_i \hat{\beta}_i(\lambda, q)+(1-\mu_i)\beta^t_i|^{2-q} }{p} < \tilde{C}, \quad a.s.
\]
\end{lemma}

\begin{proof}
For any given $0\leq \mu_i \leq 1$, it is straightforward to see
\[
|\mu_i \hat{\beta}_i(\lambda, q)+(1-\mu_i)\beta^t_i|^{2-q} \leq \max \{| \hat{\beta}_i(\lambda, q)|^{2-q}, |\beta^t_i|^{2-q} \}\leq | \hat{\beta}_i(\lambda, q)|^{2-q}+ |\beta^t_i|^{2-q}.
\]
Hence  using H$\ddot{o}$lder's inequality gives us
\begin{eqnarray*}
&&\sup_{0\leq \mu_i \leq 1}  \frac{1}{p}\sum_{i=1}^p|\mu_i \hat{\beta}_i(\lambda, q)+(1-\mu_i)\beta^t_i|^{2-q} \\
 &\leq& \frac{1}{p}\sum_{i=1}^p| \hat{\beta}_i(\lambda, q)|^{2-q}+\frac{1}{p}\sum_{i=1}^p|\beta^t_i|^{2-q} \leq  \Big (\frac{1}{p}\sum_{i=1}^p| \hat{\beta}_i(\lambda, q)|^2 \Big )^{\frac{2-q}{2}}+\Big (\frac{1}{p}\sum_{i=1}^p|\beta^t_i|^2 \Big )^{\frac{2-q}{2}}.
\end{eqnarray*}
Applying Lemma \ref{verify:1} to the above inequality finishes the proof.
\end{proof}

The next lemma is similar to Lemma 3.3 in \cite{BaMo11}. The proof is adapted from there.

\begin{lemma}\label{verify:3}
Let $\{\beta(p), X(p),w(p)\}$ be a converging sequence. Let $\{\beta^t\}_{t\geq0}$ be the sequence of estimates generated from the AMP algorithm. We have
\[
\lim_{t \rightarrow \infty} \lim_{p \rightarrow \infty} \frac{\|\nabla \mathcal{F}(\beta^t)\|^2_2}{p}=0, \quad a.s.
\]
\end{lemma}

\begin{proof}
Recall the AMP updating rule \eqref{iter:beta}:
\[
\beta^{t}=\eta_q(X^Tz^{t-1}+\beta^{t-1};\tau^{2-q}_{t-1}\chi).
\]
According to Lemma \ref{lem:toobasicpropprox} part (i) we know $\beta^t$ satisfies
\[
X^Tz^{t-1}+\beta^{t-1} =\beta^t+\tau^{2-q}_{t-1}\chi q (|\beta^t_1|^{q-1}\mbox{sign}(\beta^t_1),\ldots, |\beta^t_p|^{q-1}\mbox{sign}(\beta^t_p))^T.
\]
The rule \eqref{iter:beta} also tells us
\[
 z^t=y-X\beta^t+w_tz^{t-1},
 \]
 where $w_t$ is defined in Theorem \ref{amp:th2}. Note
 \[
\nabla \mathcal{F}(\beta^t)=-X^T(y-X\beta^t)+\lambda q (|\beta^t_1|^{q-1}\mbox{sign}(\beta^t_1),\ldots, |\beta^t_p|^{q-1}\mbox{sign}(\beta^t_p))^T.
\]
We can then upper bound $\nabla \mathcal{F}(\beta^t)$ in the following way:
\begin{eqnarray*}
&&\hspace{-0.8cm} \frac{1}{\sqrt{p}}\|\nabla \mathcal{F}(\beta^t)\|_2=\frac{1}{\sqrt{p}} \|-X^T(y-X\beta^t)+\lambda (\tau^{2-q}_{t-1}\chi)^{-1}X^Tz^{t-1}+ \beta^{t-1}-\beta^t  \|_2 \\
&&=\frac{1}{\sqrt{p}} \| -X^T(z^t-w_tz^{t-1})+  \lambda(\tau^{2-q}_{t-1}\chi)^{-1} (X^Tz^{t-1}+ \beta^{t-1}-\beta^t)   \|_2 \\
&& \hspace{-0.4cm} \leq \frac{\lambda \|\beta^{t-1}-\beta^t\|_2}{\tau^{2-q}_{t-1}\chi \sqrt{p}}+\frac{\|X^T(z^{t-1}-z^t)\|_2}{\sqrt{p}}+\frac{|\lambda+\tau^{2-q}_{t-1}\chi (w_t-1)|\cdot \|X^Tz^{t-1}\|_2}{\tau^{2-q}_{t-1}\chi \sqrt{p}} \\
&& \hspace{-0.4cm} \leq   \frac{\lambda \|\beta^{t-1}-\beta^t\|_2}{\tau^{2-q}_{t-1}\chi \sqrt{p}}+\frac{\sigma_{\max}(X) \|z^{t-1}-z^t\|_2}{\sqrt{p}}+\frac{\sigma_{\max}(X) |\lambda+\tau^{2-q}_{t-1}\chi(w_t-1)|\cdot \|z^{t-1}\|_2}{\tau^{2-q}_{t-1}\chi \sqrt{p}}.
\end{eqnarray*}
By Lemma \ref{cab:unique}, Theorem \ref{amp:th2} part (i)(ii) and Theorem \ref{BY:1993}, it is straightforward to confirm that the first two terms on the right hand side of the last inequality vanish almost surely, as $p\rightarrow \infty, t\rightarrow \infty$. For the third term, Lemma \ref{cab:unique} and Theorem \ref{amp:th2} part (iii)(iv) imply
\begin{eqnarray*}
&& \lim_{t\rightarrow \infty }\lim_{p \rightarrow \infty}\frac{ |\lambda+\tau^{2-q}_{t-1}\chi (w_t-1)|\cdot \|z^{t-1}\|_2}{\tau^{2-q}_{t-1}\chi \sqrt{p}} \\
&=&\frac{\sqrt{\delta}\tau_*}{\tau_*^{2-q}\chi}\Big |\lambda-\tau^{2-q}_*\chi \Big(1-\frac{1}{\delta}\mathbb{E}\eta'_q(B+\tau_*Z;\tau_*^{2-q}\chi)\Big) \Big |=0, \quad a.s.
\end{eqnarray*}
To obtain the last equality, we have used Equation \eqref{eq:fixedpoint21}. 
\end{proof}


We are in position to prove the asymptotic equivalence between AMP estimates and bridge regression.

\begin{proposition}\label{p:one}
Let $\{\beta(p), X(p),w(p)\}$ be a converging sequence. Denote the solution of LQLS by $\hat{\beta}(\lambda, q)$, and let $\{\beta^t\}_{t\geq0}$ be the sequence of estimates generated from the AMP algorithm. We then have
\begin{eqnarray}\label{lqls:amp}
\lim_{t\rightarrow \infty}\lim_{p\rightarrow \infty}\frac{1}{p}\|\hat{\beta}(\lambda,q)-\beta^t\|^2_2=0, \quad a.s.
\end{eqnarray}
\end{proposition}
\begin{proof}
We utilize Lemma \ref{corner:lemma}. Let $\beta+r=\hat{\beta}(\lambda,q), \beta= \beta^t$. If this pair of $\beta$ and $r$ satisfies the conditions in Lemma \ref{corner:lemma}, we will have $\frac{\|r\|_2^2}{p}=\frac{\|\hat{\beta}(\lambda,q)-\beta^t\|_2}{p}$ being very small. In the rest of the proof, we aim to verify that the conditions in Lemma \ref{corner:lemma} hold with high probability and establish the connection between the iteration numbers $t$ and $\epsilon$ in Lemma \ref{corner:lemma}.
\begin{enumerate}
\item[a.] Condition (i) follows from Lemma \ref{verify:1}: 
\[
\underset{t \rightarrow \infty}{\lim }\underset{p \rightarrow \infty}{\limsup }\frac{\|r\|_2}{\sqrt{p}}\leq  \underset{p \rightarrow \infty}{\limsup } \frac{\|\hat{\beta}(\lambda,q)\|_2}{\sqrt{p}}+ \underset{t \rightarrow \infty}{\lim }\underset{p \rightarrow \infty}{\lim } \frac{\|\beta^t\|_2}{\sqrt{p}}\leq 2\sqrt{C}, \quad a.s.
\]
\item[b.] Condition (ii) holds since $\hat{\beta}(\lambda,q)$ is the optimal solution of $\mathcal{F}(\beta)$.
\item[c.] Condition (iii) holds by Lemma \ref{verify:3}. Note that $\epsilon \rightarrow 0$, as $t\rightarrow \infty$.
\item[d.] Condition (iv) is due to Lemma \ref{verify:4}.
\item[e.] Condition (v) is the result of Theorem \ref{BY:1993}.
\item[f.] Condition (vi) is a direct application of Theorem \ref{ka:1977}.
\end{enumerate}
Note all the claims made above hold almost surely as $p \rightarrow \infty$; and $\epsilon \rightarrow 0$ as $t\rightarrow \infty$. Hence the result \eqref{lqls:amp} follows.
\end{proof}

Based on the results from Theorem \ref{amp:th},  Lemma \ref{cab:unique} and Proposition \ref{p:one}, we can use exactly the same arguments as in the proof of Theorem 1.5 from \cite{BaMo11} to finish the proof of Theorem \ref{thm:eqpseudolip}. Since the arguments are straightforward, we do not repeat it here.